\documentclass[letterpaper,11pt,american,english]{article}
\usepackage[left=1.6cm,right=1.6cm,bottom=1.3cm,top=4cm]{geometry}
\usepackage[T1]{fontenc}
\usepackage[latin9]{inputenc}
\usepackage{babel}
\usepackage{bbm} 
\usepackage{amsmath}
\usepackage{mathrsfs}  
\usepackage{amsthm}
\usepackage{amssymb}
\usepackage{setspace}
\usepackage{empheq}
\usepackage{comment}
\usepackage{etoolbox}
\usepackage{xcolor}
\usepackage{lipsum}
\usepackage[most]{tcolorbox}
\usepackage{filecontents} 
\newcommand{\mycomment}[1]{}
\usepackage{lmodern}
\usepackage[fontsize=11pt]{fontsize} 
\usepackage{graphicx}
\usepackage[normalem]{ulem} 
\usepackage[numbers]{natbib}

\makeatletter
\newenvironment{smallalign*}[2]{%
 \fontsize{#1}{#2}\selectfont
 \start@align\@ne\st@rredtrue\m@ne
}{%
 \endalign
}
\makeatother


\newcommand\scalemath[2]{\scalebox{#1}{\mbox{\ensuremath{\displaystyle #2}}}}

\usepackage[bbgreekl]{mathbbol}
\usepackage{amsfonts}
\DeclareSymbolFontAlphabet{\mathbb}{AMSb}
\DeclareSymbolFontAlphabet{\mathbbl}{bbold}

\pdfoutput=1

\usepackage{hyperref}
\hypersetup{colorlinks,allcolors=black}
 
\newcommand{\vertiii}[1]{{\left\vert\kern-0.25ex\left\vert\kern-0.25ex\left\vert #1 
\right\vert\kern-0.25ex\right\vert\kern-0.25ex\right\vert}}

\makeatletter
\newcommand*\bigcdot{\mathpalette\bigcdot@{.43}}
\newcommand*\bigcdot@[2]{\h{1pt}\mathbin{\vcenter{\hbox{\scalebox{#2}{$\m@th#1\bullet$}}}}\h{1pt}}
\makeatother

\makeatletter
\newcommand{\AlignFootnote}[1]{%
  \ifmeasuring@
  \else
    \iffirstchoice@
      \footnote{#1}%
    \fi
  \fi}
\makeatother

\newcommand{\p}{\partial}

\makeatletter
\theoremstyle{plain}
\newtheorem{thm}{\protect\theoremname}[section]
\theoremstyle{plain}
\newtheorem{lem}[thm]{\protect\lemmaname}
\theoremstyle{plain}
\newtheorem{prop}[thm]{\protect\propositionname}
\theoremstyle{thm}
\newtheorem{rem}[thm]{\protect\remarkname}
\theoremstyle{plain}

\theoremstyle{plain}
\newtheorem{definition}{Definition}[section]
\newtheorem{ass}{Assumption}[section]

\usepackage{multirow}

\usepackage{arydshln}

\usepackage{lmodern}
\newcommand{\1}{\mbox{1\hspace{-1mm}I}}
\numberwithin{equation}{section}

\usepackage{calc}
\newcommand{\customfootnotetext}[2]{{
\renewcommand{\thefootnote}{#1}
\footnotetext[0]{#2}}}

\hypersetup{colorlinks=false,
    pdfborder={0 0 0},
}

\usepackage{scalerel,stackengine}
\newcommand\pig[1]{\scalerel*[5.5pt]{\big#1}{ 
\ensurestackMath{\addstackgap[1.15pt]{\big#1}}}}
\newcommand\pigl[1]{\mathopen{\pig{#1}}}
\newcommand\pigr[1]{\mathclose{\pig{#1}}}

\makeatletter
\newcommand{\pushright}[1]{\ifmeasuring@#1\else\omit\hfill$\displaystyle#1$\fi\ignorespaces}
\newcommand{\pushleft}[1]{\ifmeasuring@#1\else\omit$\displaystyle#1$\hfill\fi\ignorespaces}
\makeatother

\newcommand{\h}{\hspace}

\newcommand{\diff}{D} 

\newcommand{\indep}{\perp\!\!\!\perp}

\makeatother

\addto\captionsamerican{\renewcommand{\corollaryname}{Corollary}}
\addto\captionsamerican{\renewcommand{\lemmaname}{Lemma}}
\addto\captionsamerican{\renewcommand{\propositionname}{Proposition}}
\addto\captionsamerican{\renewcommand{\remarkname}{Remark}}
\addto\captionsamerican{\renewcommand{\theoremname}{Theorem}}
\addto\captionsenglish{\renewcommand{\corollaryname}{Corollary}}
\addto\captionsenglish{\renewcommand{\lemmaname}{Lemma}}
\addto\captionsenglish{\renewcommand{\propositionname}{Proposition}}
\addto\captionsenglish{\renewcommand{\remarkname}{Remark}}
\addto\captionsenglish{\renewcommand{\theoremname}{Theorem}}
\providecommand{\corollaryname}{Corollary}
\providecommand{\lemmaname}{Lemma}
\providecommand{\propositionname}{Proposition}
\providecommand{\remarkname}{Remark}
\providecommand{\theoremname}{Theorem}

\begin{document}
\selectlanguage{american}%
\global\long\def\1{\mbox{1\hspace{-1mm}I}}%

\title{Mean Field Type Control Problems, Some Hilbert-space-valued FBSDEs, and Related Equations}

\author{
Alain Bensoussan\textsuperscript{1}\\
International Center for Decision and Risk Analysis,\vspace{-2pt}\\
Naveen Jindal School of Management, University of Texas at Dallas\vspace{5pt}\\
Ho Man Tai\textsuperscript{2}\\
Department of Mathematics and Department of Statistics, \vspace{-2pt}\\
The Chinese University of Hong Kong\vspace{5pt}\\
Sheung Chi Phillip Yam\textsuperscript{3}\\
Department of Statistics, The Chinese University of Hong Kong\\
}
\customfootnotetext{1}{E-mail: axb046100@utdallas.edu}
\customfootnotetext{2}{E-mail: hmtai@link.cuhk.edu.hk}
\customfootnotetext{3}{E-mail: scpyam@sta.cuhk.edu.hk}


\maketitle
\begin{abstract}
\mycomment{
The objective of this article is to provide an equivalence of the theory recently developed in \cite{CDLL19}, but now for mean field type problems, following the approach of control theory on some Hilbert spaces first introduced by the authors in \cite{BGY20}. In the mentioned paper \cite{CDLL19}, the model was more limited and reduced to that of the paper \cite{GAS}. 
On the other hand, we do not work on the torus, but on the Euclidean space $\mathbb{R}^n$. 
}

In this article, we provide an original systematic global-in-time analysis of mean field type control problems on $\mathbb{R}^n$ with generic cost functionals  by the modified approach but not the same, firstly proposed in \cite{BGY20}, as the ``lifting'' idea introduced by P. L. Lions. As an alternative to the recent popular analytical method by tackling the master equation, we resolve the control problem in a certain proper Hilbert subspace of the whole space of $L^2$ random variables, it can be regarded as  tangent space attached at the initial probability measure. The present work also fills the gap of the global-in-time solvability and extends the previous works of \cite{BGY20,BY19} which only dealt with quadratic cost functionals  in control; the problem is linked to the global solvability of the Hilbert-space-valued forward-backward stochastic differential equation (FBSDE), which is solved by variational techniques here. We also rely on the Jacobian flow of the solution to this FBSDE to establish  the regularities of the value function, including its linearly functional differentiability, which leads to the classical well-posedness of the Bellman equation. Together with the linear functional derivatives and the gradient of the linear functional derivatives of the solution to the FBSDE, we also obtain the classical well-posedness of the master equation.
\end{abstract}



\tableofcontents
\section{Introduction}
\subsection{Background and Our Approach}

Every mean field type control problem, also known as a McKean-Vlasov control problem, involves a dynamical system whose state is affected by its own probability measure on $\mathbb{R}^{n}$. Numerous researchers have made significant contributions to the modeling and techniques in the form of books or journal papers. A fundamental model is the linear quadratic setting, which has been investigated by Bensoussan, Sung, Yam and Yung \cite{BSYY16} (also see Bensoussan, Frehse and Yam \cite{BFY13}) and Carmona, Delarue and Lachapelle \cite{CDL13}. They solved both the linear quadratic mean field control problem and the mean field games, where the former handled the $N$-dimensional case and the latter was in the $1$-dimensional case; further studies can be found in Pham \cite{P16}, Yong \cite{Y13}, and other references. Most of these works resolved the problems by relating the control problem to Riccati-type differential equations, which were then solved under certain positive definiteness of the cost functionals.

A natural approach to solving mean field control problems with general cost functionals and state dynamics is to characterize the control problem using the infinite-dimensional Bellman equation. Bensoussan, Frehse, and Yam \cite{BFY13, BFY15, BFY17}, with and without the presence of common noise, first derived and discussed the Bellman equation, as well as the master equation for mean field games, under the non-linear-quadratic setting. Lauri\`ere and Pironneau \cite{LP14} also provided an independent approach in this direction in an earlier work. However, due to the nonlinear dependence of the cost functionals on the mean field term of the state, the state process may be time inconsistent, meaning that the usual Bellman optimality principle does not hold. To address this, several researchers, including Cosso, Gozzi, Kharroubi, Pham, and Rosestolato \cite{CGKPR20, CGKPR21}, Djete, Possama\"i, and Tan \cite{DPT22}, and Pham and Wei \cite{PW17,PW18}, developed the dynamic programming principle in the general setting. For instance, Cosso, Gozzi, Kharroubi, Pham, and Rosestolato \cite{CGKPR20} and Pham and Wei \cite{PW17,PW18} demonstrated the verification theorem and proved that the value function satisfies the Bellman equation in the sense of viscosity. Moreover, Pham and Wei \cite{PW17,PW18} established the uniqueness of the viscosity solution by using the standard comparison principle for viscosity solutions (see \cite{FGS17}). Cosso, Gozzi, Kharroubi, Pham, and Rosestolato \cite{CGKPR21} verified the uniqueness by using the smooth variational principle in Wasserstein space. Wu and Zhang \cite{WZ20} introduced a new definition of viscosity solution, which differs from the Crandall-Lions definition by imposing the maximum/minimum condition on the compact set of semimartingale measures on the Wasserstein space. This modification aims to reduce the difficulty of establishing uniqueness.

Another approach to solving mean field control problems is to extend the stochastic (Pontryagin) maximum principle to the mean field setting, which allows us to specify the necessary conditions for optimality in terms of the corresponding Hamiltonian. However, due to the non-Markovian nature of the state dynamics involving the mean field term, this approach typically requires certain convexity of the Hamiltonian. This direction was first studied by Andersson and Djehiche \cite{AD11}, Li \cite{L12} in the convex action space, and Buckdahn-Djehiche and Li \cite{BDL11} in general settings by using the spike variation. Many researchers have extended the stochastic maximum principle to various settings, including \cite{ABC19,BLM17, CD15}. The sufficient condition for optimality involves the solvability of the forward dynamics together with the adjoint backward dynamics, subject to the first-order condition. This system is called the mean field (or McKean-Vlasov) forward-backward stochastic differential equations (FBSDEs). The mean field stochastic differential equations were studied in Buckdahn, Li, Peng, and Rainer \cite{BLPR17}. The mean field backward SDEs were first investigated in Buckdahn, Djehiche, Li, and Peng \cite{BDLP09} and Buckdahn, Li, and Peng \cite{BLP09}, while the well-posedness of the coupled mean-field FBSDEs was first studied in Carmona and Delarue \cite{CD13} for rather general assumptions, where they did not require convexity of the coefficients but imposed certain boundedness on the coefficients with respect to the state variable. Later, Carmona and Delarue \cite{CD15} removed this boundedness assumption and established the well-posedness of FBSDEs under the framework of linear forward dynamics and convex coefficients of the backward dynamics. Bensoussan, Yam, and Zhang \cite{BYZ15} also addressed the well-posedness of FBSDEs with another set of generic assumptions on the coefficients.

\mycomment{
The pure mean field type control problem is rare in the existing literature. In \cite{BY19}, under the linear quadratic setting, the control problem in the Wasserstein space is reformulated into that in the Hilbert space, with the Bellman equation solved without the use of Wasserstein space. If the smoothness of the coefficient are sufficient, then we can obtain the global well-posedness. The alternative would be to work with the space of probability measures directly. This is the approach considered in Gangbo and {\'S}wi{\k{e}}ch \cite{GS15}. Due to the absence of any local or common noise in \cite{BY19}, the randomness of the dynamics comes solely from the initial condition, thus resulting in a deterministic control problem over the Hilbert space. In fact, we can take any generic Hilbert space. The approach works remarkably well. Both the Bellman equation and the Master equation can be completely resolved. The fact that the Hilbert space is a space of $L^2$ random variables plays only a role at the level of interpretation, especially, when we want to check that we have solved the original problem. }

The objective of this work is to extend the results in our former works \cite{BGY19} and \cite{BGY20}, which considered the mean field control problem with separable cost functionals quadratic in control variable and nonlinear in measure variable, to the case with general cost functionals. Similar to \cite{BGY19} and \cite{BGY20},  there are two types of noises in the present study: the initial probability measures and the independent Wiener process that models the evolving local noise. The state of the system depends on both the initial condition and the local noise; meanwhile, the state of the system depends on both the initial condition and the local noise. Unlike our another former article \cite{BY19}, we cannot simply generalise the ``deterministic case'' dealt with in there, that is, the stochastic control problem cannot be naively handled in an arbitrary Hilbert space. To overcome this hurdle, we separate to study (1) the initial condition of a controlled diffusion process and (2) the evolutionary part of the process from two different perspectives; namely, we consider the initial condition of a probability measure as a parameter, while the evolving part generated by the local noise of Wiener process is treated as an element in a specially designated Hilbert space as first introduced in \cite{BGY20}. The big advantage of this approach is that it remains a control problem in a Hilbert space, which is not the whole space of $L^2$ random variables, but a tangent space-like structure attached to the initial probability measure. However, a more common approach, such as that in \cite{BLPR17} calculates the derivative of the value function with respect to $m$ by using its lifted version on a Hilbert space, and this differs from \cite{BFY17,BGY19,BGY20} and our approach here.  The present article generalises our former works of \cite{BGY19,BGY20} by adopting generic cost functionals , the theory developed in \cite{BGY20} only considers the quadratic running cost functional in control, besides we comprehend \cite{BGY20} with a complete global-in-time analysis; meanwhile, our present work also covers the discussions of the book \cite{CDLL19} now in $\mathbb{R}^n$ but not on a compact torus.

Recall that the inspiring idea of ``lifting'' was first introduced by Lions \cite{L14} in his lectures at Coll\`ege de France (also see Carmona and Delarue \cite{CD18}, Cardialaguet, Delarue, Lasry and Lions \cite{CDLL19}), for the purpose of studying alternatively the Wasserstein derivatives on the space of measures; it reformulates the problem on the Wasserstein metric space of measures to another one but now on a linear space. Instead of finding a gradient flow of measures, one can look for an optimal flow in Hilbert space of random variables, whose laws match the optimal flow of measures; we instead {\color{black} consider} the controlled stochastic processes in a space of random vector fields, each of which is attached to the initial probability measure via a pushforward operation on the space of measures. These processes can be considered tangent to the space of measures at the initial measure $m$. The optimal co-state then gives the gradient of the value function in the same Hilbert space, which then projects down to a derivative on the Wasserstein space of measures by letting the initial random field be the identity map. This {\color{black}Hilbert space approach} turns out to be quite powerful in obtaining the Bellman equation and the master equation of mean field type control problems as demonstrated in the previous work \cite{BGY20} for some special cases. More precisely, we here aim to separate the derivation into two parts: the optimal control on a Hilbert space on the one hand, and the differentiation with respect to the initial distribution on the other, so that we can avoid certain stringent technical assumptions on the cost functionals regarding to the differentiability with respect to the measure argument. Besides, the regularity of the value function on the Hilbert space can be translated to that of the corresponding problem on the Wasserstein metric space, which guarantees the equivalence of solving the master equation in the Hilbert space with that lies in the Wasserstein metric space.

\mycomment{ 
However, if one does not assume that the states' probability measures have densities, the natural function space for the state of the system is the Wasserstein metric space of probability measures on $\mathbb{R}^{n}$, which is the classical approach to tackle the mean field games problem, see \cite{GM22,GS15,WYY22}. Since Wasserstein metric space is not a vector space, the classical methods of control theory are difficult to apply. Using the Hilbert space of $L^2$ random variables simplifies considerably the mathematical technicalities and makes the problem more transparent. The difficulty is to recover to the original problem. For instance, one has to check whether the dependence of the value function with respect to the random variable is only through its probability measure.  However, the difficulty is to recover to the original problem, for instance, one has to check whether the dependence of the value function with respect to the random variable is only through its probability measure.}
\mycomment{
Then we derive an optimality principle of Dynamic Programming over the Wasserstein space and translate it into the Hilbert space of $L^2$ random variables and then we use a viscosity theory argument, which finally leads to a Bellman equation in the Hilbert space, to this end, there is an approach adopted in Pham and Wei \cite{PW17}. This, of course, is not needed when one solves the problem directly on the Wasserstein space.}
\mycomment{
We keep the Wasserstein space of probability measures, but we do not use the concept of the Wasserstein gradient, which seems difficult to extend to the second derivative; instead, we use the idea of functional derivative, which extends to the second order. Using the lifting concept, it is possible to work with the Hilbert space of $L^2$ random variables. The concept of the gradient in the Hilbert space is the Fr\'echet derivative. At this stage, an ``almost'' complete equivalence is possible. However, the situation does not carry over so nicely to second-order derivatives. So we have a second-order Fr\'echet derivative in the Hilbert space and a second-order functional derivative, but not a full equivalence. On the other hand, when both exist, we have formulas to transform one concept into the other. We called these formulas rules of correspondence, see (\ref{eq:2-217}) and (\ref{eq:2-232}). The advantage of the Hilbert space approach is that equations can be written in a more synthetic way. From these equations and the rules of correspondence, the equations with functional derivatives can be written, but their rigorous study requires a direct approach and specific assumptions. This is not fully satisfactory, because our use of the lifting approach was meant precisely to circumvent the direct study of these equations. }

In mean field type control problems under the usual strict convexity setting, a unique optimal control exists, as demonstrated in Proposition \ref{prop. convex of J}, through the use of variational techniques. The standard argument via the maximum principle guarantees the existence of the optimal control and is linked to the solvability of the associated FBSDEs for the dynamics of the state and costate (adjoint processes).  It is worth noting that in mean field games, quite differently, due to the lack of functional analytic arguments warranted in mean field control setting, we have to solve the corresponding FBSDE by gluing consecutive local solutions to obtain a global one, through bounding the   sensitivity of the adjoint process; see the recent work \cite{BTWY22} for the details. Our global results on an arbitrary time horizon rely on certain monotonicity assumptions on the cost functionals with respect to the measure argument and the usual convexity in control of coefficient functions. Some of these assumptions exhibit similarity to the displacement monotonicity concept used in \cite{GM22}, particularly in relation to (\ref{assumption, bdd of F F_T}), (\ref{assumption, bdd of DF, DF_T}), (\ref{assumption, bdd of D^2F, D^2F_T}), and (\ref{assumption, convexity of D^2F, D^2F_T}). Additionally, many of the practical assumptions adopted in Section \ref{sec. assumptions} to establish the classical solvability of the Bellman equation in Proposition \ref{prop bellman} are similar to the well-received Lasry-Lions monotonicity. Further details can be found in Remark \ref{rem LL mono}.

Moreover, the assumption we make on the boundedness of the second-order Fr\'echet derivative of $F(X \otimes m)$ with respect to the random variable $X$ in (\ref{assumption, bdd of D^2F, D^2F_T}) is not that restrictive, where $F$ is a functional on the space of probability measure and $X \otimes  m$ is the pushforward of the probability measure $m$ by $X$. Under this assumption, we see that $D_X^2 F(X \otimes m)$ has to be a bounded linear operator from $\mathcal{H}_{m}$ to $\mathcal{H}_{m}$ (see Section \ref{subsec:HILBERT-SPACE} for the official definitions of the second-order Fr\'echet derivative and $\mathcal{H}_m$). In the languages of linear functional derivative on $\mathbb{R}^n$, it is sufficient to assume that $\diff_x^2 \frac{d F}{d\nu}(m)(x)$ and $\diff_{\tilde{x}}\diff_x \frac{d^{\h{.5pt}2}\h{-.7pt}  F}{d\nu^2}(m)(x,\widetilde{x})$ are bounded for any $(m,x)$ and $(m,x,\widetilde{x})$, respectively. For instance, $F(m) := \int_{\mathbb{R}}\pig( x^2 + e^{-x^2} \pig)dm(x)$, clearly far from linear quadratic, satisfies our proposed assumptions. Another plausible but highly restrictive assumption is to assume that the third-order Fr\'echet derivative of $F(X \otimes m)$ with respect to $X$ is a bounded linear operator from $\mathcal{H}_m \times\mathcal{H}_m $ to $\mathcal{H}_m$; if one adopts it, it slightly reduces the technicality of our analysis, but it puts very restrictive limitations on $F$; indeed, under such assumption, the derivative $\diff_x^3 \frac{d F}{d\nu}(m)(x)$ must be zero, which certainly reduces the family of coefficient functions to be that of linear quadratic only but nothing more. Therefore, we must not involve this here. Our present setting adopts $F(X \otimes m)$ with bounded second-order Fr\'echet derivative only, which covers a wide range of practical applications, while it also greatly simplifies the mathematical technicalities, and makes the methodology more transparent. Most importantly, it motivates all the estimates involved when we formulate the problem on Wasserstein space, also see our author's latest work \cite{WYY22}, in return to allow a bit more general setting, both generic in driving dynamics and cost functionals, than the present one stated here, for instance, without assuming the boundedness of the second-order Fr\'echet derivative of $F(X \otimes m)$ but requiring the boundedness of the second-order Wasserstein gradient of $F(m)$, though more tedious analysis has to be involved as a sacrification. To this end, we also refer to another very interesting recent work \cite{GM22} which handled master equations on Wasserstein metric space under the same dynamically controlled process as formulated here in the absence of Brownian noise. Last but not least, we believed that the methodology of this work can be applied to more exotic frameworks, we here choose a convenient setting so that most of the ideas can be illustrated in a coherent manner.

\mycomment{
Moreover, our assumption on the uniform Lipschitz continuity of the second-order Fr\'echet derivative in (\ref{assumption, lip of 3rd derivative of F and FT}) is not that restrictive. It is sufficient to assume that the third-order of Fr\'echet derivative of $F(X \otimes m)$ with respect to the random variable $X$ is bounded, where $F$ is a functional on probability measure space and $X \otimes  m$ is the pushforward of the probability measure $m$ by $X$. Under this assumption, we see that $D_X^3 F(X \otimes m)$ must be a bounded linear operator from $L^2\times L^2$ to $L^2$. In the languages of linear functional derivative, it implies that $\diff_x^3 \dfrac{d F}{d\nu}(m)(x)=0$. For example, $F(m) = \int_{\mathbb{R}^n}|x|^2 dm(x)$ satisfies this kind of assumption. Our setting covers a wide range of applications, for instance, the linear quadratic case, although the problems on Wasserstein metric space, such as \cite{WYY22,GM22}, allow a more general setting than our lifting version.}

\subsection{Organisation of the Article}
The remainder of this article is organized as follows. In Section 2, we review basic calculus in the Wasserstein and Hilbert spaces. In Section 3, we introduce the mean field type control problem and establish its connection with the associated FBSDE. Section 4 verifies some regularities of the value function, which facilitates the formulation of the Bellman equation. We investigate the classical solvability of the Bellman and master equations in Section 5. Proofs of various assertions in Sections 2 to 5 are provided in the Appendix. We would like to emphasize that the methodology presented in this article can be extended to study problems where common noise is present, as well as cases with a generic driving drift function and more general cost functionals. We are currently working on such extensions, building on the techniques developed in this article and in \cite{WYY22}, which studies the generic first-order mean field control problem.

\section{\label{sec:FORMALISM}Calculus in Hilbert Space of $L^2$-Random Variables}

\subsection{2-Wasserstein Space of Probability Measures}

Consider the space of probability measures $m$'s on $\mathbb{R}^{n}$, each of which has a finite second moment, i.e. $\int_{\mathbb{R}^{n}}|x|^{2}dm(x)<\infty$. For any such $m$ and $m'$, the 2-Wasserstein metric of them is defined by 
$$ W_2(m,m') := \inf_{\pi \in \Pi_2(m,m')}
\left[ \int_{\mathbb{R}^n \times \mathbb{R}^n} |x-x'|^2\pi(dx,dx')\right]^{1/2},
$$
where $\Pi_2(m,m')$ represents the set of joint probability measures with respective marginals $m$ and $m'$. This space equipped with the 2-Wasserstein metric is denoted by $\mathcal{P}_{2}(\mathbb{R}^{n})$. The infimum is attainable such that there are random variables $\hat{X}_{m}$ and $\hat{X}_{m'}$ (they may be dependent) associated with $m$ and $m'$ respectively so that 
\begin{equation}
W^2_{2}(m,m'):=\mathbb{E}\left|\hat{X}_{m}-\hat{X}_{m'}\right|^{2}.
\label{eq:1-2}
\end{equation}
We also recall the fact that a sequence of measures $\big\{m_{k}\big\}_{k \in \mathbb{N}}$ converges to $m$ in $\mathcal{P}_{2}(\mathbb{R}^{n})$ if
and only if it converges in the sense of weak convergence and simultaneously
\begin{equation}
\int_{\mathbb{R}^{n}}|x|^{2}dm_{k}(x)\longrightarrow\int_{\mathbb{R}^{n}}|x|^{2}dm(x);
\label{eq:1010}
\end{equation}
one can refer to Villani's textbook \cite{V} for details.

\subsection{Functionals and Their Derivatives on $\mathcal{P}_{2}(\mathbb{R}^{n})$}
Consider a continuous functional $F(m)$ on $\mathcal{P}_{2}(\mathbb{R}^{n})$, we study the concept of several derivatives of $F(m)$. Firstly, the linear functional derivative of $F(m)$ at $m$ is a function
$(m,x) \in \mathcal{P}_{2}(\mathbb{R}^{n})\times \mathbb{R}^{n} \longmapsto \dfrac{dF}{d\nu}(m)(x)$, being continuous under the product topology, satisfying 
\begin{equation}
\int_{\mathbb{R}^{n}}\left|\dfrac{dF}{d\nu}(m)(x)\right|^{2}dm(x)\leq c(m),
\label{eq:2-104}
\end{equation}
for some positive constant $c(m)$
, depending locally on $m$, which is uniformly bounded on compacta; in the rest of this article, all other $c(m)$'s represent different constants yet possessing the same boundedness nature. Besides, the Fr\'echet derivative $L_F$ of $F(m)$ with respect to $m$ (if exists) is given as usual: for any $m$, $m'\in\mathcal{P}_{2}(\mathbb{R}^{n})$, we have
\begin{equation}
\left| \dfrac{F\pig(m+\epsilon(m'-m)\pig)-F(m)}{\epsilon} -L_F(m-m')\right| \longrightarrow 0\h{15pt} \text{ as $\epsilon \to 0$}.
\label{Frechet derviative of F}
\end{equation} 
For if the densities of $m$ and $m'$ exist, denoted by $f_m$ and $f_{m'}$ respectively. By Riesz representation theorem on $L^2(\mathbb{R}^n)$, there is a $L^2(\mathbb{R}^n)$ function $\dfrac{dF}{d\nu}(m)(x)$, called the functional derivative, such that
\begin{equation}
L_F(m-m') 
= L_F(f_m-f_{m'}) 
= \int_{\mathbb{R}^n} \dfrac{dF}{d\nu}(m)(x)\pig[f_m(x)-f_{m'}(x)\pig] dx 
=  \int_{\mathbb{R}^n}
\dfrac{dF}{d\nu}(m)(x)\pig[dm(x)-dm'(x)\pig].
\label{def. linear functional derivative}
\end{equation}
The subset of $\mathcal{P}_{2}(\mathbb{R}^{n})$ consisting of all measures $m \in \mathcal{P}_{2}(\mathbb{R}^{n})$ with $L^2(\mathbb{R}^n)$ densities is clearly dense in $\mathcal{P}_{2}(\mathbb{R}^{n})$, as long as $L_F$ is closable, the representation in (\ref{def. linear functional derivative}) can be extended to all measures. Regarding to (\ref{def. linear functional derivative}), the linear functional derivative $\dfrac{dF}{d\nu}(m)(x)$ is actually the linear functional derivative $\dfrac{\delta F}{\delta m}(m)(x)$ defined in Carmona and Delarue \cite{CD18}
\footnote{Formula (\ref{eq:2-2000})
yields, by replacing $X$ by $\mathcal{I}_x$, $
\dfrac{dF}{d\nu}(m)(x)-\dfrac{dF}{d\nu}(m)(0)=\int_{0}^{1}D_{X}F(\mathcal{I}  \otimes  m)(\theta x) \cdot x\,d\theta$, which can serve as a definition of the linear functional derivative when the Fr\'echet derivative exists.}.
Further, (\ref{Frechet derviative of F}) and (\ref{def. linear functional derivative}) give
\begin{equation}
\dfrac{d}{d\theta}F\pig(m+\theta(m'-m)\pig)
=\int_{\mathbb{R}^{n}}\dfrac{dF}{d\nu}\pig(m+\theta(m'-m)\pig)(x)\pig[dm'(x)-dm(x)\pig],\label{eq:2-210}
\end{equation}
 and hence
\begin{equation}
F(m')-F(m)=\int_{0}^{1}\int_{\mathbb{R}^{n}}\dfrac{dF}{d\nu}\pig(m+\theta(m'-m)\pig)(x)\pig[dm'(x)-dm(x)\pig]d\theta.
\label{eq:2-211}
\end{equation}
We next turn to the notion of the second-order derivative, if $\dfrac{d}{d\theta}F\pig(m+\theta(m'-m)\pig)$
is continuously differentiable in $\theta$, then we can write the formula 

\begin{equation}
\dfrac{d^{\h{.5pt}2}\h{-.7pt}F }{d\theta^{2}}\pig(m+\theta(m'-m)\pig)=\int_{\mathbb{R}^{n}}\int_{\mathbb{R}^{n}}\dfrac{d^{\h{.5pt}2}\h{-.7pt}  F}{d\nu^{2}}\pig(m+\theta(m'-m)\pig)(x,\tilde{x})\pig[dm'(x)-dm(x)\pig]\pig[dm'(\tilde{x})-dm(\tilde{x})\pig],
\label{eq:2-213}
\end{equation}
where the map $(m,x,\widetilde{x}) \in \mathcal{P}_2(\mathbb{R}^n)\times\mathbb{R}^n\times\mathbb{R}^n \longmapsto\dfrac{d^{\h{.5pt}2}\h{-.7pt}  F}{d\nu^{2}}(m)(x,\widetilde{x})$, called the second-order linear functional derivative, is continuous and satisfies

\begin{equation}
\int_{\mathbb{R}^{n}}\int_{\mathbb{R}^{n}}
\left|\dfrac{d^{\h{.5pt}2}\h{-.7pt}  F}{d\nu^{2}}(m)(x,\widetilde{x})\right|^{2}dm(x)dm(\widetilde{x})\leq c(m).
\label{eq:2-212}
\end{equation}
Moreover, we have the formula 
\begin{equation}
\begin{aligned}
F(m')-F(m)=\:&\int_{\mathbb{R}^{n}}\dfrac{dF}{d\nu}(m)(x)\pig[dm'(x)-dm(x)\pig]\\
&+\int_{0}^{1}\int_{0}^{1}\int_{\mathbb{R}^n}\int_{\mathbb{R}^n}\theta\dfrac{d^{\h{.5pt}2}\h{-.7pt}  F}{d\nu^{2}}
\pig(m+\lambda\theta(m'-m)\pig)(x,\tilde{x})\pig[dm'(x)-dm(x)\pig]
\pig[dm'(\tilde{x})-dm(\tilde{x})\pig]
d\lambda d\theta.
\end{aligned}
\label{eq:2-214}
\end{equation}
Formulas (\ref{def. linear functional derivative}) and (\ref{eq:2-210}) do not change if we
add a constant to $\dfrac{dF}{d\nu}(m)(x)$, and we assume
the normalisation that $\int_{\mathbb{R}^{n}}\dfrac{dF}{d\nu}(m)(x)dm(x)=0.$ Similarly, in Formulae (\ref{eq:2-213}) and (\ref{eq:2-212}), we can replace $\dfrac{d^{\h{.5pt}2}\h{-.7pt}  F}{d\nu^{2}}(m)(x,\tilde{x})$
by $\dfrac{1}{2}\left(\dfrac{d^{\h{.5pt}2}\h{-.7pt}  F}{d\nu^{2}}(m)(x,\tilde{x})+\dfrac{d^{\h{.5pt}2}\h{-.7pt}  F}{d\nu^{2}}(m)(\tilde{x},x)\right)$
without a change of value, and we simply assume that the function $(x,\widetilde{x})\longmapsto\dfrac{d^{\h{.5pt}2}\h{-.7pt}  F}{d\nu^{2}}(m)(x,\tilde{x})$
is symmetric. Likewise, adding a function of the form $\varphi(x)+\varphi(\tilde{x})$ to a symmetric function $\dfrac{d^{\h{.5pt}2}\h{-.7pt}  F}{d\nu^{2}}(m)(x,\tilde{x})$ would result no change in value. 

\subsection{Hilbert Space of $L^2$-Random Variables}\label{subsec:HILBERT-SPACE}
Our working Hilbert space is $\mathcal{H}_{m}:=L^{2}\big(\Omega,\mathcal{A},\mathbb{P};L_{m}^{2}(\mathbb{R}^{n};\mathbb{R}^{n})\big)$ of the atomless probability space $\big(\Omega,\mathcal{A},\mathbb{P}\big)$. An element of $\mathcal{H}_{m}$ is denoted by $X_{x} = X(x,\omega)$ which maps $\mathbb{R}^n \times \Omega$ to $\mathbb{R}^n$; for each $x$, $X_x$ is a $L^{2}(\Omega,\mathcal{A},\mathbb{P};\mathbb{R}^{n})$ random variable. The inner product over $\mathcal{H}_m$ is defined by 

\begin{equation}
\langle X,Y\rangle_{\mathcal{H}_m}
:=\mathbb{E}\left[
\int_{\mathbb{R}^{n}} X\cdot Y dm(x)\right] < \infty
\h{1pt},\h{10pt}
\text{for any $X$, $Y \in \mathcal{H}_m$.}
\label{eq:2-215}
\end{equation} 
The deterministic pushforward\footnote{To follow the usual custom, we use the notation $ \# $ to denote the pushforward operation in the measure space. Nevertheless, we use $\otimes$ instead of $ \# $ to emphasise the bilinear structure of $X\otimes m$ in the space of random variables $X$ and the space of signed measures $m$. These two notations have their own importance. In future work, we may use $ \otimes $ and $\otimes$ interchangeably.} probability measure of $m \otimes \mathbb{P}  $ by $X$ on $\mathbb{R}^{n}$ is denoted by $X \#  (m \otimes \mathbb{P} ) \in \mathcal{P}_{2}(\mathbb{R}^{n})$. As $\mathbb{P}$ is fixed, we simply suppress $\mathbb{P}$ and write  $X  \otimes m$ for short. For any test function $\varphi$, we have
\begin{equation}
\int_{\mathbb{R}^{n}}\varphi(\xi)dX \otimes  m(\xi)=\mathbb{E}\left(\int_{\mathbb{R}^{n}}\varphi(X_{x})dm(x)\right).
\label{eq:2-216}
\end{equation}
In addition, for any $X \in \mathcal{H}_m$, the norm of $X$ is written as $\|X\|_{\mathcal{H}_m} = \sqrt{\langle X,X\rangle_{\mathcal{H}_m}}$. Given any functional $F$ on $\mathcal{P}_{2}(\mathbb{R}^{n})$, the map $X \longmapsto F(X \otimes  m)$ is now a functional
on $\mathcal{H}_{m}$. Since $X \longmapsto X  \otimes  m$ is continuous
from $\mathcal{H}_{m}$ to $\mathcal{P}_{2}(\mathbb{R}^{n})$, the functional
will be continuous as long as $F(m)$ is continuous on $\mathcal{P}_{2}(\mathbb{R}^{n}).$ Its G\^ateaux derivative (if exists) $D_{X}F(X  \otimes  m)$, as an element of $\mathcal{H}_{m}$, is defined by 
\begin{equation}
\dfrac{F\pig((X+\epsilon Y)  \otimes  m\pig)-F(X  \otimes  m)}{\epsilon}
\longrightarrow
\pigl\langle D_{X}F(X  \otimes  m),Y \pigr\rangle_{\mathcal{H}_m}
\h{1pt}, \h{10pt} 
\text{as $\epsilon\rightarrow 0$, for any $Y\in\mathcal{H}_{m}$.}
\label{eq:2-160}
\end{equation}
Also, we have

\begin{equation}
\dfrac{d}{d\theta}F\pig((X+\theta Y)  \otimes  m\pig)
=\pigl\langle D_{X}F\big((X+\theta Y)  \otimes  m\big),Y\pigr\rangle_{\mathcal{H}_m}.
\label{d_theta F = D_X F}
\end{equation}
Suppose that $F(m)$ has a linear functional derivative $\dfrac{dF}{d\nu}(m)(x)$ which is differentiable in $x$ for each $m \in \mathcal{P}_2(\mathbb{R}^n)$. Then it is easy to convince oneself
that when $(m,x) \longmapsto \diff_x \dfrac{dF}{d\nu}(m)(x)$ is continuous
with the growth condition 

\begin{equation}
\left|\diff_x \dfrac{dF}{d\nu}(m)(x)\right|\leq c(m)(1+|x|),
\label{eq:2-162}
\end{equation}
where {\color{black}$D_x$ is the usual differential operator with respect to $x$ in $\mathbb{R}^n$} and $c(m)$ is bounded on compacta from $\mathcal{P}_2(\mathbb{R}^n)$, then 

\begin{equation}
D_{X}F(X  \otimes  m)(\cdot,\omega)=\diff_x \dfrac{dF}{d\nu}(X  \otimes  m)(x)\bigg|_{x=X(\cdot,\omega)},
\label{eq:2-217}
\end{equation}
which is an element of $\mathcal{H}_{m}$; we refer the readers to Section 2.1 in \cite{BFY17} for a rigorous proof about the above equality. Note that, we see that $D_{X}F(X  \otimes  m)(u,\omega)$ is $\sigma(X(\cdot,\cdot))$-measurable random variable, \mycomment{what does it mean? } accounted in the augmented probability space $(  \mathbb{R}^n \times \Omega,  \mathcal{B}\times\mathcal{A},m \otimes \mathbb{P})$. On the other hand, it can also be viewed as a limiting function of the linear combinations of products of a function of $X(u,\omega)$ and a multivariate function of several integrals of different test functions against the measure $X  \otimes  m$. The linear functional derivative can be computed by the formula
\begin{equation}
\dfrac{dF}{d\nu}(X  \otimes  m)(X_{u})-\dfrac{dF}{d\nu}(X  \otimes  m)(0)
=\int_{0}^{1}D_{X}F(X  \otimes  m)(\theta X_{u})\cdot X_{u}d\theta,
\h{15pt}
\text{ for any $u \in \mathbb{R}^n$.}
\label{eq:2-2000}
\end{equation}
The identity function $\mathcal{I}_{\:\bigcdot}$ is a particular choice of $X_{\bigcdot}$ being a constant random variable, such that for any $x \in \mathbb{R}^n$, $\mathcal{I}_{x}=x$ and hence $\mathcal{I}_x  \otimes  m=m$. Therefore, (\ref{eq:2-217}) becomes 

\begin{equation}
D_{X}F(m)(x,\omega)
=D_{X}F(\mathcal{I}  \otimes  m)(x,\omega)
=\diff_x \dfrac{dF}{d\nu}(\mathcal{I}  \otimes  m)(x)=\diff_x \dfrac{dF}{d\nu}(m)(x)
,
\label{eq:2-218}
\end{equation}
it is identical to the $L$-derivative $\partial_{m}F(m)(x)$ in Carmona and Delarue \cite{CD18}; the $L$-derivative can be interpreted
as the Fr\'echet derivative of the functional $F(X  \otimes  m)$ on
$\mathcal{H}_{m}$ evaluated at $X=\mathcal{I}_x$. On the other hand, the $L$-derivative can also be realised by taking a $X_{x}$ which is totally independent of $x$, by then $X  \otimes  m=\mathcal{L}_X$
(the law of $X$) for every $m$, so $F(X  \otimes  m)=F(\mathcal{L}_X)$, and then $D_{X}F(X  \otimes  m)=\partial_{\nu}F(\mathcal{L}_X)(X)$.

Now we consider all $X_x$ such that $\dfrac{X_{x}}{(1+|x|^{2})^{\frac{1}{2}}}\in L^{\infty}(\mathbb{R}^{n};L^{2}\big(\Omega,\mathcal{A},\mathbb{P};\mathbb{R}^{n})\big)$, which means that

\begin{equation}
\mathbb{E}|X_{x}|^{2}\leq c(X)(1+|x|^{2}),
\label{eq:2-219}
\end{equation}
where $c(X)$ is a positive constant possibly depending on $X(\cdot,\cdot)$ but independent of $x$. Then $X\in\mathcal{H}_{m}$
for every $m \in \mathcal{P}_2(\mathbb{R}^n)$, and so $X$ and $m$ can be considered as two separate
arguments in $F(X  \otimes  m)$. Further, we can consider the functional
$m\longmapsto F(X  \otimes  m)$ and its linear functional derivative $\dfrac{\partial F}{\partial m}(X  \otimes  m)(x) = \dfrac{\delta F}{\delta\mu}(X  \otimes  \mu)(x)\Bigg|_{\mu =m}$, which is distinguished from $\dfrac{dF}{d\nu}(X  \otimes  m)(x) = \dfrac{\delta F}{ \delta \nu}(\nu)(x)\Bigg|_{\nu=X  \otimes m}$; in fact, by Proposition 2.12. in \cite{BGY20}, when $m\longmapsto\dfrac{dF}{d\nu}(m)(x)$ is continuous and has at most quadratic growth $\left|\dfrac{dF}{d\nu}(m)(x)\right|
\leq c(m)(1+|x|^{2})$, we have
\begin{equation}
\dfrac{\partial F}{\partial m}(X  \otimes  m)(x)
=\mathbb{E}\left[\dfrac{dF}{d\nu}(X  \otimes  m)(X_{x})\right].
\label{eq:2-220}
\end{equation} 
\begin{rem}
To facilitate further development, we allow the abuse of notations that $D_X F(Y  \otimes  m)$ always mean the Fr\'echet derivative of $F(\h{1pt}\cdot\h{1pt}  \otimes  m)$ with respect to  $\:\cdot$ ; even for the extreme case, the symbol $D_{X}F(Y_{X}  \otimes  m)$ represents $D_XF(X  \otimes  m)\big|_{X=Y_X}$.
\end{rem}

\subsection{Second-Order G\^ateaux Derivative in the Hilbert Space}

A functional $F(X  \otimes  m)$ is said to have a second-order G\^ateaux derivative
in $\mathcal{H}_{m}$, denoted by $D_{X}^{2}F(X  \otimes  m)\in\mathcal{L}(\mathcal{H}_{m};\mathcal{H}_{m})$,
if for any $Y \in\mathcal{H}_{m}$, it holds that
\begin{equation}
\dfrac{\pigl\langle D_{X}F((X+\epsilon Y)  \otimes  m)-D_{X}F(X  \otimes  m),Y \pigr\rangle_{\mathcal{H}_m}}{\epsilon}
\longrightarrow
\pigl\langle D_{X}^{2}F(X  \otimes  m)(Y),Y\pigr\rangle_{\mathcal{H}_m}\h{1pt},\h{10pt}
\text{as $\epsilon\rightarrow 0$}.
\label{def second G derivative}
\end{equation}
For any $Z$ and $W \in \mathcal{H}_m$, $D_{X}^{2}F(X  \otimes  m)(Z)$ is generally defined by
\begin{equation}
\pigl\langle D_{X}^{2}F(X  \otimes  m)(Z),W\pigr\rangle_{\mathcal{H}_m}
=\dfrac{1}{4}\left[\pigl\langle D_{X}^{2}F(X  \otimes  m)(Z+W),Z+W\pigr\rangle_{\mathcal{H}_m}
-\pigl\langle D_{X}^{2}F(X  \otimes  m)(Z-W),Z-W\pigr\rangle_{\mathcal{H}_m}\right].
\label{eq:2-1}
\end{equation}
Therefore, $D_{X}^{2}F(X  \otimes  m)$ is clearly self-adjoint by definition. 
Also 
\begin{equation}
\dfrac{d}{d\theta}\pigl\langle D_{X}F((X+\theta Z)  \otimes  m),W\pigr\rangle_{\mathcal{H}_m}
=\pigl\langle D_{X}^{2}F((X+\theta Z)  \otimes  m)(Z),W\pigr\rangle_{\mathcal{H}_m}.
\label{eq:2-228}
\end{equation}
Then, (\ref{eq:2-228}) and  (\ref{d_theta F = D_X F}) imply
\begin{equation}
\dfrac{d^{\h{.5pt}2}\h{-.7pt} F}{d\theta^{2}}\pig((X+\theta Y)  \otimes  m\pig)
=\pigl\langle D_{X}^{2}F((X+\theta Y)  \otimes  m)(Y),Y\pigr\rangle_{\mathcal{H}_m}.
\label{eq:2-229}
\end{equation}
Hence, we can rewrite the Taylor's expansion for $F$ as follows:

\begin{equation}
F\pig((X+Y)  \otimes  m)\pig)
=F(X  \otimes  m)
+\pigl\langle D_{X}F(X  \otimes  m),Y\pigr\rangle_{\mathcal{H}_m}
+\int_{0}^{1}\int_{0}^{1}\theta\pigl\langle D_{X}^{2}F((X+\theta\lambda Y)  \otimes  m)(Y),Y\pigr\rangle_{\mathcal{H}_m}
d\theta d\lambda.
\label{eq:2-230}
\end{equation}
From (\ref{eq:2-217}), it yields $\pigl\langle D_{X}F(X  \otimes  m),W\pigr\rangle_{\mathcal{H}_m}=\mathbb{E}\left[\displaystyle\int_{\mathbb{R}^{n}}\diff_x \dfrac{dF}{d\nu}(X  \otimes  m)(X_{x}) \cdot W_{x}dm(x)\right]$, it follows that
\begin{align*}
&\dfrac{1}{\epsilon}
\pigl\langle D_{X}F((X+\epsilon Z)  \otimes  m) - D_{X}F(X\otimes  m)
,W\pigr\rangle_{\mathcal{H}_m}\\
&\h{100pt}=\dfrac{1}{\epsilon}
\mathbb{E}\left[\displaystyle\int_{\mathbb{R}^{n}}
\left(\diff_x \dfrac{dF}{d\nu}((X+\epsilon Z)  \otimes  m)(X_{x}+\epsilon Z_{x})
-\diff_x \dfrac{dF}{d\nu}(X  \otimes  m)(X_{x})\right) \cdot W_{x}dm(x)\right],
\end{align*}
further, by assuming the existences of $\dfrac{d^{\h{.5pt}2}\h{-.7pt}  F}{d\nu^{2}}(m)(x,\tilde{x})$, its derivatives $\diff_x \diff_{\tilde{x}}\dfrac{d^{\h{.5pt}2}\h{-.7pt}  F}{d\nu^{2}}(m)(x,\tilde{x})$ and $\diff^2_x \dfrac{dF}{d\nu}(m)(x)$, together with continuity
properties, taking limits of both sides will give 
\begin{equation}
\begin{aligned}
\pigl\langle D_{X}^{2}F(X  \otimes  m)(Z), W\pigr\rangle_{\mathcal{H}_m}
=\:&\mathbb{E}\left[\int_{\mathbb{R}^n} \diff^2_x \dfrac{dF}{d\nu}(X  \otimes  m)(X_{x})Z_{x} \cdot W_{x} dm(x)\right]\\
&+\mathbb{E}\left\{\widetilde{\mathbb{E}}
\left[\int_{\mathbb{R}^n}\int_{\mathbb{R}^n}
\diff_x \diff_{\tilde{x}}\dfrac{d^{\h{.5pt}2}\h{-.7pt}  F}{d\nu^{2}}(X  \otimes  m)(X_{x},\widetilde{X}_{\tilde{x}})\widetilde{Z}_{\tilde{x}}\cdot W_{x}dm(\tilde{x})dm(x)\right]\right\},
\end{aligned}
\label{eq:2-231}
\end{equation}
in which $\pig(\widetilde{X}_{\tilde{x}},\widetilde{Z}_{\tilde{x}}\pig)$ are
independent copies of $\big(X_{x},Z_{x}\big)$; also refer the detailed discussions in Section 2.1 of \cite{BFY17}. We can write, consequently, that

\begin{equation}
D_{X}^{2}F(X  \otimes  m)(Z)
=\diff^2_x \dfrac{dF}{d\nu}(X  \otimes  m)(X_{x})Z_{x}
+\widetilde{\mathbb{E}}\left[\int_{\mathbb{R}^{n}}\diff_x \diff_{\tilde{x}}\dfrac{d^{\h{.5pt}2}\h{-.7pt}  F}{d\nu^{2}}(X  \otimes  m)(X_{x},\widetilde{X}_{\tilde{x}})\widetilde{Z}_{\tilde{x}}dm(\tilde{x})\right].
\label{eq:2-232}
\end{equation}
We notice the following measurability property of $D_{X}^{2}F(X  \otimes  m)(Z)$:

\begin{equation}
D_{X}^{2}F(X  \otimes  m)(Z)=M(X)Z+L(X,Z),
\label{eq:2-2320}
\end{equation}
where $M(X)$ is a matrix-valued $\sigma(X(\cdot,\cdot))$-measurable random element,
and $L(X,Z)$ is a vector-valued $\sigma(X(\cdot,\cdot))$-measurable random element, accounted in the augmented probability space $(  \mathbb{R}^n \times \Omega,  \mathcal{B}\times\mathcal{A},m \otimes \mathbb{P})$, which depends functionally on $Z$ in a linear manner. Likewise if we take $X_{x}=\mathcal{I}_{x}=x$, then we obtain

\begin{equation}
D_{X}^{2}F(m)(Z)=\diff^2_x \dfrac{dF}{d\nu}(m)(x)Z_x
+\widetilde{\mathbb{E}}\left[\int_{\mathbb{R}^{n}}\diff_x \diff_{\tilde{x}}\dfrac{d^{\h{.5pt}2}\h{-.7pt}  F}{d\nu^{2}}(m)(x,\tilde{x})\widetilde{Z}_{\tilde{x}}dm(\tilde{x})\right].
\label{eq:2-233}
\end{equation}
From Formula (\ref{eq:2-232}), it follows immediately that if $Z$
is independent of $X$ and $\mathbb{E}(Z)=0$, then the second term vanishes and so

\begin{equation}
D_{X}^{2}F(X  \otimes  m)(Z)=\diff^2_x \dfrac{dF}{d\nu}(X  \otimes  m)(X_{x})Z_{x}.
\label{eq:2-234}
\end{equation}
In order to get $D_{X}^{2}F(X  \otimes  m)(X)\in\mathcal{H}_{m}$, we
need to assume 
\begin{equation}
\left|\diff_x \dfrac{dF}{d\nu}(m)(x)\right|\leq c(m)\h{1pt},\h{8pt}
\left|\diff_x \diff_{\tilde{x}}\dfrac{d^{\h{.5pt}2}\h{-.7pt}  F}{d\nu^{2}}(m)(x,\tilde{x})\right|
\leq c(m),
\label{eq:2-235}
\end{equation}
where $|\cdot|$ in (\ref{eq:2-235}) is the matrix norm defined by $|A|:=\displaystyle\sup_{x \in \mathbb{R}^n,|x|=1}|Ax|$ for any $A \in \mathbb{R}^{n \times n}$. 

\subsection{Stochastic Calculus in the Space of $L^2$-Random Variables}\label{sec, ito cal.}

The probability space $(\Omega,\mathcal{A},\mathbb{P})$ under consideration is sufficiently large that contains a standard Wiener process in $\mathbb{R}^{n}$, denoted by $w(t)$, for $t\geq 0$, together with some additional random variables, for instance, serving the initial condition, can be independent of the filtration $\mathcal{W}_{t}^{s}:=\sigma\big(w(\tau)-w(t);\tau \in [t,s]\big)$ generated by the Wiener process. The latter random variables can be considered as $\mathcal{W}^t_0$- (or $\mathcal{W}^t$- for short) measurable for any $t \geq 0$. For each $t \in [0,T]$, we let $X_{tx} = X(x,\omega,t)$ be an arbitrary element in $\mathcal{H}_{m}$, independent of
$\mathcal{W}_{t}^T$ ($\mathcal{W}_{t}$ for short). We shall refer $X_{t\h{.5pt}\bigcdot}  \otimes  m$ or $X_{t}  \otimes  m$ for short the pushforward probability measure in $\mathcal{P}_{2}(\mathbb{R}^{n})$. We also define $\mathcal{W}_{tX}^{s}:=\sigma(X_t)\bigvee\mathcal{W}_{t}^{s}$ and denote by $\mathcal{W}_{tX}$ the filtration generated by the $\sigma-$algebras
$\mathcal{W}_{tX}^{s}$, for $s \geq t$. We denote by $L_{\mathcal{W}_{tX}}^{2}(t,T;\mathcal{H}_{m})$
the subspace of $L^{2}(t,T;\mathcal{H}_{m})$ of all elements adapted to the filtration
$\mathcal{W}_{tX}$. Based on Lemma 3.1 in \cite{BGY20}, it is important to note that, because $X_t$
is independent of $\mathcal{W}^T_{t},$ there exists an isometry between
$L_{\mathcal{W}_{tX}}^{2}(t,T;\mathcal{H}_{m})$ and $L_{\mathcal{W}_{t}}^{2}(t,T;\mathcal{H}_{X_t  \otimes  m})$; indeed, for $Y(\cdot) \in L_{\mathcal{W}_{tX}}^{2}(t,T;\mathcal{H}_{m})$, there exists a random field $Y_{t\xi}(s) \in L_{\mathcal{W}_{t}}^{2}(t,T;\mathbb{R}^n)$ for $\xi\in \mathbb{R}^{n}$ and $s \geq t$ such
that $Y(s)=Y_{t\xi}(s)\pigr|_{\xi=X_{tx}}$ and so $\mathbb{E}\big|Y(s)\big|^{2}=\mathbb{E}\left(\mathbb{E}\big|Y_{t\xi}(s)\big|^{2}\pigr|_{\xi=X_{tx}}\right)$. It follows that

\begin{align*}
\int_{t}^{T}\big\|Y(s)\big\|^{2}_{\mathcal{H}_m}ds
=\int_{t}^{T}\mathbb{E}\left[\int_{\mathbb{R}^{n}}\big|Y(s)\big|^{2}dm(x)\right]ds
&=\int_{t}^{T}\mathbb{E}\left[\int_{\mathbb{R}^{n}}\mathbb{E}\big|Y_{t\xi }(s)\big|^{2}\Big|_{\xi=X_{tx}}dm(x)\right]ds\\
&=\int_{t}^{T}\mathbb{E}\left[\int_{\mathbb{R}^{n}}\big|Y_{t \xi}(s)\big|^{2}d\pig(X_t   \otimes  m\pig)(\xi)\right]ds.
\end{align*}

Let us consider for $j=1,2,\ldots,n$, coefficient functions $\eta_{tX}^{j}(\cdot) \in L_{\mathcal{W}_{tX}}^{2}(t,T;\mathcal{H}_{m}),$
so that each of them corresponds to the random field $\eta_{t \xi}^{j}(\cdot)\in L_{\mathcal{W}_{t}}^{2}(t,T;\mathcal{H}_{X_t  \otimes  m}),$
where $X_t\in\mathcal{H}_{m}.$ We define the stochastic integral 

\begin{equation}
\sum_{j=1}^{n}\int_{t}^{s}\eta_{tX}^{j}(\tau)dw_{j}(\tau)
=\sum_{j=1}^{n}\int_{t}^{s}\eta_{t\xi}^{j}(\tau)dw_{j}(\tau)\Big|_{\xi=X},
\label{eq:3-106}
\end{equation}
which consequently defines a process in $L_{\mathcal{W}_{tX}}^{2}(t,T;\mathcal{H}_{m})$
with 

\begin{equation}
\begin{aligned}
&\mathbb{E}\left[\int_{\mathbb{R}^{n}}
\left|\sum_{j=1}^{n}\int_{t}^{s}\eta_{tX}^{j}(\tau)dw_{j}(\tau)\right|^{2}dm(x)\right]
=\int_{\mathbb{R}^{n}}\mathbb{E}\left[\h{3pt}
\left|\sum_{j=1}^{n}\int_{t}^{s}\eta_{t\xi}^{j}(\tau)dw_{j}(\tau)\right|^{2}
d\big(X_t  \otimes  m\big)(\xi)\h{3pt}\right]\\
&=\int_{\mathbb{R}^{n}}\mathbb{E}\left[\sum_{j=1}^{n}\int_{t}^{s}\pigl|\eta_{t\xi}^{j}(\tau)\pigr|^{2}d\tau\,d\big(X_t  \otimes  m\big)(\xi)
\right]
=\mathbb{E}\left[\int_{\mathbb{R}^{n}}\sum_{j=1}^{n}\int_{t}^{s}\pigl|\eta_{tX}^{j}(\tau)\pigr|^{2}d\tau dm(x)\right].
\end{aligned}
\label{eq:3-107}
\end{equation}

\begin{definition}[\bf Generalized It\^o Process]
Given $X$ measurable to $\mathcal{W}^t_{0}$ and $a_{tX}(s)\in L_{\mathcal{W}_{tX}}^{2}(t,T;\mathcal{H}_{m})$, we
can define the It\^o process as, for a given $t \in [0,T)$, 

\begin{equation}
\mathbb{X}_{tX}(s)=X+\int_{t}^{s}a_{tX}(\tau)d\tau
+\sum_{j=1}^{n}\int_{t}^{s}\eta_{tX}^{j}(\tau)dw_{j}(\tau),\h{10pt} \text{ for $s \in [t,T]$,}
\label{state X_tX, genral}
\end{equation}
which belongs to $L_{\mathcal{W}_{tX}}^{2}(t,T;\mathcal{H}_{m}).$
\end{definition}
We can then verify the following differentiation rule
\begin{empheq}[left=\empheqbiglbrace]{align}
\dfrac{d}{ds}\big\|\mathbb{X}_{tX}(s)\big\|_{\mathcal{H}_{m}}^{2}&=2\pigl\langle \mathbb{X}_{tX}(s),a_{tX}(s)\pigr\rangle_{\mathcal{H}_m}+\sum_{j=1}^{n}\pigl\|\eta_{tX}^{j}(s)\pigr\|_{\mathcal{H}_{m}}^{2},\label{diff of H norm}\\
\big\|\mathbb{X}_{tX}(t)\big\|_{\mathcal{H}_{m}}^{2}&=\|X\|^{2}_{\mathcal{H}_{m}}. \nonumber
\end{empheq}
Fix $t \in [0,T)$, we now proceed to obtain an It\^o's formula for the evolution of
$F(\mathbb{X}_{tX}(s)  \otimes  m,s)$, where $m$ is still the initial distribution at time $0$ and $s \in [t,T]$. Further, we assume the following:
\begin{align}
&\begin{aligned}
&\h{5.5pt}\text{(i)\h{2pt} Let $L^2_{\mathcal{W}_t^{\indep}}(\mathcal{H}_m) \subset \mathcal{H}_m$ the subspace of $\mathcal{H}_m$ of all elements independent of $\mathcal{W}_t$, consider the}\\
&\text{\h{23pt}map $X \in L^2_{\mathcal{W}^{\indep}_t}(\mathcal{H}_m) \longmapsto F(X  \otimes  m,s)$, and it is twice G\^ateaux differentiable over $\mathcal{H}_{m}$;}
\end{aligned}\label{subspace of H_m indep. of W_t}\\
&\h{2.5pt}\text{(ii)} \h{5pt}\pig|F(X  \otimes  m,s+\epsilon)-F(X  \otimes  m,s)\pig|\leq C_{T}\epsilon\left(1+\|X\|^{2}_{\mathcal{H}_m}\right);\label{eq:3-120}\\
&\text{(iii)} \h{5pt} D_{X}F(X  \otimes  m,s_{k}) \text{ converges to } D_{X}F(X  \otimes  m,s)\:\text{in}\:\mathcal{H}_{m}\;\text{as}\:s_{k}\downarrow s;
\label{eq:3-121}\\
&\begin{aligned}
&\text{(iv)} \h{5pt} \text{For any $X \in L^2_{\mathcal{W}^{\indep}_t} (\mathcal{H}_m)$, }\\
&\h{70pt}\big\|D_{X}F(X  \otimes  m,s)\big\|_{\mathcal{H}_m}
\leq C_{T}\pig(1+\|X\|_{\mathcal{H}_m}\pig)
\h{10pt} \text{and} \h{10pt}
\big\|D_{X}^{2}F(X  \otimes  m,s)\big\|_{\mathcal{H}_m}\leq c.
\end{aligned}
\label{eq:3-123}
\end{align}
The second-order derivative is assumed to have the following continuity properties:
\begin{align}
&\begin{aligned}
&\h{3pt}\text{(v)\h{2pt} For any sequence $\big\{X_k\big\}_{k \in \mathbb{N}} \subset L^2_{\mathcal{W}_t^{\indep}}(\mathcal{H}_m)$ converging to 
$X \in L^2_{\mathcal{W}_t^{\indep}}(\mathcal{H}_m) $ as $s_{k}\downarrow s$, then }\\
&\text{\h{30pt} $D_{X}^{2}F(X_{k}  \otimes  m,s_{k})(Y)\longrightarrow D_{X}^{2}F(X  \otimes  m,s)(Y)\;\text{uniformly in}\:\mathcal{H}_{m}$ for all $Y$ in a bounded set;}
\end{aligned}\label{eq:3-124}\\
&\begin{aligned}
&\h{23pt} \text{Furthermore, if the sequence $\{Y_k\}_{k \in \mathbb{N}}$ is bounded in $\mathcal{H}_m$, then}\\
&\h{50pt} D_{X}^{2}F(X_{k}  \otimes  m,s_{k})(Y_{k})-D_{X}^{2}F(X  \otimes  m,s_{k})(Y_{k})\longrightarrow0,
\h{10pt}\text{in}\:
L^{1}\big(\Omega,\mathcal{A},\mathbb{P};L_{m}^{1}(\mathbb{R}^{n};\mathbb{R}^{n})\big);
\end{aligned}\label{eq:3-125}
\end{align}

\h{-6pt}\text{(vi)} \h{-2pt} The processes $a_{tX}(s)$, $\eta_{tX}^{j}(s)\in L_{\mathcal{W}_{tX}}^{2}(t,T;\mathcal{H}_{m})$ satisfy the following additional properties:
\begin{align}
&\text{(a)} \h{130pt} \sup_{s\in[t,T]}\mathbb{E}\left[\int_{\mathbb{R}^{n}}
\left(\sum_{j=1}^{n}\pigl|\eta_{tX}^{j}(s)\pigr|^{2}\right)^{2}dm(x)\right]<\infty;\nonumber\\
&\text{(b)} \h{50pt}
\dfrac{1}{\epsilon}\sum^n_{j=1}\mathbb{E}\left[\int_{s}^{s+\epsilon}\int_{\mathbb{R}^{n}}
\Big|\eta_{tX}^{j}(\tau)-\eta_{tX}^{j}(s)\Big|^{2}dm(x)d\tau\right]\longrightarrow0,\h{10pt}
\text{as }\:\epsilon \rightarrow 0,\text{ a.e. }s\in[t,T];\label{eq:3-112}\\
&\text{(c)}\h{90pt}
\dfrac{1}{\epsilon}\int_{s}^{s+\epsilon}a_{tX}(\tau)d\tau\longrightarrow a_{tX}(s)\h{10pt}\text{in}\:\mathcal{H}_{m},\text{ as }\:\epsilon\rightarrow 0,\text{ a.e. }s\in[t,T].\nonumber
\end{align}
We then state the following It\^o's lemma in the mean field setting with its proof in \hyperref[app, ito thm]{Appendix}.
\begin{thm}[\textbf{Mean Field It\^o's Formula}]
\label{ito thm} Under the assumptions of  (\ref{subspace of H_m indep. of W_t}) to (\ref{eq:3-112}), the function $s\longmapsto F\big(\mathbb{X}_{tX}(s)  \otimes  m,s\big)$ is a.e
differentiable on $(t,T)$, and we also have the change-of-variable formula:
\begin{equation}
\begin{aligned}
\dfrac{d}{ds}F(\mathbb{X}_{tX}(s)  \otimes  m,s)
=\:&\dfrac{\partial}{\partial s}F(\mathbb{X}_{tX}(s)  \otimes  m,s)
+\pigl\langle D_{X}F(\mathbb{X}_{tX}(s)  \otimes  m,s),a_{tX}(s)\pigr\rangle_{\mathcal{H}_m}\\
&+\dfrac{1}{2}\left\langle D_{X}^{2}F(\mathbb{X}_{tX}(s)  \otimes  m,s)\left(\sum_{j=1}^{n}\eta_{tX}^{j}(s)\mathcal{N}_s^j\right),\sum_{j=1}^{n}\eta_{tX}^{j}(s)\mathcal{N}_s^j\right\rangle_{\mathcal{H}_m},\text{ a.e. }\:s\in(t,T),
\end{aligned}
\label{eq:3-113}
\end{equation}
where $\mathcal{N}_s^j$'s are (being independent to each other)\footnote{We can enlarge the initial $\sigma$-algebra that also contains all these $\mathcal{N}^j_s$'s and independent of $X_t$} \mycomment{this footnote correct?} Gaussian random variables, everyone  with a mean $0$ and a unit variance, and each $\mathcal{N}_s^j$ is independent of the $\sigma$-algebra $\mathcal{W}_{tX}^{s}$. If one applies Formula (\ref{eq:2-232}) together with the discussion between (\ref{eq:2-233}) and (\ref{eq:2-234}), then we also have 

\begin{equation}
\begin{aligned}
\dfrac{d}{ds}F\pig(\mathbb{X}_{tX}(s)  \otimes  m,s\pig)
=\:&\dfrac{\partial}{\partial s}F\pig(\mathbb{X}_{tX}(s)  \otimes  m,s\pig)
+\left\langle \diff_x \dfrac{dF}{d\nu}\pig(\mathbb{X}_{tX}(s)  \otimes  m,s\pig)\big( \mathbb{X}_{tx}(s) \big),a_{tX}(s)\right\rangle_{\mathcal{H}_m}\\
&+\dfrac{1}{2}\sum_{j=1}^{n}\left\langle \diff^2_x \dfrac{dF}{d\nu}(\mathbb{X}_{tX}(s)  \otimes  m,s)(\mathbb{X}_{tX}(s))\eta_{tX}^{j}(s),\eta_{tX}^{j}(s)\right\rangle_{\mathcal{H}_m},\text{ a.e. }s\in(t,T).
\end{aligned}
\label{ito lemma in gradient form}
\end{equation}
\end{thm}

\section{Formulation and Characterisation of Mean Field Type Control Problem}
\subsection{Problem Formulation}
The space of controls is set as the Hilbert space $L_{\mathcal{W}_{tX}}^{2}(t,T;\mathcal{H}_{m}).$
A control is denoted by $v_{tX}(s),$ where the initial condition at time $t$ is $X=\mathbb{X}_{tX}(t)\in\mathcal{H}_{m}$, for any $X$ measurable to the filtration $\mathcal{W}_0^t$, being independent of $\mathcal{W}_{t}^{s}=\sigma(w(\tau)-w(t) : t\leq\tau\leq s)$ for any $s \in [t,T]$. So, we have a measurable random field $v_{t\xi}(s)$, $\xi\in \mathbb{R}^{n}$, and for each $\xi$, $v_{t\xi }(\cdot)\in L_{\mathcal{W}_{t}}^{2}(t,T;\mathbb{R}^{n})$ such that $v_{t\xi }(s)\Big|_{\xi=X=\mathbb{X}_{tX}(t)} = v_{tX }(s)$; then we have the equivalence: 
\begin{equation}
\begin{aligned}
\int_{t}^{T}\|v_{tX}(s)\|^{2}_{\mathcal{H}_m}ds
=\mathbb{E}\left[\int_{t}^{T}\int_{\mathbb{R}^{n}}
|v_{t \xi}(s)|^{2}d (X   \otimes  m)(\xi)ds\right]
&=\mathbb{E}\left[\int_{t}^{T}\int_{\mathbb{R}^{n}}\mathbb{E}\left(|v_{t\xi}(s)|^{2}\right)\Big|_{\xi=X}dm(x)ds\right]\\
&=\mathbb{E}\left[\int_{t}^{T}\int_{\mathbb{R}^{n}}|v_{tX}(s)|^{2}dm(x)ds\right].
\end{aligned}
\label{eq:4-100}
\end{equation}
The state denoted by $\mathbb{X}_{tX}(s)=\mathbb{X}_{tX}(s;v_{tX}(\cdot))$ associated with a
control $v_{tX}(\cdot)$ is defined as the  It\^o process in accordance with Section \ref{sec, ito cal.}
\begin{equation}
\mathbb{X}_{tX}(s)=X + \int_{t}^{s}v_{tX}(\tau)d\tau+\eta(w(s)-w(t)),\h{10pt} \text{ for all $s \in [t,T]$,}
\label{def. X_tX, constant sigma}
\end{equation}
where $\eta$ is a constant matrix valued in $\mathbb{R}^{n \times n}$. We see
that $\mathbb{X}_{tX}(s)$ belongs to $L_{\mathcal{W}_{tX}}^{2}(t,T;\mathcal{H}_{m}).$
More precisely, $\mathbb{X}_{tX}(s)=\mathbb{X}_{t\xi}(s)\Big|_{\xi=X}=\mathbb{X}_{t \xi}(s;v_{t\xi})\Big|_{\xi=X},$  where

\begin{equation}
\mathbb{X}_{t\xi}(s)=\xi+\int_{t}^{s}v_{t\xi}(\tau)d\tau+\eta(w(s)-w(t)),\h{10pt} \text{for $\xi\in \mathbb{R}^{n}$}.
\label{eq:3-501}
\end{equation}
The probability law on $\mathbb{R}^{n},$ $\mathbb{X}_{tX}(s)  \otimes  m$ is simply the
probability law of $\mathbb{X}_{tX}(s).$ It is also the probability law of
$\mathbb{X}_{t\xi}(s)$ when $\xi$ is equipped with a probability $\mathbb{X}_{t}  \otimes  m,$
more precisely $\mathbb{X}_{t\bigcdot}  \otimes  m.$ We aim to minimize the objective functional: 
\begin{equation}
 \mbox{\fontsize{10.4}{10}\selectfont\(
J_{tX}(v_{tX} )=\displaystyle\int_{t}^{T}\int_{\mathbb{R}^n}\mathbb{E}\pig[l\pig(\mathbb{X}_{tX}(s),v_{tX}(s)\pig)\pig]dm(x)
+F\pig(\mathbb{X}_{tX}(s)  \otimes  m\pig)ds+\int_{\mathbb{R}^n} \mathbb{E}\pig[h(\mathbb{X}_{tX}(T)\pig] d m(x)
+F_{T}\pig(\mathbb{X}_{tX}(T)  \otimes  m\pig),\)}
\label{eq:3-503}
\end{equation}
over $v \in L_{\mathcal{W}_{tX}}^{2}(t,T;\mathcal{H}_{m})$, where the cost functions $l(x,v):\mathbb{R}^n \times \mathbb{R}^n \to \mathbb{R}$ and $h(x) :\mathbb{R}^n \to \mathbb{R}$. 
The functional $J_{tX}(v_{tX})$
depends on $X$ only through its probability distribution; in particular, if we are dealing an It\^o process $\mathbb{X}_{tX}(s)$ with  $\mathbb{X}_{tX}(t)=X$, then $J_{tX}(v_{tX})$ depends on $X$ only through $X  \otimes  m$. For a fixed $X$ measurable to $\mathcal{W}^t_0$, the functional $v_{tX}\longmapsto J_{tX}(v_{tX})$
is defined on the Hilbert space $L_{\mathcal{W}_{tX}}^{2}(t,T;\mathcal{H}_{m}) \subset L^{2}(t,T;\mathcal{H}_{m})$.

\subsection{Assumptions} \label{sec. assumptions}

All the following constants are positive except that $c'_l$, $c'_h$, $c'$ and $c'_T$ can be allowed to be non-positive. For any $x$, $v$, $\xi$ and $\zeta \in \mathbb{R}^n$, we assume that $l(x,v)$ and $h(x)$ are twice differentiable and the following hold:
\begin{align}
\textbf{A(i)}& \h{35pt} |l(x,v)|\leq c_{l}\pig(1+|x|^{2}+|v|^{2}\pig)\h{10pt} \text{and}\h{10pt} 
|l_{x}(x,v)|,|l_{v}(x,v)|\leq c_{l}\pig(1+|x|^{2}+|v|^{2}\pigr)^{\frac{1}{2}};
\label{assumption, bdd of l, lx, lv} \\
\textbf{A(ii)}& \h{115pt}  |l_{xx}(x,v)|
\h{1pt},\h{10pt} |l_{vx}(x,v)|
\h{1pt},\h{10pt}|l_{vv}(x,v)|\leq c_{l};
\label{assumption, bdd of lxx, lvx, lvv}\\
\textbf{A(iii)}& \h{60pt} |h(x)|\leq c_{h}\pig(1+|x|^{2}\pig)
\h{1pt},\h{10pt} |h_{x}(x)|\leq c_{h}\pig(1+|x|^{2}\pig)^{\frac{1}{2}}
\h{1pt},\h{10pt} |h_{xx}(x)|\leq c_{h};
\label{assumption, bdd of h, hx, hxx}\\
\textbf{A(iv)}& \h{95pt} l_{xx}(x,v),l_{vx}(x,v),l_{vv}(x,v),h_{xx}(x)\:\text{are continuous};
\label{assumption, cts of lxx, lvx, lvv, hxx} \\
\textbf{A(v)}& \h{70pt} l_{xx}(x,v)\xi \cdot \xi
+2l_{vx}(x,v)\zeta \cdot\xi
+l_{vv}(x,v)\zeta \cdot \zeta 
\geq\lambda|\zeta|^{2}-c'_{l}|\xi|^{2} ; 
\label{assumption, convexity of l} \\
\textbf{A(vi)}& \h{160pt} h_{xx}\xi\cdot \xi\geq-c'_{h}|\xi|^{2};  \label{assumption, convexity of h}
\end{align}
The subscript notation, for example, $l_{xv}$ represents the second-order derivative $\p_v\pig[\p_xl(x,v)\pig]$. {\color{black}We describe next the assumptions on the functionals $m \longmapsto F(m)$
and $m \longmapsto F_T(m)$ on $\mathcal{P}_{2}(\mathbb{R}^n)$. We assume that, for any $m \in \mathcal{P}_2(\mathbb{R}^n)$, 
\begin{equation}
|F( m)|\leq c \left(1+ \int_{\mathbb{R}^n}|x|^2 dm(x) \right)
\h{1pt},\h{10pt}
|F_{T}(  m)|\leq c_{T} \left(1+ \int_{\mathbb{R}^n}|x|^2 dm(x) \right),
\label{assumption, bdd of F F_T, no lift}
\end{equation}
where $c$ and $c_T$ are independent of $m$. For any $m \in \mathcal{P}_2(\mathbb{R}^n)$, $\widetilde{x} \in \mathbb{R}^n$ and $x \in \mathbb{R}^n$, the functionals also satisfy
\begin{align}
\textbf{b(i)}& \h{70pt} 
\left|\diff_x \dfrac{dF}{d\nu}(m)(x)\right|\leq \dfrac{c}{\sqrt{2}}(1+|x|)
\h{1pt},\h{10pt}
\left|\diff_x \dfrac{dF_T}{d\nu}(m)(x)\right|\leq \dfrac{c_T}{\sqrt{2}}(1+|x|); 
\label{assumption, bdd of DF, DF_T, no lift}\\
\textbf{b(ii)}& \h{5pt} 
\mbox{\fontsize{10.3}{10}\selectfont\( \left|\diff^2_x \dfrac{dF}{d\nu}(m)(x)\right|\leq \dfrac{c}{2}\h{1pt},\h{5pt}
\left|\diff_x \diff_{\tilde{x}}\dfrac{d^{\h{.5pt}2}\h{-.7pt}  F}{d\nu^{2}}(m)(\tilde{x},x)\right|
\leq \dfrac{c}{2}\h{1pt},\h{5pt}
\left|\diff^2_x \dfrac{dF_T}{d\nu}(m)(x)\right|\leq \dfrac{c_T}{2}\h{1pt},\h{5pt}
\left|\diff_x \diff_{\tilde{x}}\dfrac{d^{\h{.5pt}2}\h{-.7pt}  F_T}{d\nu^{2}}(m)(\tilde{x},x)\right|
\leq \dfrac{c_T}{2}\)};
\label{assumption, bdd of D^2F, D^2F_T, no lift}\\
\textbf{b(iii)}& \h{5pt} 
\diff^2_x \dfrac{dF}{d\nu}(m)(x),\:
\diff_x \diff_{\tilde{x}}\dfrac{d^{\h{.5pt}2}\h{-.7pt}  F}{d\nu^{2}}(m)(x,\tilde{x})
\:\text{are continuous in $(m,x)$ and $(m,x,\tilde{x})$, respectively};
\label{assumption, cts of D^2F, no lift}\\
\textbf{b(iv)}& \h{5pt} 
\diff^2_x \dfrac{dF_T}{d\nu}(m)(x),\:
\diff_x \diff_{\tilde{x}}\dfrac{d^{\h{.5pt}2}\h{-.7pt}  F_T}{d\nu^{2}}(m)(x,\tilde{x})
\:\text{are continuous in $(m,x)$ and $(m,x,\tilde{x})$, respectively};
\label{assumption, cts of D^2F_T, no lift} \\
\textbf{b(v)$^*$} &  \h{80pt} 
\diff^2_x \dfrac{dF}{d\nu}(m)(x)x\cdot x\geq -c'|x|^2\h{1pt},\h{5pt}
\diff_x \diff_{\tilde{x}}\dfrac{d^{\h{.5pt}2}\h{-.7pt}  F}{d\nu^{2}}(m)
(x,\widetilde{x})
x \cdot \widetilde{x}
\geq 0,\nonumber\\
& \h{80pt}
\diff^2_x \dfrac{dF_T}{d\nu}(m)(x)x\cdot x\geq -c'_T|x|^2\h{1pt},\h{5pt}
\diff_x \diff_{\tilde{x}}\dfrac{d^{\h{.5pt}2}\h{-.7pt}  F_T}{d\nu^{2}}(m)
(x,\widetilde{x})
x \cdot \widetilde{x}
\geq 0.
\label{assumption, convexity of D^2F, D^2F_T, no lift}\\
&\h{-20pt}\textup{or}\nonumber\\
\textbf{b(v)$^\dagger$} &  \h{80pt} 
\diff^2_x \dfrac{dF}{d\nu}(m)(x)x\cdot x+
\diff_x \diff_{\tilde{x}}\dfrac{d^{\h{.5pt}2}\h{-.7pt}  F}{d\nu^{2}}(m)
(x,\widetilde{x})
x \cdot \widetilde{x}
\geq  -c'|x|^2,\nonumber\\
& \h{80pt}
\diff^2_x \dfrac{dF_T}{d\nu}(m)(x)x\cdot x
+\diff_x \diff_{\tilde{x}}\dfrac{d^{\h{.5pt}2}\h{-.7pt}  F_T}{d\nu^{2}}(m)
(x,\widetilde{x})
x \cdot \widetilde{x}
\geq -c'_T|x|^2.
\label{assumption, convexity of D^2F, D^2F_T, no lift 2}
\end{align}
\begin{rem}
{\color{black}Although mean field control problems and mean field games are different in nature, they share some common structures in their master equations. We compare our assumptions with two commonly used monotonicity conditions in the literature dealing with master equations in mean field games, namely, the Lasry-Lions monotonicity ((2.5) in \cite{CDLL19}) and the displacement monotonicity (Assumptions 3.2, 3.5 in \cite{GMMZ22}). The corresponding linear functional derivatives $\dfrac{dF}{d\nu}(m)(x)$ and $\dfrac{dF_T}{d\nu}(m)(x)$ of the running cost functional $F$ and the terminal cost functional $F_T$ in our formulation are equivalent to $F(x,m)$ and $G(x,m)$ introduced on page 20 in \cite{CDLL19} respectively, in terms of the master equation (see the respective master equations in \eqref{eq:7-5} in this work and (2.6) in \cite{CDLL19}). The corresponding Lasry-Lions monotonicity if formulated in our setting is equivalent to the following condition (see Remark 2.4 in \cite{GMMZ22}):

\begin{equation}
LLM_{\textup{modified}}:=\mathbb{E} \mathbb{\widetilde{E}}\Big[
\diff_x\diff_{\tilde{x}}\dfrac{d^2F}{d\nu^2}(X \otimes m)(X,\widetilde{X})\widetilde{Y}\cdot Y
  \Big] \geq 0,
  \label{LLM}
\end{equation}for any $X, Y \in \mathcal{H}_m$, where $(\widetilde{X},\widetilde{Y})$ is an independent copy of $(X,Y)$. Therefore, Assumption \textup{\bf b(v)$^*$}'s \eqref{assumption, convexity of D^2F, D^2F_T, no lift} is consistent with the modified Lasry-Lions monotonicity \eqref{LLM}. For the displacement monotonicity, we have to first define the Hamiltonian by
\begin{equation}
\widetilde{H}(x,m,p):=l\pig(x,u(x,p)\pig)+u(x,p)\cdot p + \dfrac{d F}{d\nu}(m)(x),
\label{def. ham}
\end{equation}
where $u(x,p)$ solves the first order condition $l_v\pig(x,u(x,p)\pig)+p=0$ (see Lemma \ref{lem G derivative of J} and Proposition \ref{prop eq. condition of optmality}). The corresponding displacement monotonicity if formulated in our setting is equivalent to the following condition (see Remark 2.4 in \cite{GMMZ22}):
\begin{align}
DM_{\textup{modified}}:&=\int_{\mathbb{R}^n}\int_{\mathbb{R}^n} 
\left\langle
\nabla_x\nabla_{\tilde{x}}\dfrac{d}{d\nu}\widetilde{H}(x,m,\varphi(x))(\widetilde{x})v(\widetilde{x})
+\nabla_{xx}\widetilde{H}(x,m,\varphi(x))v(x),v(x)
\right\rangle_{\mathbb{R}^n}dm(x)dm(\widetilde{x})\geq 0,
\label{DM}
\end{align}
for any $m \in \mathcal{P}_2(\mathbb{R}^n)$, $\varphi \in C^1$ and $v \in L^2_m$. Under our current setting, we substitute \eqref{def. ham} into to \eqref{DM} to see that
\fontsize{9.5pt}{11pt}\begin{align*}
&DM_{\textup{modified}}\\
&=\int_{\mathbb{R}^n}\int_{\mathbb{R}^n} 
\bigg\langle
\nabla_x\nabla_{\tilde{x}}\dfrac{d^2F}{d\nu^2}(m)(x,\widetilde{x})v(\widetilde{x})\\
&\h{55pt}+\left[l_{xx}\pig(x,u(x,p)\pig)+l_{xv}\pig(x,u(x,p)\pig)\nabla_x u(x,p)
+\nabla_{x}^2\dfrac{d F}{d\nu}(m)(x)\right]v(x),v(x)
\bigg\rangle_{\mathbb{R}^n}dm(x)dm(\widetilde{x})\\
&=\int_{\mathbb{R}^n}\int_{\mathbb{R}^n} 
\bigg\langle
\nabla_x\nabla_{\tilde{x}}\dfrac{d^2F}{d\nu^2}(m)(x,\widetilde{x})v(\widetilde{x})\\
&\h{55pt}+\left[l_{xx}\pig(x,u(x,p)\pig)-l_{xv}\pig(x,u(x,p)\pig)\pig[l_{xv}\pig(x,u(x,p)\pig)\pigr]^{-1}
l_{vx}\pig(x,u(x,p)\pig)
+\nabla_{x}^2\dfrac{d F}{d\nu}(m)(x)\right]v(x),v(x)
\bigg\rangle_{\mathbb{R}^n}dm(x)dm(\widetilde{x}).
\end{align*}\normalsize
With the aid of Assumptions {\bf A(v)}'s \eqref{assumption, convexity of l} and \textup{\bf b(v)$^*$}'s \eqref{assumption, convexity of D^2F, D^2F_T, no lift} (or \textup{\bf b(v)$^\dagger$}'s \eqref{assumption, convexity of D^2F, D^2F_T, no lift 2}), we obtain that
$$ DM_{\textup{modified}} \geq -(c'+c'_l)\int_{\mathbb{R}^n}|v(x)|^2 dm(x).$$
At the first glance, our assumptions appear to be weaker than the displacement monotonicity, but if we aim to solve the problem globally, the condition in \eqref{uniqueness condition} suggests that we need to assume $-(c'+c'_l) \geq 0$. Besides, the Hamiltonian in \cite{GMMZ22} is required to be $C^5$ and its third-order derivatives are uniformly bounded (translated to our setting this requirement is equivalent that $l(x,u)$, $\dfrac{dF}{d\nu}(m)(x)$ are $C^5$ and their third-order derivatives are uniformly bounded), which enforces the cost functionals to have linear growth at the most. In contrast, our assumptions allow models beyond linear quadratic, for example, $F(m) := \int_{\mathbb{R}}\pig( x^2 + e^{-x^2} \pig)dm(x)$ in Introduction. We also remark that the displacement monotonicity does not imply Lasry-Lions monotonicity, and vice versa, but they are overlapped in some situations, see more discussion in the work of \cite{GM22mono}.
}

\label{rem LL mono}
\end{rem}
\begin{rem}
The claims of this work remain valid if we assume $F$ and $F_T$ depend on $x$ in addition to the measure argument $m$, that is $F=F(x,m):\mathbb{R}^n \times \mathcal{P}_2(\mathbb{R}^n) \to \mathbb{R}$ and $F_T=F_T(x,m):\mathbb{R}^n \times \mathcal{P}_2(\mathbb{R}^n)\to \mathbb{R}$, together with certain standard assumptions on the  regularity in the spatial argument. It does not bring in any technical difficulty but here we assume $F=F(m)$ for the sake of convenience.
\end{rem}
\begin{rem}
We provide the corresponding properties on the functionals $X\longmapsto F(X  \otimes  m)$
and $X\longmapsto F_{T}(X  \otimes  m)$ on $\mathcal{H}_{m}$ based on Assumptions \textbf{b(i)} to \textbf{b(v)}. Note that the below assumptions are necessary but not sufficient for Assumptions \textbf{b(i)} to \textbf{b(v)}, in some sense, the followings are weaker assumptions. For any $m \in \mathcal{P}_2(\mathbb{R}^n)$ and $ X \in \mathcal{H}_m$, 
\begin{equation}
|F(X  \otimes  m)|\leq c\pig(1+\|X\|^{2}_{\mathcal{H}_m}\pig)
\h{1pt},\h{10pt}
|F_{T}(X  \otimes  m)|\leq c_{T}\pig(1+\|X\|^{2}_{\mathcal{H}_m}\pig),
\label{assumption, bdd of F F_T}
\end{equation}
where $c$ and $c_T$ are the same as in \eqref{assumption, bdd of F F_T, no lift}. The functionals have Fr\'echet derivatives up to second order on $\mathcal{H}_{m}$ such that for any $X \in \mathcal{H}_m$,
\begin{align}
\textup{\textbf{B(i)}}& \h{30pt} \|D_{X}F(X  \otimes  m)\|_{\mathcal{H}_m}
\leq c\pig(1+\|X\|_{\mathcal{H}_m}\pig)
\h{1pt},\h{10pt}
\|D_{X}F_{T}(X  \otimes  m)\|_{\mathcal{H}_m}\leq c_{T}\pig(1+\|X\|_{\mathcal{H}_m}\pig); 
\label{assumption, bdd of DF, DF_T}\\
\textup{\textbf{B(ii)}}& \h{100pt}  \pigl\|D_{X}^{2}F(X  \otimes  m)\pigr\|_{\mathcal{H}_m}\leq c
\h{1pt},\h{10pt}
\pigl\|D_{X}^{2}F_{T}(X  \otimes  m)\pigr\|_{\mathcal{H}_m}\leq c_{T};
\label{assumption, bdd of D^2F, D^2F_T}\\
\textup{\textbf{B(iii)}}& \h{5pt} X\longmapsto D_{X}^{2}F(X  \otimes  m)(Y)\:\text{is continuous as a function from}\:\mathcal{H}_{m}\:\text{to }\mathcal{H}_{m},\text{ for each fixed}\:Y;
\label{assumption, cts of D^2F}\\
\textup{\textbf{B(iv)}}& \h{5pt} X\longmapsto D_{X}^{2}F_{T}(X  \otimes  m)(Y)\:\text{is continuous as a function from}\:\mathcal{H}_{m}\:\text{to }\mathcal{H}_{m},\text{ for each fixed }Y;
\label{assumption, cts of D^2F_T} \\
\textup{\textbf{B(v)}}& \h{5pt}\textup{\textbf{(a)}} \h{5pt}\text{if there is a sequence}\:\big\{X_{k}\big\}_{k \in \mathbb{N}} \text{ converging to $X$ in }\mathcal{H}_{m}, \text{ and the sequence $\{Y_{k}\}_{k \in \mathbb{N}}$ is }\nonumber\\
& \h{22pt} \text{ bounded in norm of $\mathcal{H}_m$, then }\nonumber\\
&\h{35pt}D_{X}^{2}F(X_{k}  \otimes  m)(Y_{k})-D_{X}^{2}F(X  \otimes  m)(Y_{k})\longrightarrow0,\text{ in }L^{1}(\Omega,\mathcal{A},\mathbb{P};L_{m}^{1}(\mathbb{R}^{n};\mathbb{R}^{n}))
, \h{5pt} \text{and }\footnotemark
\label{assumption, properties of D^2F, D^2F_T}\\
& \h{42pt} D_{X}^{2}F_{T}(X_{k}  \otimes  m)(Y_{k})-D_{X}^{2}F_{T}(X  \otimes  m)(Y_{k})\longrightarrow0,\text{ in }L^{1}(\Omega,\mathcal{A},\mathbb{P};L_{m}^{1}(\mathbb{R}^{n};\mathbb{R}^{n});
\nonumber\\
\textbf{\textcolor{white}{B(v)}}& \h{5pt}\textup{\textbf{(b)}} \h{10pt} 
\pigl\langle D_{X}^{2}F(X  \otimes  m)(Y),Y \pigr\rangle_{\mathcal{H}_m} 
\geq-c'\|Y\|^{2}_{\mathcal{H}_m}\h{10pt}
\text{ and }\h{10pt}
\pigl\langle D_{X}^{2}F_T(X  \otimes  m)(Y),Y \pigr\rangle_{\mathcal{H}_m}\geq-c'_{T}\|Y\|^{2}_{\mathcal{H}_m},
\label{assumption, convexity of D^2F, D^2F_T}
\mycomment{
\textbf{B(vi)}& \h{5pt}
\textcolor{red}{\text{the second-order Fr\'echet derivatives of $F(X \otimes m)$ and $F_T(X \otimes m)$ are uniformly Lipschitz in $\mathcal{H}_m$}.} 
\label{assumption, lip of 3rd derivative of F and FT}}
\end{align}

\end{rem}

\addtocounter{footnote}{0}
\footnotetext{It means that 
$\mathbb{E}\left[
\int_{\mathbb{R}^n}
\big|D_{X}^{2}F(X_{k}  \otimes  m)(Y_{k})-D_{X}^{2}F(X  \otimes  m)(Y_{k})\big|dm(x)\right] \longrightarrow 0$ as $k \to 0$.}
\vspace{-10pt}
\mycomment{
\begin{rem}
Here we provide a set of sufficient assumptions under the setting of $\mathcal{P}_2(\mathbb{R}^n)$, which is sufficient to ensure (\ref{assumption, bdd of F F_T})-(\ref{assumption, convexity of D^2F, D^2F_T}) to hold. Let $F(m)$, $F_T(m) : \mathcal{P}_2(\mathbb{R}^n) \longmapsto \mathbb{R}$ be the running and terminal cost functionals respectively. For brevity, only the assumptions for $F$ are presented, and similar conditions hold when we replace $F$ by $F_T$. First, it is sufficient to assume that
$$ |F(m)|\leq c \left(1+ \int_{\mathbb{R}^n}|x|^2 dm(x) \right),$$
for (\ref{assumption, bdd of F F_T}) to be valid. By using (\ref{eq:2-162}), the boundedness requirement in (\ref{assumption, bdd of DF, DF_T}) holds if for any $x\in \mathbb{R}^n$,
$$\left|\diff_x \dfrac{dF}{d\nu}(m)(x)\right|\leq c(1+|x|).$$
To fulfil (\ref{assumption, bdd of D^2F, D^2F_T}) and (\ref{assumption, bdd of mean of int of D^2F}), by (\ref{eq:2-232}), it is sufficient to assume that
\begin{equation}
\left|\diff^2_x \dfrac{dF}{d\nu}(m)(x)\right|\leq c\h{1pt},\h{10pt}
\left|\diff_x \diff_{\tilde{x}}\dfrac{d^{\h{.5pt}2}\h{-.7pt}  F}{d\nu^{2}}(m)(\tilde{x},x)\right|
\leq c\h{1pt}, \h{15pt} \text{ for any $x$, $\tilde{x} \in \mathbb{R}^n$.}
\label{eq:3-32}
\end{equation}
The continuity property in (\ref{assumption, cts of D^2F}) can be obtained by assuming that $
\diff^2_x \frac{dF}{d\nu}(m)(x)$ and
$\diff_x \diff_{\tilde{x}}\frac{d^{\h{.5pt}2}\h{-.7pt}  F}{d\nu^{2}}(m)(x,\tilde{x})$
are continuous in $(m,x)$ and $(m,x,\tilde{x})$, respectively, due to (\ref{eq:2-232}); the continuity of $
\diff^2_x \frac{dF}{d\nu}(m)(x)$ and
$\diff_x \diff_{\tilde{x}}\frac{d^{\h{.5pt}2}\h{-.7pt}  F}{d\nu^{2}}(m)(x,\tilde{x})$ together with the bounds in (\ref{eq:3-32}) imply the regularity requirement of (\ref{assumption, properties of D^2F, D^2F_T}).
\mycomment{
it is sufficient to assume that for any $\{\mu_k\}_{k \in \mathbb{N}} \subset \mathcal{P}_2(\mathbb{R}^n)$, $\{X_{k}\}_{k \in \mathbb{N}} \subset \mathcal{H}_m$ and $\mu \in \mathcal{P}_2(\mathbb{R}^n)$ and bounded $\{Y_{k}\}_{k \in \mathbb{N}} \subset \mathcal{H}_m$ such that $W_2(\mu_k,\mu) \longrightarrow 0$ and $X_k \to X$ in $\mathcal{H}_m$, then 
\begin{align*}
\mathbb{E}\int_{\mathbb{R}^n}\Bigg|
\diff^2_x \dfrac{dF}{d\nu}(\mu_k)(X_{k,x})Y_{k,x}
&+\widetilde{\mathbb{E}}\left[\int_{\mathbb{R}^{n}}\diff_x \diff_{\tilde{x}}\dfrac{d^{\h{.6pt}2}F}{d\nu^{2}}(\mu_k)(X_{k,x},\widetilde{X}_{k,\tilde{x}})\widetilde{Y}_{k,\tilde{x}}dm(\tilde{x})\right]\\
&- \diff^2_x \dfrac{dF}{d\nu}(\mu)(X_{k,x})Y_{k,x}
+\widetilde{\mathbb{E}}\left[\int_{\mathbb{R}^{n}}\diff_x \diff_{\tilde{x}}\dfrac{d^{\h{.6pt}2}F}{d\nu^{2}}(\mu)(X_{k,x},\widetilde{X}_{k,\tilde{x}})\widetilde{Y}_{k,\tilde{x}}dm(\tilde{x})\right]\Bigg|dm(x)
\longrightarrow0.
\end{align*}}
While (\ref{assumption, convexity of D^2F, D^2F_T}) can be fulfilled by assuming  that, for any $x$, $\xi \in \mathbb{R}^n$ and $m \in \mathcal{P}_2(\mathbb{R}^n)$, 
\mycomment{
\begin{equation}
\mathbb{E}\left[\displaystyle\int_{\mathbb{R}^n}\diff^2_x \dfrac{dF}{d\nu}(X \otimes  m)(X_x)X_x\cdot X_xdm(x)\right]\geq -c'\|X\|_{\mathcal{H}_m}^2\h{10pt} \text{and}   
\label{LL mono1}
\end{equation}

\begin{equation}
\widetilde{\mathbb{E}}\mathbb{E}
\left[\int_{\mathbb{R}^n}\int_{\mathbb{R}^n}\diff_x \diff_{\tilde{x}}\dfrac{d^{\h{.5pt}2}\h{-.7pt}  F}{d\nu^{2}}(X \otimes m)
(X_x,\widetilde{X}_{\tilde{x}})
X_x \cdot \widetilde{X}_{\tilde{x}}
dm(x)dm(\widetilde{x})\right]
\geq 0.
\label{LL mono2}
\end{equation}}

\begin{equation}
\diff^2_x \dfrac{dF}{d\nu}(m)(\xi)\xi\cdot \xi\geq -c'|\xi|^2\h{10pt} \text{and}
\label{LL mono1}
\end{equation}

\begin{equation}
\widetilde{\mathbb{E}}\mathbb{E}
\left[\int_{\mathbb{R}^n}\int_{\mathbb{R}^n}\diff_x \diff_{\tilde{x}}\dfrac{d^{\h{.5pt}2}\h{-.7pt}  F}{d\nu^{2}}(X \otimes m)
(X_x,\widetilde{X}_{\tilde{x}})
X_x \cdot \widetilde{X}_{\tilde{x}}
dm(x)dm(\widetilde{x})\right]
\geq 0,
\label{LL mono2}
\end{equation}

\end{rem}
}

\mycomment{
\begin{rem}\label{rem LL mono}
Suppose $F(m) = \displaystyle\int_{\mathbb{R}^n} f(x,m) dm(x)$ with $f : \mathbb{R}^n \times \mathcal{P}_2(\mathbb{R}^n) \longrightarrow \mathbb{R}$ having derivatives $\diff_x f(x,m)$ and $\diff_x^2 f (x,m)$, and linear functional derivative $\dfrac{d\h{.3pt}f}{d\nu}(x,m)(\widetilde{x})$ such that its derivatives $\diff_x \dfrac{d\h{.3pt}f}{d\nu}(x,m)(\widetilde{x})$ and $\diff_{\tilde{x}} \dfrac{d\h{.3pt}f}{d\nu}(x,m)(\widetilde{x})$ are all continuous in $x$ for each fixed $m \in \mathcal{P}_2(\mathbb{R}^n)$ and $\widetilde{x} \in \mathbb{R}^n$. Define $\phi(x,m):=\dfrac{d F}{d\nu}(m)(x) = f(x,m) + \displaystyle\int_{\mathbb{R}^n} \dfrac{d\h{.3pt}f}{d\nu}(\widetilde{x},m)(x)dm(\widetilde{x})$. By noting (\ref{eq:2-231}), Assumption \textbf{B(v)(b)}'s (\ref{assumption, convexity of D^2F, D^2F_T}) is equivalent to
\begin{align*}
\pigl\langle D_{X}^{2}F(X  \otimes  m)(Y_x),Y_x \pigr\rangle_{\mathcal{H}_m}
=\:&\mathbb{E}\left[ \int_{\mathbb{R}^n}
\diff^2_x \dfrac{dF}{d\nu}(X  \otimes  m)(X_{x})Y_{x}\cdot Y_{x} dm(x)\right]\\
&+\mathbb{E}\widetilde{\mathbb{E}}\left[ \int_{\mathbb{R}^n}\int_{\mathbb{R}^{n}}\diff_x \diff_{\tilde{x}}\dfrac{d^{\h{.5pt}2}\h{-.7pt}  F}{d\nu^{2}}(X  \otimes  m)(X_{x},\widetilde{X}_{\tilde{x}})\widetilde{Y}_{\tilde{x}}
\cdot Y_x dm(\tilde{x})dm(x)
\right]\\
=\:&\mathbb{E}\left[ \int_{\mathbb{R}^n}
\diff^2_x \phi(X_{x},X  \otimes  m)Y_{x}\cdot Y_{x} dm(x)\right]\\
&+\mathbb{E}\widetilde{\mathbb{E}}\left[ \int_{\mathbb{R}^n}\int_{\mathbb{R}^{n}}
\diff_x \diff_{\tilde{x}}\dfrac{d \phi}{d\nu}(X_x,X  \otimes  m)(\widetilde{X}_{\tilde{x}})\widetilde{Y}_{\tilde{x}} \cdot
Y_x dm(\tilde{x})dm(x)
\right]\\
\geq\:& -c' \|Y\|^2_{\mathcal{H}_m},
\end{align*}
as long as $\diff^2_x\phi(x,m)y\cdot y \geq -c'|y|^2$ for all $m \in \mathcal{P}_2(\mathbb{R}^n)$, $x \in \mathbb{R}^n$, $y \in \mathbb{R}^n$, and 
\begin{equation}
\mathbb{E}\widetilde{\mathbb{E}}\left[ \int_{\mathbb{R}^n}\int_{\mathbb{R}^{n}}
\diff_x \diff_{\tilde{x}}\dfrac{d \phi}{d\nu}(X_x,X  \otimes  m)(\widetilde{X}_{\tilde{x}})\widetilde{Y}_{\tilde{x}} \cdot
Y_x dm(\tilde{x})dm(x)
\right]\geq 0.
\label{LL mono}
\end{equation} 
The condition in (\ref{LL mono}) is equivalent to Lasry-Lions monotonicity assumption (2.5) in \cite{CDLL19} now applying to $\phi(x,m)$.
\end{rem}}

}

\subsection{Differentiability and Convexity of Objective Function}
We now study the {\color{black} differentiability and convexity} of the objective functional  $v_{tX} \longmapsto J_{tX}(v_{tX})$ as one on the Hilbert space $L_{\mathcal{W}_{tX}}^{2}(t,T;\mathcal{H}_{m})$, we have 
\begin{lem}
\label{lem G derivative of J}Under the assumptions (\ref{assumption, bdd of l, lx, lv}), (\ref{assumption, bdd of h, hx, hxx}) 
and (\ref{assumption, bdd of DF, DF_T}), for any $v_{tX} \in L^2_{\mathcal{W}_{tX}}(t,T;\mathcal{H}_{m})$, the functional $J_{tX}(v_{tX})$ has a G\^ateaux derivative, given by:

\begin{equation}
D_{v}J_{tX}(v_{tX})(s)=l_{v}(\mathbb{X}_{tX}(s),v_{tX}(s))+\mathbb{Z}_{tX}(s)
\h{1pt},\h{10pt}  \text{for $s \in [t,T]$,} 
\label{eq:3-5}
\end{equation}
where $\mathbb{Z}_{tX}(s)\in L_{\mathcal{W}_{tX}}^{2}(t,T;\mathcal{H}_{m})$, being the solution to the backward stochastic differential equation (BSDE)
\begin{equation}
\left\{
\begin{aligned}
-d \h{1pt} \mathbb{Z}_{tX}(s)
&=\Big[l_{x}\pig(\mathbb{X}_{tX}(s),v_{tX}(s)\pig)+D_{X}F(\mathbb{X}_{tX}(s)  \otimes  m)\Big]ds
-\sum_{j=1}^{n}\mathbbm{r}_{tX,j}(s)dw_{j}(s)\h{1pt},\h{10pt}
\text{ for $s\in [t,T]$;}\\
\mathbb{Z}_{tX}(T) &= h_x\big(\mathbb{X}_{tX}(T)\big) 
+D_X F_T\big(\mathbb{X}_{tX}(T)  \otimes  m\big), 
\end{aligned}\right.
\label{def, backward SDE}
\end{equation}
for the adapted process $\mathbbm{r}_{tX,j}(s)\in L_{\mathcal{W}_{tX}}^{2}(t,T;\mathcal{H}_{m})$ with $j=1,2,\ldots,n$. 
\end{lem}
Its proof can be found in \hyperref[{app, lem G derivative of J}]{Appendix}. We denote $\mathbbm{r}_{t X}(s):=
\pig( \mathbbm{r}_{t X,1}(s), \mathbbm{r}_{t X,2}(s),\ldots,
\mathbbm{r}_{t X,n}(s)\pig)$ the matrix having $\mathbbm{r}_{t X,j}(s)$ as its $j$-th column vector. We next state the existence and uniqueness results of the optimal control.
\begin{prop}[\bf \color{black}{Existence and Uniqueness}]
We assume (\ref{assumption, bdd of l, lx, lv}),
(\ref{assumption, bdd of lxx, lvx, lvv}),
(\ref{assumption, bdd of h, hx, hxx}), (\ref{assumption, convexity of l}), (\ref{assumption, convexity of h}), (\ref{assumption, bdd of F F_T}), (\ref{assumption, bdd of DF, DF_T}), (\ref{assumption, bdd of D^2F, D^2F_T}), (\ref{assumption, convexity of D^2F, D^2F_T}) and 
\begin{equation}
c_0:=\lambda
- (c'_T+c'_h)_+T
- \left(c'_{l}+c'\right)_+\dfrac{T^2}{2}>0,
\label{uniqueness condition}
\end{equation}
\color{black}{
then the functional $J_{tX}(v_{tX})$ is strictly convex in $v_{tX}$, and coercive in $v_{tX}$ in the sense that $J_{tX}(v_{tX})\longrightarrow\infty$ as $\int_{t}^{T}\|v_{tX}(s)\|^{2}_{\mathcal{H}_{m}}ds$ $\rightarrow\infty$, here $a_+:=\max\{a,0\}$ for any $a \in \mathbb{R}$. Consequently, there is a unique optimal control $u_{tX}(s)\in L_{\mathcal{W}_{tX}}^{2}(t,T;\mathcal{H}_{m})$ to the control problem that minimizes $J_{tX}(v_{tX})$ subject to $v_{tX}(s) \in L_{\mathcal{W}_{tX}}^{2}(t,T;\mathcal{H}_{m})$ and $\mathbb{X}_{tX}(s;v_{tX})$ satisfying (\ref{eq:3-501}).} 
\label{prop. convex of J}
\end{prop}

Its proof is put in \hyperref[{app, prop. convex of J}]{Appendix}. {\color{black} In the rest of the article, we assume that the condition (\ref{uniqueness condition}) always holds.}

\subsection{Necessary and Sufficient Condition of Optimality}\label{sec. Necessary and Sufficient Condition of Optimality}

According to Proposition \ref{prop. convex of J}, there is a {\color{black} unique optimal control $u_{tX}(s)$ which satisfies the first condition $D_{v}J_{tX}(u_{tX}) \equiv 0$} in Lemma \ref{lem G derivative of J}. Particularly, the optimal control $u_{tX}(s)$ is a feedback one; to this end, we define the Lagrangian $L:\mathbb{R}^{n} \times\mathbb{R}^{n}\times\mathbb{R}^{n}\rightarrow \mathbb{R}$ by
\begin{equation}
L(x,v,p):=l(x,v)+v\cdot p.
\label{eq:4-41}
\end{equation}
Then, from Assumption \textbf{A(v)}'s (\ref{assumption, convexity of l}), for each $x$, $p \in \mathbb{R}^n$, the function $v\longmapsto L(x,v,p)$ is strictly convex and coercive in $v$, therefore it attains its unique minimum at $v=u(x,p).$ 
We further define the Hamiltonian 
\begin{equation}
H(x,p):=\inf_{v \in \mathbb{R}^n}L(x,v,p)=l\pig(x,u(x,p)\pig)+u(x,p)\cdot p
\label{def. Hamiltonian}
\end{equation}
which is clearly $C^{1}(\mathbb{R}^n \times \mathbb{R}^n)$ with derivatives $H_{x}(x,p)=l_{x}(x,u(x,p))$ and $H_{p}(x,p)=u(x,p)$. Therefore, we can write $u_{tX}(s)$ in the following feedback form: 

\begin{equation}
u_{tX}(s)=H_{p}\pig(\mathbb{Y}_{tX}(s),D_{X}V(\mathbb{Y}_{tX}(s)  \otimes  m,s)\pig)=H_{p}\pig(\mathbb{Y}_{tX}(s),\mathbb{Z}_{tX}(s)\pig),
\label{eq:4-45}
\end{equation}
where we shall verify the last equality, that $D_{X}V(\mathbb{Y}_{tX}(s)  \otimes  m,s)=\mathbb{Z}_{tX}(s)$ in (\ref{D_X V(s)=Z(s)}). Calling $\mathbb{Y}_{tX}(s)$
the optimal state and $\pig(\mathbb{Z}_{tX}(s),\mathbbm{r}_{tX}(s)\pig)$ the solution of
the BSDE (\ref{def, backward SDE}) under the influence of $\mathbb{Y}_{tX}$, {\color{black} we have the following proposition.
\begin{prop}
For $t \in [0,T)$ and $X \in \mathcal{H}_m$ measurable to $\mathcal{W}^{t}_0$. The process $u_{tX}$ is the optimal control to the control problem that minimizes $J_{tX}(v_{tX})$ subject to $v_{tX}(s) \in L_{\mathcal{W}_{tX}}^{2}(t,T;\mathcal{H}_{m})$ and $\mathbb{X}_{tX}(s;v_{tX})$ satisfying (\ref{eq:3-501}) if and only if the process $\pig(\mathbb{Y}_{tX}(s), \mathbb{Z}_{tX}(s), u_{tX}(s), \mathbbm{r}_{tX,j}(s)\pig)$ is the solution in $L_{\mathcal{W}_{tX}}^{2}(t,T;\mathcal{H}_{m})$ to the system, for $s \in [t,T]$, 
\begin{empheq}[left=\empheqbiglbrace]{align}
\mathbb{Y}_{tX}(s)=&\:X+\int_{t}^{s}u_{tX}(\tau)d\tau+\eta(w(s)-w(t));
\label{forward FBSDE}\\
\mathbb{Z}_{tX}(s)=&\:h_{x}(\mathbb{Y}_{tX}(T))+D_{X}F_{T}(\mathbb{Y}_{tX}(T)  \otimes  m)
+\int^T_s\Big[l_x\pig(\mathbb{Y}_{tX}(\tau),u_{tX}(\tau)\pig)+D_{X}F(\mathbb{Y}_{tX}(\tau)  \otimes  m)\Big]d\tau
\label{backward FBSDE}\\
&-\int^T_s\sum_{j=1}^{n}\mathbbm{r}_{tX,j}(\tau)dw_{j}(\tau);
\nonumber
\end{empheq}

\begin{flalign}
\text{subject to}&&l_v\pig(\mathbb{Y}_{tX}(s),u_{tX}(s)\pig)+\mathbb{Z}_{tX}(s)=0.&&
\label{1st order condition}
\end{flalign}
\label{prop eq. condition of optmality}
\end{prop}
}

\begin{proof}
Suppose $u_{tX}$ is the optimal control, Lemma \ref{lem G derivative of J} and Theorem 7.2.12 in \cite{DM07} imply
$$
D_{v}J_{tX}(v_{tX})(s)=l_{v}(\mathbb{X}_{tX}(s),v_{tX}(s))+\mathbb{Z}_{tX}(s) = 0,
$$
where the process $\pig(\mathbb{Y}_{tX}(s), \mathbb{Z}_{tX}(s), u_{tX}(s), \mathbbm{r}_{tX,j}(s)\pig)$ satisfies \eqref{forward FBSDE}-\eqref{1st order condition}.

Suppose $u_{tX}$ solves the first order condition \eqref{1st order condition} such that $l_{v}(\mathbb{X}_{tX}(s),v_{tX}(s))+\mathbb{Z}_{tX}(s) =0$, where $\pig(\mathbb{Y}_{tX}(s), \mathbb{Z}_{tX}(s),$\\$ \mathbbm{r}_{tX,j}(s)\pig)$ solves \eqref{forward FBSDE}-\eqref{backward FBSDE} correspondingly. If $J_{tX}(v)$ does not attain its minimum value at $v=u_{tX}$, then there is $u^\dagger$ such that $J_{tX}(u^\dagger) < J_{tX}(u_{tX})$. The convexity of $J_{tX}(v)$ in $v$ implies that $J_{tX}(\theta u^\dagger + (1-\theta) u_{tX}) \leq \theta J_{tX}( u^\dagger)+(1-\theta)J_{tX}(  u_{tX})$ for any $\theta \in [0,1]$, that is,
$$ \dfrac{J_{tX}(\theta u^\dagger + (1-\theta) u_{tX}) - J_{tX}( u^\dagger)}{\theta} \leq J_{tX}(u^\dagger) - J_{tX}(u_{tX})<0.$$
It contradicts that $\left.\dfrac{d}{d\theta} J_{tX}(u_{tX}+\theta(u^\dagger- u_{tX}) ) \right|_{\theta=0}=0$.
\end{proof}

We remark that there is a unique $u(x,z)$ such that $l_v\big(x,u(x,z)\big)+z=0$ for any $x$, $z \in \mathbb{R}^n$, based on Assumption \textbf{A(v)}'s \eqref{assumption, convexity of l} and the inverse function theorem. Moreover, as discussed, $L_{\mathcal{W}_{tX}}^{2}(t,T;\mathcal{H}_{m})$ is
isometric to $L_{\mathcal{W}_{t}}^{2}(t,T;\mathcal{H}_{X  \otimes  m})$, there exist measurable random fields  $\mathbb{Y}_{t\xi}(s)$, $\mathbb{Z}_{t\xi}(s)$, $u_{t\xi}(s)$, $\mathbbm{r}_{t\xi,j}(s) \in L_{\mathcal{W}_{t}}^{2}(t,T;\mathcal{H}_{X_t  \otimes  m})$
such that $\mathbb{Y}_{tX}(s)=\mathbb{Y}_{t\xi}(s)\Big|_{\xi=X},\:\mathbb{Z}_{tX}(s)=\mathbb{Z}_{t\xi}(s)\Big|_{\xi=X}$, $u_{tX}(s)=u_{t\xi}(s)\Big|_{\xi=X}$, $\mathbbm{r}_{tX,j}(s)=\mathbbm{r}_{t\xi,j}(s)\Big|_{\xi=X}$.
These random fields $\mathbb{Y}_{t\xi}(s),\mathbb{Z}_{t\xi}(s),u_{t\xi}(s),\mathbbm{r}_{t\xi,j}(s)$
are solutions to the following FBSDE system, for $s \in [t,T]$, 
\begin{empheq}[left=\empheqbiglbrace]{align}
\mathbb{Y}_{t\xi}(s)=&\:\xi+\int_{t}^{s}u_{t\xi}(\tau)d\tau+\eta(w(s)-w(t)),
\label{eq:3-9}\\
\mathbb{Z}_{t\xi}(s)=&\:h_{x}(\mathbb{Y}_{t\xi}(T))+D_{X}F_{T}(\mathbb{Y}_{t\bigcdot}(T)  \otimes (X  \otimes  m))\label{eq:3-10}\\
&+\int^T_s\Big[l_x\pig(\mathbb{Y}_{t\xi}(\tau),u_{t\xi}(\tau)\pig)+D_{X}F(\mathbb{Y}_{t\bigcdot}(\tau)  \otimes (X  \otimes  m))\Big]d\tau
-\int^T_s\sum_{j=1}^{n}\mathbbm{r}_{t\xi,j}(\tau)dw_{j}(\tau),\nonumber
\end{empheq}
subject to $
l_v\pig(\mathbb{Y}_{t\xi}(s),u_{t\xi}(s)\pig)+\mathbb{Z}_{t\xi}(s)=0.$ We notice that $\mathbb{Y}_{t\xi}(s),\mathbb{Z}_{t\xi}(s),u_{t\xi}(s),\mathbbm{r}_{t\xi,j}(s)$
depend on $m$, at time $t$, only through $X  \otimes  m=X_t  \otimes  m.$
\begin{rem}
Equations (\ref{forward FBSDE})-(\ref{1st order condition}) are defined only up to a null set. If we consider infinitely many random variables $X$, the processes may not be defined. By considering the deterministic $\xi$ in (\ref{eq:3-9})-(\ref{eq:3-10}), we avoid this kind of measurability problem by putting $\xi=X$.
\end{rem}

We can express the value function 

\begin{equation}
\begin{aligned}
V(X  \otimes  m,t)=\:&\int_{t}^{T}\mathbb{E}
\left[\int_{\mathbb{R}^{n}}l\pig(\mathbb{Y}_{tX}(s),u_{tX}(s)\pig)dm(x)ds\right]
+\int_{t}^{T}F(\mathbb{Y}_{tX}(s)  \otimes  m)ds\\
&+\mathbb{E}\left[\int_{\mathbb{R}^{n}}h(\mathbb{Y}_{tX}(T))dm(x)\right]
+F_{T}(\mathbb{Y}_{tX}(T)  \otimes  m).
\end{aligned}
\label{def value function}
\end{equation}
This quantity depends only on the probability measure $X  \otimes  m=X_t  \otimes  m$ and $t.$ It can be written as follows

\begin{equation}
\begin{aligned}
V(X  \otimes  m,t)=\:&\int_{t}^{T}\mathbb{E}\left[\int_{\mathbb{R}^{n}}l\pig(\mathbb{Y}_{t\xi}(s),u_{t\xi}(s)\pig)d\big(X  \otimes  m\big)(\xi)\right]
+\int_{t}^{T}F(\mathbb{Y}_{t \cdot}(s)  \otimes (X  \otimes  m))ds\\
&+\mathbb{E}\left[\int_{\mathbb{R}^{n}}h(\mathbb{Y}_{t\xi}(T))d\big(X  \otimes  m\big)(\xi)\right]
+F_{T}(\mathbb{Y}_{t \cdot }(T)  \otimes (X  \otimes  m)).
\end{aligned}
\label{def value function, with xi}
\end{equation}

\begin{rem}
\label{rem4-1}We may enlarge the subspace of $L^{2}(t,T;\mathcal{H}_{m})$ of controls in which $u_{tX}$ remains optimal; to this point, this will facilitate the relevant comparison principle and the study of the regularity
of $X\longmapsto V(X  \otimes  m,t)$ in Sections \ref{sec. Analytic Properties of Value Function} to \ref{sec. master equation}, by introducing an additional initial element in $\mathcal{H}_{m}$, denoted by $\tilde{X}$ which is also independent
of $\mathcal{W}_{t}$ but not necessary with $X$. To see this, define the $\sigma$-algebras $\mathcal{W}_{t X\tilde{X}}^{s}=\sigma(X,\tilde{X})\bigvee\mathcal{W}_{t}^{s}$
and the filtration $\mathcal{W}_{tX\tilde{X}}:=\mathcal{W}^T_{t X\tilde{X}}$ generated by the $\sigma$-algebras
$\mathcal{W}_{tX\tilde{X}}^{s}$ for all $s \in [t,T]$. If we change the space of controls
from $L_{\mathcal{W}_{tX}}^{2}(t,T;\mathcal{H}_{m})$ to $L_{\mathcal{W}_{tX\tilde{X}}}^{2}(t,T;\mathcal{H}_{m})$, one
checks easily that the system of Equations (\ref{forward FBSDE}), (\ref{backward FBSDE}) and (\ref{1st order condition}) solve the necessary conditions of optimality of
the control problem still with the enlarged $\sigma$-algebras, and by the uniqueness
of the solution of the necessary conditions, the optimal control remains the same as before, except that we can now interpret the solution on the enlarged $\sigma$-algebras with all formulae unchanged. Of course, instead of adding one $\tilde{X}$, one could
add a finite number of elements of $\mathcal{H}_{m}$, as long as all of these are independent
of $\mathcal{W}_{t}.$
\end{rem}

\subsection{Regularity of Solutions to FBSDE (\ref{forward FBSDE})-(\ref{1st order condition}) and its Jacobian Flow}\label{sec. J Flow}
We now give the upper bounds of $\mathbb{Y}_{t X}(s), \mathbb{Z}_{t X}(s), u_{t X}(s), \mathbbm{r}_{t X,j}(s)$.
\begin{prop}[\bf Bounds of FBSDE's Solution]
\label{prop bdd of Y Z u r} Under the assumptions of Proposition \ref{prop. convex of J}, for any $X \in \mathcal{H}_m$ measurable to $\mathcal{W}^{t}_0$ and $s\in(t,T)$, the solution quadruple $\pig(\mathbb{Y}_{t X}(s), \mathbb{Z}_{t X}(s), u_{t X}(s), \mathbbm{r}_{t X,j}(s)\pig)$ to the FBSDE (\ref{forward FBSDE})-(\ref{1st order condition}) on $[t,T]$ satisfies
\begin{align} \|\mathbb{Y}_{t X}(s)\|_{\mathcal{H}_m},\:\|\mathbb{Z}_{t X}(s)\|_{\mathcal{H}_m},\:\|u_{t X}(s)\|_{\mathcal{H}_m},\:
\left[\displaystyle\sum^n_{j=1}\int_{t}^{T}\pigl\|\mathbbm{r}_{t X,j}(s)\pigr\|^{2}_{\mathcal{H}_m}ds\right]^{1/2}
\leq C_4\pig(1+\|X\|_{\mathcal{H}_m}\pig),
\label{bdd Y, Z, u, r}
\end{align}
where $C_{4}$ is a positive constant depending only on $n$, $\lambda$, $\eta$, $c$, $c_l$, $c_h$, $c_T$, $c'$, $c_l'$, $c_h'$, $c_T'$ and $T$.
\end{prop}
\begin{proof}
For simplicity, without any cause of ambiguity, we omit the subscripts $tX$ in $\mathbb{Y}_{tX}(s)$, $\mathbb{Z}_{tX}(s)$, $u_{tX}(s)$ and $\mathbbm{r}_{tX}(s)$. Recalling (\ref{backward FBSDE}) that describes the dynamics of $\mathbb{Z}(s)$ and (\ref{1st order condition}) that characterises $\mathbb{Z}(s)$, an application of product rule gives
\begin{align*}
d\Big\langle \mathbb{Z}(s),\mathbb{Y}(s)-\eta\big(w(s)-w(t)\big)\Big\rangle_{\mathcal{H}_m}
\h{-5pt}=&-\Big\langle l_{x}(\mathbb{Y}(s),u(s))+D_{X}F(\mathbb{Y}(s)  \otimes  m),\mathbb{Y}(s)-\eta\big(w(s)-w(t)\big)\Big\rangle_{\mathcal{H}_m} \h{-3pt} ds\\
&-\Big\langle l_{v}(\mathbb{Y}(s),u(s)),u(s)\Big\rangle_{\mathcal{H}_m} \h{-3pt} ds.
\end{align*}
Integrating from $t$ to $T,$ we obtain 

\begin{equation}
\begin{aligned}
\h{-4pt}\pigl\langle \mathbb{Z}(t),X\pigr\rangle_{\mathcal{H}_m}
=\:&\int_{t}^{T}\Big\langle l_{v}(\mathbb{Y}(s),u(s)),u(s)\Big\rangle_{\mathcal{H}_m}ds
+\int_{t}^{T}\Big\langle l_{x}(\mathbb{Y}(s),u(s))+D_{X}F(\mathbb{Y}(s)  \otimes  m),\mathbb{Y}(s)\Big\rangle_{\mathcal{H}_m}ds
\\
&-\int_{t}^{T}\Big\langle l_{x}(\mathbb{Y}(s),u(s))+D_{X}F(\mathbb{Y}(s)  \otimes  m),
\eta\big(w(s)-w(t)\big)\Big\rangle_{\mathcal{H}_m}ds\\
&+\Big\langle h_{x}(\mathbb{Y}(T))+\:D_XF_{T}(\mathbb{Y}(T)  \otimes  m),\mathbb{Y}(T)-\eta\big(w(T)-w(t)\Big\rangle_{\mathcal{H}_m}.
\end{aligned}
\label{eq:ApB2}
\end{equation}
On the other hand, using the definition of dynamic of $\mathbb{Z}(s)$ in (\ref{backward FBSDE}), we see
\begin{equation}
\mbox{\fontsize{10.3}{10}\selectfont\(
\pigl\langle \mathbb{Z}(t),X\pigr\rangle_{\mathcal{H}_m}
=\left\langle \displaystyle\int_{t}^{T}
l_{x}(\mathbb{Y}(s),u(s))+D_{X}F(\mathbb{Y}(s)  \otimes  m)ds,X \right\rangle_{\mathcal{H}_m}
+\Big\langle h_{x}(\mathbb{Y}(T))+\:D_XF_{T}(\mathbb{Y}(T)  \otimes  m),X\Big\rangle_{\mathcal{H}_m}\)}.
\label{eq:ApB3}
\end{equation}
Therefore, subtracting (\ref{eq:ApB3}) from (\ref{eq:ApB2}), we have

\begin{equation}
\begin{aligned}
0=\:&\int_{t}^{T}\Big\langle l_{v}(\mathbb{Y}(s),u(s)),u(s)\Big\rangle_{\mathcal{H}_m}ds
+\int_{t}^{T}\Big\langle l_{x}(\mathbb{Y}(s),u(s)) +D_{X}F(\mathbb{Y}(s)  \otimes  m),\mathbb{Y}(s)
-X
-\eta\big(w(s)-w(t)\big)\Big\rangle_{\mathcal{H}_m}ds\\
&+\Big\langle h_{x}(\mathbb{Y}(T))+\:D_XF_{T}(\mathbb{Y}(T)  \otimes  m),\mathbb{Y}(T)-X-\eta\big(w(T)-w(t)\Big\rangle_{\mathcal{H}_m}
\end{aligned}
\label{eq:ApB4}
\end{equation}
After telescoping, we obtain 
\begin{align*}
0=&\:\int_{t}^{T}\Big\langle l_{v}(\mathbb{Y}(s),u(s))-l_{v}(X+\eta(w(s)-w(t)),0),u(s)\Big\rangle_{\mathcal{H}_m}ds\\
&+\int_{t}^{T}\Big\langle l_{x}(\mathbb{Y}(s),u(s))-
l_{x}(X+\eta(w(s)-w(t)),0)+D_{X}F(\mathbb{Y}(s)  \otimes  m)-D_{X}F((X+\eta(w(s)-w(t))) \otimes  m),\\
&\h{350pt}\mathbb{Y}(s)-X-\eta\big(w(s)-w(t)\big)\Big\rangle_{\mathcal{H}_m}ds\\
&+\Big\langle h_{x}(\mathbb{Y}(T))-h_{x}\pig(X+\eta\big(w(T)-w(t)\big)\pig)
+D_XF_{T}(\mathbb{Y}(T)  \otimes  m)
-D_{X}F_T \pig((X+\eta(w(T)-w(t)))  \otimes  m\pig),\\
&\h{350pt}\mathbb{Y}(T)-X-\eta(w(T)-w(t))\Big\rangle_{\mathcal{H}_m}\\
&+\int_{t}^{T}\Big\langle l_{v}(X+\eta(w(s)-w(t)),0),u(s)\Big\rangle_{\mathcal{H}_m}ds\\
&+\int_{t}^{T}\Big\langle l_{x}(X+\eta(w(s)-w(t)),0)+D_{X}F((X+\eta(w(s)-w(t)))  \otimes  m),\mathbb{Y}(s)-X-\eta\big(w(s)-w(t)\big)\Big\rangle_{\mathcal{H}_m}ds\\
&+\left\langle h_{x}\pig(X+\eta\big(w(T)-w(t)\big)\pig)
+D_{X}F_T\pig((X+\eta(w(T)-w(t)))  \otimes  m\pig),\int_{t}^{T}u(s)ds\right\rangle_{\mathcal{H}_m}.
\end{align*}
From the mean value theorem, Assumptions \textbf{A(iii)}'s (\ref{assumption, bdd of h, hx, hxx}), \textbf{A(iv)}'s (\ref{assumption, convexity of l}), \textbf{A(vi)}'s (\ref{assumption, convexity of h}) and \textbf{B(i)}'s (\ref{assumption, bdd of DF, DF_T}), 
we have 

\begingroup
\allowdisplaybreaks
\begin{align*}
\lambda\int_{t}^{T}\|u(s)\|^{2}_{\mathcal{H}_m}ds
\leq\:& 
(c_{l}'+c')\int_{t}^{T}\|\mathbb{Y}(s)-X-\eta(w(s)-w(t))\|^{2}_{\mathcal{H}_m}ds
+(c_h'+c_{T}')\left\|\int_{t}^{T}u(s)ds\right\|_{\mathcal{H}_m}^{2}\\
&+c_{l}\pig(1+\|X+\eta(w(T)-w(t))\|_{\mathcal{H}_m}\pig)
\displaystyle\int_{t}^{T}\|u(s)\|_{\mathcal{H}_m}ds\\
&+(c_{l}+c)\pig(1+\|X+\eta(w(T)-w(t))\|_{\mathcal{H}_m}\pig)\int_{t}^{T}\|\mathbb{Y}(s)-X-\eta(w(s)-w(t))\|_{\mathcal{H}_m}ds\\
&+(c_{h}+c_{T})\pig(1+\|X+\eta(w(T)-w(t))\big\|_{\mathcal{H}_m}\pig)
\int_{t}^{T}\|u(s)\|_{\mathcal{H}_m}ds.
\end{align*}
\endgroup
Applications of Cauchy-Schwarz and Young's inequalities further give
\begin{equation*}
\begin{aligned}
\lambda\int_{t}^{T}\|u(s)\|^{2}_{\mathcal{H}_m}ds
\leq\:&(c_{l}'+c')_+\dfrac{T^2}{2}\int_{t}^{T}\|u(s)\|^{2}_{\mathcal{H}_m}ds
+\left[(c_h'+c_{T}')_++\dfrac{1}{4\kappa_4}\right]T\int_{t}^{T}\|u(s)\|_{\mathcal{H}_m}^2ds\\
&+\kappa_4 c_{l}^2\pig(1+\|X\|_{\mathcal{H}_m}+\big\|\eta(w(T)-w(t))
\big\|_{\mathcal{H}_m}\pig)^2\\
&+\kappa_5(c_{l}+c)^2\pig(1+\|X\|_{\mathcal{H}_m}+\big\|\eta(w(T)-w(t))
\big\|_{\mathcal{H}_m}\pig)^2
+\dfrac{T^2}{4\kappa_5}\int_{t}^{T}\|u(s)\|^2_{\mathcal{H}_m}ds\\
&+\kappa_6 (c_{h}+c_{T})^2\pig(1+\|X\|_{\mathcal{H}_m}+\big\|\eta(w(T)-w(t))
\big\|_{\mathcal{H}_m}\pigr)^2
+\dfrac{T}{4\kappa_6}\int_{t}^{T}\|u(s)\|_{\mathcal{H}_m}^2ds,
\end{aligned}
\end{equation*}
for some positive constants $\kappa_i$ determined later, with $i=3,4,5,6$. It is equivalent to
\begin{equation}
\begin{aligned}
&\Bigg\{\lambda 
-\left[(c_h'+c_{T}')_++\dfrac{1}{4\kappa_4}+\dfrac{T}{4\kappa_5}+\dfrac{1}{4\kappa_6}\right]T-(c_{l}'+c')_+\dfrac{T^2}{2} \Bigg\}
\int_{t}^{T}\|u(s)\|^{2}_{\mathcal{H}_m}ds\\
&\leq \kappa_4 c_{l}^2\pig(1+\|X\|_{\mathcal{H}_m}+|\eta|\sqrt{nT}\pigr)^2
+\kappa_5(c_{l}+c)^2\pig(1+\|X\|_{\mathcal{H}_m}+|\eta|\sqrt{nT}\pigr)^2
+\kappa_6 (c_{h}+c_{T})^2\pig(1+\|X\|_{\mathcal{H}_m}+|\eta|\sqrt{nT}\pigr)^2\\
&\leq 2\pig[\kappa_4 c_{l}^2 + \kappa_5(c_{l}+c)^2
+\kappa_6 (c_{h}+c_{T})^2\pig]
\|X\|_{\mathcal{H}_m}^2
+4\pig[\kappa_4 c_{l}^2 + \kappa_5(c_{l}+c)^2
+\kappa_6 (c_{h}+c_{T})^2\pig](1+|\eta|^2nT).
\end{aligned}
\label{4055}
\end{equation}
Taking suitable $\kappa_i$ with $i=3,4,5,6$ and using (\ref{uniqueness condition}) to ensure that the coefficient in the first line of (\ref{4055}) is strictly positive, we thus conclude that 

\begin{equation}
\int_{t}^{T}\|u(s)\|^{2}_{\mathcal{H}_m}ds
\leq A_1\pig(1+\|X\|^{2}_{\mathcal{H}_m}\pig),
\label{bdd int |u|^2}
\end{equation}
for some $A_1>0$ depending on $n$, $\lambda$, $\eta$, $c$, $c_l$, $c_h$, $c_T$, $c'$, $c_l'$, $c_h'$, $c_T'$ and $T$. Next, from the dynamics of $\mathbb{Y}(s)$ in (\ref{forward FBSDE}), it yields that
\begin{equation}
\|\mathbb{Y}(s)-X\|^{2}_{\mathcal{H}_m}
\leq(s-t)(1+\kappa_7)\int_{t}^{T}\|u(\tau)\|^{2}_{\mathcal{H}_m}d\tau
+n|\eta|^{2}(s-t)\left(1+\dfrac{1}{\kappa_7}\right),
\label{est. |Ys-X|^2}
\end{equation}
for some $\kappa_7>0$. Integrating (\ref{est. |Ys-X|^2}) with respect to $s$, we obtain
\begin{equation}
\int_{t}^{T}\|\mathbb{Y}(s)-X\|^{2}_{\mathcal{H}_m}ds
\leq\dfrac{(1+\kappa_7)T^{2}}{2}\int_{t}^{T}\|u(\tau)\|^{2}_{\mathcal{H}_m}d\tau+\dfrac{n|\eta|^{2}T^{2}}{2}
\left(1+\dfrac{1}{\kappa_7}\right).
\label{eq:ApB6}
\end{equation}
From (\ref{bdd int |u|^2}) and  (\ref{est. |Ys-X|^2}), we deduce that 
\begin{equation}
\sup_{s\in(t,T)}\big\|\mathbb{Y}(s)\big\|^2_{\mathcal{H}_m}
\leq 2\sup_{s\in(t,T)}\big\|\mathbb{Y}(s)-X\big\|_{\mathcal{H}_m}^2 + 2\big\|X\big\|^2_{\mathcal{H}_m}
\leq A_2\pig(1+\|X\|_{\mathcal{H}_m}^2\pig).
\label{bdd |Y|^2}
\end{equation}
From (\ref{def, backward SDE}) and It\^o's formula in (\ref{diff of H norm}), we
obtain 

\[
\dfrac{d}{ds}\big\|\mathbb{Z}(s)\big\|^{2}_{\mathcal{H}_m}
=-2\Big\langle \mathbb{Z}(s),l_x\pig(\mathbb{Y}(s),u(s)\pig)+D_{X}F(\mathbb{Y}(s)  \otimes  m)\Big\rangle_{\mathcal{H}_m}
+\sum_{j=1}^{n}\big\|\mathbbm{r}_{j}(s)\big\|^{2}_{\mathcal{H}_m}
\]
and thus 
\begin{align*}
\big\|\mathbb{Z}(s)\big\|^{2}_{\mathcal{H}_m}
&+\displaystyle\sum_{j=1}^{n}\displaystyle\int_{s}^{T}\big\|\mathbbm{r}_{j}(\tau)\big\|^{2}_{\mathcal{H}_m}d\tau\\
&=\pigl\|h_{x}(\mathbb{Y}(T))+D_{X}F_{T}(\mathbb{Y}(T) \otimes m)\pigr\|^{2}_{\mathcal{H}_m}+2\int_{s}^{T}\Big\langle \mathbb{Z}(\tau),l_{x}(\mathbb{Y}(\tau),u(\tau))+D_{X}F(\mathbb{Y}(\tau)  \otimes  m)\Big\rangle_{\mathcal{H}_m}d\tau.
\end{align*}
Using Assumptions \textbf{A(i)}'s (\ref{assumption, bdd of l, lx, lv}), \textbf{A(iii)}'s (\ref{assumption, bdd of h, hx, hxx}), \textbf{B(i)}'s (\ref{assumption, bdd of DF, DF_T}) and (\ref{1st order condition}), it follows that

\begin{align}
&\h{-15pt}\sup_{s\in(t,T)}\big\|\mathbb{Z}(s)\big\|_{\mathcal{H}_m}^2
+\sum_{j=1}^{n}\int_{t}^{T}\big\|\mathbbm{r}_{j}(s)\big\|^{2}_{\mathcal{H}_m}ds
\label{bdd Z}\\
\leq\:& 2(c_h^2+2c_T^2)\pig(1+\big\|Y(T)\big\|_{\mathcal{H}_m}^2\pig) 
+3c_l^2\int^T_{t} 1+\big\|\mathbb{Y}(s)\big\|_{\mathcal{H}_m}^2 +
\big\|u(s)\big\|_{\mathcal{H}_m}^2ds
+4c^2\int^T_{t} 1 +
\big\|\mathbb{Y}(s)\big\|_{\mathcal{H}_m}^2ds.
\nonumber
\end{align}
Using (\ref{bdd |Y|^2}) and (\ref{bdd int |u|^2}), we obtain the required upper bounds of $\displaystyle\sup_{s\in(t,T)}\big\|\mathbb{Z}(s)\big\|_{\mathcal{H}_m}^2$ and $\displaystyle\sum_{j=1}^{n}\int_{t}^{T}\big\|\mathbbm{r}_{j}(s)\big\|^{2}_{\mathcal{H}_m}ds$. Finally, from (\ref{1st order condition}), the assumptions (\ref{assumption, bdd of l, lx, lv}),
(\ref{assumption, convexity of l}), as well as Young's inequality, we have 

\begin{equation}
\lambda\big\|u(s)\big\|_{\mathcal{H}_m}^2
\leq \dfrac{\lambda}{4}\big\|u(s)\big\|_{\mathcal{H}_m}^2
+\dfrac{c_{l}^2}{\lambda}\pig(1+\|\mathbb{Y}(s)\|^{2}_{\mathcal{H}_m}\pig)
+\dfrac{1}{\lambda}\big\|\mathbb{Z}(s)\big\|_{\mathcal{H}_m}^2
+\dfrac{\lambda}{4}\big\|u(s)\big\|_{\mathcal{H}_m}^2.
\label{eq:ApB100}
\end{equation}
The results in (\ref{bdd |Y|^2}) and (\ref{bdd Z}) imply that

\begin{equation}
\sup_{s\in(t,T)}\|u(s)\|_{\mathcal{H}_m}\leq A_3\pig(1+\|X\|_{\mathcal{H}_m}\pig).
\label{bdd u}
\end{equation}
This concludes Proposition \ref{prop bdd of Y Z u r}. 
\end{proof} Given $X$, $\Psi \in L^2_{\mathcal{W}^{\indep}_{t}} (\mathcal{H}_m) \subset \mathcal{H}_{m}$, we aim to check the existence of the G\^ateaux derivative of the solution, with respect to the initial condition $X$ in the direction of $\Psi$, over $[t,T]$. Define $\mathcal{W}_{tX\Psi}^{s}:=\sigma(X,\Psi)\bigvee\mathcal{W}_{t}^{s}$ and denote $\mathcal{W}_{tX\Psi}$ the filtration generated by
the $\sigma$-algebras $\mathcal{W}_{tX\Psi}^{s}$. For $\epsilon >0 $, we define the difference processes 
\begin{equation}
\begin{aligned}
\Delta^\epsilon_\Psi\mathbb{Y}_{tX}(s)&:=\dfrac{\mathbb{Y}_{t,X+\epsilon\Psi}(s)-\mathbb{Y}_{tX}(s)}{\epsilon}
\h{1pt},\h{5pt} 
\Delta^\epsilon_\Psi \mathbb{Z}_{tX} (s)
:=\dfrac{\mathbb{Z}_{t,X+\epsilon\Psi}(s)-\mathbb{Z}_{tX}(s)}{\epsilon},\\
\Delta^\epsilon_\Psi u_{tX}(s)&:=\dfrac{u_{t,X+\epsilon\Psi}(s)-u_{tX}(s)}{\epsilon}
\h{1pt},\h{5pt} 
\Delta^\epsilon_\Psi\mathbbm{r}_{tX,j}(s):=\dfrac{\mathbbm{r}_{t,X+\epsilon\Psi,j}(s)-\mathbbm{r}_{tX,j}(s)}{\epsilon}.
\end{aligned}
\label{def diff process}
\end{equation}
We next aim to bound them.
\begin{lem} Let $X$, $\Psi \in L^2_{\mathcal{W}^{\indep}_{t}} (\mathcal{H}_m) \subset \mathcal{H}_{m}$, $\epsilon>0$ and $s\in [t,T]$. Under the assumptions of Proposition \ref{prop. convex of J},  the difference processes defined in (\ref{def diff process}) satisfy
\begin{align} \pig\|\Delta^\epsilon_\Psi\mathbb{Y}_{tX}(s)\pigr\|_{\mathcal{H}_m},\:
\pig\|\Delta^\epsilon_\Psi \mathbb{Z}_{tX} (s)\pigr\|_{\mathcal{H}_m},\:
\pig\|\Delta^\epsilon_\Psi u_{t X}(s)\pigr\|_{\mathcal{H}_m},\:
\left[\displaystyle\sum^n_{j=1}\int_{t}^{T}\pigl\|\Delta^\epsilon_\Psi\mathbbm{r}_{tX,j}(s)\pigr\|^{2}_{\mathcal{H}_m}ds\right]^{1/2}
\leq C_4'\|\Psi\|_{\mathcal{H}_m},
\label{bdd Y, Z, u, r, epsilon}
\end{align}
where $C_{4}'$ is a positive constant depending only on $\lambda$, $\eta$, $c$, $c_l$, $c_h$, $c_T$, $c'$, $c_l'$, $c_h'$, $c_T'$ and $T$.
\end{lem}
\begin{rem}
We emphasise that $C_4'$ is independent of the choice of $\epsilon$.
\end{rem}
\begin{proof}
    From (\ref{forward FBSDE})-(\ref{1st order condition}), the quadruple $\pig( \Delta^\epsilon_\Psi \mathbb{Y}_{tX} (s),
\Delta^\epsilon_\Psi \mathbb{Z}_{tX} (s),
\Delta^\epsilon_\Psi u_{tX}(s),
\Delta^\epsilon_\Psi\mathbbm{r}_{tX}(s)\pig)$ solves  the system \vspace{-15pt}

\begin{empheq}[left=\h{-10pt}\empheqbiglbrace]{align}
\Delta^\epsilon_\Psi \mathbb{Y}_{tX} (s)
=\:&\Psi+\int_{t}^{s}\Delta^\epsilon_\Psi u_{tX}(\tau)d\tau;
\label{finite diff. forward}\\
\Delta^\epsilon_\Psi \mathbb{Z}_{tX} (s)
=\:&
\displaystyle\int_{0}^{1}h_{xx}\pig(\mathbb{Y}_{t X}(T)+\theta\epsilon \Delta^\epsilon_\Psi \mathbb{Y}_{tX} (T)\pig)\Delta^\epsilon_\Psi \mathbb{Y}_{tX} (T)\nonumber\\ &\h{100pt}+D_{X}^{2}F_{T}\pig((\mathbb{Y}_{t X}(T)+\theta\epsilon \Delta^\epsilon_\Psi \mathbb{Y}_{tX} (T)) \otimes  m\pig)(\Delta^\epsilon_\Psi \mathbb{Y}_{tX} (T))d\theta\nonumber\\
&+\int^T_s\int_{0}^{1}l_{xx}\pig(\mathbb{Y}_{t X}(\tau)+\theta\epsilon \Delta^\epsilon_\Psi \mathbb{Y}_{tX} (\tau),u_{t X}(\tau) \pig)\Delta^\epsilon_\Psi \mathbb{Y}_{tX} (\tau)\nonumber\\
&\pushright{+l_{xv}\pig(\mathbb{Y}_{t,X+\epsilon\Psi}(\tau),u_{t X}(\tau)+\theta\epsilon \Delta^\epsilon_\Psi u_{tX}(\tau)\pig)\Delta^\epsilon_\Psi u_{tX}(\tau)d\theta d\tau}\h{-1pt}\nonumber\\
&+\int^T_s\int_{0}^{1}D_{X}^{2}F\pig((\mathbb{Y}_{t X}(\tau)+\theta\epsilon \Delta^\epsilon_\Psi \mathbb{Y}_{tX} (\tau)) \otimes  m\pig)(\Delta^\epsilon_\Psi \mathbb{Y}_{tX} (\tau))d\theta d\tau\nonumber\\
&-\int^T_s\sum_{j=1}^n \Delta^\epsilon_\Psi \mathbbm{r}_{t X,j}(\tau)dw_{j}(\tau),
\label{finite diff. backward}
\end{empheq}
meanwhile,\vspace{-15pt}

\begin{equation}
\scalemath{0.91}{
\int_{0}^{1}l_{vx}\pig(\mathbb{Y}_{t X}(s)+\theta\epsilon \Delta^\epsilon_\Psi \mathbb{Y}_{tX} (s),u_{t X}(s)\pig)
\Delta^\epsilon_\Psi \mathbb{Y}_{tX} (s)
+l_{vv}\pig(\mathbb{Y}_{t,X+\epsilon\Psi}(s) ,u_{t X}(s)+\theta\epsilon \Delta^\epsilon_\Psi u_{tX}(s)\pig)
\Delta^\epsilon_\Psi u_{tX}(s)d\theta+\Delta^\epsilon_\Psi \mathbb{Z}_{tX} (s)=0.}
\label{1st order finite diff.}
\end{equation}
The proof actually follows the same steps as in the proof for Proposition \ref{prop bdd of Y Z u r} by applying It\^o lemma to the inner product $\pig\langle \Delta^{\epsilon_j}_\Psi  \mathbb{Z}_{tX}(s),\Delta^{\epsilon_j}_\Psi  \mathbb{Y}_{tX}(s)  \pigr\rangle_{\mathcal{H}_m} $ and using the above equations, we omit it here.
\end{proof}

\mycomment{
$\pig( D \mathbb{Y}_{tX}^\Psi(s),
D \mathbb{Z}_{tX}^\Psi(s),
D u_{tX}^\Psi(s),
D \mathbbm{r}_{tX}^\Psi(s)\pig)$}
\begin{lem}
Let $X$, $\Psi \in L^2_{\mathcal{W}^{\indep}_{t}} (\mathcal{H}_m) \subset \mathcal{H}_{m}$. Under the assumptions of Proposition \ref{prop. convex of J}, the processes 
$ \Delta^\epsilon_\Psi \mathbb{Y}_{tX}(s),
\Delta^\epsilon_\Psi u_{tX}(s)$ and $\Delta^\epsilon_\Psi \mathbb{Z}_{tX}(s)$ converge in $L_{\mathcal{W}_{t X\Psi}}^{\infty}(t,T;\mathcal{H}_{m})$, and $\Delta^\epsilon_\Psi \mathbbm{r}_{tX,j}(s)$ converge in $L_{\mathcal{W}_{t X\Psi}}^{2}(t,T;\mathcal{H}_{m})$, as $\epsilon \to 0$. They are called the Jacobian flow (G\^ateaux derivative) of the solution to the FBSDE (\ref{forward FBSDE})-(\ref{1st order condition}). The limit of $ \pig(\Delta^\epsilon_\Psi \mathbb{Y}_{tX}(s)$, $\Delta^\epsilon_\Psi \mathbb{Z}_{tX}(s)$, $\Delta^\epsilon_\Psi \mathbbm{r}_{tX,j}(s)\pig)$ is the unique solution in $L_{\mathcal{W}_{t X\Psi}}^{2}(t,T;\mathcal{H}_{m})$ of the alternative FBSDE
\begingroup
\allowdisplaybreaks
\begin{equation}
\scalemath{0.95}{
\h{-10pt}\left\{
\begin{aligned}
D \mathbb{Y}_{tX}^\Psi  (s)
&= \Psi
+\displaystyle\int_{t}^{s}\Big[ \diff_y  u(\mathbb{Y}_{tX},\mathbb{Z}_{tX})(\tau)\Big] 
\Big[D \mathbb{Y}_{tX}^\Psi  (\tau) \Big]d\tau
+\int_{t}^{s} 
\Big[\diff_z  u (\mathbb{Y}_{tX},\mathbb{Z}_{tX}) (\tau)\Big] 
\Big[D \mathbb{Z}_{tX}^\Psi  (\tau)\Big]  d\tau;\\
D \mathbb{Z}_{tX}^\Psi (s)
&=h_{xx}(\mathbb{Y}_{tX}(T))D \mathbb{Y}_{tX}^\Psi (T)
+D_{X}^2F_{T}(\mathbb{Y}_{tX}(T)  \otimes  m)(D \mathbb{Y}_{tX}^\Psi (T))\\
&\h{10pt}+\displaystyle\int^T_s\bigg\{l_{xx}(\mathbb{Y}_{tX}(\tau),u(\tau))D \mathbb{Y}_{tX}^\Psi (\tau)
+l_{xv}(\mathbb{Y}_{tX}(\tau),u(\tau))\Big[ \diff_y  u(\mathbb{Y}_{tX},\mathbb{Z}_{tX})(\tau)\Big] 
D \mathbb{Y}_{tX}^\Psi (\tau)\\
&\h{45pt}+l_{xv}(\mathbb{Y}_{tX}(\tau),u(\tau))
\Big[ \diff_z  u (\mathbb{Y}_{tX},\mathbb{Z}_{tX})(\tau)\Big] 
D \mathbb{Z}_{tX}^\Psi (\tau)
+D^2_{X}F(\mathbb{Y}_{tX}(\tau)  \otimes  m)\pig(
D\mathbb{Y}_{tX}^\Psi (\tau)\pig)\bigg\}d\tau \h{-30pt}\\
&\h{10pt}-\int^T_s\sum_{j=1}^{n}D\mathbbm{r}^\Psi_{tX,j}(\tau)dw_{j}(\tau),
\end{aligned}\right.}
\label{J flow of FBSDE}
\end{equation}
\endgroup
\label{lem, Existence of J flow}
\end{lem}
\begin{rem}\label{rem def of Du = DY +DZ}
We call that the control process $u_{t X}(s)$ satisfying the first order condition (\ref{1st order condition}) is a map $u(y,z)$ evaluated at $(y,z) = (\mathbb{Y}_{t X}(s), \mathbb{Z}_{t X}(s))$ such that it is solved implicitly in the equation $l_v(y,u(y,z))+z=0$. The Jacobian matrices $\diff_y  u $ and $\diff_z  u $, with respect to the first and second arguments respectively, are bounded by
$$
\pig| \diff_y  u  \pig| 
\leq \dfrac{c_l}{\lambda} 
\h{15pt} \text{ and } \h{15pt}
\pig| \diff_z  u  \pig| \leq \dfrac{1}{\lambda},
$$
which can be verified by differentiating the first order condition (\ref{1st order condition}).
\end{rem}
\mycomment{
\begin{rem}
For any $\Psi \in L^2_{\mathcal{W}^{\indep}_{tX}} (\mathcal{H}_m)$ such that $\mathbb{E}(\Psi)=0$, then the backward equation in the system (\ref{J flow of FBSDE}) becomes 
\begin{align*}
&\h{-10pt}D\mathbb{Z}^{\Psi}_{t X}(s)\\
=&\:h_{xx}(\mathbb{Y}_{t X}(T))
D\mathbb{Y}^{\Psi}_{t X}(T)
+\diff^2_x \dfrac{dF_{T}}{d\nu}(\mathbb{Y}_{t X}(T)  \otimes  m)(\mathbb{Y}_{t X}(T))D\mathbb{Y}^{\Psi}_{t X}(T)\nonumber\\
&+\int^T_s
\bigg\{l_{xx}(\mathbb{Y}_{t X}(\tau),u(\tau))D \mathbb{Y}_{t X}^\Psi (\tau)
+l_{xv}(\mathbb{Y}_{t X}(\tau),u(\tau))\Big[ \diff_y  u(\mathbb{Y}_{t X},\mathbb{Z}_{t X})(\tau)\Big] 
D \mathbb{Y}_{t X}^\Psi (\tau)\nonumber\\
&\h{35pt}+l_{xv}(\mathbb{Y}_{t X}(\tau),u(\tau))
\Big[ \diff_z  u (\mathbb{Y}_{t X},\mathbb{Z}_{t X})(\tau)\Big]D \mathbb{Z}_{t X}^\Psi (\tau)
+\diff^2_x 
\dfrac{dF}{d\nu}(\mathbb{Y}_{t X}(\tau)  \otimes  m)(\mathbb{Y}_{t X}(\tau))D\mathbb{Y}^\Psi_{t X}(\tau)\Bigg\}d\tau\h{-30pt}
\nonumber\\
&-\int^T_s\sum_{j=1}^{n}D\mathbbm{r}^\Psi_{t X,j}(\tau)dw_{j}(\tau);
\end{align*}
indeed, from formula (\ref{eq:2-232}), it is sufficient to show that
\[
\widetilde{\mathbb{E}}\left[\int_{\mathbb{R}^{n}}\diff_x \diff_{\tilde{x}}
\dfrac{d^{\h{.5pt}2}\h{-.7pt}  F}{d\nu^{2}}(\mathbb{Y}_{t X}(s)  \otimes  m)\pig(\mathbb{Y}_{t X}(s),\widetilde{\mathbb{Y}}_{t\tilde{X}}(s)\pig)
D\widetilde{\mathbb{Y}}^{\tilde{\Psi}}_{t\tilde{X}}(s)dm(\tilde{x})\right]=0.
\]
Referring to the system (\ref{J flow of FBSDE, X=I, matrix}) and relation (\ref{eq:4-178}), there is a matrix-valued process $D_X\widetilde{\mathbb{Y}}_{t \tilde{X}}
\in $\\$L_{\mathcal{W}_{t\tilde{X}}}^{2}(t,T;\mathcal{L}(\mathcal{H}_{m};\mathcal{H}_{m}))$ such that we can write 
$D\widetilde{\mathbb{Y}}_{t\tilde{X}}^{\tilde{\Psi}}
=\pig[D_X\widetilde{\mathbb{Y}}_{t \tilde{X}}\pig]
\widetilde{\Psi}$. Therefore, due to the fact that $\widetilde{\mathbb{E}}(\widetilde{\Psi})=0$ and $\widetilde{\Psi}$ is independent of $\mathcal{W}_{t\tilde{X}}$, we have  
\[
\int_{\mathbb{R}^{n}}
\widetilde{\mathbb{E}}\left[\diff_x \diff_{\tilde{x}}\dfrac{d^{\h{.5pt}2}\h{-.7pt}  F}{d\nu^{2}}(\mathbb{Y}_{tX}(s)  \otimes  m)\pig(\mathbb{Y}_{tX}(s),\widetilde{\mathbb{Y}}_{t\tilde{X}}(s)\pig)
D_X\widetilde{\mathbb{Y}}_{t \tilde{X}}\right]
\widetilde{\mathbb{E}}(\widetilde{\Psi})
dm(\tilde{x})=0.
\]
\end{rem}}
The proof of Lemma \ref{lem, Existence of J flow} can be found in \hyperref[app, Existence of J flow]{Appendix}.
\mycomment{
\textcolor{red}{
The idea of proof consists of first establishing the existence of solution to (\ref{J flow of FBSDE}) can apply the standard martingale representation theorem and fixed point arguments to  $(\mathscr{D}\mathbb{Y}, \mathscr{D}\mathbb{Z})$ to (\ref{J flow of FBSDE}). Since the system (\ref{J flow of FBSDE}) is linear, the proof is standard and thus omitted here. Then, we can easily verify that the solution $(\mathscr{D}\mathbb{Y}, \mathscr{D}\mathbb{Z})$ to (\ref{J flow of FBSDE}) is exactly the derivatives of $\mathbb{Y}_{t,X}(s)$ and $\mathbb{Z}_{t,X}(s)$ with respect to the initial random variable $X$ by considering, for example, 
$$ \dfrac{\mathbb{Y}_{t,X+\varepsilon\Psi}(s)
- \mathbb{Y}_{t,X}(s)}{\varepsilon} - \mathscr{D}\mathbb{Y},$$
following the details as in \cite{PP92}. }}

\begin{lem}
 Let $X$, $\Psi \in L^2_{\mathcal{W}^{\indep}_{t}} (\mathcal{H}_m) \subset \mathcal{H}_{m}$ and $s\in [t,T]$. Under the assumptions of Proposition \ref{prop. convex of J}, the Jacobian flow $\pig(D \mathbb{Y}_{tX}^\Psi(s),
D \mathbb{Z}_{tX}^\Psi(s),D u_{tX}^\Psi(s), $ $D \mathbbm{r}_{tX,j}^\Psi(s)\pig)$ is linear in $\Psi$, and each of them is partially continuous in $X$ for a given $\Psi$. Thus, their Fr\'echet derivatives exist, and they are denoted by $D \mathbb{Y}_{tX}(s),
D \mathbb{Z}_{tX}(s),D u_{tX}(s)$, $D \mathbbm{r}_{tX,j}(s)$, respectively.
\label{lem, Existence of Frechet derivatives}
\end{lem}
The proof can be found in \hyperref[app, Existence of Frechet derivatives]{Appendix}.

\section{Regularities of Value Function and its Functional Derivatives}\label{sec. Analytic Properties of Value Function}
In this section, the quadruple $\pig(\mathbb{Y}_{tX}(s), \mathbb{Z}_{tX}(s), u_{tX}(s), \mathbbm{r}_{tX}(s)\pig)$ is the solution to the FBSDE (\ref{forward FBSDE})-(\ref{1st order condition}) and $V(X \otimes m,t)$ is the corresponding value function defined in (\ref{def value function}).
\subsection{Growth and Regularities of Value Function and its Derivatives}
Due to the quadratic growths in $x$ of the running and terminal cost functionals, we have the same growth for the value function in $X$.
\begin{prop}[\bf Quadratic Growth of $V$ in $X$]
\label{prop bdd of V} Under the assumptions of Proposition \ref{prop. convex of J}, for any $X \in \mathcal{H}_m$ measurable to $\mathcal{W}^t_0$, we have 
\begin{align}
|V(X  \otimes  m,t)|\leq C_{5}\pig(1+\|X\|^{2}_{\mathcal{H}_m}\pig),
\label{bdd V}
\end{align}
where $C_5$ is a positive constants depending only on $n$, $\lambda$, $\eta$, $c$, $c_l$, $c_h$, $c_T$, $c'$, $c_l'$, $c_h'$, $c_T'$ and $T$.
\end{prop}
\begin{proof}
Recall Formula (\ref{def value function}), the value function reads
$$
\mbox{\fontsize{9.7}{10}\selectfont\(
V(X  \otimes  m,t)=\displaystyle\int_{t}^{T}\mathbb{E}
\left[\int_{\mathbb{R}^{n}}l\pig(\mathbb{Y}_{tX}(s),u_{tX}(s)\pig)dm(x)\right]
+F(\mathbb{Y}_{tX}(s)  \otimes  m)ds
+\mathbb{E}\left[\int_{\mathbb{R}^{n}}h(\mathbb{Y}_{tX}(T))dm(x)\right]
+F_{T}(\mathbb{Y}_{tX}(T)  \otimes  m).\)}
$$
Applying Assumptions \textbf{A(i)}'s (\ref{assumption, bdd of l, lx, lv}), \textbf{A(iii)}'s (\ref{assumption, bdd of h, hx, hxx}), (\ref{assumption, bdd of F F_T}) and the bounds of $\mathbb{Y}_{tX}$, $u_{tX}$ in Proposition \ref{prop bdd of Y Z u r}, we can then deduce the estimate (\ref{bdd V}).
\end{proof}

Next, we have the Lipschitz nature of the derivative $D_X V(X  \otimes  m,t)$ with respect to
$X \in L^2_{\mathcal{W}_t^{\indep}}(\mathcal{H}_m)$ which is defined in (\ref{subspace of H_m indep. of W_t}). 
\begin{prop}[\bf Lipschitz Continuity of $D_X V$ in $X$]
\label{prop4-2} Under the assumptions in Proposition \ref{prop. convex of J}, the functional $X\in L^2_{\mathcal{W}_t^{\indep}}(\mathcal{H}_m)
\longmapsto V(X  \otimes  m,t)$ is Fr\'echet differentiable
and 

\begin{equation}
D_{X}V(X  \otimes  m,t)=\mathbb{Z}_{tX}(t).
\label{D_X V = Z}
\end{equation}
Moreover, for any $X^{1},X^{2}\in L^2_{\mathcal{W}_t^{\indep}}(\mathcal{H}_m)$, the derivative is Lipschitz continuous such that
\begin{equation}
\pigl\|D_{X}V(X^{1}  \otimes  m,t)-D_{X}V(X^{2}  \otimes  m,t)\pigr\|_{\mathcal{H}_m}
\leq C_6 \big\|X^{1}-X^{2}\big\|_{\mathcal{H}_m},
\label{lip cts of D_XV in X}
\end{equation}
where $C_6$ depends only on $\lambda$, $c$, $c_l$, $c_h$, $c_T$, $c'$, $c_h'$, $c_T'$, $c_l'$ and $T$.
\label{prop D_X V = Z and lip of D_X V}
\end{prop}

The proof can be found in \hyperref[app, prop DX V = Z and lip of DX V]{Appendix}. As a consequence of the optimality principle mentioned in Section \ref{Bellman Equation}, we have the time consistency, namely, the control $u_{tX}(\tau)$, for $\tau \in [s,T)$ with $s \in [t,T)$, remains optimal for the problem with initial conditions $(s,\mathbb{Y}_{tX}(s)$) and thus we also have 
\begin{equation}
D_{X}V(\mathbb{Y}_{tX}(s)  \otimes  m,s)
=\mathbb{Z}_{s\mathbb{Y}_{tX}(s)}(s)
=\mathbb{Z}_{tX}(s),
\label{D_X V(s)=Z(s)}
\end{equation}
due to the flow property (\ref{flow property}). We then proceed on the regularity in time of the value function and its derivative.
\begin{prop}[\bf H\"older Continuity in Time of $V$ and $D_X V$]\label{prop4-5}
With the assumptions in Proposition \ref{prop. convex of J}, it holds that for any $X\in L^2_{\mathcal{W}_t^{\indep}}(\mathcal{H}_m)$ and $t_1$, $t_2 \in [0,T]$, 

\begin{equation}
\pig|V(X  \otimes  m,t_1)-V(X  \otimes  m,t_2)\pig|
\leq C_{7}\pig(1+\|X\|^{2}_{\mathcal{H}_m}\pig)\pig|t_1-t_2\pig|
\h{10pt}\text{and}
\label{eq:4-26}
\end{equation}

\begin{equation}
\pig\|D_{X}V(X  \otimes  m,t_1)-D_{X}V(X  \otimes  m,t_2)\pigr\|_{\mathcal{H}_m}
\leq C_{8}\Big(\big|t_1-t_2\big|^{\frac{1}{2}} 
+ \|X\|_{\mathcal{H}_m} \big|t_1-t_2\big|
\Big),
\label{cts of D_X V in t}
\end{equation}
where $C_7$ and $C_8$ depend only on $n$, $\lambda$, $\eta$, $c$, $c_l$, $c_h$, $c_T$, $c'$, $c_l'$, $c_h'$, $c_T'$ and $T$.
\label{prop holder of V D_X V in time}
\end{prop}

The proof can be found in \hyperref[app, prop holder of V DX V in time]{Appendix}.
\begin{rem}
\label{rem7-1} With $t_1<t_2$, since $\mathbb{Z}_{t_1X}(t_2)=\mathbb{Z}_{t_2,\mathbb{Y}_{t_1X}(t_2)}(t_2)$ (see (\ref{flow property})), together with (\ref{D_X V = Z}), (\ref{lip cts of D_XV in X}), (\ref{cts of D_X V in t}) and (\ref{est. |Ys-X|^2}), we have $
\pig\|\mathbb{Z}_{t_1X}(t_2)-\mathbb{Z}_{t_2X}(t_2)\pigr\|_{\mathcal{H}_m}\leq C_{6}\pig\|\mathbb{Y}_{t_1X}(t_2)-X\pigr\|_{\mathcal{H}_m}$ and $\pig\|\mathbb{Z}_{t_1X}(t_2)-\mathbb{Z}_{t_1X}(t_1)\pigr\|_{\mathcal{H}_m}
\leq C_{8}'\Big(\big|t_1-t_2\big|^{\frac{1}{2}} + \big|t_1-t_2\big|\|X\|_{\mathcal{H}_m}\Big)$.
\end{rem}

\subsection{Second-order G\^ateaux Derivative}

In the following proposition, we characterise the second-order G\^ateaux derivative of the value function by the Jacobian flow as the solution to the backward equation in (\ref{J flow of FBSDE}).
\begin{prop}
Let $X$, $\Psi\in\mathcal{H}_{m} \in L^2_{\mathcal{W}^{\indep}_t} (\mathcal{H}_m)$. Under the assumptions (\ref{assumption, bdd of l, lx, lv})-(\ref{assumption, convexity of h}), (\ref{assumption, bdd of F F_T})-(\ref{assumption, convexity of D^2F, D^2F_T}) and (\ref{uniqueness condition}), the value function $V(X  \otimes  m,t)$ has a second-order G\^ateaux derivative with respect to $X$ given by
\begin{equation}
D_{X}^{2}V(X  \otimes  m,t)(\Psi)=D\mathbb{Z}^\Psi_{tX}(t)
\h{10pt}\text{and}\h{10pt}
\label{DX^2 V = DZ}
\end{equation}

\begin{equation}
\pig\|D_{X}^{2}V(X  \otimes  m,t)(\Psi)\pigr\|_{\mathcal{H}_m}\leq C_{9}\|\Psi\|_{\mathcal{H}_{m}},
\label{DX^2 V < psi}
\end{equation}
where $C_{9}$ is a positive constant depending only on $n$, $\lambda$, $\eta$, $c$, $c_l$, $c_h$, $c_T$, $c'$, $c_l'$, $c_h'$, $c_T'$ and $T$.
\label{prop DXX V}
\end{prop}
\begin{proof}
Since $\mathbb{Z}_{tX}(t)=D_{X}V(X \otimes  m,t)$ due to Proposition \ref{prop D_X V = Z and lip of D_X V}, by differentiating both sides and using the definition of G\^ateaux  second derivative in (\ref{def second G derivative}) and the strong convergence in Lemma \ref{lem, Existence of J flow}, we obtain 
\begin{align*}
\pig\langle D_{X}^{2}V(X \otimes  m,t)(\Psi),\Psi \pigr\rangle_{\mathcal{H}_m}
&=\lim_{\epsilon \to 0} \dfrac{1}{\epsilon}
\pig\langle D_X V((X+\epsilon \Psi) \otimes  m,t) - D_X V(X \otimes  m,t) ,\Psi \pigr\rangle_{\mathcal{H}_m}
=\lim_{\epsilon \to 0} 
\pig\langle  \Delta^\epsilon_{\Psi}\mathbb{Z}_{t X}(t),
\Psi\pigr\rangle_{\mathcal{H}_m}\\
&=\pig\langle  
D\mathbb{Z}_{t X}^{\Psi}(t),
\Psi\pigr\rangle_{\mathcal{H}_m},
\end{align*}
together with the bound in (\ref{bdd Y, Z, u, r, epsilon}),  as a limit, we also obtain the upper bound in (\ref{DX^2 V < psi}).
\end{proof}
We then state additional pointwise continuity properties for the second-order derivative, which is similar to (\ref{eq:3-125}), given by
\begin{prop}
Let $\{t_{k}\}_{k\in\mathbb{N}}$ be a sequence decreasing to $t\in [0,T)$. Suppose that  $\big\{X_{k}\big\}_{k\in\mathbb{N}}$ and $\big\{\Psi_{k}\big\}_{k\in\mathbb{N}}$ are sequences converging to $X$ and $\Psi$ in $\mathcal{H}_{m}$, respectively, such that each $X_k$, $\Psi_k\in L^2_{\mathcal{W}^{\indep}_{t_k}} (\mathcal{H}_m)$. Under the assumptions of Proposition \ref{prop DXX V}, we have
\begin{equation}
D_{X}^{2}V(X_{k}  \otimes  m,t_{k})(\Psi_{k})\longrightarrow D_{X}^{2}V(X  \otimes  m,t)(\Psi)
\h{10pt} \text{in}\;\mathcal{H}_{m}.
\label{ct of D^2_X V}
\end{equation}
\label{prop cts of DXX V}
\end{prop}\vspace{-20pt}
\begin{rem}
As a consequence, using Cauchy-Schwarz inequality, we can deduce the $L^1$-convergence

$$D_{X}^{2}V(X_{k}  \otimes  m,t_{k})(\Psi_{k})
-D_{X}^{2}V(X  \otimes  m,t_{k})(\Psi_{k})\longrightarrow0
\h{10pt} 
\text{in }L^{1}(\Omega,\mathcal{A},\mathbb{P};L_{m}^{1}(\mathbb{R}^{n}));$$
indeed, the quantity
$$\scalemath{0.97}{
\begin{aligned}
&\h{-10pt}\int_{\mathbb{R}^n}
\mathbb{E}\left[
\pig|D_{X}^{2}V(X_{k}  \otimes  m,t_{k})(\Psi_{k})
-D_{X}^{2}V(X  \otimes  m,t_{k})(\Psi_{k})\pig|\right]dm(x)\\
\leq\:&
\int_{\mathbb{R}^n}
\mathbb{E}\left[
\pig|D_{X}^{2}V(X_{k}  \otimes  m,t_{k})(\Psi_{k})
-D_{X}^{2}V(X  \otimes  m,t)(\Psi)\pig|+
\pig|D_{X}^{2}V(X  \otimes  m,t)(\Psi)
-D_{X}^{2}V(X  \otimes  m,t_{k})(\Psi_{k})\pig|\right]dm(x)\\
\leq\:&
\pig\|D_{X}^{2}V(X_{k}  \otimes  m,t_{k})(\Psi_{k})
-D_{X}^{2}V(X  \otimes  m,t)(\Psi)\pigr\|_{\mathcal{H}_m}
+\pig\|D_{X}^{2}V(X  \otimes  m,t)(\Psi)
-D_{X}^{2}V(X  \otimes  m,t_{k})(\Psi_{k})\pigr\|_{\mathcal{H}_m}
\mycomment{\leq\:&
\left\{\mathbb{E}\left[\int_{\mathbb{R}^n}
\pig|D_{X}^{2}V(X_{k}  \otimes  m,t_{k})(\Psi_{k})
-D_{X}^{2}V(X  \otimes  m,t)(\Psi)\pig|^2dm(x)\right]\right\}^{1/2}
+\left\{\mathbb{E}\left[
\int_{\mathbb{R}^n}\pig|D_{X}^{2}V(X  \otimes  m,t)(\Psi)
-D_{X}^{2}V(X  \otimes  m,t_{k})(\Psi_{k})\pig|^2dm(x)\right]\right\}^{1/2}}
\end{aligned}}
$$
tends to zero as $k \to \infty$. This $L^1$-convergence ensures a valid  application of It\^o lemma (Theorem \ref{ito thm}) in Section \ref{sec. master equation}.
\end{rem}

The proof of Proposition \ref{prop cts of DXX V} can be found in  \hyperref[app, prop cts of DXX V]{Appendix}.

\subsection{Linear Functional Derivatives}
In the discussion before, we use the notations of the processes $\mathbb{Y}_{tX}(s)$, $\mathbb{Z}_{tX}(s)$, $u_{tX}(s)$, $\mathbbm{r}_{tX,j}(s)$ to emphasise the  dependence on the initial random variable $X$. In this section, we take $X=\mathcal{I}_x=x$, and then we denote the processes $\mathbb{Y}_{t\mathcal{I}}(s)$, $\mathbb{Z}_{t\mathcal{I}}(s)$, $u_{t\mathcal{I}}(s)$, $\mathbbm{r}_{t\mathcal{I},j}(s)$ by $\mathbb{Y}_{txm}(s)$, $\mathbb{Z}_{txm}(s)$, $u_{txm}(s)$, $\mathbbm{r}_{txm,j}(s)$ so as to emphasise the dependence on the initial distribution $m$ and the initial point $x$. With the new notations, the system (\ref{forward FBSDE})-(\ref{1st order condition}) can be rewritten as
\begin{empheq}[left=\h{-10pt}\empheqbiglbrace]{align}
\mathbb{Y}_{txm}(s)=&\:x+\int_{t}^{s}u_{txm}(\tau)d\tau+\eta(w(s)-w(t));
\label{forward FBSDE, X=I}\\
\mathbb{Z}_{txm}(s)=&\:h_{x}(\mathbb{Y}_{txm}(T))
+\diff_x \dfrac{dF_T}{d\nu}(\mathbb{Y}_{t\bigcdot m}(T) \otimes  m)\pig(\mathbb{Y}_{txm}(T)\pig)
\nonumber\\
&+\int^T_sl_x\pig(\mathbb{Y}_{txm}(\tau),u_{txm}(\tau)\pig)
+\diff_x \dfrac{dF}{d\nu}(\mathbb{Y}_{t\bigcdot m}(\tau) \otimes  m)\pig(\mathbb{Y}_{txm}(\tau)\pig)d\tau-\int^T_s\sum_{j=1}^{n}\mathbbm{r}_{txm,j}(\tau)dw_{j}(\tau);\h{-3pt}
\label{backward FBSDE, X=I}
\end{empheq}
\begin{flalign}
\text{subject to}&&
l_v\pig(\mathbb{Y}_{txm}(s),u_{txm}(s)\pig)+\mathbb{Z}_{txm}(s)=0. &&
\label{1st order condition, X=I}
\end{flalign}
\mycomment{Proposition \ref{prop bdd of Y Z u r} shows that the optimal control $u_{txm}$ belongs to the space
of processes $v_{tx}$: $x\longmapsto\dfrac{v_{tx}}{(1+|x|^{2})^{\frac{1}{2}}}\in L^2_m\pig(\mathbb{R}^n;L^{\infty}\big(t,T;L^2(\Omega,\mathcal{A},\mathbb{P};\mathbb{R}^{n})\big)\pig)$, which is a subspace of $L_{\mathcal{W}_{t}}^{2}(t,T;\mathcal{H}_{m})$
for any $m$.}
Proposition \ref{prop bdd of Y Z u r} shows that, since $1+|x|^2$ is the same as $(1+|x|)^2$ asymptotically for large $x$, for almost every $x \in \mathbb{R}^n$ and $s\in [t,T]$, the expected value $\mathbb{E}\left[\frac{(v_{tx}(s))^2}{1+|x|^{2}}\right]$ is finite, and thus $\frac{v_{tx}(s)}{\sqrt{1+|x|^{2}}} \in L^2(\Omega,\mathcal{A},\mathbb{P};\mathbb{R}^{n})$ for every $s \in [t,T]$. Therefore, for almost every fixed $x\in \mathbb{R}^n$, the process $\frac{v_{tx}}{\sqrt{1+|x|^{2}}} \in L^{\infty}\big(t,T;L^2(\Omega,\mathcal{A},\mathbb{P};\mathbb{R}^{n})\big)$. Thus, the optimal control $u_{txm}$ belongs to the space {\small$ L^2_m\pig(\mathbb{R}^n;L^{\infty}\big(t,T;L^2(\Omega,\mathcal{A},\mathbb{P};\mathbb{R}^{n})\big)\pig)$} as a function of $x$, for each $m \in \mathcal{P}_2(\mathbb{R}^n)$. In order to study the master equation in Section \ref{sec. master equation}, it is necessary to consider the linear functional derivative of the value function with respect to $m$. Clearly $\mathcal{I}_x \otimes m=m$, we
write $V(m,t)=V(\mathcal{I} \otimes  m,t)$ , which is given by  

\begin{equation}
\scalemath{0.93}{
\begin{aligned}
V(m,t)=\int_{t}^{T}\mathbb{E}
\left[\int_{\mathbb{R}^{n}}l\pig(\mathbb{Y}_{txm}(s),u_{txm}(s)\pig)dm(x)\,\right]
+F(\mathbb{Y}_{t\bigcdot m}(s) \otimes  m)ds
+\mathbb{E}\left[\int_{\mathbb{R}^{n}}h(\mathbb{Y}_{txm}(T))dm(x)\right]+F_{T}(\mathbb{Y}_{t\bigcdot m}(T) \otimes  m).
\end{aligned}}
\label{eq:4-91}
\end{equation}
Recalling from Proposition \ref{prop D_X V = Z and lip of D_X V}, the G\^ateaux derivative of $V$ equals
\begin{equation}
D_{X}V(m,t)
=D_X V(X  \otimes  m)\Big|_{X=\mathcal{I}}
=\mathbb{Z}_{t\mathcal{I}}(t)
=\mathbb{Z}_{txm}(t).
\label{eq:4-92}
\end{equation}
Next, we want to discuss the linear functional derivative $\dfrac{dV}{d\nu}(m,t).$
\begin{prop}
\label{prop4-3}
Under the assumptions of Proposition $\ref{prop DXX V}$, the value function $V(m,t)$ has the linear functional derivative $\dfrac{dV}{d\nu}(m,t)(x)$ satisfying
\vspace{-5pt}
\begin{flalign}
\text{(a)}&&\dfrac{dV}{d\nu}(m,t)(x)
=\:
\int_{t}^{T}\left\{\mathbb{E}\pig[l(\mathbb{Y}_{txm}(s),u_{txm}(s))\pig]
+ \mathbb{E}\left[\dfrac{dF}{d\nu}(\mathbb{Y}_{t\bigcdot m}(s) \otimes  m)(\mathbb{Y}_{txm}(s))\right]\right\}ds&&\nonumber\\
&&+\mathbb{E}\pig[h(\mathbb{Y}_{txm}(T))\pig]
+\mathbb{E}\left[\dfrac{dF_T}{d\nu}(\mathbb{Y}_{t\bigcdot m}(T) \otimes  m)(\mathbb{Y}_{txm}(T))\right],\h{30pt}&&
\label{eq:4-10}
\end{flalign}
\vspace{-30pt}
\begin{flalign*}
&\h{20pt}\text{with all the expectations $\mathbb{E}$'s being taken with respect to the Brownian motion only;}&&
\end{flalign*}
\vspace{-15pt}
\begin{flalign}
\text{(b)}&&
\diff_x \dfrac{dV}{d\nu}(m,t)(x)=\mathbb{Z}_{txm}(t);&&
\label{D_x p_nu V = Z}
\end{flalign}
\vspace{-15pt}
\begin{flalign}
\text{(c)}&&\dfrac{dV}{d\nu}(m,t)(x)=\int_{0}^{1}\mathbb{Z}_{t,\theta x,m}(t)\cdot xd\theta+C(t),&&
\label{eq:4-110}
\end{flalign}
where $C(t)$ depends only on $t$ but not $x$, and which can be chosen to be $0$ if the normalisation condition $\displaystyle\int_{\mathbb{R}^{n}}\dfrac{dV}{d\nu}(m,t)(x)dm(x)=0$ is taken.
\end{prop}
\begin{rem}
Formulae (\ref{D_X V(s)=Z(s)}) and (\ref{D_x p_nu V = Z}) tell us that $
\mathbb{Z}_{txm}(s)=\diff_x \,\dfrac{dV}{d\nu}(\mathbb{Y}_{t\bigcdot m}(s)  \otimes   m,s)(\mathbb{Y}_{txm}(s))$ for $s \in (t,T]$.
\end{rem}
\begin{proof}
Since the G\^ateaux  derivative $D_{X}V(X \otimes  m,t)$ exists due to Proposition \ref{prop D_X V = Z and lip of D_X V} and satisfies the regularity in (\ref{bdd Y, Z, u, r}), the linear functional derivative also exists and is given by (\ref{eq:2-2000}), because of Proposition 4.1 in \cite{BY19}. Utilizing (\ref{eq:2-217}) and the fact that  $\mathbb{Z}_{tX}(t)=D_{X}V(X  \otimes   m,t)$, we know that (\ref{D_x p_nu V = Z}) holds, and a simple integration and the mean value theorem give (\ref{eq:4-110}). We now turn to prove (\ref{eq:4-10}). 

Denote the function on the right hand side
of (\ref{eq:4-10}) by $\mathcal{S}(x,m,t)$, it suffices to show that

\begin{equation}
\mathbb{Z}_{txm}(t)=\diff_x  \mathcal{S}(x,m,t)\label{eq:4-500}
\end{equation}
since the linear functional derivative is defined up to a constant function independent of $x$, but possibly depending on $t$ (also see (\ref{def. linear functional derivative})). To this end, we consider the control problem described as follows: for any control $\alpha_{tx} \in L_{\mathcal{W}_{t}}^{2}(t,T;\mathbb{R}^{n})$, consider the state process by  
\begin{equation}
\mathbb{A}_{tx}(s)=x+\int_{t}^{s}\alpha_{tx}(\tau)d\tau+\eta(w(s)-w(t)).
\label{eq:4-13}
\end{equation}
We aim to minimize the following objective function:

\begin{equation}
\begin{aligned}
\mathcal{S}^*_{tx}(\alpha_{tx}):=\:&\mathbb{E}\left[\int_{t}^{T}l(\mathbb{A}_{tx}(s),\alpha_{tx}(s))ds\right]
+\int_{t}^{T}\mathbb{E}\left[\dfrac{dF}{d\nu}(\mathbb{Y}_{t\bigcdot m}(s)  \otimes   m)(\mathbb{A}_{tx}(s))ds\right]
+\mathbb{E}\pig[h(\mathbb{A}_{tx}(T))\pig]\\
&+\mathbb{E}\left[\dfrac{dF_{T}}{d\nu}(\mathbb{Y}_{t\bigcdot m}(T)  \otimes   m)(\mathbb{A}_{tx}(T))\right],
\end{aligned}
\label{eq:4-14}
\end{equation}
where $\mathbb{Y}_{t x m}(s)$ is given by the system (\ref{forward FBSDE, X=I})-(\ref{1st order condition, X=I}). The necessary conditions of optimality, working similarly to Lemma \ref{lem G derivative of J}, is given by \begin{equation}
l_v(\mathbb{A}_{tx}(s),\alpha_{tx}(s))+\mathbb{K}_{tx}(s)=0,
\label{1st condition, S}
\end{equation} where
\begin{equation}
\left\{
\begin{aligned}
-d\mathbb{K}_{tx}(s)
&=\left[l_{x}\pig(\mathbb{A}_{tx}(s),\alpha_{tx}(s)\pig)+\diff_x 
\dfrac{dF}{d\nu}(\mathbb{Y}_{t\bigcdot m}(s)  \otimes   m)\pig(\mathbb{A}_{tx}(s)\pig)\right]ds
-\sum_{j=1}^{n}\kappa_{tx,j}(s)dw_{j}(s)\h{1pt},\h{10pt}
\text{ for $s\in [t,T]$;}\\
\mathbb{K}_{tx}(T) &= h_x\big(\mathbb{A}_{tx}(T)\big) 
+\diff_x 
\dfrac{dF_T}{d\nu}(\mathbb{Y}_{t\bigcdot m}(T)  \otimes   m)\pig(\mathbb{A}_{tx}(T)\pig), 
\end{aligned}\right.
\label{def of Z*, S}
\end{equation}
for some adapted processes $\kappa_{tx,j}(s)\in L_{\mathcal{W}_{t}}^{2}(t,T;\mathcal{H}_{m})$ with $j=1,2,\ldots,n$. It is clear that $\mathbb{A}_{tx}=\mathbb{Y}_{txm}$, $\mathbb{K}_{tx}=\mathbb{Z}_{txm}$, $\alpha_{tx}=u_{txm}$, $\kappa_{tx,j}=\mathbbm{r}_{txm,j}$ satisfy (\ref{eq:4-13}), (\ref{def of Z*, S}) and (\ref{1st condition, S}) respectively. Since (\ref{uniqueness condition}) holds by the assumption of this proposition, it implies the strict convexity of the objective function $\mathcal{S}^*_{tx}(\alpha_{tx})$, and thus the optimal controlled process is unique. Therefore, the optimal control is $u_{txm}(s)$, and the
optimal state is $\mathbb{Y}_{txm}(s)$. The Bellman equation for
(\ref{eq:4-13})-(\ref{eq:4-14}) reads

\begin{equation}
\left\{
\begin{aligned}
&\dfrac{\partial\Phi}{\partial s}+H(x,\diff_x \Phi)+\dfrac{dF}{d\nu}(\mathbb{Y}_{t\bigcdot m}(s)  \otimes   m)(x)+\dfrac{1}{2}\sum_{j=1}^{n}D^2_x\Phi\eta^{j}\cdot\eta^{j}=0,
\\
&\Phi(x,T)=h(x)+\dfrac{dF_{T}}{d\nu}(\mathbb{Y}_{t\bigcdot m}(T)  \otimes   m)(x),
\end{aligned}
\right.
\label{eq:4-15}
\end{equation}
where $\eta^{j}$ is the $j$-th column vector of the matrix $\eta$. By the uniqueness of the Bellman equation, we see that $\Phi(x,t)=\mathcal{S}(x,m,t)$ and $\mathbb{Z}_{txm}(t)=\diff_x \mathcal{S}(x,m,t)$
\mycomment{\begin{equation}
\mathbb{Z}_{txm}(s)=\diff_x \Phi(\mathbb{Y}_{txm}(s),s)
\h{10pt}\text{and}\h{10pt}
\mathbbm{r}_{txm,j}(s)=\diff^2_x \Phi(\mathbb{Y}_{txm}(s),s)\eta^{j}
.\label{eq:4-160}
\end{equation}
Clearly
\begin{equation}
\Phi(x,t)=L(x,m,t)
\h{10pt}\text{and}\h{10pt}
\mathbb{Z}_{txm}(t)=\diff_x L(x,m,t).
\label{eq:4-161}
\end{equation}}
which concludes (\ref{eq:4-500}).
\end{proof}

We next study the second-order derivative of the linear functional derivative of the value function $V$ and its connection with the second-order G\^ateaux  derivative of $V$. With $X=\mathcal{I}_x=x$, just like as above, we adopt the notation of Jacobian flow $D\mathbb{Y}^{\Psi}_{tX}(s)=D\mathbb{Y}^{\Psi}_{txm}(s)$, $D\mathbb{Z}^{\Psi}_{tX}(s)=D\mathbb{Z}^{\Psi}_{txm}(s)$, $Du^{\Psi}_{tX}(s)=Du^{\Psi}_{txm}(s)$ and $D\mathbbm{r}^{\Psi}_{tX}(s)=D\mathbbm{r}^{\Psi}_{txm}(s)$ such that the governing system (\ref{J flow of FBSDE}) can be rewritten as

\begin{equation}
\scalemath{0.92}{
\h{-10pt}\left\{
\begin{aligned}
D \mathbb{Y}_{txm}^\Psi  (s)
&=\Psi
+\displaystyle\int_{t}^{s}\Big[ \diff_y  u(\mathbb{Y}_{txm},\mathbb{Z}_{txm})(\tau)\Big]
\Big[D \mathbb{Y}_{txm}^\Psi  (\tau) \Big]
+
\Big[\diff_z  u (\mathbb{Y}_{txm},\mathbb{Z}_{txm}) (\tau)\Big]
\Big[D \mathbb{Z}_{txm}^\Psi  (\tau)\Big]  d\tau;\\
D \mathbb{Z}_{txm}^\Psi (s)
&=h_{xx}(\mathbb{Y}_{txm}(T))D \mathbb{Y}_{txm}^\Psi (T)
+D_{X}^2F_{T}(\mathbb{Y}_{txm}(T)  \otimes  m)(D \mathbb{Y}_{txm}^\Psi (T))\\
&\h{10pt}+\displaystyle\int^T_s\bigg\{l_{xx}(\mathbb{Y}_{txm}(\tau),u(\tau))D \mathbb{Y}_{txm}^\Psi (\tau)
+l_{xv}(\mathbb{Y}_{txm}(\tau),u(\tau))\Big[ \diff_y  u(\mathbb{Y}_{txm},\mathbb{Z}_{txm})(\tau)\Big]
D \mathbb{Y}_{txm}^\Psi (\tau)\\
&\h{45pt}+l_{xv}(\mathbb{Y}_{txm}(\tau),u(\tau))
\Big[ \diff_z  u (\mathbb{Y}_{txm},\mathbb{Z}_{txm})(\tau)\Big]
D \mathbb{Z}_{txm}^\Psi (\tau)
+D^2_{X}F(\mathbb{Y}_{t \bigcdot m}(\tau)  \otimes  m)\pig(
D\mathbb{Y}_{txm}^\Psi (\tau)\pig)\bigg\}d\tau\h{-15pt} \\
&\h{10pt}-\int^T_s\sum_{j=1}^{n}D\mathbbm{r}^\Psi_{txm,j}(\tau)dw_{j}(\tau).
\end{aligned}\right.}
\label{J flow of FBSDE, X=I}
\end{equation}
The existence of solution to the above system is clearly established in Lemma \ref{lem, Existence of J flow}. Result (\ref{DX^2 V = DZ}) can also be rewritten as
\begin{equation}
D\mathbb{Z}_{txm}^\Psi(t)=D_{X}^{2}V(m,t)(\Psi).
\label{eq:4-173}
\end{equation}
On the other hand, we consider the derivative of $\pig(\mathbb{Y}_{txm}(s), \mathbb{Z}_{txm}(s), u_{txm}(s), \mathbbm{r}_{txm,j}(s)\pig)$ with respect to $x$, for instance, the notation $\diff_x \mathbb{Y}_{txm}(s)$ denotes the matrix-valued random process with $(i,j)$ entry $\pig[D \mathbb{Y}^{e_j}_{txm}(s)\pigr]_i$, here the $l$-th coordinate of $e_j$ is $\delta_{jl}$. Without involving the test random vector $\Psi$, the system (\ref{J flow of FBSDE, X=I}) can be rewritten as 
\begin{equation}
\scalemath{0.91}{
\h{-10pt}\left\{
\begin{aligned}
\diff_x  \mathbb{Y}_{txm}  (s)
&=Id
+\displaystyle\int_{t}^{s}\Big[ \diff_y  u(\mathbb{Y}_{txm},\mathbb{Z}_{txm})(\tau)\Big] 
\Big[\diff_x  \mathbb{Y}_{txm}  (\tau) \Big]
+\Big[\diff_z  u (\mathbb{Y}_{txm},\mathbb{Z}_{txm}) (\tau)\Big] 
\Big[\diff_x  \mathbb{Z}_{txm}  (\tau)\Big]  d\tau;\\
\diff_x  \mathbb{Z}_{txm} (s)
&=h_{xx}(\mathbb{Y}_{txm}(T))\diff_x  \mathbb{Y}_{txm} (T)
+D_{X}^2F_{T}(\mathbb{Y}_{txm}(T)  \otimes  m)(\diff_x  \mathbb{Y}_{txm} (T))\\
&\h{10pt}+\displaystyle\int^T_s\bigg\{l_{xx}(\mathbb{Y}_{txm}(\tau),u(\tau))\diff_x  \mathbb{Y}_{txm} (\tau)
+l_{xv}(\mathbb{Y}_{txm}(\tau),u(\tau))\Big[ \diff_y  u(\mathbb{Y}_{txm},\mathbb{Z}_{txm})(\tau)\Big] 
\diff_x  \mathbb{Y}_{txm} (\tau)\h{-20pt}\\
&\h{45pt}+l_{xv}(\mathbb{Y}_{txm}(\tau),u(\tau))
\Big[ \diff_z  u (\mathbb{Y}_{txm},\mathbb{Z}_{txm})(\tau)\Big] 
\diff_x  \mathbb{Z}_{txm} (\tau)
+D^2_{X}F(\mathbb{Y}_{t \bigcdot m}(\tau)  \otimes  m)\pig(
\diff_x \mathbb{Y}_{txm} (\tau)\pig)\bigg\}d\tau\h{-30pt} \\
&\h{10pt}-\int^T_s\sum_{k=1}^{n}\diff_x \mathbbm{r}_{txm,k}(\tau)dw_{j}(\tau).
\end{aligned}\right.}
\label{J flow of FBSDE, X=I, matrix}
\end{equation}
Due to the uniqueness of the system (\ref{J flow of FBSDE}) and Lemma \ref{lem, Existence of Frechet derivatives}, we compare the systems (\ref{J flow of FBSDE, X=I}) and (\ref{J flow of FBSDE, X=I, matrix}) to obtain that
\begin{equation}
D \mathbb{Y}_{txm}^\Psi(s)=\diff_x  \mathbb{Y}_{txm}(s)\Psi\;,
D \mathbb{Z}_{txm}^\Psi(s)=\diff_x  \mathbb{Z}_{txm}(s)\Psi\;,
D u_{txm}^\Psi(s)=\diff_x  u_{txm}(s)\Psi\;,
D\mathbbm{r}_{txm,k}^\Psi(s)=\diff_x  \mathbbm{r}_{txm,k}(s)\Psi,
\label{eq:4-178}
\end{equation}
where each term on the right hand side is  the matrix product of the gradient matrix and the random vector $\Psi$. Recalling from (\ref{eq:2-232}), we have
\begin{equation}
\h{-17pt}\begin{aligned}
D_{X}^{2}F(\mathbb{Y}_{t \bigcdot m}  (s)  \otimes  m)(D \mathbb{Y}_{txm}^\Psi  (s))
=\:&\diff^2_x \dfrac{dF}{d\nu}(\mathbb{Y}_{t \bigcdot m}  \otimes  m)(\mathbb{Y}_{txm})D \mathbb{Y}_{txm}^\Psi  (s)\\
&+\widetilde{\mathbb{E}}\left[\int_{\mathbb{R}^{n}}\h{-5pt}\diff_x \diff_{\tilde{x}}\dfrac{d^{\h{.5pt}2}\h{-.7pt}  F}{d\nu^{2}}( \mathbb{Y}_{t \bigcdot m}  (s)  \otimes  m)(\mathbb{Y}_{t x m}  (s),\widetilde{\mathbb{Y}}_{t\tilde{x}m}  (s))
D \widetilde{\mathbb{Y}}_{t \tilde{x} m}^{\tilde{\Psi}}  (s)
dm(\tilde{x})\right].\h{-20pt}
\end{aligned}
\end{equation}
Here $(\widetilde{\mathbb{Y}}_{t \tilde{x} m} (s),D \widetilde{\mathbb{Y}}_{t \tilde{x} m}^{\tilde{\Psi}}  (s))$ is an independent copy of $(\mathbb{Y}_{t x m} (s),D \mathbb{Y}_{tx m}^{\Psi} (s))$. A similar expression holds for $F_T$. If $\Psi$ is independent of $\mathcal{W}^s_0$ with $\mathbb{E}(\Psi)=0$, then 

\noindent(a)\: we use (\ref{eq:4-178}) to obtain
\begin{equation}
\begin{aligned}
D_{X}^{2}F(\mathbb{Y}_{t \bigcdot m}  (s)  \otimes  m)(D \mathbb{Y}_{txm}^\Psi  (s))
=\:& \diff^2_x \dfrac{dF}{d\nu}(\mathbb{Y}_{t \bigcdot m}  \otimes  m)(\mathbb{Y}_{txm})D \mathbb{Y}_{txm}^\Psi  (s)\\
&\mbox{\fontsize{10.5}{10}\selectfont\(
+\displaystyle\int_{\mathbb{R}^{n}}
\widetilde{\mathbb{E}}\left[\diff_x \diff_{\tilde{x}}\dfrac{d^{\h{.5pt}2}\h{-.7pt}  F}{d\nu^{2}}( \mathbb{Y}_{t \bigcdot m}  (s)  \otimes  m)(\mathbb{Y}_{t x m}  (s),\widetilde{\mathbb{Y}}_{t\tilde{x}m}  (s))
\diff_x  \widetilde{\mathbb{Y}}_{t \tilde{x} m}  (s)
\right] \widetilde{\mathbb{E}}\pig( \widetilde{\Psi}\pig)
dm(\tilde{x})\)}\\
=\:& \diff^2_x \dfrac{dF}{d\nu}(\mathbb{Y}_{t \bigcdot m}  \otimes  m)(\mathbb{Y}_{txm})D \mathbb{Y}_{txm}^\Psi  (s);
\end{aligned}
\label{3099}
\end{equation}
(b)\: moreover, under the assumptions of Proposition \ref{prop DXX V}, similar to (\ref{3099}), Proposition \ref{prop DXX V} and (\ref{eq:4-178}) imply
\begin{equation}
\diff^2_x \dfrac{dV}{d\nu}(m,t)(x)\Psi
=D_{X}^{2}V(m,t)(\Psi)
=D\mathbb{Z}_{txm}^\Psi(t)
=\diff_x \mathbb{Z}_{txm}(t)\Psi.
\label{D^2_x dV=DZ Psi}
\end{equation}
\mycomment{
So we can state
\begin{prop}\label{prop4-4}
Under the assumptions of Proposition \ref{prop DXX V}, for $\Psi \in\mathcal{H}_{m}$ measurable to $\mathcal{W}^t_0$ with $\mathbb{E}(\Psi)=0$, we have 
\end{prop}

\begin{equation}
D_{X}^{2}V(m,t)(\Psi)=\diff^2_x \dfrac{dV}{d\nu}(m,t)(x)\Psi
\label{eq:4-180}
\end{equation}}

\mycomment{
We can also interpret the system (\ref{J flow of FBSDE, X=I}) as a control problem following the proof of Proposition \ref{prop DXX V}. Taking $\Psi \in \mathcal{H}_m$ independent of $\mathcal{W}_{t}$ and controls in $\mathscr{V}_{t\Psi} \in L_{\mathcal{W}_{t \Psi}}^{2}(t,T;\mathcal{H}_{m})$, we have the state equation 
\begin{equation}
\mathscr{X}_{t\Psi}(s)=\Psi+\int_{t}^{s}\mathscr{V}_{t\Psi}(\tau)d\tau\label{eq:8-1-1}
\end{equation}
and the objective function is 

\begin{equation}
\begin{aligned}
\mathscr{J}_{txm\Psi}(\mathscr{V}_{t\Psi})
=\:&\dfrac{1}{2}\int_{t}^{T}\left\langle \left[l_{xx}(\mathbb{Y}_{txm}(s),u_{txm}(s))+\diff^2_x \dfrac{dF}{d\nu}(\mathbb{Y}_{txm}(s)  \otimes   m)(\mathbb{Y}_{txm}(s))\right]\mathscr{X}_{t\Psi}(s),\mathscr{X}_{t\Psi}(s)\right\rangle_{\mathcal{H}_m}\\
&\h{20pt}+2\pig\langle l_{xv}(\mathbb{Y}_{txm}(s),u_{txm}(s))\mathscr{V}_{t\Psi}(s),\mathscr{X}_{t\Psi}(s)\pigr\rangle_{\mathcal{H}_m}
+\pig\langle l_{vv}(\mathbb{Y}_{txm}(s),u_{txm}(s))\mathscr{V}_{t\Psi}(s),\mathscr{V}_{t\Psi}(s)\pigr\rangle_{\mathcal{H}_m} ds\\
&+\dfrac{1}{2}\left\langle \left[h_{xx}(\mathbb{Y}_{txm}(T))+\diff^2_x \dfrac{dF_T}{d\nu}(\mathbb{Y}_{txm}(T)  \otimes   m)(\mathbb{Y}_{txm}(T))\right]\mathscr{X}_{t\Psi}(T),\mathscr{X}_{t\Psi}(T)\right\rangle_{\mathcal{H}_m}
\end{aligned}
\label{eq:8-2-1}
\end{equation}
The optimal control is $D u_{txm}^\Psi(s) = \diff_x  u_{txm}(s) \Psi$ and the optimal state is $D \mathbb{Y}_{txm}^\Psi(s) =\diff_x  \mathbb{Y}_{txm}(s)\Psi.$}

\section{Bellman and Master Equations}\label{sec. master equation}

\subsection{Optimality Principle and Bellman Equation}\label{Bellman Equation}
The dynamical programming principle is the key to establishing the Bellman equation, it is read as: for $\epsilon \in (0,T-t)$,
\begin{equation}
V(X  \otimes  m,t)=\int_{t}^{t+\epsilon}\int_{\mathbb{R}^{n}}\mathbb{E}\Big[l\pig(\mathbb{Y}_{tX}(s),u_{tX}(s)\pig)\Big]dm(x)
+F(\mathbb{Y}_{tX}(s)  \otimes  m)\h{1pt}ds
+V(\mathbb{Y}_{tX}(t+\epsilon)  \otimes  m,t+\epsilon).
\label{eq:3-16}
\end{equation}
The above follows by the usual arguments consisting of showing
that for $X^{\epsilon}=\mathbb{Y}_{tX}(t+\epsilon)$ and $s \in [t+\epsilon,T]$, the flow property of
\begin{equation}
\mathbb{Y}_{t+\epsilon, X^{\epsilon}}(s)=\mathbb{Y}_{tX}(s),\;
\mathbb{Z}_{t+\epsilon, X^{\epsilon}}(s)=\mathbb{Z}_{tX}(s),\;
u_{t+\epsilon, X^{\epsilon}}(s)=u_{tX}(s) \h{10pt} \text{and}\h{10pt}
\mathbbm{r}_{t+\epsilon, X^{\epsilon},j}(s)=\mathbbm{r}_{tX,j}(s),
\label{flow property}
\end{equation}
which is clearly true due to the uniqueness of solution to the FBSDE \eqref{forward FBSDE}-\eqref{1st order condition}. It implies that
\begin{equation}
\begin{aligned}
V(\mathbb{Y}_{tX}(t+\epsilon)  \otimes  m,t+\epsilon)
=\:&\int_{t+\epsilon}^{T}\int_{\mathbb{R}^{n}}
\mathbb{E}\Big[l\pig(\mathbb{Y}_{tX}(s),u_{tX}(s)\pig)\Big]dm(x)
+F(\mathbb{Y}_{tX}(s)  \otimes  m)\h{2pt}ds\\
&+\int_{\mathbb{R}^n} \mathbb{E}\pig[h(\mathbb{Y}_{tX}(T)\pig] d m(x)
+F_{T}\pig(\mathbb{Y}_{tX}(T)  \otimes  m\pig).
\end{aligned}
\end{equation}
We can now state the following.
\begin{thm}
Let $X\in L^2_{\mathcal{W}^{\indep}_t} (\mathcal{H}_m)$. Under the assumptions (\ref{assumption, bdd of l, lx, lv})-(\ref{assumption, convexity of h}), (\ref{assumption, bdd of F F_T})-(\ref{assumption, convexity of D^2F, D^2F_T}) and (\ref{uniqueness condition}), we have the following:
\begin{itemize}
\item[(a).] The value function $V(X  \otimes  m,t)$ defined in (\ref{def value function}) is a solution to the Bellman equation
\begin{equation}
\h{-5pt}\left\{
\begin{aligned}
&\mbox{\fontsize{8.7}{10}\selectfont\(
\dfrac{\partial \Phi}{\partial t}(X  \otimes  m,t)+\mathbb{E}\left[\displaystyle\int_{\mathbb{R}^{n}}H\pig(X,D_{X}\Phi(X  \otimes  m,t)\pig)\,dm(x)\right]
+F(X  \otimes  m)+\dfrac{1}{2}\left\langle D_{X}^{2}\Phi(X  \otimes  m,t)\left(\displaystyle\sum_{j=1}^{n}\eta^{j}\mathcal{N}_{t}^{j}\right),\displaystyle\sum_{j=1}^{n}\eta^{j}\mathcal{N}_{t}^{j}\right\rangle_{\mathcal{H}_m}\h{-10pt}=0;
\)}\\
&\Phi(X  \otimes  m,T)=\mathbb{E}\left[\displaystyle\int_{\mathbb{R}^{n}}h(X)dm(x)\right]+F_{T}(X  \otimes  m),
\end{aligned}\right.
\label{eq. bellman eq.}
\end{equation}
where $\eta^{j}$ is the $j$-th column vector of the matrix $\eta$, $H$ is the Hamiltonian defined in (\ref{def. Hamiltonian}); and $\mathcal{N}_{t}^{j}$'s are Gaussian random variables each of them having a mean $0$ and a unit variance, and they are independent of each other and also of $X$;
\item[(b).] Among all functions satisfying the regularity properties in (\ref{bdd V}), (\ref{lip cts of D_XV in X}), (\ref{cts of D_X V in t}), (\ref{eq:4-26}), (\ref{DX^2 V < psi}) and (\ref{ct of D^2_X V}), $V(X  \otimes  m,t)$ is the unique solution to the Bellman equation (\ref{eq. bellman eq.}).
\end{itemize}
\label{prop bellman} 
\end{thm}
Its proof is placed in \hyperref[app, prop bellman]{Appendix}. Particularly, when $X=\mathcal{I}_x$, thanks to
(\ref{D^2_x dV=DZ Psi}), we have

\begin{equation}
\begin{aligned}
\left\langle 
D_{X}^{2}V(m,t)\left(\sum_{j=1}^{n}\eta^{j}\mathcal{N}_{t}^{j}\right),\sum_{j=1}^{n}\eta^{j}\mathcal{N}_{t}^{j}\right\rangle_{\mathcal{H}_m}
&=\left\langle 
\diff^2_x\dfrac{dV}{d\nu}(m,t)(x)\sum_{j=1}^{n}\eta^{j}\mathcal{N}_{t}^{j},\sum_{j=1}^{n}\eta^{j}\mathcal{N}_{t}^{j}\right\rangle_{\mathcal{H}_m}\\
&=\int_{\mathbb{R}^{n}}\sum_{j=1}^{n}\diff^2_x \dfrac{dV}{d\nu}(m,t)(x)\eta^{j}\cdot\eta^{j}dm(x),
\end{aligned}
\end{equation}
by then the Bellman equation in this case reads
\begin{equation}
\left\{
\begin{aligned}
-&\dfrac{\partial V}{\partial t}(m,t)
-\dfrac{1}{2}\sum^n_{j=1}\int_{\mathbb{R}^{n}}\diff^2_x \dfrac{dV}{d\nu}(m,t)(x)\eta^{j}\cdot\eta^{j}dm(x)=\int_{\mathbb{R}^{n}}H\left(x,\diff_x \dfrac{dV}{d\nu}(m,t)(x)\right)\,dm(x)+F(m),\\
&V(m,T)=\int_{\mathbb{R}^{n}}h(x)dm(x)+F_{T}(m).
\end{aligned}
\label{eq:5-24}
\right.
\end{equation}

\subsection{Master Equation}
Up to the last section, the mean field type control problem (\ref{eq:3-501})-(\ref{eq:3-503}) has already been completely resolved. In this section, we introduce the master equation of the mean field type control problem, and show how we apply the assertions we have so far to solve the master equation. By using (\ref{D^2_x dV=DZ Psi}), the Bellman equation in (\ref{eq:5-24}) is written as follows, 

\begin{equation}
-\dfrac{\partial V}{\partial t}(m,t)-\dfrac{1}{2}\sum_{j=1}^{n}\int_{\mathbb{R}^{n}}
\diff_\xi  \mathbb{Z}_{t\xi m}(t)\eta^{j} \cdot \eta^{j}dm(\xi)=\int_{\mathbb{R}^{n}}H(\xi,\mathbb{Z}_{t\xi m}(t))dm(\xi)+F(m).
\label{eq:5-25}
\end{equation}
For if the linear functional derivatives of $\mathbb{Z}_{t \xi m}(s)$
and $\diff_\xi  \mathbb{Z}_{t\xi m}(s)$ with respect to $m$ exist (their existence will be established in Propositions \ref{prop bdd of functiona d of Y Z u r} and \ref{prop bdd of Dxi dnu Y Z u r} respectively), which are denoted by $ \dfrac{d \mathbb{Z}_{t \xi m}}{d\nu} (x,s)$
and $ \dfrac{d\diff_\xi  \mathbb{Z}_{t\xi m}}{d\nu}(x,s)$ respectively, further, if we  set
\begin{equation}
U(x,m,t):=\dfrac{d V}{d\nu}(m,t)(x),
\label{eq:7-1}
\end{equation} 
by differentiating (\ref{eq:5-25}) both sides with respect to $m$, we can write 
\begin{equation}
\begin{aligned}
&-\dfrac{\partial U}{\partial t}(x,m,t)
-\dfrac{1}{2}\sum_{j=1}^{n}\diff_x  \mathbb{Z}_{txm}(t)\eta^{j}\cdot \eta^{j}
-\dfrac{1}{2}\sum_{j=1}^{n}\int_{\mathbb{R}^{n}}\dfrac{d\diff_\xi \mathbb{Z}_{t\xi m}}{d\nu}(x,t)\eta^{j}\cdot\eta^{j}dm(\xi)\\
&=H(x,\mathbb{Z}_{txm}(t))+\int_{\mathbb{R}^{n}}H_{p}(\xi,\mathbb{Z}_{t \xi m}(t))
\dfrac{d \mathbb{Z}_{t \xi m}}{d\nu} (x,t)dm(\xi)+\dfrac{dF}{d\nu}(m)(x).
\end{aligned}
\label{eq:7-3}
\end{equation}
The relation of (\ref{D_x p_nu V = Z}) implies that
\begin{equation}
\diff_x  U(x,m,t)=\mathbb{Z}_{txm}(t)
\h{10pt}\text{and}\h{10pt}
\diff^2_x  U(x,m,t)=\diff_x  \mathbb{Z}_{txm}(t).
\label{eq:7-2}
\end{equation}
In particular, we have used Lemma \ref{lem, Existence of Frechet derivatives} to justify the Fr\'echet differentiability of the second term in (\ref{eq:7-2}), which is obtained by putting $X=\mathcal{I}_x=x$ in $D\mathbb{Z}_{tX}(t)$ mentioned in Lemma \ref{lem, Existence of Frechet derivatives} as shown in (\ref{eq:4-178}), and the equation of $D\mathbb{Z}_{tX}(t)$ reduces to equation of $\diff_x \mathbb{Z}_{txm}(s)$ in (\ref{J flow of FBSDE, X=I, matrix}). It further implies, by differentiating with respect to $m$ that,
\begin{equation}
\diff_\xi 
\dfrac{dU}{d\nu}(\xi,m,t)(x)=\dfrac{d \mathbb{Z}_{t \xi m}}{d\nu} (x,t)
\h{10pt}\text{and}\h{10pt}
\diff_\xi ^{2}\dfrac{dU}{d\nu}(\xi,m,t)(x)= \dfrac{d\diff_\xi  \mathbb{Z}_{t\xi m}}{d\nu}(x,t).
\label{eq:7-4}
\end{equation}
Therefore, by plugging (\ref{eq:7-4}) into Equation (\ref{eq:7-3}), we finally obtain the Cauchy problem of the master equation
\begin{equation}
\left\{
\begin{aligned}
&-\dfrac{\partial U}{\partial t}(x,m,t)
-\dfrac{1}{2}\sum_{j=1}^{n}\diff^2_x U(x,m,t)\eta^{j}\cdot\eta^{j}-\dfrac{1}{2}\sum_{j=1}^{n}\int_{\mathbb{R}^{n}}\diff_\xi ^{2}\dfrac{dU}{d\nu}(\xi,m,t)(x)\eta^{j}\cdot \eta^{j}dm(\xi)\\
&\h{50pt}=
H(x,\diff_x U(x,m,t))+\int_{\mathbb{R}^{n}}H_{p}(\xi,\diff_\xi  U(\xi,m,t))\diff_\xi \dfrac{dU}{d\nu}(\xi,m,t)(x)dm(\xi)+\dfrac{dF}{d\nu}(m)(x)\\
&U(x,m,T)=h(x)+\dfrac{dF_T}{d\nu}(m)(x),
\end{aligned}
\right.
\label{eq:7-5}
\end{equation}
also see the discussions in \cite{BFY15,BFY17}. We next aim to justify the existences of the derivatives $ \dfrac{d \mathbb{Z}_{t \xi m}}{d\nu} (x,s)$
and $ \dfrac{d\diff_\xi  \mathbb{Z}_{t\xi m}}{d\nu}(x,s)$. To this end, we consider the linear system for the linear functional derivatives of $\mathbb{Y}_{t \xi m}(s)$, $\mathbb{Z}_{t \xi m}(s)$, $u_{t \xi m}(s)$, $\mathbbm{r}_{t \xi m,j}(s)$, which are denoted by $\dfrac{d \mathbb{Y}_{t \xi m}}{d\nu}(x,s)$, $\dfrac{d \mathbb{Z}_{t \xi m}}{d\nu}(x,s)$, $\dfrac{d u_{t \xi m}}{d\nu}(x,s)$, $\dfrac{d\mathbbm{r}_{t \xi m,j}}{d\nu}(x,s)$, respectively. From the system (\ref{forward FBSDE, X=I})-(\ref{backward FBSDE, X=I}), using the rules of differentiation with respect to $m$, we obtain
\begin{empheq}[left=\h{-10pt}\empheqbiglbrace]{align}
\dfrac{d \mathbb{Y}_{t \xi m}}{d\nu}(x,s)
=&\: \int_{t}^{s}\dfrac{d u_{t \xi m}}{d\nu}(x,\tau)d\tau;
\label{dnu, forward FBSDE, X=I}\\
\dfrac{d \mathbb{Z}_{t \xi m}}{d\nu}(x,s)
=&\:h_{xx}(\mathbb{Y}_{t\xi m}(T))
\dfrac{d \mathbb{Y}_{t \xi m}}{d\nu}(x,T)
+\diff_x^2\dfrac{dF_T}{d\nu}(\mathbb{Y}_{t\bigcdot m}(T) \otimes  m)\pig(\mathbb{Y}_{t\xi m}(T)\pig)\dfrac{d \mathbb{Y}_{t \xi m}}{d\nu}(x,T)
\nonumber\\
&+\widetilde{\mathbb{E}}\left[ \int_{\mathbb{R}^n} \diff_x\diff_{\tilde{x}} \dfrac{d^{\h{.5pt}2}\h{-.7pt}  F_T}{d\nu^2}(\mathbb{Y}_{t\bigcdot m}(T) \otimes  m)\pig(\mathbb{Y}_{t\xi m}(T),\widetilde{\mathbb{Y}}_{t\tilde{\xi} m}(T)\pig)
\dfrac{d \widetilde{\mathbb{Y}}_{t \tilde{\xi} m}}{d\nu}(x,T)
dm(\widetilde{\xi})\right]\nonumber\\
&+\widetilde{\mathbb{E}}\left[\diff_x  \dfrac{d^{\h{.5pt}2}\h{-.7pt}  F_T}{d\nu^2}(\mathbb{Y}_{t\bigcdot m}(T) \otimes  m)\pig(\mathbb{Y}_{t\xi m}(T),\widetilde{\mathbb{Y}}_{tx m}(T)\pig)\right]\nonumber\\
&+\int^T_s \left[l_{xx}\pig(\mathbb{Y}_{t\xi m}(\tau),u_{t\xi m}(\tau)\pig)
+\diff^2_x \dfrac{dF}{d\nu}(\mathbb{Y}_{t\bigcdot m}(\tau) \otimes  m)\pig(\mathbb{Y}_{t\xi m}(\tau)\pig)\right]\dfrac{d \mathbb{Y}_{t \xi m}}{d\nu}(x,\tau) d\tau\nonumber\\
&+\int^T_s l_{xv}\pig(\mathbb{Y}_{t\xi m}(\tau),u_{t\xi m}(\tau)\pig)
\dfrac{d u_{t \xi m}}{d\nu}(x,\tau) d\tau\nonumber\\
&+\int^T_s\widetilde{\mathbb{E}}\left[ \int_{\mathbb{R}^n} \diff_x\diff_{\tilde{x}} \dfrac{d^{\h{.5pt}2}\h{-.7pt}  F}{d\nu^2}(\mathbb{Y}_{t\bigcdot m}(\tau) \otimes  m)\pig(\mathbb{Y}_{t\xi m}(\tau),\widetilde{\mathbb{Y}}_{t\tilde{\xi} m}(\tau)\pig)
\dfrac{d \widetilde{\mathbb{Y}}_{t \tilde{\xi} m}}{d\nu}(x,\tau)
dm(\widetilde{\xi})\right]d\tau\nonumber\\
&+\int^T_s\widetilde{\mathbb{E}}\left[\diff_x  \dfrac{d^{\h{.5pt}2}\h{-.7pt}  F}{d\nu^2}(\mathbb{Y}_{t\bigcdot m}(\tau) \otimes  m)\pig(\mathbb{Y}_{t\xi m}(\tau),\widetilde{\mathbb{Y}}_{tx m}(\tau)\pig)\right]d\tau-\int^T_s\sum_{j=1}^{n}\dfrac{d\mathbbm{r}_{t \xi m,j}}{d\nu}(x,\tau) dw_{j}(\tau).
\label{dnu, backward FBSDE, X=I}
\end{empheq}
In order to guarantee the validity of the system (\ref{dnu, forward FBSDE, X=I})-(\ref{dnu, backward FBSDE, X=I}), we assume further regularity conditions on $F$ and $F_T$.
\begin{ass}
For any $x$, $\widetilde{x} \in \mathbb{R}^n$, we assume
\begin{itemize}
\item[(i).] the regularities of the second-order linear functional derivative: 
\begin{equation}
\left|\diff_x\dfrac{d^{\h{.5pt}2}\h{-.7pt}  F}{d\nu^{2}}(m)(x,\widetilde{x})\right|
\leq c(1+|\widetilde{x}|)\h{1pt},\h{10pt}
\left|\diff_x\dfrac{d^{\h{.5pt}2}\h{-.7pt}  F_T}{d\nu^{2}}(m)(x,\widetilde{x})\right|
\leq c_{T}(1+|\widetilde{x}|);
\label{assumption bdd of Dd^2F, F_T}
\end{equation}

\begin{equation}
\left|\diff_x \diff_{\tilde{x}}\dfrac{d^{\h{.5pt}2}\h{-.7pt}  F}{d\nu^{2}}(m)(x,\widetilde{x})\right|
\leq c\h{1pt},\h{10pt}
\left|\diff_x \diff_{\tilde{x}}\dfrac{d^{\h{.5pt}2}\h{-.7pt}  F_T}{d\nu^{2}}(m)(x,\widetilde{x})\right|
\leq c_{T};
\label{assumption bdd of D^2d^2F, F_T}
\end{equation}
\item[(ii).] all the derivatives in (i) and (ii) are continuous in $(m,x)$ and $(m,x,\widetilde{x})$ respectively.
\end{itemize}
\label{ass. ass 1}
\end{ass}

We then have 
\begin{prop}
\label{prop bdd of functiona d of Y Z u r} Under the assumptions (\ref{assumption, bdd of l, lx, lv})-(\ref{assumption, convexity of h}), (\ref{assumption, bdd of F F_T})-(\ref{assumption, convexity of D^2F, D^2F_T}), (\ref{uniqueness condition}) and  Assumption \ref{ass. ass 1}, if there is a $\delta_1 \in (0,1)$ such that 
\begin{equation}
(1-\delta_1)\lambda - \big( c_h'+c_T'+c_T \big)T
-\big( c+c'+c_l' \big)\dfrac{T^2}{2}>0,
\label{ass in prop 6.2}
\end{equation} then the system (\ref{dnu, forward FBSDE, X=I})-(\ref{dnu, backward FBSDE, X=I}) has the unique solution $\left(\dfrac{d \mathbb{Y}_{t \xi m}}{d\nu}(x,s), \dfrac{d \mathbb{Z}_{t \xi m}}{d\nu}(x,s), \dfrac{d u_{t \xi m}}{d\nu}(x,s), \dfrac{d\mathbbm{r}_{t \xi m,j}}{d\nu}(x,s)\right)$, and they are the linear functional derivatives of $\mathbb{Y}_{t \xi m}(s),\mathbb{Z}_{t \xi m}(s),u_{t \xi m}(s),\mathbbm{r}_{t \xi m,j}(s)$. Moreover, for any $s \in [t,T]$, they satisfies the following $L^2$-boundedness

\begin{equation}
\begin{aligned}
&\mathbb{E}\left[\int_{\mathbb{R}^{n}}\left|\dfrac{d \mathbb{Y}_{t \xi m}}{d\nu}(x,s)\right|^{2}dm(\xi)\right]
\h{1pt},\h{10pt}\mathbb{E}\left[\int_{\mathbb{R}^{n}}\left|\dfrac{d \mathbb{Z}_{t \xi m}}{d\nu}(x,s)\right|^{2}dm(\xi)\right],\\
&\mathbb{E}\left[\int_{\mathbb{R}^{n}}\left|\dfrac{d u_{t \xi m}}{d\nu}(x,s)\right|^{2}dm(\xi)\right]
\h{1pt},\h{10pt}\sum_{j=1}^{n}\int_{t}^{T}\mathbb{E}\left[\int_{\mathbb{R}^n}\left|\dfrac{d \mathbbm{r}_{t \xi m,j}}{d\nu}(x,s)\right|^{2}dm(\xi)\right]ds\leq C_{10}(1+|x|^{2}),
\end{aligned}
\label{eq:7-6}
\end{equation}
where $C_{10}$ is a positive constant depending only on $\delta_1$, $n$, $\lambda$, $\eta$, $c$, $c_l$, $c_h$, $c_T$, $c'$, $c_l'$, $c_h'$, $c_T'$ and $T$.
\end{prop}
\begin{proof}
As in the proof of Lemma \ref{lem, Existence of J flow}, by considering the finite differences of the processes, we can show that the unique solution to (\ref{dnu, forward FBSDE, X=I})-(\ref{dnu, backward FBSDE, X=I}) equals the linear functional derivatives of $\pig(\mathbb{Y}_{t \xi m}(s),\mathbb{Z}_{t \xi m}(s),u_{t \xi m}(s),\mathbbm{r}_{t \xi m,j}(s)\pig)$. We omit the proof here. To establish the estimates in (\ref{eq:7-6}), we differentiate the first order condition (\ref{1st order condition}) with respect to $m$ such that
\begin{equation}
 l_{vx}\pig(\mathbb{Y}_{t \xi m}(s),u_{t \xi m}(s)\pig)\dfrac{d \mathbb{Y}_{t \xi m}}{d\nu}(x,s)
+l_{vv}\pig(\mathbb{Y}_{t \xi m}(s),u_{t \xi m}(s)\pig)\dfrac{d u_{t \xi m}}{d\nu}(x,s)
+\dfrac{d \mathbb{Z}_{t \xi m}}{d\nu}(x,s)=0.
\label{1st order condition, dnu}
\end{equation}
We then apply It\^o lemma to the inner product $\left\langle \dfrac{d \mathbb{Y}_{t \xi m}}{d\nu}(x,s),\dfrac{d \mathbb{Z}_{t \xi m}}{d\nu}(x,s)\right\rangle_{\mathcal{H}_m}$ and then integrate from $t$ to $T$, together with (\ref{1st order condition, dnu}), to obtain
\begingroup
\allowdisplaybreaks
\begin{align*}
&\Bigg\langle \dfrac{d \mathbb{Y}_{t \xi m}}{d\nu}(x,T),
h_{xx}(\mathbb{Y}_{txm}(T))
\dfrac{d \mathbb{Y}_{t \xi m}}{d\nu}(x,T)
+\diff_x^2\dfrac{dF_T}{d\nu}(\mathbb{Y}_{t\bigcdot m}(T) \otimes  m)\pig(\mathbb{Y}_{txm}(T)\pig)\dfrac{d \mathbb{Y}_{t \xi m}}{d\nu}(x,T)
\nonumber\\
&\h{10pt}+\widetilde{\mathbb{E}}\left[ \int_{\mathbb{R}^n} \diff_x\diff_{\tilde{x}} \dfrac{d^{\h{.5pt}2}\h{-.7pt}  F_T}{d\nu^2}(\mathbb{Y}_{t\bigcdot m}(T) \otimes  m)\pig(\mathbb{Y}_{t\xi m}(T),\widetilde{\mathbb{Y}}_{t\tilde{\xi} m}(T)\pig)
\dfrac{d \widetilde{\mathbb{Y}}_{t \tilde{\xi} m}}{d\nu}(x,T)
dm(\widetilde{\xi})\right]\nonumber\\
&\h{190pt}+\widetilde{\mathbb{E}}\left[\diff_x  \dfrac{d^{\h{.5pt}2}\h{-.7pt}  F_T}{d\nu^2}(\mathbb{Y}_{t\bigcdot m}(T) \otimes  m)\pig(\mathbb{Y}_{t\xi m}(T),\widetilde{\mathbb{Y}}_{tx m}(T)\pig)\right]\Bigg\rangle_{\mathcal{H}_m}\nonumber\\
=\:&-\int^T_t  \Bigg\langle \dfrac{d \mathbb{Y}_{t \xi m}}{d\nu}(x,\tau), \left[l_{xx}\pig(\mathbb{Y}_{t\xi m}(\tau),u_{t\xi m}(\tau)\pig)
+\diff^2_x \dfrac{dF}{d\nu}(\mathbb{Y}_{t\bigcdot m}(\tau) \otimes  m)\pig(\mathbb{Y}_{t\xi m}(\tau)\pig)\right]\dfrac{d \mathbb{Y}_{t \xi m}}{d\nu}(x,\tau) \nonumber\\
&+ l_{xv}\pig(\mathbb{Y}_{t\xi m}(\tau),u_{t\xi m}(\tau)\pig)
\dfrac{d u_{t \xi m}}{d\nu}(x,\tau) 
+  \widetilde{\mathbb{E}}\left[ \int_{\mathbb{R}^n} \diff_x\diff_{\tilde{x}} \dfrac{d^{\h{.5pt}2}\h{-.7pt}  F}{d\nu^2}(\mathbb{Y}_{t\bigcdot m}(\tau) \otimes  m)\pig(\mathbb{Y}_{t\xi m}(\tau),\widetilde{\mathbb{Y}}_{t\tilde{\xi} m}(\tau)\pig)
\dfrac{d \widetilde{\mathbb{Y}}_{t \tilde{\xi} m}}{d\nu}(x,\tau)
dm(\widetilde{\xi})\right] \nonumber\\
&\h{250pt}+  \widetilde{\mathbb{E}}\left[\diff_x  \dfrac{d^{\h{.5pt}2}\h{-.7pt}  F}{d\nu^2}(\mathbb{Y}_{t\bigcdot m}(\tau) \otimes  m)\pig(\mathbb{Y}_{t\xi m}(\tau),\widetilde{\mathbb{Y}}_{tx m}(\tau)\pig)\right]  \Bigg\rangle_{\mathcal{H}_m} d\tau\\
&-\int^T_t \left\langle \dfrac{d u_{t \xi m}}{d\nu}(x,\tau),
l_{vx}\pig(\mathbb{Y}_{t \xi m}(\tau),u_{t \xi m}(\tau)\pig)\dfrac{d \mathbb{Y}_{t \xi m}}{d\nu}(x,\tau)
+l_{vv}\pig(\mathbb{Y}_{t \xi m}(\tau),u_{t \xi m}(\tau)\pig)\dfrac{d u_{t \xi m}}{d\nu}(x,\tau)
\right\rangle_{\mathcal{H}_m}d\tau.
\end{align*}
\endgroup
Note the integration on $\mathbb{R}^n$ involved in the $\mathcal{H}_m$-inner product corresponds to the variable $\xi$. Using the assumptions in (\ref{assumption bdd of Dd^2F, F_T})-(\ref{assumption bdd of D^2d^2F, F_T}), Assumptions \textbf{A(v)}'s (\ref{assumption, convexity of l}), \textbf{A(vi)}'s (\ref{assumption, convexity of h}), \textbf{b(i)}'s \eqref{assumption, bdd of DF, DF_T, no lift},
\textbf{b(ii)}'s \eqref{assumption, bdd of D^2F, D^2F_T, no lift},
\textbf{b(v)}'s \eqref{assumption, convexity of D^2F, D^2F_T, no lift},
\textbf{B(v)(b)}'s (\ref{assumption, convexity of D^2F, D^2F_T}), we obtain 
\begin{equation}
\scalemath{0.94}{
\begin{aligned}
\lambda\int^T_t\left\|\dfrac{d u_{t \xi m}}{d\nu}(x,\tau)\right\|^2_{\mathcal{H}_m} d\tau
\leq\:&\big( c+c'+c_l' \big)\int^T_t\left\|\dfrac{d \mathbb{Y}_{t \xi m}}{d\nu}(x,\tau)\right\|^2_{\mathcal{H}_m}d\tau
+\big( c_h'+c_T'+c_T \big) \left\|\dfrac{d \mathbb{Y}_{t \xi m}}{d\nu}(x,T)\right\|^2_{\mathcal{H}_m} \\
&+c\int^T_t  \left\| \dfrac{d \mathbb{Y}_{t \xi m}}{d\nu}(x,\tau)\right\|_{\mathcal{H}_m} \widetilde{\mathbb{E}}\left[1+\big|\widetilde{\mathbb{Y}}_{tx m}(\tau)\big|\right]  d\tau
+c_T\left\| \dfrac{d \mathbb{Y}_{t \xi m}}{d\nu}(x,T)\right\|_{\mathcal{H}_m} \widetilde{\mathbb{E}}\left[1+\big|\widetilde{\mathbb{Y}}_{tx m}(T)\big|\right].
\end{aligned}}
\label{3456}
\end{equation}
The equation in (\ref{dnu, forward FBSDE, X=I}) implies that
\begin{equation}
\left\| \dfrac{d \mathbb{Y}_{t \xi m}}{d\nu}(x,T)\right\|_{\mathcal{H}_m}^2
\leq T\int^T_t  \left\| \dfrac{d u_{t \xi m}}{d\nu}(x,\tau)\right\|_{\mathcal{H}_m}^2d\tau
\h{5pt}\text{and}\h{5pt}
\int^T_t  \left\| \dfrac{d \mathbb{Y}_{t \xi m}}{d\nu}(x,\tau)\right\|_{\mathcal{H}_m}^2d\tau
\leq \dfrac{T^2}{2}\int^T_t  \left\| \dfrac{d u_{t \xi m}}{d\nu}(x,\tau)\right\|_{\mathcal{H}_m}^2 d\tau.
\label{bdd |dnuY|}
\end{equation}
By repeating the proof of Proposition \ref{prop bdd of Y Z u r}, we also have the fact that $\displaystyle\sup_{\tau\in(t,T)}\mathbb{E}\pig(|\mathbb{Y}_{txm}(\tau)|^{2}\pig)\leq C_{4}(1+|x|^{2})$ which is put into the second line of (\ref{3456}), together with (\ref{bdd |dnuY|}), the inequality in (\ref{3456}) can be rewritten as
\begin{equation}
\begin{aligned}
\left[(1-\delta_1)\lambda - \big( c_h'+c_T'+c_T \big)T
-\big( c+c'+c_l' \big)\dfrac{T^2}{2}\right]
\int^T_t\left\|\dfrac{d u_{t \xi m}}{d\nu}(x,\tau)\right\|^2_{\mathcal{H}_m} d\tau
\leq\:&\dfrac{C_4(c^2+c_T^2)}{\lambda\delta_1}(1+|x|^2),
\end{aligned}
\label{3462}
\end{equation}
for any $\delta_1 \in (0,1)$. The required estimate for $\dfrac{d u_{t \xi m}}{d\nu}(x,\tau)$ thus holds, and similarly, the estimates for $\dfrac{d \mathbb{Y}_{t \xi m}}{d\nu}(x,s),$\\$ \dfrac{d \mathbb{Z}_{t \xi m}}{d\nu}(x,s), \dfrac{d\mathbbm{r}_{t \xi m,j}}{d\nu}(x,s)$ also hold by following the step as in the proof of Proposition \ref{prop bdd of Y Z u r} and using (\ref{3462}). This completes the proof.
\end{proof}
To further proceed to the existence of the gradient of the linear functional derivatives, we assume:
\begin{ass}
For any $x$, $\widetilde{x}$, $v \in \mathbb{R}^n$, 
\begin{itemize}
\item[(i).] the boundedness of the derivatives:
\begin{equation}
|l_{xxx}(x,v)|\h{1pt},\h{5pt}|l_{xxv}(x,v)|\h{1pt},\h{5pt}
|l_{xvv}(x,v)|\h{1pt},\h{5pt}|l_{vvv}(x,v)|\leq c\h{1pt},\h{5pt}|h_{xxx}(x)|\leq c_{T};
\label{assumption, new 1}
\end{equation}

\item[(ii).] the regularities of the first-order linear functional derivatives:
\begin{equation}
\left|\diff^{3}_x\dfrac{dF}{d\nu}(m)(x)\right|
\leq c\h{1pt},\h{5pt}
\left|\diff^{3}_x\dfrac{dF_T}{d\nu}(m)(x)\right|\leq c_{T};
\label{assumption, new 2}
\end{equation}
\item[(iii).] the regularities of the second-order linear functional derivatives:
\begin{equation}
\left|\diff_{x}^{2}\dfrac{d^{\h{.5pt}2}\h{-.7pt} F }{d\nu^{2}}(m)(x,\widetilde{x})\right|
\leq c(1+|\widetilde{x}|)\h{1pt},\h{5pt}
\left|\diff_{x}^{2}\dfrac{d^{\h{.5pt}2}\h{-.7pt} F_T }{d\nu^{2}}(m)(x,\widetilde{x})\right|
\leq c_T(1+|\widetilde{x}|) ;
\label{assumption, new 3}
\end{equation}
\begin{equation}
\left|\diff_{x}^{2}\diff_{\tilde{x}}
\dfrac{d^{\h{.5pt}2}\h{-.7pt} F}{d\nu^{2}}(m)(x,\widetilde{x})\right|
\leq c\h{1pt},\h{5pt}
\left|\diff_{x}^{2}\diff_{\tilde{x}}
\dfrac{d^{\h{.5pt}2}\h{-.7pt} F_T}{d\nu^{2}}(m)(x,\widetilde{x})\right|
\leq c_T.
\label{assumption, new 4}
\end{equation}
\item[(iv).] all the derivatives in (i), (ii) and (iii) are continuous in $(x,v)$, $(m,x)$ and $(m,x,\widetilde{x})$, respectively.
\end{itemize}
\label{ass. ass 2}
\end{ass}
From (\ref{dnu, forward FBSDE, X=I})-(\ref{dnu, backward FBSDE, X=I}), by taking
the gradient in $\xi$, we obtain the system 
{\footnotesize
\begin{empheq}[left=\h{-10pt}\empheqbiglbrace]{align}
\diff_\xi\dfrac{d \mathbb{Y}_{t \xi m}}{d\nu}(x,s)
=&\: \int_{t}^{s}\diff_\xi\dfrac{d u_{t \xi m}}{d\nu}(x,\tau)d\tau;
\label{Dxi dnu, forward FBSDE, X=I}\\
\diff_\xi \dfrac{d \mathbb{Z}_{t \xi m}}{d\nu}(x,s)
=&\:\left[h_{xx}(\mathbb{Y}_{t\xi m}(T))
+\diff_x^2\dfrac{dF_T}{d\nu}(\mathbb{Y}_{t\bigcdot m}(T) \otimes  m)\pig(\mathbb{Y}_{t \xi m}(T)\pig)\right]
\diff_\xi\dfrac{d \mathbb{Y}_{t \xi m}}{d\nu}(x,T)
\nonumber\\
&+\left[h_{xxx}(\mathbb{Y}_{t\xi m}(T))
+\diff_x^3\dfrac{dF_T}{d\nu}(\mathbb{Y}_{t\bigcdot m}(T) \otimes  m)\pig(\mathbb{Y}_{t \xi m}(T)\pig)\right]
\diff_\xi \mathbb{Y}_{t \xi m} (T)
\dfrac{d \mathbb{Y}_{t \xi m}}{d\nu}(x,T)
\nonumber\\
&+\widetilde{\mathbb{E}}\left[ \int_{\mathbb{R}^n} \diff_x^2\diff_{\tilde{x}} \dfrac{d^{\h{.5pt}2}\h{-.7pt}  F_T}{d\nu^2}(\mathbb{Y}_{t\bigcdot m}(T) \otimes  m)\pig(\mathbb{Y}_{t\xi m}(T),\widetilde{\mathbb{Y}}_{t\tilde{\xi} m}(T)\pig)
\diff_\xi \mathbb{Y}_{t \xi m} (T)
\dfrac{d \widetilde{\mathbb{Y}}_{t \tilde{\xi} m}}{d\nu}(x,T)
dm(\widetilde{\xi})\right]\nonumber\\
&+\widetilde{\mathbb{E}}\left[\diff_x^2 \dfrac{d^{\h{.5pt}2}\h{-.7pt}  F_T}{d\nu^2}(\mathbb{Y}_{t\bigcdot m}(T) \otimes  m)\pig(\mathbb{Y}_{t\xi m}(T),\widetilde{\mathbb{Y}}_{tx m}(T)\pig)
 \right]\diff_\xi \mathbb{Y}_{t \xi m}(T)\nonumber\\
&+\int^T_s \left[l_{xx}\pig(\mathbb{Y}_{t\xi m}(\tau),u_{t\xi m}(\tau)\pig)
+\diff^2_x \dfrac{dF}{d\nu}(\mathbb{Y}_{t\bigcdot m}(\tau) \otimes  m)\pig(\mathbb{Y}_{t\xi m}(\tau)\pig)\right]
\diff_\xi\dfrac{d \mathbb{Y}_{t \xi m}}{d\nu}(x,\tau) d\tau\nonumber\\
&+\int^T_s \left[l_{xxx}\pig(\mathbb{Y}_{t\xi m}(\tau),u_{t\xi m}(\tau)\pig)
+\diff^3_x \dfrac{dF}{d\nu}(\mathbb{Y}_{t\bigcdot m}(\tau) \otimes  m)\pig(\mathbb{Y}_{t\xi m}(\tau)\pig)\right]
\diff_\xi \mathbb{Y}_{t \xi m} (\tau)
\dfrac{d \mathbb{Y}_{t \xi m}}{d\nu}(x,\tau) d\tau\nonumber\\
&+\int^T_s \left[l_{xxv}\pig(\mathbb{Y}_{t\xi m}(\tau),u_{t\xi m}(\tau)\pig)\right]
\diff_\xi u_{t \xi m} (\tau)
\dfrac{d \mathbb{Y}_{t \xi m}}{d\nu}(x,\tau) d\tau\nonumber\\
&+\int^T_s l_{xv}\pig(\mathbb{Y}_{t\xi m}(\tau),u_{t\xi m}(\tau)\pig)
\diff_\xi\dfrac{d u_{t \xi m}}{d\nu}(x,\tau) 
+l_{xvx}\pig(\mathbb{Y}_{t\xi m}(\tau),u_{t\xi m}(\tau)\pig)
\diff_\xi \mathbb{Y}_{t \xi m} (\tau)
\dfrac{d u_{t \xi m}}{d\nu}(x,\tau)d\tau\nonumber\\
&+\int^T_s l_{xvv}\pig(\mathbb{Y}_{t\xi m}(\tau),u_{t\xi m}(\tau)\pig)
\diff_\xi u_{t\xi m}(\tau)
\dfrac{d u_{t \xi m}}{d\nu}(x,\tau) 
d\tau\nonumber\\
&+\int^T_s\widetilde{\mathbb{E}}\left[ \int_{\mathbb{R}^n} \diff_x^2\diff_{\tilde{x}} \dfrac{d^{\h{.5pt}2}\h{-.7pt}  F}{d\nu^2}(\mathbb{Y}_{t\bigcdot m}(\tau) \otimes  m)\pig(\mathbb{Y}_{t\xi m}(\tau),\widetilde{\mathbb{Y}}_{t\tilde{\xi} m}(\tau)\pig)
\diff_\xi \mathbb{Y}_{t\xi m}(\tau)
\dfrac{d \widetilde{\mathbb{Y}}_{t \tilde{\xi} m}}{d\nu}(x,\tau)
dm(\widetilde{\xi})\right]d\tau\nonumber\\
&+\int^T_s\widetilde{\mathbb{E}}\left[\diff_x^2 \dfrac{d^{\h{.5pt}2}\h{-.7pt}  F}{d\nu^2}(\mathbb{Y}_{t\bigcdot m}(\tau) \otimes  m)\pig(\mathbb{Y}_{t\xi m}(\tau),\widetilde{\mathbb{Y}}_{tx m}(\tau)\pig)
\diff_\xi \mathbb{Y}_{t\xi m}(\tau)\right]d\tau\nonumber\\
&-\int^T_s\sum_{j=1}^{n}\diff_\xi\dfrac{d\mathbbm{r}_{t \xi m,j}}{d\nu}(x,\tau) dw_{j}(\tau).
\label{Dxi dnu, backward FBSDE, X=I}
\end{empheq}}
We can then state:
\begin{prop}
\label{prop bdd of Dxi dnu Y Z u r} Under the assumptions (\ref{assumption, bdd of l, lx, lv})-(\ref{assumption, convexity of h}), (\ref{assumption, bdd of F F_T})-(\ref{assumption, convexity of D^2F, D^2F_T}), (\ref{uniqueness condition}), (\ref{ass in prop 6.2}), Assumptions \ref{ass. ass 1} and \ref{ass. ass 2}, the system (\ref{Dxi dnu, forward FBSDE, X=I})-(\ref{Dxi dnu, backward FBSDE, X=I}) has the unique solution, and they are the gradient of
$\dfrac{d \mathbb{Y}_{t \xi m}}{d\nu}(x,s), \dfrac{d \mathbb{Z}_{t \xi m}}{d\nu}(x,s),$ $\dfrac{d u_{t \xi m}}{d\nu}(x,s),$ $\dfrac{d\mathbbm{r}_{t \xi m,j}}{d\nu}(x,s)$ with respect to $\xi$ such that 

\begin{equation}
\begin{aligned}
&\mathbb{E}\left[\int_{\mathbb{R}^{n}}\left|\diff_\xi\dfrac{d \mathbb{Y}_{t \xi m}}{d\nu}(x,s)\right|^{2}dm(\xi)\right]
\h{1pt},\h{10pt}\mathbb{E}\left[\int_{\mathbb{R}^{n}}\left|\diff_\xi\dfrac{d \mathbb{Z}_{t \xi m}}{d\nu}(x,s)\right|^{2}dm(\xi)\right],\\
&\mathbb{E}\left[\int_{\mathbb{R}^{n}}\left|\diff_\xi\dfrac{d u_{t \xi m}}{d\nu}(x,s)\right|^{2}dm(\xi)\right]
\h{1pt},\h{10pt}\sum_{j=1}^{n}\int_{t}^{T}\mathbb{E}\left[\int_{\mathbb{R}^n}\left|\diff_\xi\dfrac{d \mathbbm{r}_{t \xi m,j}}{d\nu}(x,s)\right|^{2}dm(\xi)\right]\leq C_{11}(1+|x|^{2}),
\end{aligned}
\label{eq:6-17}
\end{equation}
where $C_{11}$ is a positive constant depending only on $\delta_1$, $n$, $\lambda$, $\eta$, $c$, $c_l$, $c_h$, $c_T$, $c'$, $c_l'$, $c_h'$, $c_T'$ and $T$.
\end{prop}

\begin{proof}
Similar to the proof of Proposition \ref{prop bdd of functiona d of Y Z u r}.
\end{proof}
This justifies the differentiation with respect to $m$ of the Bellman
equation, and hence $U(x,m,t)=\dfrac{d V}{d\nu}(m,t)(x)$ is a solution to the Master equation in (\ref{eq:7-5}). If $U^*(x,m,t)$ is another solution to (\ref{eq:7-5}), then we define $\mathbb{Z}^*_{txm}(t) := \diff_xU^*(x,m,t)$,  $\mathbb{Y}^*_{txm}(s) := \int^s_t u\pig(\mathbb{Y}^*_{txm}(\tau), \mathbb{Z}^*_{\tau \mathbb{Y}^*_{txm}(\tau) m}(\tau) \pig)d\tau + \eta\pig(w(s)-w(t)\pig)$ and $\mathbb{Z}^*_{t x m}(\tau):=\mathbb{Z}^*_{\tau \mathbb{Y}^*_{txm}(\tau) m}(\tau)= \diff_x U^*(\mathbb{Y}^*_{txm}(\tau),m,\tau)$. We then differentiate (\ref{eq:7-5}) with respect to $x$ and evaluate the resulting equation at $x = \mathbb{Y}^*_{txm}(s)$. By using the mean-field It\^o lemma in Theorem \ref{ito thm}, we work backward and realise that $\mathbb{Z}^*_{txm}(s)$ satisfies the backward dynamics (\ref{def, backward SDE}). Since the FBSDE (\ref{forward FBSDE})-(\ref{backward FBSDE}) has the unique solution, therefore we know $\mathbb{Y}_{txm}(s)=\mathbb{Y}^*_{txm}(s)$, $\mathbb{Z}_{txm}(s)=\mathbb{Z}^*_{txm}(s)$ and hence $U=U^*$ due to the terminal condition in (\ref{eq:7-5}). The well-posedness of the master equation in (\ref{eq:7-5}) is concluded by the  proposition.

\begin{prop}
Under the assumptions (\ref{assumption, bdd of l, lx, lv})-(\ref{assumption, convexity of h}), (\ref{assumption, bdd of F F_T})-(\ref{assumption, convexity of D^2F, D^2F_T}) and (\ref{uniqueness condition}), the value function $V(m,t)$  satisfies the Bellman equation in (\ref{Bellman Equation}) classically. Furthermore, if Assumptions \ref{ass. ass 1}, \ref{ass. ass 2} and (\ref{ass in prop 6.2}) are fulfilled, the linear functional derivative $U(x,m,t)=\dfrac{d V}{d\nu}(m,t)(x)$ of $V(m,t)$ is the unique classical solution to the master equation in (\ref{eq:7-5}) in the pointwise sense, with all the derivatives $\dfrac{\p U}{\p t}(x,m,t)$, $\diff_x U(x,m,t)$, $\diff^2_x U(x,m,t)$, $\diff_\xi \dfrac{d U}{d \nu}(\xi,m,t)(x)$ and $\diff^2_\xi \dfrac{d U}{d \nu}(\xi,m,t)(x)$ being existed.
\end{prop}

\vspace{10pt}

\textbf{Acknowledgement.} The authors would like to express their sincere gratitude for the inspiring suggestions from the attendants in the talk by Phillip Yam in ICMS workshop ``Mean-field games, energy systems, and other applications'' at the University of Edinburgh in April 2019. The primitive ideas of this article had germinated since the final stage of Ph.D. study of Michael Man Ho Chau at Imperial College London and University of Hong Kong 2016 under the supervision of Phillip Yam, and some particular results had been incorporated in his Ph.D. dissertation. The authors would also like to extend their heartfelt gratitude to Michael Chau for his remarkable work. {\color{black} Alain Bensoussan is supported by the
National Science Foundation under grants NSF-DMS-1905449 and NSF-DMS-2204795, and grant
from the SAR Hong Kong RGC GRF 14301321.} Ho Man Tai extends his warmest thanks to Professor Yam and the Department of Statistics at The Chinese University of Hong Kong for the financial support. 
Phillip Yam acknowledges the financial supports from HKGRF-14301321 with the project title ``General Theory for Infinite Dimensional Stochastic Control: Mean Field and Some Classical Problems''. He also thanks Columbia University for the kind invitation to be a visiting faculty member in the Department of Statistics during
his sabbatical leave. He also recalled the unforgettable moments and the happiness shared with his beloved father and used this work in memory of his father's brave battle against liver cancer. 

\section{Appendix}
\subsection{Proof of Statements in Section 2}
\subsubsection{Proof of Theorem \ref{ito thm}}\label{app, ito thm}
We note that the constant $C$ in the following proof may be different from line to line, but we still denote by the same symbol $C$ without ambiguity.
We begin with some preliminaries. If the sequence $\{Y_{k}\}_{k \in \mathbb{N}}$ in assumption (\ref{eq:3-125}) further satisfies
\begin{equation}
\sup_k \mathbb{E}\left[\int_{\mathbb{R}^{n}}|Y_{k}|^{4}dm(x)\right] < \infty,
\label{eq:3-114}
\end{equation}
then 
\begin{equation}
\pigl\langle D_{X}^{2}F(X_{k}  \otimes  m,s_{k})(Y_{k})-D_{X}^{2}F(X  \otimes  m,s_{k})(Y_{k}),Y_{k}\pigr\rangle_{\mathcal{H}_m}\longrightarrow0 \h{10pt} \text{as $k \to \infty$;}
\label{eq:3-115}
\end{equation}
indeed, denote $\mathscr{I}_{\h{.5pt}k}$ for the expression (\ref{eq:3-115}), we write
$\mathscr{I}_{k}=:\mathscr{I}_{\h{.5pt}\epsilon k}^{1}+\mathscr{I}_{\h{.5pt}\epsilon k}^{2}$, for any $\epsilon>0$, where 

\[
\mathscr{I}_{\h{.5pt}\epsilon k}^{1} := \left\langle D_{X}^{2}F(X_{k}  \otimes  m,s_{k})(Y_{k})-D_{X}^{2}F(X  \otimes  m,s_{k})(Y_{k}),\dfrac{Y_{k}}{1+\epsilon|Y_{k}|}\right\rangle_{\mathcal{H}_m},
\]

\[
\mathscr{I}_{\h{.5pt}\epsilon k}^{2}:=\epsilon \left\langle D_{X}^{2}F(X_{k}  \otimes  m,s_{k})(Y_{k})-D_{X}^{2}F(X  \otimes  m,s_{k})(Y_{k}),\dfrac{Y_{k}|Y_{k}|}{1+\epsilon|Y_{k}|}\right\rangle_{\mathcal{H}_m}.
\]
Thanks to Assumptions (\ref{eq:3-123}) and (\ref{eq:3-114}), we have $\big|\mathscr{I}_{\h{.5pt}\epsilon k}^{2}\big|\leq C\epsilon$  for some $C>0$ independent of $\epsilon$; and thanks to (\ref{eq:3-125}), $\mathscr{I}_{\h{.5pt}\epsilon k}^{1}\longrightarrow0$, as $k\rightarrow \infty,$ for each $\epsilon>0$. Combining the two convergences, we get $\mathscr{I}_{k}\longrightarrow 0$ as $k\rightarrow\infty.$ Next, using
(\ref{eq:2-230}), we write 

\[
\dfrac{1}{\epsilon}\Big[F\pig(\mathbb{X}_{tX}(s+\epsilon)  \otimes  m,s+\epsilon\pig)
-F\pig(\mathbb{X}_{tX}(s)  \otimes  m,s+\epsilon\pig)\Big]
=:\textup{I}_{\h{.5pt}\epsilon}+\textup{II}_{\h{.5pt}\epsilon}+\textup{III}_{\h{.5pt}\epsilon}+\textup{IV}_{\epsilon},
\]
where 

$$
\textup{I}_{\h{.5pt}\epsilon}:=\dfrac{1}{\epsilon}
\left\langle D_{X}F\pig(\mathbb{X}_{tX}(s)  \otimes  m,s+\epsilon\pig), \int_{s}^{s+\epsilon}a_{tX}(\tau)d\tau + \sum_{j=1}^n\int_{s}^{s+\epsilon}\eta_{tX}^{j}(\tau)dw_{j}(\tau) \right\rangle_{\mathcal{H}_m};
$$
 
\begin{align*}
\textup{II}_{\h{.5pt}\epsilon}
:=\dfrac{1}{\epsilon}
\int_{0}^{1}\int_{0}^{1}
\theta\Bigg\langle 
D_{X}^{2}F\Big(\pig\{\mathbb{X}_{tX}(s)+\theta\lambda\pig[\mathbb{X}_{tX}(s+\epsilon)
&-\mathbb{X}_{tX}(s)\pig]\pig\}  \otimes  m,s+\epsilon\Big)
\left(\int_{s}^{s+\epsilon}a_{tX}(\tau)d\tau\right),\\
&\int_{s}^{s+\epsilon}a_{tX}(\tau)d\tau+2\sum_{j=1}^n\int_{s}^{s+\epsilon}\eta_{tX}^{j}(\tau)dw_{j}(\tau) \Bigg\rangle_{\mathcal{H}_m} 
d \theta d \lambda\h{1pt}, \h{5pt} 
\end{align*}
(the presence of factor $2$ is due to the self-adjointness of $D^2_X F(X   \otimes  m,s)$);
\begin{align*}
\textup{III}_{\h{.5pt}\epsilon}
:=\dfrac{1}{\epsilon}\int_{0}^{1}\int_{0}^{1}
\theta\Bigg\langle D_{X}^{2}
& F\Big(\pig\{\mathbb{X}_{tX}(s)
+\theta\lambda\pig[\mathbb{X}_{tX}(s+\epsilon)
-\mathbb{X}_{tX}(s)\pig]\pig\}  \otimes  m,s+\epsilon\Big)
\left(\sum_{j=1}^n\int_{s}^{s+\epsilon}\eta_{tX}^{j}(\tau)dw_{j}(\tau)\right)\\
&-D_{X}^{2}F(\mathbb{X}_{tX}(s)  \otimes  m,s+\epsilon)
\left(\sum_{j=1}^n\int_{s}^{s+\epsilon}\eta_{tX}^{j}(\tau)dw_{j}(\tau)\right)
,\sum_{j=1}^n\int_{s}^{s+\epsilon}\eta_{tX}^{j}(\tau)dw_{j}(\tau)\Bigg\rangle_{\mathcal{H}_m}
d\theta d\lambda;
\end{align*}
and
\[
\textup{IV}_{\epsilon}:=\dfrac{1}{2\epsilon}\left\langle D_{X}^{2}F(\mathbb{X}_{tX}(s)  \otimes  m,s+\epsilon)
\left(\sum^n_{j=1}\int_{s}^{s+\epsilon}\eta_{tX}^{j}(\tau)dw_{j}(\tau)
\right),\sum_{j=1}^n\int_{s}^{s+\epsilon}\eta_{tX}^{j}(\tau)dw_{j}(\tau)
\right\rangle_{\mathcal{H}_m}.
\]
Firstly, we expand the first term $\textup{I}_{\h{.5pt}\epsilon}$ by writing
\begin{align*}
\textup{I}_{\h{.5pt}\epsilon}=\:&\pigl\langle D_{X}F(\mathbb{X}_{tX}(s)  \otimes  m,s+\epsilon),a_{tX}(s)\pigr\rangle_{\mathcal{H}_m}
+\left\langle D_{X}F(\mathbb{X}_{tX}(s)  \otimes  m,s+\epsilon),\dfrac{1}{\epsilon}\int_{s}^{s+\epsilon}a_{tX}(\tau)d\tau -a_{tX}(s)\right\rangle_{\mathcal{H}_m}  \\
&+\left\langle D_{X}F\pig(\mathbb{X}_{tX}(s)  \otimes  m,s+\epsilon\pig),
\dfrac{1}{\epsilon}\sum_{j=1}^n\int_{s}^{s+\epsilon}\eta_{tX}^{j}(\tau)dw_{j}(\tau)\right\rangle_{\mathcal{H}_m}.
\end{align*}
From Assumption (c) of (\ref{eq:3-112}), together with (\ref{eq:3-123}), the second term in the first line tends to zero. While for the term in the second line, due to the future Brownian increment of the stochastic integral over $[s,s+\epsilon]$, it vanishes as well after an application of tower property. Finally, for the remaining term, in light of Assumption (\ref{eq:3-121}), we get
\begin{equation}
\textup{I}_{\h{.5pt}\epsilon}\longrightarrow
\pigl\langle D_{X}F(\mathbb{X}_{tX}(s)  \otimes  m,s),a_{tX}(s)\pigr\rangle_{\mathcal{H}_m}
\h{1pt},\h{10pt}
\text{as}\;\epsilon\rightarrow0.
\label{eq:3-116}
\end{equation}
For the term $\textup{II}_{\h{.5pt}\epsilon}$, we apply the second assumption in (\ref{eq:3-123}), assumptions (b) and (c) of (\ref{eq:3-112}), we have
\begin{equation}
\textup{II}_{\h{.5pt}\epsilon}\longrightarrow0
\h{1pt},\h{10pt}
\text{as}\;\epsilon\rightarrow0.
\label{eq:3-117}
\end{equation}
For the term $\textup{III}_{\h{.5pt}\epsilon}$, the use of B\"urkholder-Davis-Gundy and Cauchy-Schwarz inequalities gives
$$
\dfrac{1}{\epsilon^{2}}
\mathbb{E}\left[\h{3pt}\left|\displaystyle\sum_{j=1}^{n}
\displaystyle\int_{s}^{s+\epsilon}\eta_{t\xi }^{j}(\tau)dw_{j}(\tau)\right|^{4}\h{3pt}\right]
\leq \dfrac{4}{\epsilon^{2}}
\mathbb{E}\left[\left( \displaystyle\int_{s}^{s+\epsilon}\displaystyle\sum_{j=1}^{n}\big|\eta_{t\xi }(\tau)\big|^{2}d\tau\right)^2\right]
\leq \dfrac{4}{\epsilon}
\mathbb{E}\left[ \displaystyle\int_{s}^{s+\epsilon}
\left(\displaystyle\sum_{j=1}^{n}
\big|\eta_{t\xi }(\tau)\big|^{2}\right)^2 d\tau\right],
$$
for any $\xi\in \mathbb{R}^{n}$. Therefore,

\begin{align*}
\dfrac{1}{\epsilon^2}
\mathbb{E}_{\mathcal{W}_t}\left[\int_{\mathbb{R}^{n}}\left|\sum_{j=1}^{n}\int_{s}^{s+\epsilon}\eta_{tX}^{j}(\tau)dw_{j}(\tau)\right|^{4}dm(x)\right]
&\leq \dfrac{C}{\epsilon}
\mathbb{E}_{\mathcal{W}_t}\left[\int_{\mathbb{R}^{n}}\int_{s}^{s+\epsilon}\left(\sum_{j=1}^{n}|\eta_{tX}(\tau)|^{2}\right)^{2}d\tau dm(x)\right]\\
&\leq C\sup_{s\in [t,T]} \mathbb{E}_{\mathcal{W}_t}\left[\int_{\mathbb{R}^{n}}\left(\sum_{j=1}^{n}|\eta_{tX}(s)|^{2}\right)^{2} dm(x)\right]\\
&\leq C
\end{align*}
by the first assumption (a) of (\ref{eq:3-112}). Based on this bound, from Assumption
(\ref{eq:3-125}), one can use the claim (\ref{eq:3-115}) to conclude the convergence of the integrand of $\textup{III}_{\h{.5pt}\epsilon}$ to zero for each $(\lambda,\theta)$. Next, under the second assumption of (\ref{eq:3-123}), a standard application of the bounded convergence theorem asserts that

\begin{equation}
\textup{III}_{\h{.5pt}\epsilon}\longrightarrow0,\h{10pt} \text{as}\;\epsilon\rightarrow0.
\label{eq:3-118}
\end{equation}
Moreover, from the second assumption (\ref{eq:3-112}), we have

\[\dfrac{1}{\sqrt{\epsilon}}
\sum_{j=1}^{n}\int_{s}^{s+\epsilon}\pig[\eta_{tX}^{j}(\tau)-\eta_{tX}^{j}(s)\pig]dw_{j}(\tau)
\longrightarrow 0, \h{10pt} \text{as $\epsilon \to 0$, in $\mathcal{H}_{m}$}
\]
and thus by using the second assumption of (\ref{eq:3-123})
\[
\textup{IV}_{\epsilon}
-\dfrac{1}{2}\left\langle 
D_{X}^{2}F(\mathbb{X}_{tX}(s)  \otimes  m,s+\epsilon)
\left(\sum_{j=1}^n\eta_{tX}^{j}(s)\dfrac{w_{j}(s+\epsilon)-w_{j}(s)}{\sqrt{\epsilon}}\right),\sum_{j=1}^n\eta_{tX}^{j}(s)\dfrac{w_{j}(s+\epsilon)-w_{j}(s)}{\sqrt{\epsilon}}\right\rangle_{\mathcal{H}_m}\longrightarrow0.
\]
On the other hand, by formula (\ref{eq:2-231}), it holds that
\begingroup
\allowdisplaybreaks
\begin{align*}
&\left\langle D_{X}^{2}F(\mathbb{X}_{tX}(s)  \otimes  m,s+\epsilon)
\left(\sum_{j=1}^n\eta_{tX}^{j}(s)\dfrac{w_{j}(s+\epsilon)-w_{j}(s)}{\sqrt{\epsilon}}\right)
,\sum_{j=1}^n\eta_{tX}^{j}(s)\dfrac{w_{j}(s+\epsilon)-w_{j}(s)}{\sqrt{\epsilon}}\right\rangle_{\mathcal{H}_m}\\
&=\left\langle D_{X}^{2}F(\mathbb{X}_{tX}(s)  \otimes  m,s+\epsilon)
\left(\sum_{j=1}^{n}\eta_{tX}^{j}(s)\mathcal{N}_{s}^{j}\right),\sum_{j=1}^{n}\eta_{tX}^{j}(s)\mathcal{N}_{s}^{j}\right\rangle_{\mathcal{H}_m},
\end{align*}
\endgroup
for any standard normal distributed $\mathcal{N}_{s}^{j}$ being independent of $\mathcal{W}_{tX}^{s}$ since the law of $\dfrac{w_{j}(s+\epsilon)-w_{j}(s)}{\sqrt{\epsilon}}$ is normally distributed with a mean $0$ and variance $1$, and hence is independent of the choice of $\epsilon$.

We have proven, from assumption (\ref{eq:3-124}), we conclude
\begin{align*}
\mbox{\fontsize{9}{10}\selectfont\(\left\langle D_{X}^{2}F(\mathbb{X}_{tX}(s)  \otimes  m,s+\epsilon)\left(\displaystyle\sum_{j=1}^{n}\eta_{tX}^{j}(s)\mathcal{N}_{s}^{j}\right),\displaystyle\sum_{j=1}^{n}\eta_{tX}^{j}(s)\mathcal{N}_{s}^{j}\right\rangle_{\mathcal{H}_m}\h{-5pt}\rightarrow\left\langle D_{X}^{2}F(\mathbb{X}_{tX}(s)  \otimes  m,s)\left(\displaystyle\sum_{j=1}^{n}\eta_{tX}^{j}(s)\mathcal{N}_{s}^{j}\right),\displaystyle\sum_{j=1}^{n}\eta_{tX}^{j}(s)\mathcal{N}_{s}^{j}\right\rangle_{\mathcal{H}_m}\)},
\end{align*}
as $\epsilon \to 0$. Collecting all these results, we have proven that 
\begin{equation}
\begin{aligned}
&\dfrac{1}{\epsilon}\Big[F(\mathbb{X}_{tX}(s+\epsilon)  \otimes  m,s+\epsilon)-F(\mathbb{X}_{tX}(s)  \otimes  m,s+\epsilon)\Big]
\longrightarrow\\
&\pigl\langle D_{X}F(\mathbb{X}_{tX}(s)  \otimes  m,s),a_{tX}(s)\pigr\rangle_{\mathcal{H}_m}
+\dfrac{1}{2}\left\langle D_{X}^{2}F(\mathbb{X}_{tX}(s)  \otimes  m,s)\left(\sum_{j=1}^{n}\eta_{tX}^{j}(s)\mathcal{N}_{s}^{j}\right),\sum_{j=1}^{n}\eta_{tX}^{j}(s)\mathcal{N}_{s}^{j}\right\rangle_{\mathcal{H}_m}.
\end{aligned}
    \label{eq:3-119}
\end{equation}
Finally, from assumption (\ref{eq:3-120}), in accordance with Rademacher's theorem,
\[
\dfrac{1}{\epsilon}\pig[F(\mathbb{X}_{tX}(s)  \otimes  m,s+\epsilon)-F(\mathbb{X}_{tX}(s)  \otimes  m,s)\pig]
\longrightarrow\dfrac{\partial}{\partial s}F(\mathbb{X}_{tX}(s)  \otimes  m,s),\h{10pt} \text{ a.e. }s \in [t,T].
\]
Therefore, the result in (\ref{eq:3-113}) is obtained. If the formula (\ref{eq:2-232})
applies, by combining the discussion between (\ref{eq:2-233}) and (\ref{eq:2-234}),
\begin{align*}
D_{X}^{2}F(\mathbb{X}_{tX}(s)  \otimes  m,s)\left(\sum_{j=1}^{n}\eta_{tX}^{j}(s)\mathcal{N}_{s}^{j}\right)
=\:&\diff^2_x \dfrac{dF}{d\nu}(\mathbb{X}_{tX}(s)  \otimes  m,s)(\mathbb{X}_{tX}(s))
\left(\sum_{j=1}^{n}\eta_{tX}^{j}(s)\mathcal{N}_{s}^{j}\right)\\
&+\mbox{\fontsize{9}{10}\selectfont\(
\widetilde{\mathbb{E}}\left[\displaystyle\int_{\mathbb{R}^{n}}\diff_x \diff_{\tilde{x}}\dfrac{d^{\h{.5pt}2}\h{-.7pt}  F}{d\nu^{2}}(\mathbb{X}_{tX}(s)  \otimes  m,s)
\Big(\mathbb{X}_{tX}(s),\widetilde{\mathbb{X}}_{t\widetilde{X}}(s)\Big)
\left(\displaystyle\sum_{j=1}^{n}\eta_{t\widetilde{X}}^{j}(s)\widetilde{\mathcal{N}}_{s}^{j}\right)dm(\tilde{x})\right]\)}\\
=\:&\diff^2_x \dfrac{dF}{d\nu}(\mathbb{X}_{tX}(s)  \otimes  m,s)
\pig(\mathbb{X}_{tX}(s)\pig)\left(\sum_{j=1}^{n}\eta_{tX}^{j}(s)\mathcal{N}_{s}^{j}\right),
\end{align*}
where $\widetilde{\mathcal{N}}_{s}^{j}$ is independent of $\mathcal{W}_{t\widetilde{X}}^{s}$. This concludes the proof. \hfill$\blacksquare$

\subsubsection{Complements to Theorem \ref{ito thm}}
If we choose $X_t=\mathcal{I}_x$ in the formula (\ref{ito lemma in gradient form}), then the
process $\mathbb{X}_{t\mathcal{I}}(s)$ is simply a classical It\^o process indexed by
a parameter $x$, which is the initial condition at the time $t$, denoted
by $\mathbb{X}_{tx}(s)$. It reads

\begin{equation}
\mathbb{X}_{tx}(s)=x+\int_{t}^{s}a_{tx}(\tau)d\tau+\sum_{j=1}^{n}\int_{t}^{s}\eta_{tx}^{j}(\tau)dw_{j}(\tau)\:, \h{10pt} \text{ for $s\geq t$.}
\label{eq:3-131}
\end{equation}
Both the processes $a_{tx}(s)$ and $\eta_{tx}^{j}(s)$ are adapted to the filtration $\mathcal{W}_{t}^{s}$, and such that 

\[
\mathbb{E}\left[\int_{t}^{T}\int_{\mathbb{R}^{n}}|a_{tx}(s)|^{2}ds\:dm(x)\right]
\h{10pt} \text{and} \h{10pt}
\mathbb{E}\left[\int_{t}^{T}\int_{\mathbb{R}^{n}}|\eta_{tx}^{j}(s)|^{2}ds\:dm(x)\right]<\infty.
\]
We can write (\ref{ito lemma in gradient form}) as

\begin{equation}
\begin{aligned}
&\dfrac{d}{ds}F(\mathbb{X}_{tx}(s)  \otimes  m,s)=\dfrac{\partial}{\partial s}F(\mathbb{X}_{tx}(s)  \otimes  m,s)\\
&\pushright{+\mathbb{E}\left\{\int_{\mathbb{R}^{n}}
\bigg[\diff_x \dfrac{dF}{d\nu}(\mathbb{X}_{tx}(s)  \otimes  m,s)(\mathbb{X}_{tx}(s))\cdot a_{tx}(s)
+\dfrac{1}{2}\sum_{j=1}^{n}\diff^2_x \dfrac{dF}{d\nu}(\mathbb{X}_{tx}(s)  \otimes  m,s)(\mathbb{X}_{tx}(s))\eta_{tx}^{j}(s)\cdot \eta_{tx}^{j}(s)\bigg]dm(x)\right\}.}
\end{aligned}
\label{eq:3-132}
\end{equation}
For instance, taking $F(m,s)=\displaystyle\int_{\mathbb{R}^{n}}\Psi(x,s)dm(x),$ then $F(\mathbb{X}_{tx}(s)  \otimes  m,s)=\mathbb{E}\left[\displaystyle\int_{\mathbb{R}^{n}}\Psi(\mathbb{X}_{tx}(s),s)dm(x)\right]$
and $\dfrac{dF}{d\nu}(m,s)(x)=\Psi(x,s).$ Therefore, (\ref{eq:3-132})
reduces to 

\begin{equation}
\begin{aligned}
\dfrac{d}{ds}\mathbb{E}
\left[\int_{\mathbb{R}^{n}}\Psi(\mathbb{X}_{tx}(s),s)dm(x)\right]
=\mathbb{E}\Bigg\{\int_{\mathbb{R}^{n}}\Bigg[\dfrac{\partial}{\partial s}\Psi(\mathbb{X}_{tx}(s),s)+\diff_x \Psi(\mathbb{X}_{tx}(s),s)\cdot a_{tx}(s)&\\
+\dfrac{1}{2}\sum_{j=1}^{n}\diff^2_x \Psi(\mathbb{X}_{tx}(s),s)\eta_{tx}^{j}(s) & \cdot \eta_{tx}^{j}(s)\Bigg]dm(x)\Bigg\},
\end{aligned}
\label{eq:3-133}
\end{equation}
which is the same as the classical It\^o's formula when applies to an arbitrary test function $\Psi$. 

\subsection{Proof of Statements in Section 3}
\subsubsection{Proof of Lemma \ref{lem G derivative of J}}\label{app, lem G derivative of J}

For the perturbed control $v_{tX}(s)+\theta \tilde{v}_{tX}(s)$ with $\theta \in \mathbb{R}$ and $\tilde{v}_{tX}(s) \in L_{\mathcal{W}_{tX}}^{2}(t,T;\mathcal{H}_{m})$, the corresponding state is $\mathbb{X}_{tX}(s)+\theta\int_{t}^{s}\tilde{v}_{tX}(\tau)d\tau$. Using (\ref{assumption, bdd of l, lx, lv}), (\ref{assumption, bdd of h, hx, hxx}), (\ref{assumption, bdd of DF, DF_T}) and dominated convergence theorem, we first check that 

\begin{equation}
\begin{aligned}
\dfrac{d}{d\theta}J_{tX}\pig(v_{tX}+\theta\tilde{v}_{tX}\pig)\bigg|_{\theta=0}
=\:&\int_{t}^{T}\pigl\langle l_{v}\big(\mathbb{X}_{tX}(s),v_{tX}(s)\big),\tilde{v}_{tX}(s)\pigr\rangle_{\mathcal{H}_m}ds\\
&+\int_{t}^{T}\left\langle l_{x}\pig(\mathbb{X}_{tX}(s),v_{tX}(s)\pig)
+D_{X}F\pig(\mathbb{X}_{tX}(s)  \otimes  m\pig)
,\int_{t}^{s}\widetilde{v}_{tX}(\tau)d\tau\right\rangle_{\mathcal{H}_m}ds\\
&+\left\langle h_{x}\pig(\mathbb{X}_{tX}(T)\pig)
+D_{X}F_T\pig(\mathbb{X}_{tX}(T)  \otimes  m\pig)
,\int_{t}^{T}\widetilde{v}_{tX}(\tau)d\tau\right\rangle_{\mathcal{H}_m}. 
\end{aligned}
\label{eq:8-100}
\end{equation}
To deal with the second and third terms in (\ref{eq:8-100}), we define $\Gamma_{tX}$ by the following:

\begin{equation}
\Gamma_{tX}
:=\int_{t}^{T}
l_{x}\pig(\mathbb{X}_{tX}(s),v_{tX}(s)\pig)
+D_{X}F\pig(\mathbb{X}_{tX}(s)  \otimes  m\pig)ds
+h_{x}\pig(\mathbb{X}_{tX}(T)\pig)
+D_{X}F_T\pig(\mathbb{X}_{tX}(T)  \otimes  m\pig),
\label{eq:8-101}
\end{equation}
which is clearly a random variable taking values in $\mathbb{R}^{n}$, and is $\mathcal{W}_{tX}^{T}$-measurable, so it can be written as a $\mathcal{W}_{t}^{T}$-measurable random field $\Gamma_{t\xi}$, for almost every $\xi \in \mathbb{R}^n$, and then we substitute $\xi$ by $X$. Now,  

\[
\mathbb{E}\big(\Gamma_{tX}\big|\mathcal{W}_{tX}^{s}\big)
=\mathbb{E}\big(\Gamma_{t\xi}\big|\mathcal{W}_{t}^{s}\big)\Big|_{\xi=X}.
\]
By the standard martingale representation theorem adapted to the Wiener filtration, we can write 

\[
\mathbb{E}(\Gamma_{t\xi}|\mathcal{W}_{t}^{s})
=\mathbb{E}\big(\Gamma_{t\xi}\big)
+\sum_{j=1}^{n}\int_{t}^{s}\mathbbm{r}_{t\xi,j}(\tau)dw_{j}(\tau).
\]
Therefore, we can write
\[
\mathbb{E}\big(\Gamma_{tX}\big|\mathcal{W}_{tX}^{s}\big)
=\mathbb{E}\big(\Gamma_{tX}\big|X\big)
+\sum_{j=1}^{n}\int_{t}^{s}\mathbbm{r}_{tX,j}(\tau)dw_{j}(\tau),
\]
with $\mathbbm{r}_{tX,j}(s)\in L_{\mathcal{W}_{tX}}^{2}(t,T;\mathcal{H}_{m}).$ Next, we define the It\^o process:

\begin{equation}
\mathbb{Z}_{tX}(s):=\mathbb{E}(\Gamma_{tX}|X)
-\int_{t}^{s}
l_{x}\pig(\mathbb{X}_{tX}(\tau),v_{tX}(\tau)\pig)
+D_{X}F\pig(\mathbb{X}_{tX}(\tau)  \otimes  m\pig)d\tau
+\sum_{j=1}^{n}\int_{t}^{s}\mathbbm{r}_{tX,j}(\tau)dw_{j}(\tau).
\label{eq:8-102}
\end{equation}
We note that $\mathbb{Z}_{tX}(T) = h_x\big(\mathbb{X}_{tX}(T)\big) 
+D_X F_T\big(\mathbb{X}_{tX}(T)  \otimes  m\big) $ since $\Gamma_{tX}$ is independent of $\mathcal{W}^T_{tX}$. This fact together with integration by parts yield
\begin{align*}
\int_{t}^{T}\pigl\langle \mathbb{Z}_{tX}(s),\widetilde{v}_{tX}(s)\pigr\rangle_{\mathcal{H}_m}ds
=&\int_{t}^{T}\left\langle l_{x}\pig(\mathbb{X}_{tX}(s),v_{tX}(s)\pig)
+D_{X}F\pig(\mathbb{X}_{tX}(s)  \otimes  m\pig)
,\int_{t}^{s}\widetilde{v}_{tX}(\tau)d\tau
\right\rangle_{\mathcal{H}_m}ds\\
&+\left\langle h_{x}(\mathbb{X}_{tX}(T))+D_{X}F_T(\mathbb{X}_{tX}(T)  \otimes  m),\int_{t}^{T}\widetilde{v}_{tX}(\tau)d\tau\right\rangle_{\mathcal{H}_m}.
\end{align*}
Hence 
\[
\dfrac{d}{d\theta}J_{tX}\pig(v_{tX}+\theta\tilde{v}_{tX}\pig)\bigg|_{\theta=0}
=\int_{t}^{T}\pigl\langle l_{v}\big(\mathbb{X}_{tX}(s),v_{tX}(s)\big)+\mathbb{Z}_{tX}(s),\tilde{v}_{tX}(s)\pigr\rangle_{\mathcal{H}_m}ds,
\]
which gives the result (\ref{eq:3-5}) when we take the first variation with respect to $v_{tX}$. \hfill$\blacksquare$

\subsubsection{Proof of Proposition \ref{prop. convex of J}}\label{app, prop. convex of J}

\noindent {\bf Part 1A. Convexity:}\\
Consider two controls $v_{tX}^{1}(s), v_{tX}^{2}(s) \in L_{\mathcal{W}_{tX}}^{2}(t,T;\mathcal{H}_{m})$. We are going
to verify the strong convexity, in the sense that 

\begin{equation}
\begin{aligned}
\int_{t}^{T}\Big\langle D_{v}J_{tX}\pig(v_{tX}^{1}\pig)(s)
-D_{v}J_{tX}\pig(v_{tX}^{2}&\pig)(s),v_{tX}^{1}(s)-v_{tX}^{2}(s)\Big\rangle_{\mathcal{H}_m}ds
\geq c_0\int_{t}^{T}\pigl\|v_{tX}^{1}(s)-v_{tX}^{2}(s)\pigr\|^{2}_{\mathcal{H}_m}ds.
\end{aligned}
\label{eq:Ap1}
\end{equation}
for some constant $c_0>0$. Then, the claim in the proposition will follow immediately. To simplify the notations, we write $v^{1}(s)=v_{tX}^{1}(s),\:v^{2}(s)=v_{tX}^{2}(s)$, $
\mathbb{X}^{1}(s)=\mathbb{X}_{tX} \pig(s;v_{tX}^{1}\pig)$ and $\mathbb{X}^{2}(s)=\mathbb{X}_{tX}\pig(s;v_{tX}^{2}\pig)$. Let $\mathbb{Z}^{1}(s)$ and $\mathbb{Z}^{2}(s)$ be the corresponding solutions of (\ref{def, backward SDE}) with respect to $\mathbb{X}^1(s)$ and $\mathbb{X}^2(s)$ respectively.

From the formula (\ref{eq:3-5}), we have 
\begin{equation}
\begin{aligned}
\int_{t}^{T}\Big\langle D_{v}J_{tX}(v^1)(s)-D_{v}J_{tX}(v^{2})(s)&,v^{1}(s)-v^{2}(s)\Big\rangle_{\mathcal{H}_m}ds\\
=\:&\int_{t}^{T}\Big\langle l_{v}\pig(\mathbb{X}^{1}(s),v^{1}(s)\pig)-l_{v}\pig(\mathbb{X}^{2}(s),v^{2}(s)\pig),v^{1}(s)-v^{2}(s)\Big\rangle_{\mathcal{H}_m}ds\\
&+\int_{t}^{T}\Big\langle \mathbb{Z}^{1}(s)-\mathbb{Z}^{2}(s),v^{1}(s)-v^{2}(s)\Big\rangle_{\mathcal{H}_m}ds.
\end{aligned}
\label{eq. of int DJ1-DJ2}
\end{equation}
Next, since $v^{1}(s)-v^{2}(s)=\dfrac{d}{ds}\pig[\mathbb{X}^{1}(s)-\mathbb{X}^{2}(s)\pig]$
and $\mathbb{X}^{1}(t)-\mathbb{X}^{2}(t)=0,$ the equality in (\ref{eq. of int DJ1-DJ2}) is equivalent to
\begin{equation}
\begin{aligned}
&\h{-20pt}\int_{t}^{T}\Big\langle D_{v}J_{tX}(v^1)(s)-D_{v}J_{tX}(v^{2})(s),v^{1}(s)-v^{2}(s)\Big\rangle_{\mathcal{H}_m}ds\\
=\:&\int_{t}^{T}\Big\langle l_{v}\pig(\mathbb{X}^{1}(s),v^{1}(s)\pig)-l_{v}\pig(\mathbb{X}^{2}(s),v^{2}(s)\pig),v^{1}(s)-v^{2}(s)\Big\rangle_{\mathcal{H}_m}ds\\
&+\int_{t}^{T}\Big\langle l_{x}(\mathbb{X}^{1}(s),v^{1}(s))-l_{x}(\mathbb{X}^{2}(s),v^{2}(s))+D_{X}F(\mathbb{X}^{1}(s)  \otimes  m)-D_{X}F(\mathbb{X}^{2}(s)  \otimes  m),\mathbb{X}^{1}(s)-\mathbb{X}^{2}(s)\Big\rangle_{\mathcal{H}_m}ds\\
&+\Big\langle  h_x(\mathbb{X}^1(T)) 
+ D_X  F_T(\mathbb{X}^1(T)  \otimes  m) 
-h_x(\mathbb{X}^2(T)) 
- D_X F_T(\mathbb{X}^2(T)  \otimes  m),
\mathbb{X}^1(T)-\mathbb{X}^2(T)
\Big\rangle_{\mathcal{H}_m}.
\end{aligned}
\label{Z1-Z2 geq v +X}
\end{equation}
The mean value theorem, Assumptions \textbf{A(ii)}, \textbf{A(v)}, \textbf{A(vi)} and \textbf{B(v)(b)} tell us that
\begin{equation}
\begin{aligned}
&\h{-10pt}\int_{t}^{T}\Big\langle D_{v}J_{tX}(v^{1})(s)-D_{v}J_{tX}(v^{2})(s),v^{1}(s)-v^{2}(s)\Big\rangle_{\mathcal{H}_m}ds\\
\geq\,&\int_{t}^{T}\inf_{x,v\in\mathbb{R}^n}\bigg[\Big\langle l_{vv}(x,v)\pig[v^{1}(s)-v^{2}(s)\pig]+l_{vx}(x,v)\pig[\mathbb{X}^{1}(s)-\mathbb{X}^{2}(s)\pig],v^{1}(s)-v^{2}(s)\Big\rangle_{\mathcal{H}_m}\\
&\h{60pt}+\Big\langle l_{xv}(x,v)\pig[v^{1}(s)-v^{2}(s)\pig]+l_{xx}(x,v)\pig[\mathbb{X}^{1}(s)-\mathbb{X}^{2}(s)\pig],\mathbb{X}^{1}(s)-\mathbb{X}^{2}(s)\Big\rangle_{\mathcal{H}_m}\bigg]ds\\
&- c' \int_{t}^{T}\pigl\|\mathbb{X}^{1}(s)-\mathbb{X}^{2}(s)\pigr\|^{2}_{\mathcal{H}_m}ds
- (c'_T+c'_h)  \pigl\|\mathbb{X}^{1}(T)-\mathbb{X}^{2}(T)\pigr\|^{2}_{\mathcal{H}_m}\\
\geq\,&\lambda
\int_{t}^{T}\pigl\|v^{1}(s)-v^{2}(s)\pigr\|^{2}_{\mathcal{H}_m}ds
-( c' +c'_l)\int_{t}^{T}\pigl\|\mathbb{X}^{1}(s)-\mathbb{X}^{2}(s)\pigr\|^{2}_{\mathcal{H}_m}ds
- (c'_T+c'_h)  \pigl\|\mathbb{X}^{1}(T)-\mathbb{X}^{2}(T)\pigr\|^{2}_{\mathcal{H}_m}
\end{aligned}
\label{convex, est. DJ1-DJ2}
\end{equation}
The fact that $\mathbb{X}^{1}(s)-\mathbb{X}^{2}(s)=\int_{t}^{s}v^{1}(\tau)-v^{2}(\tau)d\tau$ for any $s \in(t,T]$, together with a simple application of Cauchy-Schwarz inequality imply
\begin{equation}
    \pigl\|\mathbb{X}^{1}(T)-\mathbb{X}^{2}(T)\pigr\|^{2}_{\mathcal{H}_m}
\leq (T-t) \int_{t}^{T}\pigl\|v^{1}(s)-v^{2}(s)\pigr\|^{2}_{\mathcal{H}_m}ds\h{10pt} \text{ and }
\label{convex, est. |X_1(T)-X_2(T)|^2}
\end{equation}
\begin{equation}
\int_{t}^{T}\pigl\|\mathbb{X}^{1}(s)-\mathbb{X}^{2}(s)\pigr\|^{2}_{\mathcal{H}_m}ds\leq\dfrac{(T-t)^{2}}{2}\int_{t}^{T}\pigl\|v^{1}(s)-v^{2}(s)\pigr\|^{2}_{\mathcal{H}_m}ds.
\label{convex, est. int |X_1(s)-X_2(s)|^2}
\end{equation}
Bringing (\ref{convex, est. |X_1(T)-X_2(T)|^2}) and (\ref{convex, est. int |X_1(s)-X_2(s)|^2}) into (\ref{convex, est. DJ1-DJ2}), the strict convexity is proven for
\begin{equation}
c_0:=\lambda
- (c'_T+c'_h)_+T
- \left(c'_{l}+c'\right)_+\dfrac{T^2}{2}>0.
\label{c_0 with k_1 and k_2}
\end{equation}
\noindent {\bf Part 1B. Coercivity:}\\
For the coercivity, from the formula (\ref{eq:Ap1}), we have
\begin{equation}
\int_{t}^{T}\Big\langle D_{v}J_{tX}\pig(v_{tX}\pig)(s)
-D_{v}J_{tX}(0)(s),v_{tX}(s)\Big\rangle_{\mathcal{H}_m}ds
\geq c_0\int_{t}^{T}\bigl\|v_{tX}(s)\bigr\|^{2}_{\mathcal{H}_m}ds,
\label{eq:Ap2}
\end{equation}
In addition, one can write
\begin{align*}
J_{tX}\big(v_{tX}\big)-J_{tX}(0)
=\int^1_0 \dfrac{d}{d\theta}J_{tX}\pig( \theta v_{tX} \pig) d \theta 
=\int_{0}^{1} \int^T_t \Big\langle D_{v}J_{tX}(\theta v_{tX})(s),v_{tX}(s)\Big\rangle_{\mathcal{H}_m}dsd\theta.
\end{align*}
Combined with (\ref{eq:Ap2}), we obtain

\[
J_{tX}(v_{tX})-J_{tX}(0)
\geq\int_{t}^{T}\Big\langle D_{v}J_{tX}(0)(s),v_{tX}(s)\Big\rangle_{\mathcal{H}_m}ds
+\dfrac{c_{0}}{2}\int_{t}^{T}\big\|v_{tX}(s)\big\|^{2}_{\mathcal{H}_m}ds,
\]
where we note that $\int^1_0\theta d \theta =\frac{1}{2}$. This completes the proof of the coercivity.

{\color{black}
\noindent {\bf Part 2. Existence and Uniqueness of the Optimal Control:}\\
Since $l$, $h$, $F$ and $F_T$ are continuous according to Assumptions \textbf{A(i)}'s \eqref{assumption, bdd of l, lx, lv}, \textbf{A(iii)}'s \eqref{assumption, bdd of h, hx, hxx} and \textbf{B(i)}'s \eqref{assumption, bdd of DF, DF_T}, together with the fact that $\mathbb{X}_{tX}$ is continuous in $v_{tX}$  under the strong $L_{\mathcal{W}_{tX}}^{2}(t,T;\mathcal{H}_{m})$-norm by \eqref{def. X_tX, constant sigma}, the functional $J_{tX}\big(v_{tX}\big)$ is clearly continuous in $v_{tX}$. Combining the convexity and coercivity of the functional, the existence and uniqueness of the optimal control is then guaranteed by Theorem 7.2.12. in \cite{DM07}. We here include our own proof for the sake of  completeness.

\noindent {\bf Part 2A. Weakly Sequentially Lower Semi-continuity of $J_{tX}\big(v_{tX}\big)$:}\\
Let $\{u_n\}^\infty_{n=1} \subset L_{\mathcal{W}_{tX}}^{2}(t,T;\mathcal{H}_{m})$ such that $u_n \longrightarrow u_\infty$ weakly. Let $\gamma_1 := \displaystyle\liminf_{n\to \infty} J_{tX}\big(u_n\big)$, there is a subsequence $\{u_{n_k}\}^\infty_{k=1}$ such that $J_{tX}\big(u_{n_k}\big) \longrightarrow \gamma_1$. Take $\epsilon>0$ be small, for any large enough $k$, we have $u_{n_k} \in L_{\mathcal{W}_{tX}}^{2}(t,T;\mathcal{H}_{m}) \cap \pig\{J_{tX}\big(u\big) \leq \gamma_1+\epsilon\pig\}$. Since $J_{tX}\big(u\big)$ is continuous, then the set $L_{\mathcal{W}_{tX}}^{2}(t,T;\mathcal{H}_{m}) \cap \pig\{J_{tX}\big(u\big) \leq \gamma_1+\epsilon\pig\}$ is closed. Moreover, $J_{tX}\big(u\big)$ is convex in $u$, therefore the set $L_{\mathcal{W}_{tX}}^{2}(t,T;\mathcal{H}_{m}) \cap \pig\{J_{tX}\big(u\big) \leq \gamma_1+\epsilon\pig\}$ is also convex and hence is weakly closed (see Exercise 2.1.40. in \cite{DM07} which is a simple consequence of Hahn-Banach theorem by arguing that the convex sets $L_{\mathcal{W}_{tX}}^{2}(t,T;\mathcal{H}_{m}) \cap \pig\{J_{tX}\big(u\big) \leq \gamma_1+\epsilon\pig\}$ and $\{u_\infty\}$ cannot be separated by linear functionals as $u_\infty$ is the weak limit of $u_{n_k}$). As $\epsilon>0$ is arbitrary, we see that $u_\infty \in L_{\mathcal{W}_{tX}}^{2}(t,T;\mathcal{H}_{m}) \cap \pig\{J_{tX}\big(u\big) \leq \gamma_1\pig\}$ and thus $J_{tX}\big(u_\infty\big)\leq \gamma_1= \displaystyle\liminf_{n\to \infty} J_{tX}\big(u_n\big)$.

\noindent {\bf Part 2B. Conclusion:}\\
Let $\gamma_2 = \displaystyle\inf_{v \in L_{\mathcal{W}_{tX}}^{2}(t,T;\mathcal{H}_{m})}J_{tX}\big(v\big) $ and $\{v_n\}^\infty_{n=1} \subset L_{\mathcal{W}_{tX}}^{2}(t,T;\mathcal{H}_{m})$ be a minimizing sequence such that $J_{tX}\big(v_n\big) \longrightarrow \gamma_2$. By the coercivity of $J_{tX}\big(v\big)$, there is a $\Gamma_1>0$ such that if $v \in L_{\mathcal{W}_{tX}}^{2}(t,T;\mathcal{H}_{m})$ and $\|u\|_{L_{\mathcal{W}_{tX}}^{2}(t,T;\mathcal{H}_{m})} \geq \Gamma_1$, then $J_{tX}\big(v\big)>2\gamma_2$. Since $J_{tX}\big(v_n\big) \longrightarrow \gamma_2$, thus for large enough $n$, we can assume that $\|v_n\|_{L_{\mathcal{W}_{tX}}^{2}(t,T;\mathcal{H}_{m})} \leq \Gamma_1$. So there is a subsequence $\{v_{n_k}\}^\infty_{n=1}$ such that $v_{n_k} \longrightarrow v_\infty$ weakly and thus the weakly sequentially lower semi-continuity of $J_{tX}\big(v_{tX}\big)$ implies 
$$\displaystyle\inf_{v \in L_{\mathcal{W}_{tX}}^{2}(t,T;\mathcal{H}_{m})}J_{tX}\big(v\big)
\leq J_{tX}\big(v_\infty\big)\leq \displaystyle\liminf_{k\to \infty} J_{tX}\big(v_{n_k}\big) = \gamma_2
=\displaystyle\inf_{v \in L_{\mathcal{W}_{tX}}^{2}(t,T;\mathcal{H}_{m})}J_{tX}\big(v\big).$$
It concludes that $v_\infty$ is the optimal control. The uniqueness follows easily by the strict convexity as stated in Part 1A of this proof.

}

\hfill $\blacksquare$

\subsubsection{Proof of Lemma \ref{lem, Existence of J flow}}\label{app, Existence of J flow}

\mycomment{
First, we aim to find a local solution to the system (\ref{forward eq DY_tX psi, no expand in u, A})-(\ref{backward eq Z_tX psi, no expand in u, A}) over $[t,T]$. More precisely, for the sequence $\big\{\tau_{ik}\big\}_{k=1}^{n^i} \subset [t,T]$ such that $t=\tau_{i1}<\tau_{i2}<\ldots<\tau_{i,{n^i-1}}<\tau_{i,{n^i}}=T$, we look for a solution $\pig( \mathscr{D}^\Psi\mathbb{Y}_{\tau_{ik}X}(s),\mathscr{D}^\Psi\mathbb{Z}_{\tau_{ik}X}(s), \mathscr{D}^\Psi\mathbbm{r}_{\tau_{ik}X,j}(s)\pig)$ to the system, for $s \in [\tau_{ik},T]$ with $k=1,2,\ldots,n^i$, 
\vspace{-5pt}
{\small
\begin{empheq}[left=\h{-10pt}\empheqbiglbrace]{align}
\mathscr{D}^\Psi\mathbb{Y}_{\tau_{ik} X}(s)
=&\:\Psi
+\displaystyle\int_{\tau_{ik}}^{s}
\Big[ \diff_y  u(\mathbb{Y}_{tX},\mathbb{Z}_{tX})(\tau)\Big] 
\mathscr{D}^\Psi\mathbb{Y}_{\tau_{ik}X}(\tau)d\tau
+\int_{\tau_{ik}}^{s} 
\Big[\diff_z  u (\mathbb{Y}_{tX},\mathbb{Z}_{tX}) (\tau)\Big] 
\mathscr{D}^\Psi\mathbb{Z}_{\tau_{ik}X}(\tau)  d\tau
\label{forward eq DY_tX psi, no expand in u, A};\\
\mathscr{D}^\Psi\mathbb{Z}_{\tau_{ik}X}(s)
=&\:h_{xx}(\mathbb{Y}_{tX}(T))\mathscr{D}^\Psi\mathbb{Y}_{\tau_{ik}X}(T)+D_{X}^{2}F_{T}(\mathbb{Y}_{tX}(T)  \otimes  m)(\mathscr{D}^\Psi\mathbb{Y}_{\tau_{ik}X}(T))\nonumber\\
&+\int_s^T
\h{-3pt}\pig[l_{xx}\big(\mathbb{Y}_{tX}(\tau),u_{tX}(\tau)\big)\mathscr{D}^\Psi\mathbb{Y}_{\tau_{ik}X}(\tau)
+l_{xv}\big(\mathbb{Y}_{tX}(\tau),u_{tX}(\tau)\big)\Big[ \diff_y  u(\mathbb{Y}_{tX},\mathbb{Z}_{tX})(\tau)\Big] 
\mathscr{D}^\Psi\mathbb{Y}_{\tau_{ik}X}(\tau)\h{-70pt}\nonumber\\
&\h{25pt}+l_{xv}\big(\mathbb{Y}_{tX}(\tau),u_{tX}(\tau)\big)\Big[ \diff_z  u(\mathbb{Y}_{tX},\mathbb{Z}_{tX})(\tau)\Big] 
\mathscr{D}^\Psi\mathbb{Z}_{\tau_{ik}X}(\tau)
+D_{X}^{2}F(\mathbb{Y}_{tX}(\tau)  \otimes  m)(\mathscr{D}^\Psi\mathbb{Y}_{\tau_{ik}X}(\tau))\pig]d\tau
\h{-70pt}\nonumber\\
&-\int_s^T\sum_{j=1}^{n}\mathscr{D}^\Psi\mathbbm{r}_{\tau_{ik}X,j}(\tau)dw_{j}(\tau).
\label{backward eq Z_tX psi, no expand in u, A}
\end{empheq}}
Note that the system (\ref{forward eq DY_tX psi, no expand in u, A})-(\ref{backward eq Z_tX psi, no expand in u, A}) is linear in coefficients and the terminal condition, the coefficients are also fixed if the solution to the FBSDE (\ref{forward FBSDE})-(\ref{1st order condition}) is given, therefore, we can adopt the method in \cite{HP95} to establish the global existence. Or, the construction of the global solution can follow the same procedure as in Section \ref{sec. local existence}. To this end, the critical step is to construct the Lipschitz terminal map $P^i_k$ in $L^2_{\mathcal{W}_{\tau_{ik}}^{\indep} }(\mathcal{H}_m) $ for the local system (\ref{forward eq DY_tX psi, no expand in u, A})-(\ref{backward eq Z_tX psi, no expand in u, A}), similar to Steps 1 to 3 in Section \ref{sec. local existence} for those generic $P_i$'s. More precisely, first, it is clear that (\ref{forward eq DY_tX psi, no expand in u, A})-(\ref{backward eq Z_tX psi, no expand in u, A}) has a solution on $[\tau_{i,n^i-1},T]$ by following the same method as in Lemmas \ref{local forward estimate}-\ref{lem, local backward estimate} and contraction mapping arguments as in Section \ref{sec. local existence},  provided that $\big|T-\tau_{i,n^i-1}\big|$ is small enough. Note that (\ref{forward eq DY_tX psi, no expand in u, A})-(\ref{backward eq Z_tX psi, no expand in u, A}) is a linear system with fixed coefficients, the proof of local existence on $[\tau_{i,n^i-1},T]$ is much simpler than that in Section \ref{sec. local existence}. Let $\Psi_1$, $\Psi_2 \in L^2_{\mathcal{W}_{\tau_{i,n^i-1}}^{\indep} }\h{-5pt}(\mathcal{H}_m) $, thanks to (\ref{uniqueness condition}), Proposition \ref{prop bdd of Y Z u r} and the linearity of system (\ref{forward eq DY_tX psi, no expand in u, A})-(\ref{backward eq Z_tX psi, no expand in u, A}), we have the estimate,
\begin{equation}
\sup_{s \in \big[\tau_{i,n^i-1},T\big]}\big\|\mathscr{D}^{\Psi_1}\mathbb{Z}_{\tau_{i,n^i-1}X}(s)-
\mathscr{D}^{\Psi_2}\mathbb{Z}_{\tau_{i,n^i-1}X}(s)\pigr\|_{\mathcal{H}_m}\:
\leq C'_4\|\Psi_1-\Psi_2\|_{\mathcal{H}_m},
\label{bdd of DU DY DZ, approx of J flow}
\end{equation}
where the positive constant $C'_4$ is independent of $\Psi_1$, $\Psi_2$ and $\tau_{i,n^i-1}$. Therefore, $P^i_{n^i-2}$ is Lipschitz under the norm over $\mathcal{H}_m$. Again, due to the linearity of the system, we easily follow Step 1 to Step 3 mentioned in Section \ref{sec. local existence}, without invoking the estimates in Lemmas \ref{lem cruical est.}-\ref{lem, bdd of DZ}, to establish the Lipschitz continuity of $P^i_{k}$, for all $k=1,2,\ldots,n^i-1$, with the Lipschitz constant being independent of the partition $\{\tau_{ik}\}^{n^i}_{k=1}$. Hence, the existence of the global solution to (\ref{forward eq DY_tX psi, no expand in u, A})-(\ref{backward eq Z_tX psi, no expand in u, A}) on $[t,T]$ is established with the same reasoning for Theorem \ref{thm global existence},  as long as $\displaystyle\max_k|\tau_{ik}-\tau_{i,k+1}|$ is sufficiently small. We remark that using (\ref{uniqueness condition}) and Proposition \ref{prop bdd of Y Z u r} again, we see that 
\begin{equation}
\big\|\mathscr{D}^\Psi\mathbb{Y}_{tX}(s)\big\|_{\mathcal{H}_m},\:
\big\|\mathscr{D}^\Psi\mathbb{Z}_{tX}(s)\big\|_{\mathcal{H}_m},\:
\left[\displaystyle\sum^n_{j=1}\int^T_{t}\big\|\mathscr{D}^\Psi\mathbbm{r}_{tX,j}(\tau)\big\|_{\mathcal{H}_m}^2
\right]^{1/2} d\tau
\leq C'_4\|\Psi\|_{\mathcal{H}_m}.
\label{bdd of DU DY DZ, approx of J flow}
\end{equation}

On the other hand, the quadruple $\pig( \Delta^\epsilon_\Psi \mathbb{Y}_{tX} (s),
\Delta^\epsilon_\Psi \mathbb{Z}_{tX} (s),
\Delta^\epsilon_\Psi u_{tX}(s),
\Delta^\epsilon_\Psi\mathbbm{r}_{tX}(s)\pig)$ solves  the system 

\begin{empheq}[left=\h{-10pt}\empheqbiglbrace]{align}
\Delta^\epsilon_\Psi \mathbb{Y}_{tX} (s)
=\:&\Psi+\int_{t}^{s}\Delta^\epsilon_\Psi u_{tX}(\tau)d\tau;
\label{finite diff. forward}\\
\Delta^\epsilon_\Psi \mathbb{Z}_{tX} (s)
=\:&
\displaystyle\int_{0}^{1}h_{xx}\pig(\mathbb{Y}_{t X}(T)+\theta\epsilon \Delta^\epsilon_\Psi \mathbb{Y}_{tX} (T)\pig)\Delta^\epsilon_\Psi \mathbb{Y}_{tX} (T)\nonumber\\ &\h{100pt}+D_{X}^{2}F_{T}\pig((\mathbb{Y}_{t X}(T)+\theta\epsilon \Delta^\epsilon_\Psi \mathbb{Y}_{tX} (T)) \otimes  m\pig)(\Delta^\epsilon_\Psi \mathbb{Y}_{tX} (T))d\theta\nonumber\\
&+\int^T_s\int_{0}^{1}l_{xx}\pig(\mathbb{Y}_{t X}(\tau)+\theta\epsilon \Delta^\epsilon_\Psi \mathbb{Y}_{tX} (\tau),u_{t X}(\tau) \pig)\Delta^\epsilon_\Psi \mathbb{Y}_{tX} (\tau)\nonumber\\
&\pushright{+l_{xv}\pig(\mathbb{Y}_{t,X+\epsilon\Psi}(\tau),u_{t X}(\tau)+\theta\epsilon \Delta^\epsilon_\Psi u_{tX}(\tau)\pig)\Delta^\epsilon_\Psi u_{tX}(\tau)d\theta d\tau}\h{-1pt}\nonumber\\
&+\int^T_s\int_{0}^{1}D_{X}^{2}F\pig((\mathbb{Y}_{t X}(\tau)+\theta\epsilon \Delta^\epsilon_\Psi \mathbb{Y}_{tX} (\tau)) \otimes  m\pig)(\Delta^\epsilon_\Psi \mathbb{Y}_{tX} (\tau))d\theta d\tau\nonumber\\
&-\int^T_s\sum_{j=1}^n \Delta^\epsilon_\Psi r_{t X,j}(\tau)dw_{j}(\tau)d\tau,
\label{finite diff. backward}
\end{empheq}
meanwhile,
{\footnotesize
\begin{equation}
\int_{0}^{1}l_{vx}\pig(\mathbb{Y}_{t X}(s)+\theta\epsilon \Delta^\epsilon_\Psi \mathbb{Y}_{tX} (s),u_{t X}(s)\pig)
\Delta^\epsilon_\Psi \mathbb{Y}_{tX} (s)
+l_{vv}\pig(\mathbb{Y}_{t,X+\epsilon\Psi}(s) ,u_{t X}(s)+\theta\epsilon \Delta^\epsilon_\Psi u_{tX}(s)\pig)
\Delta^\epsilon_\Psi u_{tX}(s)d\theta+\Delta^\epsilon_\Psi \mathbb{Z}_{tX} (s)=0.
\label{1st order finite diff.}
\end{equation}}
Now, by subtracting (\ref{forward eq DY_tX psi, no expand in u, A})-(\ref{backward eq Z_tX psi, no expand in u, A}) from (\ref{finite diff. forward})-(\ref{finite diff. backward}) correspondingly, we consider the system for 
$$\pig( \Delta^\epsilon_\Psi \mathbb{Y}_{tX} (s)
- D \mathbb{Y}^\Psi_{tX} (s), 
\Delta^\epsilon_\Psi \mathbb{Z}_{tX} (s)
-D \mathbb{Z}^\Psi_{tX} (s),
\Delta^\epsilon_\Psi\mathbbm{r}_{t X,j}(s)
-\mathscr{D}^\Psi \mathbbm{r}_{t X,j}(s)
\pig).$$
By using the same arguments as in the proof of Proposition \ref{prop bdd of Y Z u r}, together with Assumptions \textbf{A(vii)}'s (\ref{assumption, lip of 3rd derivative of l and h}) and \textbf{B(vi)}'s (\ref{assumption, lip of 3rd derivative of F and FT}), we can easily deduce that, for instance, $$\pig\|\mathscr{D}^\Psi\mathbb{Y}_{t X}(s)
-\Delta_\Psi^\epsilon\mathbb{Y}_{t X}(s)
\pigr\|_{\mathcal{H}_m}
\leq \epsilon C''_4,$$
where $C_{4}''$ is a positive constant independent of $\epsilon$ and $s$. Therefore, we conclude, as $\epsilon \to 0^+$, \scalebox{0.85}{$\pig( \Delta^\epsilon_\Psi \mathbb{Y}_{tX} (s),
\Delta^\epsilon_\Psi \mathbb{Z}_{tX} (s)\pig)$}\\
$\longrightarrow 
\pig( \mathscr{D}^\Psi\mathbb{Y}_{tX}(s),\mathscr{D}^\Psi\mathbb{Z}_{tX}(s)\pig)$ in $L_{\mathcal{W}_{t X\Psi}}^{\infty}(t,T;\mathcal{H}_{m})$ and $\Delta^\epsilon_\Psi\mathbbm{r}_{tX,j}(s) \longrightarrow \mathscr{D}^\Psi\mathbbm{r}_{t X,j}(s)$ in $L_{\mathcal{W}_{t X\Psi}}^{2}(t,T;\mathcal{H}_{m})$.

======================================}

\noindent {\bf Step 1. Weak Convergence:}\\
From (\ref{forward FBSDE})-(\ref{1st order condition}), the quadruple $\pig( \Delta^\epsilon_\Psi \mathbb{Y}_{tX} (s),
\Delta^\epsilon_\Psi \mathbb{Z}_{tX} (s),
\Delta^\epsilon_\Psi u_{tX}(s),
\Delta^\epsilon_\Psi\mathbbm{r}_{tX}(s)\pig)$ solves  the system \vspace{-15pt}

\begin{empheq}[left=\h{-10pt}\empheqbiglbrace]{align}
\Delta^\epsilon_\Psi \mathbb{Y}_{tX} (s)
=\:&\Psi+\int_{t}^{s}\Delta^\epsilon_\Psi u_{tX}(\tau)d\tau;
\label{finite diff. forward}\\
\Delta^\epsilon_\Psi \mathbb{Z}_{tX} (s)
=\:&
\displaystyle\int_{0}^{1}h_{xx}\pig(\mathbb{Y}_{t X}(T)+\theta\epsilon \Delta^\epsilon_\Psi \mathbb{Y}_{tX} (T)\pig)\Delta^\epsilon_\Psi \mathbb{Y}_{tX} (T)\nonumber\\ &\h{100pt}+D_{X}^{2}F_{T}\pig((\mathbb{Y}_{t X}(T)+\theta\epsilon \Delta^\epsilon_\Psi \mathbb{Y}_{tX} (T)) \otimes  m\pig)(\Delta^\epsilon_\Psi \mathbb{Y}_{tX} (T))d\theta\nonumber\\
&+\int^T_s\int_{0}^{1}l_{xx}\pig(\mathbb{Y}_{t X}(\tau)+\theta\epsilon \Delta^\epsilon_\Psi \mathbb{Y}_{tX} (\tau),u_{t X}(\tau) \pig)\Delta^\epsilon_\Psi \mathbb{Y}_{tX} (\tau)\nonumber\\
&\pushright{+l_{xv}\pig(\mathbb{Y}_{t,X+\epsilon\Psi}(\tau),u_{t X}(\tau)+\theta\epsilon \Delta^\epsilon_\Psi u_{tX}(\tau)\pig)\Delta^\epsilon_\Psi u_{tX}(\tau)d\theta d\tau}\h{-1pt}\nonumber\\
&+\int^T_s\int_{0}^{1}D_{X}^{2}F\pig((\mathbb{Y}_{t X}(\tau)+\theta\epsilon \Delta^\epsilon_\Psi \mathbb{Y}_{tX} (\tau)) \otimes  m\pig)(\Delta^\epsilon_\Psi \mathbb{Y}_{tX} (\tau))d\theta d\tau\nonumber\\
&-\int^T_s\sum_{j=1}^n \Delta^\epsilon_\Psi \mathbbm{r}_{t X,j}(\tau)dw_{j}(\tau),
\label{finite diff. backward}
\end{empheq}
meanwhile,\vspace{-15pt}

\begin{equation}
\scalemath{0.91}{
\int_{0}^{1}l_{vx}\pig(\mathbb{Y}_{t X}(s)+\theta\epsilon \Delta^\epsilon_\Psi \mathbb{Y}_{tX} (s),u_{t X}(s)\pig)
\Delta^\epsilon_\Psi \mathbb{Y}_{tX} (s)
+l_{vv}\pig(\mathbb{Y}_{t,X+\epsilon\Psi}(s) ,u_{t X}(s)+\theta\epsilon \Delta^\epsilon_\Psi u_{tX}(s)\pig)
\Delta^\epsilon_\Psi u_{tX}(s)d\theta+\Delta^\epsilon_\Psi \mathbb{Z}_{tX} (s)=0.}
\label{1st order finite diff.}
\end{equation}
Since $\Delta^\epsilon_\Psi \mathbb{Y}_{tX} (s)$, $\Delta^\epsilon_\Psi \mathbb{Z}_{tX} (s)$ and $\Delta^\epsilon_\Psi u_{tX}(s)$ are uniformly bounded in $L^\infty_{\mathcal{W}_{t X\Psi}}(t,T;\mathcal{H}_m)$ for all choices of $\epsilon$ by (\ref{bdd Y, Z, u, r, epsilon}), then by Banach-Alaoglu theorem, these finite difference processes converge to the weak limits $\mathscr{D} \mathbb{Y}^\Psi_{tX} (s)$, $\mathscr{D} \mathbb{Z}^\Psi_{tX} (s)$ and $\mathscr{D}u^\Psi_{tX} (s)$ in $L^2_{\mathcal{W}_{t X\Psi}}(t,T;\mathcal{H}_m)$, as $\epsilon \to 0$ along a subsequence. Equation (\ref{finite diff. forward}) becomes, as $\epsilon \to 0$ along the subsequence, 
\begin{equation}
\Delta^\epsilon_\Psi \mathbb{Y}_{tX} (s)\longrightarrow
 \mathscr{D}\mathbb{Y}^\Psi_{tX} (s)
=\:\Psi+\int_{t}^{s}\mathscr{D}u^\Psi_{tX}(\tau)d\tau,\h{10pt}\text{weakly in $L^2_{\mathcal{W}_{t X\Psi}}(t,T;\mathcal{H}_m)$,}
\label{forward, app}
\end{equation}
where $\mathscr{D}u^\Psi_{tX}(\tau)$ can be expressed as \begin{equation}
\mathscr{D}u^\Psi_{tX}(\tau)
=\Big[\diff_y  u (\mathbb{Y}_{tX},\mathbb{Z}_{tX}) (\tau)\Big] 
\Big[\mathscr{D} \mathbb{Y}_{tX}^\Psi  (\tau)\Big]
+\Big[\diff_z  u (\mathbb{Y}_{tX},\mathbb{Z}_{tX}) (\tau)\Big] 
\Big[\mathscr{D} \mathbb{Z}_{tX}^\Psi  (\tau)\Big]
\label{Du = DY+DZ}
\end{equation} due to the first order condition (\ref{1st order condition}) and Remark \ref{rem def of Du = DY +DZ}. We also have 
\begin{equation}
\Delta^\epsilon_\Psi \mathbb{Y}_{tX} (T)\longrightarrow
 \mathscr{D}\mathbb{Y}^\Psi_{tX} (T)
=\:\Psi+\int_{t}^{T}\mathscr{D}u^\Psi_{tX}(\tau)d\tau,\h{10pt}\text{weakly in $L^2_{\mathcal{W}_{t X\Psi}}(\mathcal{H}_m)$ up to a subsequence.}
\label{forward, app, at T}
\end{equation}
Moreover, $\Delta^\epsilon_\Psi \mathbbm{r}_{tX,j}(s)$ is uniformly bounded for all choices of $\epsilon$ in $L^2_{\mathcal{W}_{t X\Psi}}(t,T;\mathcal{H}_m)$ by (\ref{bdd Y, Z, u, r, epsilon}), then it converges to the weak limits $\mathscr{D} \mathbbm{r}^\Psi_{tX,j} (s)$, up to a subsequence. Define $\mathbb{Y}_{\theta\epsilon}(s) := \mathbb{Y}_{{tX}}(s) 
+ \epsilon \theta \Delta^\epsilon_\Psi \mathbb{Y}_{{tX}}(s) $ and $u_{\theta\epsilon}(s) := u_{{tX}}(s) 
+ \epsilon \theta \Delta^\epsilon_\Psi u_{{tX}}(s) $, we rewrite the right hand side of (\ref{finite diff. backward}) by augmenting the pointwisely bounded test random variable $\varphi \in L^\infty_{\mathcal{W}_{t X\Psi}}(t,T;\mathcal{H}_m)$ under the inner product, after telescoping, 
\begin{equation}
\h{-9pt}\begin{aligned}
&\h{-10pt}\bigg\langle
\int^1_0 \Big[h_{xx}(\mathbb{Y}_{\theta\epsilon}(T) )
+D_{X}^{2}F_{T}\pig(\mathbb{Y}_{\theta\epsilon}(T) \otimes  m\pig)\Big](\Delta^\epsilon_\Psi \mathbb{Y}_{tX} (T))d\theta
,\varphi
\bigg\rangle_{\mathcal{H}_m}\\
&+\bigg\langle\int_s^{T}
 \int^1_0 l_{xv}(\mathbb{Y}_{t,X+\epsilon\Psi}(\tau),u_{\theta\epsilon}(\tau))
\Delta^\epsilon_\Psi u_{tX}(\tau)\\
&\pushright{+\Big[l_{xx}(\mathbb{Y}_{\theta\epsilon}(\tau),u_{tX}(\tau))
+D_{X}^{2}F(\mathbb{Y}_{\theta\epsilon}(\tau) \otimes  m)\Big](\Delta^\epsilon_\Psi\mathbb{Y}_{tX}(\tau)) d\theta d\tau,
\varphi
\bigg\rangle_{\mathcal{H}_m} \h{-6pt} }\\
=&\:\bigg\langle
\int^1_0 \Big[h_{xx}(\mathbb{Y}_{\theta\epsilon}(T) )
+D_{X}^{2}F_{T}\pig(\mathbb{Y}_{\theta\epsilon}(T) \otimes  m\pig)
-h_{xx}(\mathbb{Y}_{tX}(T) )
-D_{X}^{2}F_{T}\pig(\mathbb{Y}_{tX}(T) \otimes  m\pig)\Big](\Delta^\epsilon_\Psi \mathbb{Y}_{tX} (T))d\theta
,\varphi
\bigg\rangle_{\mathcal{H}_m}\h{-30pt}\\
&+\bigg\langle
\int^1_0 \Big[h_{xx}(\mathbb{Y}_{tX}(T) )
+D_{X}^{2}F_{T}\pig(\mathbb{Y}_{tX}(T) \otimes  m\pig)\Big](\Delta^\epsilon_\Psi \mathbb{Y}_{tX} (T))d\theta
,\varphi
\bigg\rangle_{\mathcal{H}_m}\\
&+\bigg\langle\int_s^{T} 
 \int^1_0 \Big[l_{xv}(\mathbb{Y}_{t,X+\epsilon\Psi}(\tau),u_{\theta\epsilon}(\tau))
-l_{xv}(\mathbb{Y}_{t,X}(\tau),u_{t,X}(\tau))\Big]
\Delta^\epsilon_\Psi u_{tX}(\tau) d\theta d\tau,
\varphi
\bigg\rangle_{\mathcal{H}_m} \h{-6pt} \\
&+\bigg\langle\int_s^{T} 
 \int^1_0 l_{xv}(\mathbb{Y}_{t,X }(\tau),u_{t,X }(\tau))
\Delta^\epsilon_\Psi u_{tX}(\tau) d\theta d\tau,
\varphi
\bigg\rangle_{\mathcal{H}_m} \h{-6pt} \\
&+\bigg\langle\int_s^{T} 
\int^1_0 \Big[l_{xx}(\mathbb{Y}_{\theta\epsilon}(\tau),u_{tX}(\tau))
-l_{xx}(\mathbb{Y}_{tX}(\tau),u_{tX}(\tau))
\\
&\h{150pt}+D_{X}^{2}F(\mathbb{Y}_{\theta\epsilon}(\tau) \otimes  m)
-D_{X}^{2}F(\mathbb{Y}_{tX}(\tau) \otimes  m)\Big](\Delta^\epsilon_\Psi\mathbb{Y}_{tX}(\tau)) d\theta d\tau,
\varphi
\bigg\rangle_{\mathcal{H}_m} \h{-6pt} \h{-30pt}\\
&+\bigg\langle \int_s^{T}\h{-3pt}
 \int^1_0 \Big[l_{xx}(\mathbb{Y}_{tX}(\tau),u_{tX}(\tau))
+D_{X}^{2}F(\mathbb{Y}_{tX}(\tau) \otimes  m)\Big]
(\Delta^\epsilon_\Psi\mathbb{Y}_{tX}(\tau)) d\theta d\tau,
\varphi
\bigg\rangle_{\mathcal{H}_m}.
\end{aligned}
\label{<DZ_e(t),Psi,j flow>2}
\end{equation}
We claim that the following terms
\begin{align}
&\scalemath{0.93}{\bigg\langle
\int^1_0 \Big[h_{xx}(\mathbb{Y}_{\theta\epsilon}(T) )
+D_{X}^{2}F_{T}\pig(\mathbb{Y}_{\theta\epsilon}(T) \otimes  m\pig)
-h_{xx}(\mathbb{Y}_{tX}(T) )
-D_{X}^{2}F_{T}\pig(\mathbb{Y}_{tX}(T) \otimes  m\pig)\Big](\Delta^\epsilon_\Psi \mathbb{Y}_{tX} (T))d\theta
,\varphi
\bigg\rangle_{\mathcal{H}_m}},
\h{-30pt}\label{(1) h_xx}\\
&\bigg\langle\int_s^{T} 
\int^1_0 \Big[l_{xv}(\mathbb{Y}_{t,X+\epsilon\Psi}(\tau),u_{\theta\epsilon}(\tau))
-l_{xv}(\mathbb{Y}_{t,X}(\tau),u_{t,X}(\tau))\Big]
\Delta^\epsilon_\Psi u_{tX}(\tau) d\theta d\tau,
\varphi
\bigg\rangle_{\mathcal{H}_m} ,\label{(2) l_xv}\\
&
\scalemath{0.83}{
\bigg\langle \displaystyle\int_s^{T}\h{-3pt}
 \int^1_0 \Big[l_{xx}(\mathbb{Y}_{\theta\epsilon}(\tau),u_{tX}(\tau))
-l_{xx}(\mathbb{Y}_{tX}(\tau),u_{tX}(\tau))
+D_{X}^{2}F(\mathbb{Y}_{\theta\epsilon}(\tau) \otimes  m)
-D_{X}^{2}F(\mathbb{Y}_{tX}(\tau) \otimes  m)\Big](\Delta^\epsilon_\Psi\mathbb{Y}_{tX}(\tau)) d\theta d\tau,
\varphi
\bigg\rangle_{\mathcal{H}_m},}\label{(3) l_xx}
\end{align}
converge to zero, as $\epsilon \to 0$ up to a subsequence. For instance, for (\ref{(1) h_xx}),
\mycomment{{
\small \begin{align*}
&\left|\bigg\langle
\int^1_0 \Big[h_{xx}(\mathbb{Y}_{\theta\epsilon}(T) )
+D_{X}^{2}F_{T}\pig(\mathbb{Y}_{\theta\epsilon}(T) \otimes  m\pig)
-h_{xx}(\mathbb{Y}_{tX}(T) )
-D_{X}^{2}F_{T}\pig(\mathbb{Y}_{tX}(T) \otimes  m\pig)\Big](\Delta^\epsilon_\Psi \mathbb{Y}_{tX} (T))d\theta
,\varphi
\bigg\rangle_{\mathcal{H}_m}\right|\\
&\leq\|\varphi\|_{L^\infty}\cdot
\int^1_0 \left\|h_{xx}(\mathbb{Y}_{\theta\epsilon}(T) )
-h_{xx}(\mathbb{Y}_{tX}(T) )
+D_{X}^{2}F_{T}\pig(\mathbb{Y}_{\theta\epsilon}(T) \otimes  m\pig)
-D_{X}^{2}F_{T}\pig(\mathbb{Y}_{tX}(T) \otimes  m\pig)\right\|_{\mathcal{H}_m}
\left\|\Delta^\epsilon_\Psi \mathbb{Y}_{tX} (T))\right\|_{\mathcal{H}_m}d\theta\\
&\leq\|\varphi\|_{L^\infty}
\cdot\left\|\Delta^\epsilon_\Psi \mathbb{Y}_{tX} (T))\right\|_{\mathcal{H}_m}\cdot
\h{-3pt}
\int^1_0 \pig\|h_{xx}(\mathbb{Y}_{\theta\epsilon}(T) )
-h_{xx}(\mathbb{Y}_{tX}(T))
\pigr\|_{\mathcal{H}_m}
+\left\|D_{X}^{2}F_{T}\pig(\mathbb{Y}_{\theta\epsilon}(T) \otimes  m\pig)
-D_{X}^{2}F_{T}\pig(\mathbb{Y}_{tX}(T) \otimes  m\pig)\right\|_{\mathcal{H}_m}
\h{-5pt}d\theta.
\end{align*}}}
\begin{align*}
&\left|\bigg\langle
\int^1_0 \Big[h_{xx}(\mathbb{Y}_{\theta\epsilon}(T) )
-h_{xx}(\mathbb{Y}_{tX}(T) )\Big](\Delta^\epsilon_\Psi \mathbb{Y}_{tX} (T))d\theta
,\varphi
\bigg\rangle_{\mathcal{H}_m}\right|\\
&\leq\|\varphi\|_{L^\infty}
\cdot\left\|\Delta^\epsilon_\Psi \mathbb{Y}_{tX} (T))\right\|_{\mathcal{H}_m}\cdot
\h{-3pt}
\int^1_0 \pig\|h_{xx}(\mathbb{Y}_{\theta\epsilon}(T) )
-h_{xx}(\mathbb{Y}_{tX}(T))
\pigr\|_{\mathcal{H}_m}d\theta.
\end{align*}
The definition of $\mathbb{Y}_{\theta\epsilon}(\tau)$ and the bounds in (\ref{bdd Y, Z, u, r, epsilon}) tell us that
\begin{equation}
\int^1_0\pig\|\mathbb{Y}_{\theta\epsilon}(T)
-\mathbb{Y}_{{tX}}(T) \pigr\|_{\mathcal{H}_m} d\theta
=   \int^1_0 \epsilon \theta
\pig\|\Delta^\epsilon_\Psi \mathbb{Y}_{{tX}}(T)\pigr\|_{\mathcal{H}_m} d\theta
\leq  \dfrac{\epsilon  C'_4}{2} \|\Psi\|_{\mathcal{H}_m} \longrightarrow 0 \h{10pt}
\text{as $\epsilon\to 0$.}
\label{strong conv. of Y theta u theta}
\end{equation}
The strong convergence in  (\ref{strong conv. of Y theta u theta}) and Borel-Cantelli lemma show that there is a subsequence of $\epsilon$ such that 
\begin{align}
\mathbb{Y}_{\theta\epsilon}(T) - \mathbb{Y}_{tX}(T)\longrightarrow 0, \h{10pt} \text{$m \otimes \mathbb{P}$-a.s., a.e. $\theta \in [0,1]$,  as $\epsilon \to 0$, and}
\label{pointwise conv. of Y}\\
\h{-20pt}\pig\|\mathbb{Y}_{\theta\epsilon}(T)
-\mathbb{Y}_{{tX}}(T) \pigr\|_{\mathcal{H}_m} \longrightarrow 0,
\h{10pt}
\text{for a.e. $\theta \in [0,1]$ as $\epsilon \to 0$.}
\label{strong conv. of Y}
\end{align} 
Along this subsequence, the convergence in  (\ref{pointwise conv. of Y}) and the continuity of $h_{xx}$ in Assumptions \textbf{A(iv)}'s (\ref{assumption, cts of lxx, lvx, lvv, hxx}) tell us that
\begin{equation}
h_{xx}(\mathbb{Y}_{\theta\epsilon}(T) )
-h_{xx}(\mathbb{Y}_{tX}(T) )\longrightarrow 0 \h{10pt}
\text{$m \otimes \mathbb{P}$-a.s., a.e. $\theta \in [0,1]$,  as $\epsilon \to 0$.}
\label{hxxY conv.point}
\end{equation} 
Since $h_{xx}$ is pointwisely bounded due to \textbf{A(iii)}'s (\ref{assumption, bdd of h, hx, hxx}), then $|h_{xx}(\mathbb{Y}_{\theta\epsilon}(T) )
-h_{xx}(\mathbb{Y}_{tX}(T) ) | \leq 2\displaystyle\sup_{x\in \mathbb{R}^n}|h_{xx}(x)| \leq 2c_h$, $m \otimes \mathbb{P}$-a.s., a.e. $\theta \in [0,1]$. Using the dominated convergence theorem and (\ref{hxxY conv.point}), we see that as $\epsilon \to 0$, it holds that 
\begin{equation}
\int^1_0
\left\|h_{xx}(\mathbb{Y}_{\theta\epsilon}(T) )
-h_{xx}(\mathbb{Y}_{tX}(T) )
\right\|_{\mathcal{H}_m}
d\theta \longrightarrow 0.
\label{conv. hxx theta 1}
\end{equation}
By the $\mathcal{H}_m$-boundedness of $\Delta^\epsilon_\Psi\mathbb{Y}_{tX}(\tau)$ in (\ref{bdd Y, Z, u, r, epsilon}) and the strong convergence in (\ref{strong conv. of Y}) for a.e. $\theta \in [0,1]$, Assumption \textbf{B(v)(a)}'s (\ref{assumption, properties of D^2F, D^2F_T}) implies
\begin{align*}
\mathbb{E}\left[\int_{\mathbb{R}^n}\left|D_{X}^{2}F_{T}\pig(\mathbb{Y}_{\theta\epsilon}(T) \otimes  m\pig)\big(\Delta^\epsilon_\Psi \mathbb{Y}_{tX} (T)\big)
-D_{X}^{2}F_{T}\pig(\mathbb{Y}_{tX}(T) \otimes  m\pig)\big(\Delta^\epsilon_\Psi \mathbb{Y}_{tX} (T)\big)\right|dm(x)\right] \longrightarrow 0,
\end{align*}
for a.e. $\theta \in [0,1]$ as $\epsilon \to 0$. By Assumption \textbf{B(ii)}'s (\ref{assumption, bdd of D^2F, D^2F_T}), Cauchy-Schwarz inequality and (\ref{bdd Y, Z, u, r, epsilon}), we also see that
\begin{align}
&\mathbb{E}\left[\int_{\mathbb{R}^n}\left|D_{X}^{2}F_{T}\pig(\mathbb{Y}_{\theta\epsilon}(T) \otimes  m\pig)\big(\Delta^\epsilon_\Psi \mathbb{Y}_{tX} (T)\big)
-D_{X}^{2}F_{T}\pig(\mathbb{Y}_{tX}(T) \otimes  m\pig)\big(\Delta^\epsilon_\Psi \mathbb{Y}_{tX} (T)\big)\right|dm(x)\right]\nonumber\\
&\leq \left\|D_{X}^{2}F_{T}\pig(\mathbb{Y}_{\theta\epsilon}(T) \otimes  m\pig)\big(\Delta^\epsilon_\Psi \mathbb{Y}_{tX} (T)\big)
-D_{X}^{2}F_{T}\pig(\mathbb{Y}_{tX}(T) \otimes  m\pig)\big(\Delta^\epsilon_\Psi \mathbb{Y}_{tX} (T)\big)\right\|_{\mathcal{H}_m}\nonumber\\
&\leq \left\|D_{X}^{2}F_{T}\pig(\mathbb{Y}_{\theta\epsilon}(T) \otimes  m\pig)\big(\Delta^\epsilon_\Psi \mathbb{Y}_{tX} (T)\big)\right\|_{\mathcal{H}_m}
+\left\|D_{X}^{2}F_{T}\pig(\mathbb{Y}_{tX}(T) \otimes  m\pig)\big(\Delta^\epsilon_\Psi \mathbb{Y}_{tX} (T)\big)\right\|_{\mathcal{H}_m}\nonumber\\
&\leq 2c_T\left\|\Delta^\epsilon_\Psi \mathbb{Y}_{tX} (T)\big)\right\|_{\mathcal{H}_m}\nonumber\\
&\leq 2c_T C_4'\|\Psi\|_{\mathcal{H}_m}
\label{bdd of diff of |D^2_X F_T|}
\end{align}
is uniformly bounded for a.e. $\theta\in [0,1]$. We apply the dominated convergence theorem to obtain that as $\epsilon \to 0$
\begin{equation}
\begin{aligned}
&\left|\bigg\langle
\int^1_0 \Big[
D_{X}^{2}F_{T}\pig(\mathbb{Y}_{\theta\epsilon}(T) \otimes  m\pig)
-D_{X}^{2}F_{T}\pig(\mathbb{Y}_{tX}(T) \otimes  m\pig)\Big](\Delta^\epsilon_\Psi \mathbb{Y}_{tX} (T))d\theta
,\varphi
\bigg\rangle_{\mathcal{H}_m}\right|\label{DF conv.}\\
&\leq\|\varphi\|_{L^\infty}\cdot
\int^1_0 \mathbb{E}\left[\int_{\mathbb{R}^n}\left|D_{X}^{2}F_{T}\pig(\mathbb{Y}_{\theta\epsilon}(T) \otimes  m\pig)\big(\Delta^\epsilon_\Psi \mathbb{Y}_{tX} (T)\big)
-D_{X}^{2}F_{T}\pig(\mathbb{Y}_{tX}(T) \otimes  m\pig)\big(\Delta^\epsilon_\Psi \mathbb{Y}_{tX} (T)\big)\right|dm(x)\right]d\theta\\
&\longrightarrow 0.
\end{aligned}
\end{equation}
Combining (\ref{conv. hxx theta 1}) and (\ref{DF conv.}), we see that (\ref{(1) h_xx}) converges to zero as $\epsilon \to 0$.
Therefore, in a similar manner, (\ref{(2) l_xv})-(\ref{(3) l_xx}) also converge to zero as $\epsilon \to 0$. Hence, by the weak convergences of $\Delta^\epsilon_\Psi\mathbb{Y}_{tX}(\tau)$, $\Delta^\epsilon_\Psi u_{tX}(\tau)$, the weak convergence in (\ref{forward, app, at T}) and the convergences of (\ref{(1) h_xx})-(\ref{(3) l_xx}), (\ref{<DZ_e(t),Psi,j flow>2}) converges, as $\epsilon \to 0$ along a subsequence, 
\begin{equation}
\begingroup
\allowdisplaybreaks
\begin{aligned}
&\int^T_{t}\bigg\langle
\int^1_0 \Big[h_{xx}(\mathbb{Y}_{\theta\epsilon}(T) )
+D_{X}^{2}F_{T}\pig(\mathbb{Y}_{\theta\epsilon}(T) \otimes  m\pig)\Big](\Delta^\epsilon_\Psi \mathbb{Y}_{tX} (T))d\theta
,\varphi
\bigg\rangle_{\mathcal{H}_m}ds\\
&\h{10pt}+\int_{t}^{T}\int_{s}^{T}\h{-3pt}
\bigg\langle \int^1_0 l_{xv}(\mathbb{Y}_{t,X+\epsilon\Psi}(\tau),u_{\theta\epsilon}(\tau))
\Delta^\epsilon_\Psi u_{tX}(\tau)\\
&\h{100pt}+\Big[l_{xx}(\mathbb{Y}_{\theta\epsilon}(\tau),u_{tX}(\tau))
+D_{X}^{2}F(\mathbb{Y}_{\theta\epsilon}(\tau) \otimes  m)\Big](\Delta^\epsilon_\Psi\mathbb{Y}_{tX}(\tau)) d\theta,
\varphi
\bigg\rangle_{\mathcal{H}_m} \h{-6pt} d\tau ds\\
&\h{-5pt}\longrightarrow 
\int^T_{t}\bigg\langle
 \Big[h_{xx}(\mathbb{Y}_{tX}(T) )
+D_{X}^{2}F_{T}\pig(\mathbb{Y}_{tX}(T) \otimes  m\pig)\Big](\mathscr{D} \mathbb{Y}^\Psi_{tX} (T))
,\varphi
\bigg\rangle_{\mathcal{H}_m}ds\\
&\h{20pt}+\int^T_{t}\int_{s}^{T}\h{-3pt}
\bigg\langle l_{xv}(\mathbb{Y}_{t,X }(\tau),u_{t,X }(\tau))
\mathscr{D} u^\Psi_{tX}(\tau) \\
&\h{100pt} +\Big[l_{xx}(\mathbb{Y}_{tX}(\tau),u_{tX}(\tau))
+D_{X}^{2}F(\mathbb{Y}_{tX}(\tau) \otimes  m)\Big]
(\mathscr{D} \mathbb{Y}^\Psi_{tX}(\tau)) ,
\varphi
\bigg\rangle_{\mathcal{H}_m} \h{-6pt} d\tau ds,
\end{aligned}
\endgroup
\label{<DZ_e(t),Psi,j flow>3}
\end{equation}
since the remaining terms in (\ref{<DZ_e(t),Psi,j flow>2}) only involve linear operators acting on $\Delta^\epsilon_\Psi \mathbb{Y}_{tX}$ and $\Delta^\epsilon_\Psi u_{tX}$. The backward equation of (\ref{finite diff. backward}) and the convergence of (\ref{<DZ_e(t),Psi,j flow>3}) show that
\begin{align}
\mathscr{D} \mathbb{Z}^\Psi_{tX} (s)
=\:&
h_{xx}\pig(\mathbb{Y}_{t X}(T)\pig)\mathscr{D} \mathbb{Y}^\Psi_{tX}(T)
+D_{X}^{2}F_{T}\pig((\mathbb{Y}_{t X}(T)) \otimes  m\pig)(\mathscr{D} \mathbb{Y}^\Psi_{tX} (T))\nonumber\\
&+\int^T_s l_{xx}\pig(\mathbb{Y}_{t X}(\tau), u_{t X}(\tau)\pig)
\mathscr{D} \mathbb{Y}^\Psi_{tX} (\tau)
+l_{xv}\pig(\mathbb{Y}_{t,X}(\tau),u_{t X}(\tau)\pig)\mathscr{D} u^\Psi_{tX}(\tau) d\tau\h{-1pt}\nonumber\\
&+\int^T_s D_{X}^{2}F\pig((\mathbb{Y}_{t X}(\tau) \otimes  m\pig)(\mathscr{D} \mathbb{Y}^\Psi_{tX} (\tau)) d\tau
-\int^T_s\sum_{j=1}^n \mathscr{D} \mathbbm{r}^\Psi_{t X,j}(\tau)dw_{j}(\tau).
\label{backward, app}
\end{align}
Similarly, we also note that  as $\epsilon \to 0$ along a subsequence, 
\begin{equation}
\begin{aligned}
\Delta^\epsilon_\Psi \mathbb{Z}_{tX} (t)
=\:&
\displaystyle\int_{0}^{1}h_{xx}\pig(\mathbb{Y}_{t X}(T)+\theta\epsilon \Delta^\epsilon_\Psi \mathbb{Y}_{tX} (T)\pig)\Delta^\epsilon_\Psi \mathbb{Y}_{tX} (T)\nonumber\\ &\h{100pt}+D_{X}^{2}F_{T}\pig((\mathbb{Y}_{t X}(T)+\theta\epsilon \Delta^\epsilon_\Psi \mathbb{Y}_{tX} (T)) \otimes  m\pig)(\Delta^\epsilon_\Psi \mathbb{Y}_{tX} (T))d\theta\nonumber\\
&+\int^T_{t}\int_{0}^{1}l_{xx}\pig(\mathbb{Y}_{t X}(\tau)+\theta\epsilon \Delta^\epsilon_\Psi \mathbb{Y}_{tX} (\tau),u_{t X}(\tau) \pig)\Delta^\epsilon_\Psi \mathbb{Y}_{tX} (\tau)\nonumber\\
&\pushright{+l_{xv}\pig(\mathbb{Y}_{t,X+\epsilon\Psi}(\tau),u_{t X}(\tau)+\theta\epsilon \Delta^\epsilon_\Psi u_{tX}(\tau)\pig)\Delta^\epsilon_\Psi u_{tX}(\tau)d\theta d\tau}\h{-1pt}\nonumber\\
&+\int^T_{t}\int_{0}^{1}D_{X}^{2}F\pig((\mathbb{Y}_{t X}(\tau)+\theta\epsilon \Delta^\epsilon_\Psi \mathbb{Y}_{tX} (\tau)) \otimes  m\pig)(\Delta^\epsilon_\Psi \mathbb{Y}_{tX} (\tau))d\theta d\tau
-\int^T_{t}\sum_{j=1}^n \Delta^\epsilon_\Psi \mathbbm{r}_{t X,j}(\tau)dw_{j}(\tau)\\
\longrightarrow\:& h_{xx}\pig(\mathbb{Y}_{t X}(T)\pig)\mathscr{D} \mathbb{Y}^\Psi_{tX}(T)
+D_{X}^{2}F_{T}\pig((\mathbb{Y}_{t X}(T)) \otimes  m\pig)(\mathscr{D} \mathbb{Y}^\Psi_{tX} (T))\nonumber\\
&+\int^T_{t} l_{xx}\pig(\mathbb{Y}_{t X}(\tau), u_{t X}(\tau)\pig)
\mathscr{D} \mathbb{Y}^\Psi_{tX} (\tau)
+l_{xv}\pig(\mathbb{Y}_{t,X}(\tau),u_{t X}(\tau)\pig)\mathscr{D} u^\Psi_{tX}(\tau) d\tau\h{-1pt}\nonumber\\
&+\int^T_{t} D_{X}^{2}F\pig((\mathbb{Y}_{t X}(\tau) \otimes  m\pig)(\mathscr{D} \mathbb{Y}^\Psi_{tX} (\tau)) d\tau
-\int^T_s\sum_{j=1}^n \mathscr{D} \mathbbm{r}^\Psi_{t X,j}(\tau)dw_{j}(\tau)\\
=\:&\mathscr{D} \mathbb{Z}^\Psi_{tX} (t),\h{20pt}\text{weakly in $L^2_{\mathcal{W}_{t X\Psi}}(\mathcal{H}_m)$.}
\end{aligned}
\label{Delta Z tau_i weak conv}
\end{equation}
Subject to the first order condition (\ref{1st order condition}), we see that (\ref{J flow of FBSDE}) is a linear system and thus its solution can be shown to be unique, where this claim will be proven below. Taking this for granted, we see that the solution of (\ref{forward, app}) and (\ref{backward, app}) is unique and equal to the weak limit of $\pig( \Delta^\epsilon_\Psi \mathbb{Y}_{tX} (s),
\Delta^\epsilon_\Psi \mathbb{Z}_{tX} (s),
\Delta^\epsilon_\Psi\mathbbm{r}_{tX}(s)\pig)$, with $u_{tX}(s)$ satisfying the first order condition (\ref{1st order condition}). As a consequence, it implies that the weak limit of $\pig( \Delta^\epsilon_\Psi \mathbb{Y}_{tX} (s),
\Delta^\epsilon_\Psi \mathbb{Z}_{tX} (s),
\Delta^\epsilon_\Psi\mathbbm{r}_{tX}(s)\pig)$ along any subsequence as $\epsilon \to 0$ is the same, if it exists. Therefore, we can now identify  $\pig( \mathscr{D} \mathbb{Y}^\Psi_{tX} (s),
\mathscr{D} \mathbb{Z}^\Psi_{tX} (s),
\mathscr{D} u^\Psi_{tX}(s),
\mathscr{D}\mathbbm{r}^\Psi_{tX}(s)\pig)$ with $\pig( D \mathbb{Y}^\Psi_{tX} (s),
D \mathbb{Z}^\Psi_{tX} (s),
D u^\Psi_{tX}(s),
D\mathbbm{r}^\Psi_{tX}(s)\pig)$ which is the Jacobian flow indeed.

\noindent {\bf Uniqueness of (\ref{J flow of FBSDE})\footnote{It is well-known that the global uniqueness of general linear FBSDE is not true even for the case with constant coefficients, for instance, it can be verified by the system $dX_t=Y_t$ and $dY_t=-X_t$ with $X_0=Y_{\pi/2}=0$, we see that $(X_t,Y_t)=(a\sin t,a\cos t)$ can be a solution for any $a\in \mathbb{R}$. The uniqueness claim of (\ref{J flow of FBSDE}) heavily on certain convexity structures of the FBSDE, namely (\ref{uniqueness condition}).}:}\\
If we have two sets of solution with the same initial condition $\Psi$ to the system (\ref{J flow of FBSDE}), then it is enough to show that the differences between them are zero; also note that these differences would satisfy the following first order condition (\ref{1st order finite diff. uni}). Since the system (\ref{J flow of FBSDE}) is linear, it is sufficient  to show that any solution $\pig( D \mathbb{Y}^* (s),
D\mathbb{Z}^* (s),
D u^*(s),
D\mathbbm{r}^*(s)\pig)$ to the system with a zero initial condition, that is,
\begin{equation}
\scalemath{0.93}{
\h{-10pt}\left\{
\begin{aligned}
D \mathbb{Y}^*  (s)
&= 
\displaystyle\int_{t}^{s}
D u^*(\tau)d\tau;\\
D \mathbb{Z}^* (s)
&=h_{xx}(\mathbb{Y}_{tX}(T))D \mathbb{Y}^*(T)
+D_{X}^2F_{T}(\mathbb{Y}_{tX}(T)  \otimes  m)(D \mathbb{Y}^* (T))\\
&\h{10pt}+\displaystyle\int^T_s\bigg\{l_{xx}(\mathbb{Y}_{tX}(\tau),u(\tau))D \mathbb{Y}^*(\tau)
+l_{xv}(\mathbb{Y}_{tX}(\tau),u(\tau))D u^*(\tau)
+D^2_{X}F(\mathbb{Y}_{tX}(\tau)  \otimes  m)\pig(
D\mathbb{Y}^* (\tau)\pig)\bigg\}d\tau \h{-10pt}\\
&\h{10pt}-\int^T_s\sum_{j=1}^{n}D\mathbbm{r}^*_{j}(\tau)dw_{j}(\tau),
\end{aligned}\right.}
\label{J flow of FBSDE, app, zero ini}
\end{equation}
must be vanished. From the first order condition in (\ref{1st order condition}), we obtain that
\begin{equation}
l_{vx}\pig(\mathbb{Y}_{t X}(s),u_{t X}(s)\pig)
D\mathbb{Y}^* (s)
+l_{vv}\pig(\mathbb{Y}_{tX}(s),u_{t X}(s)\pig)
D u^* (s)+D \mathbb{Z}^*  (s)=0.
\label{1st order finite diff. uni}
\end{equation}
We consider the inner product 
$\pig\langle D\mathbb{Y}^*(s), D\mathbb{Z}^*(s) \pigr\rangle_{\mathcal{H}_m}$, together with (\ref{1st order finite diff. uni}), we obtain the following equation
$$
\scalemath{0.94}{
\h{-20pt}\begin{aligned}
&\h{-10pt}\bigg\langle
h_{xx}(\mathbb{Y}_{tX}(T) )
D \mathbb{Y}^* (T)
+D_{X}^{2}F_{T}\pig(\mathbb{Y}_{tX}(T) \otimes  m\pig)(D \mathbb{Y}^* (T))
,D \mathbb{Y}^*(T)
\bigg\rangle_{\mathcal{H}_m}\\
=\:&
-\int_{t}^{T}
\Big\langle 
l_{vx}\pig(\mathbb{Y}_{t X}(\tau),u_{t X}(\tau)\pig)
D\mathbb{Y}^* (\tau)
+l_{vv}\pig(\mathbb{Y}_{tX}(\tau),u_{t X}(\tau)\pig)D u^*(\tau),
D u^*(\tau)\Big\rangle_{\mathcal{H}_m} d\tau\\
&-\int_{t}^{T}\h{-3pt}
\bigg\langle   l_{xv}(\mathbb{Y}_{tX}(\tau),u_{tX}(\tau))
D u^*(\tau)
+l_{xx}(\mathbb{Y}_{tX}(\tau),u_{tX}(\tau))D\mathbb{Y}^*(\tau)
+D_{X}^{2}F(\mathbb{Y}_{tX}(\tau) \otimes  m)(D\mathbb{Y}^*(\tau)),
D \mathbb{Y}^*(\tau)
\bigg\rangle_{\mathcal{H}_m} \h{-6pt} d\tau.\h{-40pt}
\end{aligned}}
$$
Assumptions \textbf{A(v)}'s (\ref{assumption, convexity of l}), \textbf{A(vi)}'s (\ref{assumption, convexity of h}), \textbf{B(v)(b)}'s (\ref{assumption, convexity of D^2F, D^2F_T}) imply
\begin{equation}
\begin{aligned}
\int_{t}^{T}
\lambda\big\|D  u^* (s) \big\|_{\mathcal{H}_m}^{2}
-\left(c'_{l}+c'\right) 
\big\|D \mathbb{Y}^* (s)\big\|_{\mathcal{H}_m}^{2}ds
-(c'_{h}+c'_{T})\big\|D \mathbb{Y}^* (T)\big\|_{\mathcal{H}_m}^{2}\leq 0.
\end{aligned}
\label{est Jk > 5}
\end{equation}
The equation of $D \mathbb{Y}^* (s)$ in (\ref{J flow of FBSDE, app, zero ini}) with a simple application of Cauchy-Schwarz inequality gives
\begin{equation}
\big\|D \mathbb{Y}^* (s)\big\|_{\mathcal{H}_m}^{2} \leq s\int^T_{t}
\big\|D  u^* (\tau) \big\|_{\mathcal{H}_m}^{2}d\tau \h{10pt} \text{ and }\h{10pt} 
\int^T_{t}\big\|D \mathbb{Y}^* (\tau)\big\|_{\mathcal{H}_m}^{2} 
d\tau
\leq \dfrac{T^2}{2}\int^T_{t}
\big\|D u^*(\tau) \big\|_{\mathcal{H}_m}^{2}d\tau.
\label{4657}
\end{equation}
Putting (\ref{4657}) into (\ref{est Jk > 5}), we have
$$\left[\lambda
-(c'_{h}+c'_{T})_+T
-\left(c'_{l}+c'\right)_+\dfrac{T^2}{2}
\right]\int_{t}^{T}
\big\|D  u^* (s) \big\|_{\mathcal{H}_m}^{2}ds
\leq 0.$$
The condition in (\ref{uniqueness condition}) implies $\int^T_{t}
\big\|D u^*(s) \big\|_{\mathcal{H}_m}^{2}ds=0$ which further deduces $Du^*(s)=0$, $m \otimes \mathbb{P}$-a.s. for a.e. $s \in [t,T]$. Therefore, since the processes $D\mathbb{Y}^*(s)$ and $D\mathbb{Z}^*(s)$ are continuous in time $s$, we easily see that $D\mathbb{Y}^*(s)=D\mathbb{Z}^*(s)=0$, $m \otimes \mathbb{P}$-a.s. for all $s \in [t,T]$, and $D\mathbbm{r}^*(s)=0$, $m \otimes \mathbb{P}$-a.s. for a.e. $s \in [t,T]$, from (\ref{J flow of FBSDE, app, zero ini}) and (\ref{1st order finite diff. uni}).

\noindent {\bf Step 2. Strong Convergence:}\\
Now we next prove that the finite difference process $\pig( \Delta^\epsilon_\Psi \mathbb{Y}_{tX} (s),
\Delta^\epsilon_\Psi \mathbb{Z}_{tX} (s),
\Delta^\epsilon_\Psi u_{tX}(s),
\Delta^\epsilon_\Psi\mathbbm{r}_{tX}(s)\pig)$ converges to $\pig( D \mathbb{Y}^\Psi_{tX} (s),
D \mathbb{Z}^\Psi_{tX} (s),
D u^\Psi_{tX}(s),
D\mathbbm{r}^\Psi_{tX}(s)\pig)$ strongly in $L^2_{\mathcal{W}_{t X\Psi}}(t,T;\mathcal{H}_m)$ as $\epsilon \to 0$. For if not the case, there is a sequence $\{\epsilon_k\}_{k\in\mathbb{N}}$ such that $\epsilon_k \longrightarrow 0$ and, without loss of generality, we assume 
\begin{equation}
\lim_{k \to \infty} \sup_{s\in [t,T]} \pig\| \Delta^{\epsilon_k}_\Psi \mathbb{Y}_{tX} (s)- D \mathbb{Y}^\Psi_{tX} (s) \pigr\|_{\mathcal{H}_m}>0.
\label{not conv DY}
\end{equation}
According to Step 1 in this proof, we can extract a subsequence $\{\epsilon_{k_j}\}_{j\in \mathbb{N}}$ from $\{\epsilon_{k}\}_{k\in \mathbb{N}}$ such that $\pig( \Delta^{\epsilon_{k_j}}_\Psi \mathbb{Y}_{tX} (s)$,\\
$\Delta^{\epsilon_{k_j}}_\Psi \mathbb{Z}_{tX} (s),
\Delta^{\epsilon_{k_j}}_\Psi u_{tX}(s),
\Delta^{\epsilon_{k_j}}_\Psi\mathbbm{r}_{tX}(s)\pig)$ converges weakly to $\pig( D \mathbb{Y}^\Psi_{tX} (s),
D \mathbb{Z}^\Psi_{tX} (s),
D u^\Psi_{tX}(s),
D\mathbbm{r}^\Psi_{tX}(s)\pig)$ as $j \to 0$, which is the unique solution of (\ref{J flow of FBSDE}) with $D u^\Psi_{tX}(s)$ satisfying (\ref{Du = DY+DZ}). For simplicity, we write $\epsilon_j$ in place of $\epsilon_{k_j}$ to avoid cumbersome notations without affecting the main arguments in the rest of our proof.

By noting that $\Delta^{\epsilon_j}_\Psi\mathbb{Y}_{tX}(s)$ is a finite variation process, we apply the traditional It\^o lemma to\\ $\pig\langle \Delta^{\epsilon_j}_\Psi\mathbb{Y}_{tX}(s), \Delta^{\epsilon_j}_\Psi\mathbb{Z}_{tX}(s) \pigr\rangle_{\mathcal{H}_m}$ to give the equation
$$
\begin{aligned}
\pig\langle \Delta^{\epsilon_j}_\Psi \mathbb{Z}_{tX}(t),\Psi  \pigr\rangle_{\mathcal{H}_m}
=\:&
\bigg\langle
\int^1_0 h_{xx}(\mathbb{Y}_{\theta{\epsilon_j}}(T) )\Delta^{\epsilon_j}_\Psi \mathbb{Y}_{tX} (T)
+D_{X}^{2}F_{T}\pig(\mathbb{Y}_{\theta{\epsilon_j}}(T) \otimes  m\pig)(\Delta^{\epsilon_j}_\Psi \mathbb{Y}_{tX} (T))d\theta
,\Delta^{\epsilon_j}_\Psi \mathbb{Y}_{tX}(T)
\bigg\rangle_{\mathcal{H}_m}\\
&-\int_{t}^{T}
\Big\langle 
\Delta^{\epsilon_j}_\Psi \mathbb{Z}_{tX}(\tau),
\Delta^{\epsilon_j}_\Psi u_{tX}(\tau)\Big\rangle_{\mathcal{H}_m} d\tau\\
&+\int_{t}^{T}\h{-3pt}
\bigg\langle \int^1_0 l_{xv}(\mathbb{Y}_{t,X+{\epsilon_j}\Psi}(\tau),u_{\theta{\epsilon_j}}(\tau))
\Delta^{\epsilon_j}_\Psi u_{tX}(\tau)
+l_{xx}(\mathbb{Y}_{\theta{\epsilon_j}}(\tau),u_{tX}(\tau))\Delta^{\epsilon_j}_\Psi\mathbb{Y}_{tX}(\tau)\\
&\pushright{+D_{X}^{2}F(\mathbb{Y}_{\theta{\epsilon_j}}(\tau) \otimes  m)(\Delta^{\epsilon_j}_\Psi\mathbb{Y}_{tX}(\tau)) d\theta,
\Delta^{\epsilon_j}_\Psi \mathbb{Y}_{tX}(\tau)
\bigg\rangle_{\mathcal{H}_m} \h{-6pt} d\tau,\h{-13pt}}
\end{aligned}
$$
together with the first order condition in (\ref{1st order finite diff.}),
we further have
\begin{equation}
\scalemath{0.94}{
\h{-10pt}\begin{aligned}
&\h{-13pt}\pig\langle \Delta^{\epsilon_j}_\Psi \mathbb{Z}_{tX}(t),\Psi  \pigr\rangle_{\mathcal{H}_m}\\
=\:&
\bigg\langle
\int^1_0 h_{xx}(\mathbb{Y}_{\theta{\epsilon_j}}(T) )
\Delta^{\epsilon_j}_\Psi \mathbb{Y}_{tX} (T)
+D_{X}^{2}F_{T}\pig(\mathbb{Y}_{\theta{\epsilon_j}}(T) \otimes  m\pig)(\Delta^{\epsilon_j}_\Psi \mathbb{Y}_{tX} (T))d\theta
,\Delta^{\epsilon_j}_\Psi \mathbb{Y}_{tX}(T)
\bigg\rangle_{\mathcal{H}_m}\\
&+\int_{t}^{T}
\bigg\langle 
\int_{0}^{1}l_{vx}\pig(\mathbb{Y}_{\theta{\epsilon_j}} (\tau),u_{t X}(\tau)\pig)
\Delta^{\epsilon_j}_\Psi \mathbb{Y}_{tX} (\tau)
+l_{vv}\pig(\mathbb{Y}_{t,X+{\epsilon_j}\Psi}(\tau) ,u_{\theta{\epsilon_j}}(\tau)\pig)
\Delta^{\epsilon_j}_\Psi u_{tX}(\tau)d\theta,
\Delta^{\epsilon_j}_\Psi u_{tX}(\tau)\bigg\rangle_{\mathcal{H}_m} d\tau\\
&+\int_{t}^{T}\h{-3pt}
\bigg\langle \int^1_0 l_{xv}(\mathbb{Y}_{t,X+{\epsilon_j}\Psi}(\tau),u_{\theta{\epsilon_j}}(\tau))
\Delta^{\epsilon_j}_\Psi u_{tX}(\tau)
+l_{xx}(\mathbb{Y}_{\theta{\epsilon_j}}(\tau),u_{tX}(\tau))\Delta^{\epsilon_j}_\Psi\mathbb{Y}_{tX}(\tau)\\
&\pushright{+D_{X}^{2}F(\mathbb{Y}_{\theta{\epsilon_j}}(\tau) \otimes  m)(\Delta^{\epsilon_j}_\Psi\mathbb{Y}_{tX}(\tau)) d\theta,
\Delta^{\epsilon_j}_\Psi \mathbb{Y}_{tX}(\tau)
\bigg\rangle_{\mathcal{H}_m} \h{-6pt} d\tau}.
\end{aligned}}
\label{<DZ_e(t),Psi,j flow>}
\end{equation}
By defining the linear operators $Q_{1{\epsilon_j}}:=h_{xx}(\mathbb{Y}_{\theta{\epsilon_j}}(T) )
+D_{X}^{2}F_{T}\pig(\mathbb{Y}_{\theta{\epsilon_j}}(T) \otimes  m\pig)$, 
$Q_{2{\epsilon_j}}:=l_{vx}\pig(\mathbb{Y}_{\theta{\epsilon_j}} (\tau),u_{t X}(\tau)\pig)$, 
$Q_{3{\epsilon_j}}:=l_{vv}(\mathbb{Y}_{t,X+{\epsilon_j}\Psi}(\tau),u_{\theta{\epsilon_j}}(\tau))$,
$Q_{4{\epsilon_j}}:=l_{xv}(\mathbb{Y}_{t,X+{\epsilon_j}\Psi}(\tau),u_{\theta{\epsilon_j}}(\tau))$,\\
$Q_{5{\epsilon_j}}:=l_{xx}(\mathbb{Y}_{\theta{\epsilon_j}}(\tau),u_{tX}(\tau))+D_{X}^{2}F(\mathbb{Y}_{\theta{\epsilon_j}}(\tau) \otimes  m)$, we can also write, by using the backward equation of Jacobian flow of (\ref{backward, app}) and It\^o's lemma directly, 
\begin{equation}
\begin{aligned}
\pig\langle D\mathbb{Z}_{tX}^\Psi(t), \Psi \pigr\rangle_{\mathcal{H}_m}
=\:&
\Big\langle
Q_{1}^*(D\mathbb{Y}^\Psi_{tX} (T)),
D\mathbb{Y}^\Psi_{tX}(T)
\Big\rangle_{\mathcal{H}_m} 
+\int_{t}^{T}
\Big\langle
Q_{2}^*(D\mathbb{Y}^\Psi_{tX}(\tau)),
Du^\Psi_{tX}(\tau)\Big\rangle_{\mathcal{H}_m} d \tau \\
&+\int_{t}^{T}
\Big\langle
Q_{3}^*(Du^\Psi_{tX}(\tau)),
Du^\Psi_{tX}(\tau)\Big\rangle_{\mathcal{H}_m} d \tau 
+\int_{t}^{T}
\Big\langle Q_{4}^*
(Du^\Psi_{tX} (\tau)),
D\mathbb{Y}^\Psi_{tX} (\tau)
\Big\rangle_{\mathcal{H}_m} d\tau\\
&+\int_{t}^{T}\Big\langle Q_{5}^*(D\mathbb{Y}^\Psi_{tX}(\tau)),
D\mathbb{Y}^\Psi_{tX} (\tau)
\Big\rangle_{\mathcal{H}_m} d\tau,\h{-40pt}
\end{aligned}
\label{<DZ_infty(t),Psi,j flow>}
\end{equation}
where, for example, $Q_{1}^*= h_{xx}(\mathbb{Y}_{t X}(T))
+D_{X}^2F_{T}(\mathbb{Y}_{tX}(T)  \otimes  m)$, by noting the strong convergence of (\ref{strong conv. of Y theta u theta}), and the continuities of $h_{xx}$ and $D^2_X F_T$ under Assumptions \textbf{A(iv)}'s (\ref{assumption, cts of lxx, lvx, lvv, hxx}) and \textbf{B(iv)}'s (\ref{assumption, cts of D^2F_T}) respectively. 

\noindent {\bf Step 2A. Estimate of $\pig\langle \Delta^{\epsilon_j}_\Psi  \mathbb{Z}_{tX}(t),\Psi  \pigr\rangle_{\mathcal{H}_m} 
- \pig\langle D\mathbb{Z}^\Psi_{tX}(t),\Psi \pigr\rangle_{\mathcal{H}_m}$:}\\
Referring to (\ref{<DZ_e(t),Psi,j flow>}) and (\ref{<DZ_infty(t),Psi,j flow>}), we want to study the limit of $\pig\langle \Delta^{\epsilon_j}_\Psi  \mathbb{Z}_{tX}(t),\Psi  \pigr\rangle_{\mathcal{H}_m} 
- \pig\langle D\mathbb{Z}^\Psi_{tX}(t),\Psi \pigr\rangle_{\mathcal{H}_m}$ by first checking the term
\begin{equation}
\h{-5pt}
\begin{aligned}
&\h{-10pt}\int^1_0\Big\langle
Q_{1{\epsilon_j}}(\Delta^{\epsilon_j}_\Psi  \mathbb{Y}_{tX} (T)),
\Delta^{\epsilon_j}_\Psi  \mathbb{Y}_{tX} (T)
 \Big\rangle_{\mathcal{H}_m} 
 -\Big\langle
Q_{1}^*(D\mathbb{Y}^\Psi_{t X} (T)),
D\mathbb{Y}^\Psi_{tX}(T)
 \Big\rangle_{\mathcal{H}_m} d\theta \\
=\:&\int^1_0\Big\langle
Q_{1{\epsilon_j}}\pig(\Delta^{\epsilon_j}_\Psi  \mathbb{Y}_{tX} (T)-D\mathbb{Y}^\Psi_{tX} (T)\pig)
 ,\Delta^{\epsilon_j}_\Psi  \mathbb{Y}_{tX} (T)-D\mathbb{Y}^\Psi_{tX} (T)
 \Big\rangle_{\mathcal{H}_m} \\
&+\Big\langle
Q_{1{\epsilon_j}}(D\mathbb{Y}^\Psi_{tX} (T)),\Delta^{\epsilon_j}_\Psi  \mathbb{Y}_{tX}(T)
-D\mathbb{Y}^\Psi_{tX}(T)
\Big\rangle_{\mathcal{H}_m}
\h{-5pt}+\Big\langle
Q_{1{\epsilon_j}}(\Delta^{\epsilon_j}_\Psi  \mathbb{Y}_{tX} (T))
-Q_{1}^*(D\mathbb{Y}^\Psi_{tX} (T)),D\mathbb{Y}^\Psi_{tX}(T)
\Big\rangle_{\mathcal{H}_m}d\theta.
\h{-50pt}
\end{aligned}
\label{Q1e-Q1infty}
\end{equation}
The first term in the third line of (\ref{Q1e-Q1infty}) reads
\begin{align}
&\int^1_0\Big\langle
Q_{1{\epsilon_j}}(D\mathbb{Y}^\Psi_{tX} (T)),\Delta^{\epsilon_j}_\Psi  \mathbb{Y}_{tX}(T)
-D\mathbb{Y}^\Psi_{tX}(T)
\Big\rangle_{\mathcal{H}_m}d\theta\label{4813}\\
&=\int^1_0 \h{-5pt} \Big\langle
(Q_{1{\epsilon_j}}-Q_{1}^*)(D\mathbb{Y}^\Psi_{tX} (T)),\Delta^{\epsilon_j}_\Psi  \mathbb{Y}_{tX}(T)
-D\mathbb{Y}^\Psi_{tX}(T)
\Big\rangle_{\mathcal{H}_m}
\h{-10pt}+\Big\langle
Q_{1}^*(D\mathbb{Y}^\Psi_{tX} (T)),\Delta^{\epsilon_j}_\Psi  \mathbb{Y}_{tX}(T)
-D\mathbb{Y}^\Psi_{tX}(T)
\Big\rangle_{\mathcal{H}_m}\h{-5pt}d\theta.
\nonumber
\end{align}
The second term in the second line of (\ref{4813}) converges to zero since $\Delta^{\epsilon_j}_\Psi  \mathbb{Y}_{tX}(T)$ weakly converges to $D\mathbb{Y}^\Psi_{tX}(T)$ by (\ref{forward, app, at T}). The first term in the second line of (\ref{4813}) can be estimated by
\begin{align*}
&\int^1_0\left|\Big\langle
(Q_{1{\epsilon_j}}-Q_{1}^*)(D\mathbb{Y}^\Psi_{tX} (T)),\Delta^{\epsilon_j}_\Psi  \mathbb{Y}_{tX}(T)
-D\mathbb{Y}^\Psi_{tX}(T)
\Big\rangle_{\mathcal{H}_m}\right|d\theta\nonumber\\
&\leq \int^1_0\left|\Big\langle
\pig[h_{xx}(\mathbb{Y}_{\theta{\epsilon_j}}(T))
-h_{xx}(\mathbb{Y}_{t X}(T)) \pig]
(D\mathbb{Y}^\Psi_{tX} (T)),\Delta^{\epsilon_j}_\Psi  \mathbb{Y}_{tX}(T)
-D\mathbb{Y}^\Psi_{tX}(T)
\Big\rangle_{\mathcal{H}_m}\right|\nonumber\\
&\h{20pt}+\left|\Big\langle
\pig[D_{X}^{2}F_{T}\pig(\mathbb{Y}_{\theta{\epsilon_j}}(T) \otimes  m\pig)
-D_{X}^{2}F_{T}\pig(\mathbb{Y}_{t X}(T) \otimes  m\pig)\pig](D\mathbb{Y}^\Psi_{tX} (T)),\Delta^{\epsilon_j}_\Psi  \mathbb{Y}_{tX}(T)
-D\mathbb{Y}^\Psi_{tX}(T)
\Big\rangle_{\mathcal{H}_m}\right|d\theta\nonumber\\
&\leq \int^1_0\left\|
\pig[h_{xx}(\mathbb{Y}_{\theta{\epsilon_j}}(T))
-h_{xx}(\mathbb{Y}_{t X}(T)) \pig]
(D\mathbb{Y}^\Psi_{tX} (T))\right\|_{\mathcal{H}_m}
\left\|\Delta^{\epsilon_j}_\Psi  \mathbb{Y}_{tX}(T)
-D\mathbb{Y}^\Psi_{tX}(T)
\right\|_{\mathcal{H}_m}\\
&\h{20pt}+\left\|
\pig[D_{X}^{2}F_{T}\pig(\mathbb{Y}_{\theta{\epsilon_j}}(T) \otimes  m\pig)
-D_{X}^{2}F_{T}\pig(\mathbb{Y}_{t X}(T) \otimes  m\pig)\pig](D\mathbb{Y}^\Psi_{tX} (T))\right\|_{\mathcal{H}_m}
\left\|\Delta^{\epsilon_j}_\Psi  \mathbb{Y}_{tX}(T)
-D\mathbb{Y}^\Psi_{tX}(T)
\right\|_{\mathcal{H}_m}d\theta.
\end{align*}
By (\ref{bdd Y, Z, u, r, epsilon}), in light of the nature of linear operator norm, Fatou's lemma and the weak convergence of $\Delta^{\epsilon_j}_\Psi  \mathbb{Y}_{tX}(T)$ by (\ref{forward, app, at T}), we see that 
\begin{equation}
\left\|D\mathbb{Y}^\Psi_{tX}(T)\right\|_{\mathcal{H}_m}\leq C'_4\|\Psi\|_{\mathcal{H}_m}.
\label{|DY| bdd}
\end{equation}
Thus \begin{align}
&\int^1_0\left|\Big\langle
(Q_{1{\epsilon_j}}-Q_{1}^*)(D\mathbb{Y}^\Psi_{tX} (T)),\Delta^{\epsilon_j}_\Psi  \mathbb{Y}_{tX}(T)
-D\mathbb{Y}^\Psi_{tX}(T)
\Big\rangle_{\mathcal{H}_m}\right|d\theta\nonumber\\
&\leq 2C'_4 \|\Psi\|_{\mathcal{H}_m}
\int^1_0\left\|
\pig[h_{xx}(\mathbb{Y}_{\theta{\epsilon_j}}(T))
-h_{xx}(\mathbb{Y}_{t X}(T)) \pig]
(D\mathbb{Y}^\Psi_{tX} (T))\right\|_{\mathcal{H}_m}
d\theta\label{4843}\\
&\h{10pt}+2C'_4 \|\Psi\|_{\mathcal{H}_m}
\int^1_0\left\|
\pig[D_{X}^{2}F_{T}\pig(\mathbb{Y}_{\theta{\epsilon_j}}(T) \otimes  m\pig)
-D_{X}^{2}F_{T}\pig(\mathbb{Y}_{t X}(T) \otimes  m\pig)\pig](D\mathbb{Y}^\Psi_{tX} (T))\right\|_{\mathcal{H}_m}
d\theta.
\label{4849}
\end{align}
Since $|h_{xx}|$ is also bounded due to Assumptions \textbf{A(iii)}'s (\ref{assumption, bdd of h, hx, hxx}), then by (\ref{hxxY conv.point}), (\ref{|DY| bdd}) and dominated convergence theorem, we see that (\ref{4843}) converges to zero as $\epsilon_j \to 0$ up to a subsequence. Due to the strong convergence of $\mathbb{Y}_{\theta\epsilon}(T)$ in (\ref{strong conv. of Y}), the boundedness in (\ref{|DY| bdd}) and (\ref{bdd of diff of |D^2_X F_T|}), an application of Assumptions \textbf{B(v)(a)}'s (\ref{assumption, properties of D^2F, D^2F_T}) shows that 
$$\int^1_0\mathbb{E}\left[\int_{\mathbb{R}^n}\Big|\pig[D_{X}^{2}F_{T}\pig(\mathbb{Y}_{\theta\epsilon}(T) \otimes  m\pig)
-D_{X}^{2}F_{T}\pig(\mathbb{Y}_{t X}(T) \otimes  m\pig)\pig](D\mathbb{Y}^\Psi_{tX} (T))\Big|dm(x)\right]d\theta \longrightarrow 0,$$ which further implies that there is a subsequence of $\epsilon_j$ such that $$\pig[D_{X}^{2}F_{T}\pig(\mathbb{Y}_{\theta\epsilon}(T) \otimes  m\pig)
-D_{X}^{2}F_{T}\pig(\mathbb{Y}_{t X}(T) \otimes  m\pig)\pig](D\mathbb{Y}^\Psi_{tX} (T)) \longrightarrow 0,$$  $m \otimes \mathbb{P}$-a.s. for a.e. $\theta \in [0,1]$ as $\epsilon_j \to 0$. Therefore, by the boundedness of $D^2_XF_T$ in Assumption \textbf{B(ii)}'s (\ref{assumption, bdd of D^2F, D^2F_T}), (\ref{|DY| bdd}) and the dominated convergence theorem, we see that (\ref{4849}) also converges to $0$ up to a subsequence. It concludes that (\ref{4813}) converges to $0$ as $\epsilon_j \to 0$ up to a subsequence. The second term in the third line of (\ref{Q1e-Q1infty}) can be handled similarly to obtain its convergence to 0 up to a subsequence of $\epsilon_j$. Therefore, it yields that
\begin{align*}
\int^1_0\Big\langle
Q_{1\epsilon_j}(\Delta^{\epsilon_j}_\Psi  \mathbb{Y}_{tX} (T)),
\Delta^{\epsilon_j}_\Psi  \mathbb{Y}_{tX} (T)
 \Big\rangle_{\mathcal{H}_m} 
&-\Big\langle
Q_{1}^*(D\mathbb{Y}^\Psi_{t X} (T)),
D\mathbb{Y}^\Psi_{tX}(T)
 \Big\rangle_{\mathcal{H}_m} \\
 &-\Big\langle
Q_{1\epsilon_j}\pig(\Delta^{\epsilon_j}_\Psi  \mathbb{Y}_{tX} (T)-D\mathbb{Y}^\Psi_{tX} (T)\pig)
 ,\Delta^{\epsilon_j}_\Psi  \mathbb{Y}_{tX} (T)-D\mathbb{Y}^\Psi_{tX} (T)
 \Big\rangle_{\mathcal{H}_m} d\theta
\end{align*}
tends to $0$ as $\epsilon_j \to 0$ up to a subsequence. By estimating the remaining terms in $\pig\langle \Delta^{\epsilon_j}_\Psi  \mathbb{Z}_{tX}(t),\Psi  \pigr\rangle_{\mathcal{H}_m} 
- \pig\langle D\mathbb{Z}^\Psi_{tX}(t),\Psi \pigr\rangle_{\mathcal{H}_m}$ involving $Q_{2\epsilon_j},\ldots,Q_{5\epsilon_j}$ similarly, we can deduce that as $\epsilon_j \to 0$  up to a subsequence,
\begin{equation}
 \pig\langle \Delta^{\epsilon_j}_\Psi  \mathbb{Z}_{tX}(t),\Psi  \pigr\rangle_{\mathcal{H}_m} 
- \pig\langle D\mathbb{Z}^\Psi_{tX}(t),\Psi \pigr\rangle_{\mathcal{H}_m}
-\mathscr{J}^\dagger_{\epsilon_j}\longrightarrow 0,
\label{Delta Z - DZ - J to 0}
\end{equation}
where
\begingroup
\allowdisplaybreaks
\begin{align*}
\mathscr{J}^\dagger_{\epsilon_j}:=\int^1_0&\Big\langle Q_{1\epsilon_j}\pig(\Delta^{\epsilon_j}_\Psi  \mathbb{Y}_{tX} (T)-D\mathbb{Y}^\Psi_{tX} (T)\pig),
\Delta^{\epsilon_j}_\Psi  \mathbb{Y}_{tX} (T)-D\mathbb{Y}^\Psi_{tX} (T)
\Big\rangle_{\mathcal{H}_m}\\
&+\int_{t}^{T}
\Big\langle Q_{2\epsilon_j}
\pig(\Delta^{\epsilon_j}_\Psi  \mathbb{Y}_{tX} (s)-D\mathbb{Y}^\Psi_{tX} (s)\pig),
\Delta^{\epsilon_j}_\Psi  u_{tX} (s)-Du^\Psi_{tX} (s)
\Big\rangle_{\mathcal{H}_m}\\
&\h{20pt}+\Big\langle Q_{3\epsilon_j}
\pig(\Delta^{\epsilon_j}_\Psi  u_{tX} (s)-Du^\Psi_{tX} (s)\pig),
\Delta^{\epsilon_j}_\Psi  u_{tX} (s)-Du^\Psi_{tX} (s)\Big\rangle_{\mathcal{H}_m}\\
&\h{20pt}+\Big\langle
Q_{4\epsilon_j}
\pig(\Delta^{\epsilon_j}_\Psi  u_{tX} (s)-Du^\Psi_{tX} (s)\pig),
\Delta^{\epsilon_j}_\Psi  \mathbb{Y}_{tX} (s)-D\mathbb{Y}^\Psi_{tX} (s)\Big\rangle_{\mathcal{H}_m}\\
&\h{20pt}+\Big\langle Q_{5\epsilon_j}
\pig(\Delta^{\epsilon_j}_\Psi  \mathbb{Y}_{tX} (s)-D\mathbb{Y}^\Psi_{tX} (s)\pig),
\Delta^{\epsilon_j}_\Psi  \mathbb{Y}_{tX} (s)-D\mathbb{Y}^\Psi_{tX} (s)\Big\rangle_{\mathcal{H}_m}dsd\theta\\
=\int^1_0&\Big\langle \pig[h_{xx}(\mathbb{Y}_{\theta\epsilon_j}(T))
+D_{X}^{2}F_{T}(\mathbb{Y}_{\theta\epsilon_j}(T) \otimes  m)\pig]\pig(\Delta^{\epsilon_j}_\Psi  \mathbb{Y}_{tX} (T)-D\mathbb{Y}^\Psi_{tX} (T)\pig),
\Delta^{\epsilon_j}_\Psi  \mathbb{Y}_{tX} (T)-D\mathbb{Y}^\Psi_{tX} (T)
\Big\rangle_{\mathcal{H}_m}\\
&+\int_{t}^{T}
\Big\langle l_{vx}(\mathbb{Y}_{\theta\epsilon_j}(s),u_{tX}(s))
\pig(\Delta^{\epsilon_j}_\Psi  \mathbb{Y}_{tX} (s)-D\mathbb{Y}^\Psi_{tX} (s)\pig),
\Delta^{\epsilon_j}_\Psi  u_{tX} (s)-Du^\Psi_{tX} (s)
\Big\rangle_{\mathcal{H}_m}\\
&\h{20pt}+\Big\langle l_{vv}(\mathbb{Y}_{t,X+\epsilon_j \Psi}(s),u_{\theta\epsilon_j}(s))
\pig(\Delta^{\epsilon_j}_\Psi  u_{tX} (s)-Du^\Psi_{tX} (s)\pig),
\Delta^{\epsilon_j}_\Psi  u_{tX} (s)-Du^\Psi_{tX} (s)\Big\rangle_{\mathcal{H}_m}\\
&\h{20pt}+\Big\langle
l_{xv}(\mathbb{Y}_{t,X+\epsilon_j \Psi}(s),u_{\theta\epsilon_j}(s))
\pig(\Delta^{\epsilon_j}_\Psi  u_{tX} (s)-Du^\Psi_{tX} (s)\pig),
\Delta^{\epsilon_j}_\Psi  \mathbb{Y}_{tX} (s)-D\mathbb{Y}^\Psi_{tX} (s)\Big\rangle_{\mathcal{H}_m}\\
&\h{20pt}+\Big\langle l_{xx}\pig(\mathbb{Y}_{\theta\epsilon_j}(s),u_{tX}(s)\pig)
\pig(\Delta^{\epsilon_j}_\Psi  \mathbb{Y}_{tX} (s)-D\mathbb{Y}^\Psi_{tX} (s)\pig),
\Delta^{\epsilon_j}_\Psi  \mathbb{Y}_{tX} (s)-D\mathbb{Y}^\Psi_{tX} (s)\Big\rangle_{\mathcal{H}_m}\\
&\h{20pt}+\Big\langle D_{X}^{2}F(\mathbb{Y}_{\theta\epsilon_j}(s) \otimes  m)\pig(\Delta^{\epsilon_j}_\Psi  \mathbb{Y}_{tX} (s)-D\mathbb{Y}^\Psi_{tX} (s)\pig),
\Delta^{\epsilon_j}_\Psi  \mathbb{Y}_{tX} (s)-D\mathbb{Y}^\Psi_{tX} (s)\Big\rangle_{\mathcal{H}_m}dsd\theta\\
=\int^1_0&\Big\langle \pig[h_{xx}(\mathbb{Y}_{\theta\epsilon_j}(T))
+D_{X}^{2}F_{T}(\mathbb{Y}_{\theta\epsilon_j}(T) \otimes  m)\pig]\pig(\Delta^{\epsilon_j}_\Psi  \mathbb{Y}_{tX} (T)-D\mathbb{Y}^\Psi_{tX} (T)\pig),
\Delta^{\epsilon_j}_\Psi  \mathbb{Y}_{tX} (T)-D\mathbb{Y}^\Psi_{tX} (T)
\Big\rangle_{\mathcal{H}_m}\\
&+\int_{t}^{T}
\Big\langle l_{vx}(\mathbb{Y}_{\theta\epsilon_j}(s),u_{tX}(s))
\pig(\Delta^{\epsilon_j}_\Psi  \mathbb{Y}_{tX} (s)-D\mathbb{Y}^\Psi_{tX} (s)\pig),
\Delta^{\epsilon_j}_\Psi  u_{tX} (s)-Du^\Psi_{tX} (s)
\Big\rangle_{\mathcal{H}_m}\\
&\h{20pt}+\Big\langle l_{vv}(\mathbb{Y}_{\theta\epsilon_j}(s),u_{tX}(s))
\pig(\Delta^{\epsilon_j}_\Psi  u_{tX} (s)-Du^\Psi_{tX} (s)\pig),
\Delta^{\epsilon_j}_\Psi  u_{tX} (s)-Du^\Psi_{tX} (s)\Big\rangle_{\mathcal{H}_m}\\
&\h{20pt}+\Big\langle
l_{xv}(\mathbb{Y}_{\theta\epsilon_j}(s),u_{tX}(s))
\pig(\Delta^{\epsilon_j}_\Psi  u_{tX} (s)-Du^\Psi_{tX} (s)\pig),
\Delta^{\epsilon_j}_\Psi  \mathbb{Y}_{tX} (s)-D\mathbb{Y}^\Psi_{tX} (s)\Big\rangle_{\mathcal{H}_m}\\
&\h{20pt}+\Big\langle l_{xx}(\mathbb{Y}_{\theta\epsilon_j}(s),u_{tX}(s))
\pig(\Delta^{\epsilon_j}_\Psi  \mathbb{Y}_{tX} (s)-D\mathbb{Y}^\Psi_{tX} (s)\pig),
\Delta^{\epsilon_j}_\Psi  \mathbb{Y}_{tX} (s)-D\mathbb{Y}^\Psi_{tX} (s)\Big\rangle_{\mathcal{H}_m}\\
&\h{20pt}+\Big\langle D_{X}^{2}F(\mathbb{Y}_{\theta\epsilon_j}(s) \otimes  m)\pig(\Delta^{\epsilon_j}_\Psi  \mathbb{Y}_{tX} (s)-D\mathbb{Y}^\Psi_{tX} (s)\pig),
\Delta^{\epsilon_j}_\Psi  \mathbb{Y}_{tX} (s)-D\mathbb{Y}^\Psi_{tX} (s)\Big\rangle_{\mathcal{H}_m}dsd\theta+\mathscr{R}_{\epsilon_j},
\end{align*}
\endgroup

{\small
\begin{align*}
&\h{-10pt}\mathscr{R}_{\epsilon_j}=\\
\int^1_0&\int_{t}^{T}
\Big\langle \pig[l_{vv}(\mathbb{Y}_{t,X+\epsilon_j \Psi}(s),u_{\theta\epsilon_j}(s))
-l_{vv}(\mathbb{Y}_{\theta\epsilon_j}(s),u_{tX}(s))\pig]
\pig(\Delta^{\epsilon_j}_\Psi  u_{tX} (s)-Du^\Psi_{tX} (s)\pig),
\Delta^{\epsilon_j}_\Psi  u_{tX} (s)-Du^\Psi_{tX} (s)\Big\rangle_{\mathcal{H}_m}\\
&\h{20pt}+\Big\langle
\pig[l_{xv}(\mathbb{Y}_{t,X+\epsilon_j \Psi}(s),u_{\theta\epsilon_j}(s))
-l_{xv}(\mathbb{Y}_{\theta\epsilon_j}(s),u_{tX}(s))\pig]
\pig(\Delta^{\epsilon_j}_\Psi  u_{tX} (s)-Du^\Psi_{tX} (s)\pig),
\Delta^{\epsilon_j}_\Psi  \mathbb{Y}_{tX} (s)-D\mathbb{Y}^\Psi_{tX} (s)\Big\rangle_{\mathcal{H}_m}dsd\theta.
\end{align*}}
\noindent {\bf Step 2B. Estimate of $\mathscr{J}^\dagger_{\epsilon_j}$ and Conclusion:}\\
Together with the convergences in (\ref{Delta Z - DZ - J to 0}) and (\ref{Delta Z tau_i weak conv}), we have $\mathscr{J}^\dagger_{\epsilon_j} \longrightarrow 0$ as $\epsilon_j \to 0$ up to a subsequence. The convexity conditions of $h_{xx}$, $D_X^2F_T$, $D_X^2F$ and the second-order derivatives of $l$ in Assumptions \textbf{A(v)}'s (\ref{assumption, convexity of l}), \textbf{A(vi)}'s (\ref{assumption, convexity of h}) and \textbf{B(v)(b)}'s (\ref{assumption, convexity of D^2F, D^2F_T}) imply that
\begin{equation}
\begin{aligned}
\mathscr{J}^\dagger_{\epsilon_j}\geq\:&\int_{t}^{T}
\lambda\big\|\Delta^{\epsilon_j}_\Psi  u_{tX} (s)-Du^\Psi_{tX} (s)\big\|_{\mathcal{H}_m}^{2}
-\left(c'_{l}+c' \right) 
\big\|\Delta^{\epsilon_j}_\Psi  \mathbb{Y}_{tX} (s)-D\mathbb{Y}^\Psi_{tX} (s)\big\|_{\mathcal{H}_m}^{2}ds\\
&-(c'_{h}+c'_{T})\big\|\Delta^{\epsilon_j}_\Psi  \mathbb{Y}_{tX} (T)-D\mathbb{Y}^\Psi_{tX} (T)\big\|_{\mathcal{H}_m}^{2} + \mathscr{R}_{\epsilon_j}.
\end{aligned}
\label{est Jk > 1}
\end{equation}
We first estimate the term $\mathscr{R}_{\epsilon_j}$,
{\small
\begingroup
\allowdisplaybreaks
\begin{align*}
&\h{-10pt}|\mathscr{R}_{\epsilon_j}|\leq\\
\int^1_0&\int_{t}^{T}
\Big\| \pig[l_{vv}(\mathbb{Y}_{t,X+\epsilon_j \Psi}(s),u_{\theta\epsilon_j}(s))
-l_{vv}(\mathbb{Y}_{\theta\epsilon_j}(s),u_{tX}(s))\pig]
\pig(\Delta^{\epsilon_j}_\Psi  u_{tX} (s)-Du^\Psi_{tX} (s)\pig)
\Big\|_{\mathcal{H}_m}
\Big\|\Delta^{\epsilon_j}_\Psi  u_{tX} (s)-Du^\Psi_{tX} (s)\Big\|_{\mathcal{H}_m}\\
&\h{15pt}+\Big\|
\pig[l_{xv}(\mathbb{Y}_{t,X+\epsilon_j \Psi}(s),u_{\theta\epsilon_j}(s))
-l_{xv}(\mathbb{Y}_{\theta\epsilon_j}(s),u_{tX}(s))\pig]
\pig(\Delta^{\epsilon_j}_\Psi  u_{tX} (s)-Du^\Psi_{tX} (s)\pig)\Big\|_{\mathcal{H}_m}
\Big\|\Delta^{\epsilon_j}_\Psi  \mathbb{Y}_{tX} (s)-D\mathbb{Y}^\Psi_{tX} (s)\Big\|_{\mathcal{H}_m}dsd\theta\\
&\leq 2c_l\int_{t}^{T}
\Big\|
\pig(\Delta^{\epsilon_j}_\Psi  u_{tX} (s)-Du^\Psi_{tX} (s)\pig)
\Big\|_{\mathcal{H}_m}^2
+\Big\|
\pig(\Delta^{\epsilon_j}_\Psi  u_{tX} (s)-Du^\Psi_{tX} (s)\pig)\Big\|_{\mathcal{H}_m}
\Big\|\Delta^{\epsilon_j}_\Psi  \mathbb{Y}_{tX} (s)-D\mathbb{Y}^\Psi_{tX} (s)\Big\|_{\mathcal{H}_m}ds,
\end{align*}
\endgroup}
where we have used Assumption \textbf{A(ii)}'s (\ref{assumption, bdd of lxx, lvx, lvv}). It shows that $\mathscr{R}_{\epsilon_j}$ is integrable by the weak convergences of $\Delta^{\epsilon_j}_\Psi  u_{tX} (s)$ and $\Delta^{\epsilon_j}_\Psi  \mathbb{Y}_{tX} (s)$, as well as \eqref{bdd Y, Z, u, r, epsilon}. Moreover, the definition of $\mathbb{Y}_{\theta\epsilon_j}(s)$ and \eqref{bdd Y, Z, u, r, epsilon} show that
\begin{align*}
\int^1_0 \int^T_t \pig\|
\mathbb{Y}_{t,X+\epsilon_j \Psi}(s)-\mathbb{Y}_{\theta\epsilon_j}(s)
\pigr\|_{\mathcal{H}_m}^2dsd\theta 
= \epsilon^2\int^1_0 (1-\theta)^2\int^T_t \pig\|
\Delta^{\epsilon_j}_\Psi  \mathbb{Y}_{tX} (s)
\pigr\|_{\mathcal{H}_m}^2dsd\theta \longrightarrow 0 \h{10pt} \text{ as $j \to \infty$.}
\end{align*}
Borel-Cantelli lemma shows that there is a subsequence of $\epsilon_j$, still denoted by $\epsilon_j$, such that 
\begin{align}
\mathbb{Y}_{t,X+\epsilon_j \Psi}(s)-\mathbb{Y}_{\theta\epsilon_j}(s)\longrightarrow 0, \h{10pt} \text{$m \otimes \mathbb{P}$-a.s., a.e. $\theta \in [0,1]$, a.e. $s \in [t,T]$,  as $\epsilon_j \to 0$.}
\end{align} 
The continuity of the second order derivatives of $l$ also imply
\begin{align*}
l_{vv}(\mathbb{Y}_{t,X+\epsilon_j \Psi}(s),u_{\theta\epsilon_j}(s))
-l_{vv}(\mathbb{Y}_{\theta\epsilon_j}(s),u_{tX}(s))\longrightarrow 0, \h{10pt} \text{$m \otimes \mathbb{P}$-a.s., a.e. $\theta \in [0,1]$, a.e. $s \in [t,T]$,  as $\epsilon_j \to 0$}
\end{align*} 
 and
\begin{align*}
l_{xv}(\mathbb{Y}_{t,X+\epsilon_j \Psi}(s),u_{\theta\epsilon_j}(s))
-l_{xv}(\mathbb{Y}_{\theta\epsilon_j}(s),u_{tX}(s))\longrightarrow 0, \h{10pt} \text{$m \otimes \mathbb{P}$-a.s., a.e. $\theta \in [0,1]$, a.e. $s \in [t,T]$,  as $\epsilon_j \to 0$.}
\end{align*} 
Therefore, Lebesgue dominated convergence yields that $\mathscr{R}_{\epsilon_j} \longrightarrow 0$ as $j \to \infty$. With an application of Cauchy-Schwarz inequality, the equation of $\Delta^{\epsilon_j}_\Psi  \mathbb{Y}_{tX} (s)-D\mathbb{Y}^\Psi_{tX} (s)$ imply
\begin{equation}
\begin{aligned}
&\big\|\Delta^{\epsilon_j}_\Psi  \mathbb{Y}_{tX} (s)-D\mathbb{Y}^\Psi_{tX} (s)\big\|_{\mathcal{H}_m}^{2} \leq (s-t)\int^s_{t}
\big\|\Delta^{\epsilon_j}_\Psi  u_{tX} (\tau)-Du^\Psi_{tX} (\tau)\big\|_{\mathcal{H}_m}^{2}d\tau \h{10pt} \text{ and }\\
&\int^T_{t}\big\|\Delta^{\epsilon_j}_\Psi  \mathbb{Y}_{tX} (\tau)-D\mathbb{Y}^\Psi_{tX} (\tau)\big\|_{\mathcal{H}_m}^{2} 
d\tau
\leq \dfrac{T^2}{2}\int^T_{t}
\big\|\Delta^{\epsilon_j}_\Psi  u_{tX} (\tau)-Du^\Psi_{tX} (\tau)\big\|_{\mathcal{H}_m}^{2}d\tau.
\end{aligned}
\label{4514}
\end{equation}
Putting (\ref{4514}) into (\ref{est Jk > 1}), together with the assumption in (\ref{uniqueness condition}) and the fact that $\mathscr{R}_{\epsilon_j} \to 0$, we see that as $\epsilon_j \to 0$ up to a subsequence,
\begin{equation}
\int_{t}^{T}
\big\|\Delta^{\epsilon_j}_\Psi  u_{tX} (s)-Du^\Psi_{tX} (s)\big\|_{\mathcal{H}_m}^{2}ds \longrightarrow 0.
\label{4520}
\end{equation}
We also bring (\ref{4520}) into (\ref{4514}) to yield that as $\epsilon_j \to 0$ up to a subsequence,
\begin{equation}
\big\|\Delta^{\epsilon_j}_\Psi  \mathbb{Y}_{tX} (s)-D\mathbb{Y}^\Psi_{tX} (s)\big\|_{\mathcal{H}_m}^{2} \longrightarrow 0 \h{10pt} \text{uniformly for all $s\in [t,T]$}.
\label{4526}
\end{equation}
It contradicts (\ref{not conv DY}), therefore the strong convergence of $\Delta^{\epsilon_j}_\Psi  \mathbb{Y}_{tX} (s)$ should follow. 

The strong convergences of 
$\Delta^{\epsilon_j}_\Psi  \mathbb{Z}_{tX} (s)$ and $\Delta^{\epsilon_j}_\Psi  \mathbbm{r}_{tX} (s)$ are concluded by subtracting the equation of $\Delta^{\epsilon_j}_\Psi  \mathbb{Z}_{tX} (s)$ in (\ref{finite diff. backward}) from the equation of $D  \mathbb{Z}^\Psi_{tX} (s)$ in (\ref{J flow of FBSDE}) and then using It\^o's lemma, together with the convergences in (\ref{4520}) and (\ref{4526}). Finally, the strong convergence of $\Delta^{\epsilon_j}_\Psi  u_{tX} (s)$ is deduced by the first order condition in (\ref{1st order condition}) and the strong convergences of $\Delta^{\epsilon_j}_\Psi  \mathbb{Y}_{tX} (s)$, $\Delta^{\epsilon_j}_\Psi  \mathbb{Z}_{tX} (s)$ and $\Delta^{\epsilon_j}_\Psi \mathbbm{r}_{tX} (s)$ just obtained.

\hfill$\blacksquare$

\subsubsection{Proof of Lemma \ref{lem, Existence of Frechet derivatives}}\label{app, Existence of Frechet derivatives}

\noindent\textbf{Part 1. Linearity in $\Psi$:}\\
Given $X$, $\Psi_1$, $\Psi_2 \in L^2_{\mathcal{W}^{\indep}_{t}} (\mathcal{H}_m) \subset \mathcal{H}_{m}$, by summing up the equations in (\ref{J flow of FBSDE}) with $\Psi=\Psi_1$ and $\Psi=\Psi_2$, we have

\begin{equation}
\scalemath{0.99}{
\left\{
\begin{aligned}
D \mathbb{Y}_{tX}^{\Psi_1}  (s)
+D \mathbb{Y}_{tX}^{\Psi_2}  (s)
=\:& \Psi_1+\Psi_2
+\displaystyle\int_{t}^{s}\Big[ \diff_y  u(\mathbb{Y}_{tX},\mathbb{Z}_{tX})(\tau)\Big] 
\Big[D \mathbb{Y}_{tX}^{\Psi_1}  (\tau) 
+D \mathbb{Y}_{tX}^{\Psi_2}  (\tau)\Big]d\tau\\
&\h{37pt}+\int_{t}^{s} 
\Big[\diff_z  u (\mathbb{Y}_{tX},\mathbb{Z}_{tX}) (\tau)\Big] 
\Big[D \mathbb{Z}_{tX}^{\Psi_1}  (\tau)
+D \mathbb{Z}_{tX}^{\Psi_2}  (\tau)\Big]  d\tau;\\
D \mathbb{Z}_{tX}^{\Psi_1} (s)
+D \mathbb{Z}_{tX}^{\Psi_2} (s)
=\:&h_{xx}(\mathbb{Y}_{tX}(T)) \pig[D\mathbb{Y}_{tX}^{\Psi_1} (T)
+D\mathbb{Y}_{tX}^{\Psi_2} (T)\pig]
+D_{X}^2F_{T}(\mathbb{Y}_{tX}(T)  \otimes  m)\pig(D\mathbb{Y}_{tX}^{\Psi_1} (T)
+D\mathbb{Y}_{tX}^{\Psi_2} (T)\pig)\h{-40pt}\\
&+\displaystyle\int^T_s\bigg\{l_{xx}(\mathbb{Y}_{tX}(\tau),u(\tau))
\pig[D \mathbb{Y}_{tX}^{\Psi_1} (\tau)
+D \mathbb{Y}_{tX}^{\Psi_2} (\tau)\pig]\\
&\h{40pt}+l_{xv}(\mathbb{Y}_{tX}(\tau),u(\tau))\Big[ \diff_y  u(\mathbb{Y}_{tX},\mathbb{Z}_{tX})(\tau)\Big] 
\pig[D \mathbb{Y}_{tX}^{\Psi_1} (\tau)
+D \mathbb{Y}_{tX}^{\Psi_2} (\tau)\pig]\\
&\h{40pt}+l_{xv}(\mathbb{Y}_{tX}(\tau),u(\tau))
\Big[ \diff_z  u (\mathbb{Y}_{tX},\mathbb{Z}_{tX})(\tau)\Big] 
\pig[D \mathbb{Z}_{tX}^{\Psi_1} (\tau)
+D \mathbb{Z}_{tX}^{\Psi_2} (\tau)\pig]\\
&\h{40pt}+D^2_{X}F(\mathbb{Y}_{tX}(\tau)  \otimes  m)\pig(
D\mathbb{Y}_{tX}^{\Psi_1} (\tau)
+D\mathbb{Y}_{tX}^{\Psi_2} (\tau)\pig)\bigg\}d\tau \h{-10pt}\\
&-\int^T_s\sum_{j=1}^{n}D\mathbbm{r}^{\Psi_1}_{tX,j}(\tau)
+D\mathbbm{r}^{\Psi_2}_{tX,j}(\tau)dw_{j}(\tau).
\end{aligned}\right.}
\label{J flow 1+2}
\end{equation}
By comparing the equations in (\ref{J flow of FBSDE}) under $\Psi=\Psi_1+\Psi_2$ with the equations in (\ref{J flow 1+2}), the uniqueness of solution stated in Lemma \ref{lem, Existence of J flow} shows that 
$$(D \mathbb{Y}_{tX}^{\Psi_1+\Psi_2}(s),
D \mathbb{Z}_{tX}^{\Psi_1+\Psi_2}(s),
D u_{tX}^{\Psi_1+\Psi_2}(s))
=(D \mathbb{Y}_{tX}^{\Psi_1}(s)
+D \mathbb{Y}_{tX}^{\Psi_2}(s),
D \mathbb{Z}_{tX}^{\Psi_1}(s)
+D \mathbb{Z}_{tX}^{\Psi_2}(s),
D u_{tX}^{\Psi_1}(s)
+D u_{tX}^{\Psi_2}(s))$$
for any $s \in [t,T]$ and $D \mathbbm{r}_{tX,j}^{\Psi_1+\Psi_2}(s)=D \mathbbm{r}_{tX,j}^{\Psi_1}(s)+D \mathbbm{r}_{tX,j}^{\Psi_2}(s)$ for a.e. $s\in [t,T]$. The homogeneity property is given directly by the definition of G\^ateaux derivative, for instance,
\begin{align*}
0&=\lim_{\epsilon \to 0} 
\left\|
\dfrac{\mathbb{Y}_{t,X+\epsilon c\Psi}(s)-\mathbb{Y}_{tX}(s)}{\epsilon}
-D \mathbb{Y}_{tX}^{c\Psi}(s) \right\|_{\mathcal{H}_m}
=c\lim_{\epsilon' \to 0} 
\left\|
\dfrac{\mathbb{Y}_{t,X+\epsilon '\Psi}(s)-\mathbb{Y}_{tX}(s)}{\epsilon'}
-\dfrac{1}{c} D \mathbb{Y}_{tX}^{c\Psi}(s) \right\|_{\mathcal{H}_m}\\
&=c\lim_{\epsilon' \to 0} 
\left\|
\dfrac{\mathbb{Y}_{t,X+\epsilon '\Psi}(s)-\mathbb{Y}_{tX}(s)}{\epsilon'}
-D \mathbb{Y}_{tX}^{\Psi}(s)
+D \mathbb{Y}_{tX}^{\Psi}(s)
-\dfrac{1}{c} D \mathbb{Y}_{tX}^{c\Psi}(s) \right\|_{\mathcal{H}_m},
\end{align*}
for any non-zero constant $c \in \mathbb{R}$. Thus $D \mathbb{Y}_{tX}^{c\Psi}(s)=cD \mathbb{Y}_{tX}^{\Psi}(s)$.

\noindent\textbf{Part 2. Partial Continuity in $X$:}\\
For $X$, $\Psi \in L^2_{\mathcal{W}^{\indep}_{t}} (\mathcal{H}_m) \subset \mathcal{H}_{m}$, we consider a sequence $\{X_k\}_{k \in \mathbb{N}} \subset L^2_{\mathcal{W}^{\indep}_{t}} (\mathcal{H}_m)$ such that $X_k \longrightarrow X$ in $\mathcal{H}_m$. Applying It\^o's lemma to the inner product $\pig\langle D \mathbb{Z}_{tX_k}^{\Psi}(s)
-D \mathbb{Z}_{tX}^{\Psi}(s),
D \mathbb{Y}_{tX_k}^{\Psi}(s)
-D \mathbb{Y}_{tX}^{\Psi}(s)
\pigr\rangle_{\mathcal{H}_m}$, together with the first order condition in (\ref{1st order condition}), we have
\begingroup
\allowdisplaybreaks
\begin{equation}
\begin{aligned}
&\h{-11pt}\pig\langle h_{xx}(\mathbb{Y}_{tX_k}(T))D \mathbb{Y}_{tX_k}^{\Psi} (T)
-h_{xx}(\mathbb{Y}_{tX}(T))D \mathbb{Y}_{tX}^{\Psi} (T),
D \mathbb{Y}_{tX_k}^{\Psi}(T)
-D \mathbb{Y}_{tX}^{\Psi}(T)
\pigr\rangle_{\mathcal{H}_m}\\
&+\pig\langle D_{X}^2F_{T}(\mathbb{Y}_{tX_k}(T)  \otimes  m)\pig(D\mathbb{Y}_{tX_k}^{\Psi} (T)\pig)
-D_{X}^2F_{T}(\mathbb{Y}_{tX}(T)  \otimes  m)\pig(D\mathbb{Y}_{tX}^{\Psi} (T)\pig),
D \mathbb{Y}_{tX_k}^{\Psi}(T)
-D \mathbb{Y}_{tX}^{\Psi}(T)
\pigr\rangle_{\mathcal{H}_m}\\
=\:&-2\int_{t}^{T}
\Big\langle 
l_{vx}\pig(\mathbb{Y}_{t X_k}(\tau),u_{t X_k}(\tau)\pig)
D \mathbb{Y}^\Psi_{tX_k} (\tau)
-l_{vx}\pig(\mathbb{Y}_{t X}(\tau),u_{t X}(\tau)\pig)
D \mathbb{Y}^\Psi_{tX} (\tau),
D u_{tX_k}^{\Psi}(\tau)
-D u_{tX}^{\Psi}(\tau)\Big\rangle_{\mathcal{H}_m} d\tau\\
&-\int_{t}^{T}
\Big\langle 
l_{vv}\pig(\mathbb{Y}_{t X_k}(\tau),u_{t X_k}(\tau)\pig)
D u^\Psi_{tX_k} (\tau)
-l_{vv}\pig(\mathbb{Y}_{t X}(\tau),u_{t X}(\tau)\pig)
D u^\Psi_{tX} (\tau),
D u_{tX_k}^{\Psi}(\tau)
-D u_{tX}^{\Psi}(\tau)\Big\rangle_{\mathcal{H}_m} d\tau\\
&-\int_{t}^{T}\h{-3pt}
\bigg\langle l_{xx}(\mathbb{Y}_{t,X_k}(\tau),u_{t,X_k}(\tau))
D \mathbb{Y}^\Psi_{tX_k}(\tau)
-l_{xx}(\mathbb{Y}_{t,X}(\tau),u_{t,X}(\tau))
D \mathbb{Y}^\Psi_{tX}(\tau),
D \mathbb{Y}_{tX_k}^{\Psi}(\tau)
-D \mathbb{Y}_{tX}^{\Psi}(\tau)\Big\rangle_{\mathcal{H}_m} d\tau\\
&-\int_{t}^{T}\h{-3pt}
\bigg\langle D_{X}^{2}F(\mathbb{Y}_{tX_k}(\tau) \otimes  m)(D\mathbb{Y}_{tX_k}^\Psi(\tau)) 
-D_{X}^{2}F(\mathbb{Y}_{tX}(\tau) \otimes  m)(D\mathbb{Y}_{tX}^\Psi(\tau)) ,
D \mathbb{Y}_{tX_k}^{\Psi}(\tau)
-D \mathbb{Y}_{tX}^{\Psi}(\tau)
\bigg\rangle_{\mathcal{H}_m} \h{-6pt} d\tau.
\end{aligned}
\label{5185}
\end{equation}
\endgroup
The first line in (\ref{5185}) can be estimated by using Young's inequality and Assumption \textbf{A(vi)}'s (\ref{assumption, convexity of h}) such that we have
\begin{align*}
&\h{-10pt}\pig\langle h_{xx}(\mathbb{Y}_{tX_k}(T))D \mathbb{Y}_{tX_k}^{\Psi} (T)
-h_{xx}(\mathbb{Y}_{tX}(T))D \mathbb{Y}_{tX}^{\Psi} (T),
D \mathbb{Y}_{tX_k}^{\Psi}(T)
-D \mathbb{Y}_{tX}^{\Psi}(T)
\pigr\rangle_{\mathcal{H}_m}\\
=\:&\pig\langle h_{xx}(\mathbb{Y}_{tX_k}(T))
\pig[D \mathbb{Y}_{tX_k}^{\Psi} (T)
-D \mathbb{Y}_{tX}^{\Psi} (T)
\pig]
+\pig[h_{xx}(\mathbb{Y}_{tX_k}(T))
-h_{xx}(\mathbb{Y}_{tX}(T))\pig]
D \mathbb{Y}_{tX}^{\Psi} (T),\\
&\h{350pt}D \mathbb{Y}_{tX_k}^{\Psi}(T)
-D \mathbb{Y}_{tX}^{\Psi}(T)
\pigr\rangle_{\mathcal{H}_m}\\
\geq\:& -c_h'\pig\|D \mathbb{Y}_{tX_k}^{\Psi} (T)
-D \mathbb{Y}_{tX}^{\Psi} (T)\pigr\|_{\mathcal{H}_m}^2
-\kappa_8\pig\|\pig[h_{xx}(\mathbb{Y}_{tX_k}(T))
-h_{xx}(\mathbb{Y}_{tX}(T))\pig]
D \mathbb{Y}_{tX}^{\Psi} (T)\pigr\|_{\mathcal{H}_m}^2\\
&-\dfrac{1}{4\kappa_8}
\pig\|D \mathbb{Y}_{tX_k}^{\Psi}(T)
-D \mathbb{Y}_{tX}^{\Psi}(T)\pigr\|_{\mathcal{H}_m}^2,
\end{align*}
for some $\kappa_8>0$ to be determined later. All the other terms in (\ref{5185}) can be decomposed and estimated in a similar manner as above, then together with Assumptions \textbf{A(v)}'s (\ref{assumption, convexity of l}), \textbf{B(v)(b)}'s (\ref{assumption, convexity of D^2F, D^2F_T}), we have
\begin{equation}
\scalemath{0.93}{
\begin{aligned}
&\h{-10pt}\int^T_{t}
\lambda\pig\|D u_{tX_k}^{\Psi}(\tau)
-D u_{tX}^{\Psi}(\tau)\pigr\|_{\mathcal{H}_m}^2
-(c'_l+c')
\pig\|D \mathbb{Y}_{tX_k}^{\Psi}(\tau)
-D \mathbb{Y}_{tX}^{\Psi}(\tau)\pigr\|_{\mathcal{H}_m}^2 d\tau
-(c_h'+c_T')\pig\|D \mathbb{Y}_{tX_k}^{\Psi} (T)
-D \mathbb{Y}_{tX}^{\Psi} (T)\pigr\|_{\mathcal{H}_m}^2\\
\leq\:&\dfrac{1}{4\kappa_8}
\pig\|D \mathbb{Y}_{tX_k}^{\Psi}(T)
-D \mathbb{Y}_{tX}^{\Psi}(T)\pigr\|_{\mathcal{H}_m}^2
+\kappa_8\pig\|\pig[h_{xx}(\mathbb{Y}_{tX_k}(T))
-h_{xx}(\mathbb{Y}_{tX}(T))\pig]
D \mathbb{Y}_{tX}^{\Psi} (T)\pigr\|_{\mathcal{H}_m}^2\\
&+\dfrac{1}{4\kappa_9}
\pig\|D \mathbb{Y}_{tX_k}^{\Psi}(T)
-D \mathbb{Y}_{tX}^{\Psi}(T)\pigr\|_{\mathcal{H}_m}^2
+\kappa_9\pig\|\pig[D_{X}^{2}F(\mathbb{Y}_{tX_k}(T) \otimes  m)
-D_{X}^{2}F(\mathbb{Y}_{tX}(T) \otimes  m)\pig]
(D \mathbb{Y}_{tX}^{\Psi} (T))\pigr\|_{\mathcal{H}_m}^2\\
&+
\int^T_{t}\dfrac{1}{2\kappa_{10}}
\pig\|
D u_{tX_k}^{\Psi}(\tau)
-D u_{tX}^{\Psi}(\tau)\pigr\|_{\mathcal{H}_m}^2
+2\kappa_{10}\pig\|\pig[l_{vx}\pig(\mathbb{Y}_{t X_k}(\tau),u_{t X_k}(\tau)\pig)
-l_{vx}\pig(\mathbb{Y}_{t X}(\tau),u_{t X}(\tau)\pig)\pig]
D \mathbb{Y}_{tX}^{\Psi} (\tau)\pigr\|_{\mathcal{H}_m}^2d\tau\\
&+
\int^T_{t}\dfrac{1}{4\kappa_{11}}
\pig\|
D u_{tX_k}^{\Psi}(\tau)
-D u_{tX}^{\Psi}(\tau)\pigr\|_{\mathcal{H}_m}^2
+\kappa_{11}\pig\|\pig[l_{vv}\pig(\mathbb{Y}_{t X_k}(\tau),u_{t X_k}(\tau)\pig)
-l_{vv}\pig(\mathbb{Y}_{t X}(\tau),u_{t X}(\tau)\pig)\pig]
D u_{tX}^{\Psi} (\tau)\pigr\|_{\mathcal{H}_m}^2d\tau\\
&+
\int^T_{t}\dfrac{1}{4\kappa_{12}}
\pig\|
D \mathbb{Y}_{tX_k}^{\Psi}(\tau)
-D \mathbb{Y}_{tX}^{\Psi}(\tau)\pigr\|_{\mathcal{H}_m}^2
+\kappa_{12}\pig\|\pig[l_{xx}\pig(\mathbb{Y}_{t X_k}(\tau),u_{t X_k}(\tau)\pig)
-l_{xx}\pig(\mathbb{Y}_{t X}(\tau),u_{t X}(\tau)\pig)\pig]
D \mathbb{Y}_{tX}^{\Psi} (\tau)\pigr\|_{\mathcal{H}_m}^2d\tau\\
&+
\int^T_{t}\dfrac{1}{4\kappa_{13}}
\pig\|
D \mathbb{Y}_{tX_k}^{\Psi}(\tau)
-D \mathbb{Y}_{tX}^{\Psi}(\tau)\pigr\|_{\mathcal{H}_m}^2
+\kappa_{13}\pig\|\pig[D_{X}^{2}F(\mathbb{Y}_{tX_k}(\tau) \otimes  m)
-D_{X}^{2}F(\mathbb{Y}_{tX}(\tau) \otimes  m)\pig]
D \mathbb{Y}_{tX}^{\Psi} (\tau)\pigr\|_{\mathcal{H}_m}^2d\tau,
\end{aligned}}
\label{5282}
\end{equation}
for some positive constants $\kappa_9,\ldots,\kappa_{13}$ to be determined. With an application of Cauchy-Schwarz inequality, the equation of $D \mathbb{Y}_{tX_k}^{\Psi}(\tau)
-D \mathbb{Y}_{tX}^{\Psi}(\tau)$ implies that
\begin{equation}
\pig\|
D \mathbb{Y}_{tX_k}^{\Psi}(s)
-D \mathbb{Y}_{tX}^{\Psi}(s)\pigr\|_{\mathcal{H}_m}^2
\leq s \int^s_{t}\pig\|
D u_{tX_k}^{\Psi}(\tau)
-D u_{tX}^{\Psi}(\tau)\pigr\|_{\mathcal{H}_m}^2 d\tau,
\label{DY-DY X-X_k}
\end{equation}
\begin{flalign}
\text{and}&&\int^T_{t}\pig\|
D \mathbb{Y}_{tX_k}^{\Psi}(\tau)
-D \mathbb{Y}_{tX}^{\Psi}(\tau)\pigr\|_{\mathcal{H}_m}^2d\tau
\leq \dfrac{T^2}{2} \int^T_{t}\pig\|
D u_{tX_k}^{\Psi}(\tau)
-D u_{tX}^{\Psi}(\tau)\pigr\|_{\mathcal{H}_m}^2 d\tau.&&
\label{int DY-DY X-X_k}
\end{flalign}
Substituting (\ref{DY-DY X-X_k}) and (\ref{int DY-DY X-X_k}) into (\ref{5282}), we have
\begin{equation}
\scalemath{0.93}{
\begin{aligned}
&\h{-10pt}\int^T_{t}
\left[\lambda-\dfrac{1}{2\kappa_{10}}-\dfrac{1}{4\kappa_{11}}
- \left(\dfrac{1}{4\kappa_8}
+\dfrac{1}{4\kappa_9}\right)T
- \left(\dfrac{1}{4\kappa_{12}}
+\dfrac{1}{4\kappa_{13}}\right)\dfrac{T^2}{2}
\right]
\pig\|D u_{tX_k}^{\Psi}(\tau)
-D u_{tX}^{\Psi}(\tau)\pigr\|_{\mathcal{H}_m}^2\\
&\h{10pt}-(c'_l+c')_+\dfrac{T^2}{2}
\pig\|D u_{tX_k}^{\Psi}(\tau)
-D u_{tX}^{\Psi}(\tau)\pigr\|_{\mathcal{H}_m}^2-\left(c_h'+c_T'\right)_+T\pig\|D u_{tX_k}^{\Psi}(\tau)
-D u_{tX}^{\Psi}(\tau)\pigr\|_{\mathcal{H}_m}^2d\tau\\
\leq\:&\kappa_8\pig\|\pig[h_{xx}(\mathbb{Y}_{tX_k}(T))
-h_{xx}(\mathbb{Y}_{tX}(T))\pig]
D \mathbb{Y}_{tX}^{\Psi} (T)\pigr\|_{\mathcal{H}_m}^2
+\kappa_9\pig\|\pig[D_{X}^{2}F(\mathbb{Y}_{tX_k}(T) \otimes  m)
-D_{X}^{2}F(\mathbb{Y}_{tX}(T) \otimes  m)\pig]
(D \mathbb{Y}_{tX}^{\Psi} (T))\pigr\|_{\mathcal{H}_m}^2\\
&+
\int^T_{t}
2\kappa_{10}\pig\|\pig[l_{vx}\pig(\mathbb{Y}_{t X_k}(\tau),u_{t X_k}(\tau)\pig)
-l_{vx}\pig(\mathbb{Y}_{t X}(\tau),u_{t X}(\tau)\pig)\pig]
D \mathbb{Y}_{tX}^{\Psi} (T)\pigr\|_{\mathcal{H}_m}^2d\tau\\
&+
\int^T_{t}
\kappa_{11}\pig\|\pig[l_{vv}\pig(\mathbb{Y}_{t X_k}(\tau),u_{t X_k}(\tau)\pig)
-l_{vv}\pig(\mathbb{Y}_{t X}(\tau),u_{t X}(\tau)\pig)\pig]
D u_{tX}^{\Psi} (T)\pigr\|_{\mathcal{H}_m}^2d\tau\\
&+
\int^T_{t}\kappa_{12}\pig\|\pig[l_{xx}\pig(\mathbb{Y}_{t X_k}(\tau),u_{t X_k}(\tau)\pig)
-l_{xx}\pig(\mathbb{Y}_{t X}(\tau),u_{t X}(\tau)\pig)\pig]
D \mathbb{Y}_{tX}^{\Psi} (T)\pigr\|_{\mathcal{H}_m}^2d\tau\\
&+
\int^T_{t}
\kappa_{13}\pig\|\pig[D_{X}^{2}F(\mathbb{Y}_{tX_k}(\tau) \otimes  m)
-D_{X}^{2}F(\mathbb{Y}_{tX}(\tau) \otimes  m)\pig]
(D \mathbb{Y}_{tX}^{\Psi} (\tau))\pigr\|_{\mathcal{H}_m}^2d\tau.
\end{aligned}}
\label{5429}
\end{equation}
We next prove that the sequence of processes $\pig( D \mathbb{Y}_{tX_k}^\Psi (s),
D \mathbb{Z}_{tX_k}^\Psi (s),
D u_{tX_k}^\Psi (s),
D \mathbbm{r}_{tX_k}^\Psi (s)\pig)$ converges strongly in norm to $\pig( D \mathbb{Y}^\Psi_{tX} (s),
D \mathbb{Z}^\Psi_{tX} (s),
D u^\Psi_{tX}(s),
D\mathbbm{r}^\Psi_{tX}(s)\pig)$ as $k \to \infty$. For if not the case, there is a subsequence, without relabelling for simplicity,  such that, for instance, 
\begin{equation}
\lim_{k \to \infty} \sup_{s\in [t,T]} \pig\| D \mathbb{Y}^\Psi_{tX_k} (s)
- D \mathbb{Y}^\Psi_{tX} (s) \pigr\|_{\mathcal{H}_m}>0.
\label{not conv DY, X_k}
\end{equation}
By setting $\epsilon=1$ and taking $\Psi=X_k-X$ in the bound in (\ref{bdd Y, Z, u, r, epsilon}), we have
\begin{align} \pig\| \mathbb{Y}_{tX_k}(T)
-\mathbb{Y}_{tX}(T)\pigr\|_{\mathcal{H}_m}
\leq C_4'\|X_k-X\|_{\mathcal{H}_m}\h{15pt} \text{and},
\label{bdd DY, X_k-X}
\end{align}
\begin{flalign} 
&&
\int^T_{t}\pig\| \mathbb{Y}_{tX_k}(\tau)
-\mathbb{Y}_{tX}(\tau)\pigr\|_{\mathcal{H}_m} d\tau,\:
 \int^T_{t}\pig\|u_{tX_k}(\tau)
-u_{tX}(\tau)\pigr\|_{\mathcal{H}_m}d\tau
\leq C_4'\h{1pt}T\|X_k-X\|_{\mathcal{H}_m},&&
\label{bdd DY, Du, X_k-X}
\end{flalign}
where $C'_4$ mentioned in (\ref{bdd Y, Z, u, r, epsilon}) is independent of the sequence $\{X_k\}_{k\in \mathbb{N}}$ and $X$. The strong convergences in (\ref{bdd DY, X_k-X}) and (\ref{bdd DY, Du, X_k-X}) imply that there is a subsequence such that $\mathbb{Y}_{tX_k}(T)$ converges to $\mathbb{Y}_{tX}(T)$, $m \otimes \mathbb{P}$-a.s., 
$\mathbb{Y}_{tX_k}(\tau)$ converges to $\mathbb{Y}_{tX}(\tau)$, $m \otimes \mathbb{P}$-a.s. for a.e. $\tau \in [t,T]$, and $u_{tX_k}(\tau)$ converges to $u_{tX}(\tau)$, $m \otimes \mathbb{P}$-a.s. for a.e. $\tau \in [t,T]$. Similar to Step 2A in the proof of Lemma \ref{lem, Existence of J flow}, by the continuities and boundedness of $D^2_X F$, $D^2_X F_T$, $h_{xx}$, the second-order derivatives of $l$ in Assumptions \textbf{B(iii)}'s (\ref{assumption, cts of D^2F}),
\textbf{B(iv)}'s (\ref{assumption, cts of D^2F_T}),
\textbf{A(iv)}'s (\ref{assumption, cts of lxx, lvx, lvv, hxx}), \textbf{B(ii)}'s (\ref{assumption, bdd of D^2F, D^2F_T}), \textbf{A(ii)}'s (\ref{assumption, bdd of lxx, lvx, lvv}), 
\textbf{A(iii)}'s (\ref{assumption, bdd of h, hx, hxx}), by applying the dominated convergence theorem to (\ref{5429}) deduces the subsequential convergence of the right hand side of (\ref{5429}) to zero and
\begin{equation*}
\begin{aligned}
&\h{-10pt}\int^T_{t}
\left[\lambda-\dfrac{1}{2\kappa_{10}}-\dfrac{1}{4\kappa_{11}}
- \left(\dfrac{1}{4\kappa_8}
+\dfrac{1}{4\kappa_9}\right)T
- \left(\dfrac{1}{4\kappa_{12}}
+\dfrac{1}{4\kappa_{13}}\right)\dfrac{T^2}{2}
\right]
\pig\|D u_{tX_k}^{\Psi}(\tau)
-D u_{tX}^{\Psi}(\tau)\pigr\|_{\mathcal{H}_m}^2\\
&\h{10pt}-(c'_l+c')_+\dfrac{T^2}{2}
\pig\|D u_{tX_k}^{\Psi}(\tau)
-D u_{tX}^{\Psi}(\tau)\pigr\|_{\mathcal{H}_m}^2-\left(c_h'+c_T'\right)_+T\pig\|D u_{tX_k}^{\Psi}(\tau)
-D u_{tX}^{\Psi}(\tau)\pigr\|_{\mathcal{H}_m}^2d\tau \\
&\h{-10pt}\longrightarrow 0,\h{20pt}
\text{as $k \to \infty$ up to a subsequence.}
\end{aligned}
\end{equation*}
Choosing small enough $\kappa_i$'s, the condition in (\ref{uniqueness condition}) yields that 
\begin{equation}
\int^T_{t}
\pig\|D u_{tX_k}^{\Psi}(\tau)
-D u_{tX}^{\Psi}(\tau)\pigr\|_{\mathcal{H}_m}^2 d\tau\longrightarrow 0,
\h{20pt}
\text{as $k \to \infty$ up to a subsequence.}
\label{5397}
\end{equation}
Together with (\ref{DY-DY X-X_k}), we have 
$\sup_{s\in[t,T]}\pig\|
D \mathbb{Y}_{tX_k}^{\Psi}(s)
-D \mathbb{Y}_{tX}^{\Psi}(s)\pigr\|_{\mathcal{H}_m}$ converges to zero as $k \to \infty$ up to a subsequence. This contradicts (\ref{not conv DY, X_k}), therefore the strong convergence of $D  \mathbb{Y}_{tX_k}^\Psi (s)$ and the continuity of $D  \mathbb{Y}_{tX}^\Psi (s)$ with respect to $X$ in norm should follow. 

The strong convergences of 
$D  \mathbb{Z}_{tX_k}^\Psi (s)$ and $D \mathbbm{r}_{tX_k}^\Psi (s)$ are concluded by subtracting their equations and then using It\^o's lemma, together with the convergences in (\ref{5397}), that of $D  \mathbb{Y}_{tX_k}^\Psi (s)$, continuities and boundedness of $D^2_X F$, $D^2_X F_T$, $h_{xx}$, the second-order derivatives of $l$ in Assumptions \textbf{B(iii)}'s (\ref{assumption, cts of D^2F}),
\textbf{B(iv)}'s (\ref{assumption, cts of D^2F_T}),
\textbf{A(iv)}'s (\ref{assumption, cts of lxx, lvx, lvv, hxx}), \textbf{B(ii)}'s (\ref{assumption, bdd of D^2F, D^2F_T}), \textbf{A(ii)}'s (\ref{assumption, bdd of lxx, lvx, lvv}), 
\textbf{A(iii)}'s (\ref{assumption, bdd of h, hx, hxx}). Finally, the strong convergence of $D u_{tX_k}^\Psi (s)$ is deduced by differentiating the first order condition in (\ref{1st order condition}), continuities of the second-order derivatives of $l$ in Assumption \textbf{A(iv)}'s (\ref{assumption, cts of lxx, lvx, lvv, hxx}), and the strong convergences of $D \mathbb{Y}_{tX}^\Psi (s)$, $D \mathbb{Z}_{tX_k}^\Psi (s)$ just obtained.

\noindent\textbf{Part 3. Existence of G\^ateaux  derivative:}\\
To conclude, since $(D \mathbb{Y}_{tX}^\Psi(s),
D \mathbb{Z}_{tX}^\Psi(s),D u_{tX}^\Psi(s),D \mathbbm{r}_{tX,j}^\Psi(s))$ is linear in $\Psi$ and continuous in $X$ for a given $\Psi$, therefore, by Proposition 3.2.15 in \cite{DM07} together with the separability of the Hilbert space $\mathcal{H}_m$, we obtain the existence of the Fr\'echet derivatives.

\hfill $\blacksquare$

\subsection{Proof of Statements in Section 4}
\subsubsection{Proof of Proposition \ref{prop4-2} }\label{app, prop DX V = Z and lip of DX V}
This proof follows the same arguments as in Theorem 2.1 in \cite{BY19}; also see Theorem 5.3 in \cite{BGY19}. Let the initial random variables $X^{1},X^{2}\in L^2_{\mathcal{W}_t^{\indep}}(\mathcal{H}_m)$, for $\nu=1,2$, and consider the corresponding solution to the FBSDE (\ref{forward FBSDE})-(\ref{1st order condition}). For simplicity, we write $\mathbb{Y}^\nu(s)=\mathbb{Y}_{tX^\nu}(s)$, $\mathbb{Z}^\nu(s)=\mathbb{Z}_{tX^\nu}(s)$, $u^{\nu}(s)=u_{tX^\nu}(s)$ and $\mathbbm{r}^{\nu}(s)=\mathbbm{r}_{tX^{\nu}}(s)$, for $\nu=1,2$. We therefore consider the corresponding systems
\begin{empheq}[left=\h{-15pt}\empheqbiglbrace]{align} 
\mathbb{Y}^\nu(s)=&\:X^{\nu}+\int_{t}^{s}u^\nu(\tau)d\tau+\eta(w(s)-w(t));
\label{forward, X^nu}\\
\mathbb{Z}^\nu(s)=&\:
\mbox{\fontsize{9}{10}\selectfont\(h_{x}(\mathbb{Y}^\nu(T))+D_{X}F_{T}(\mathbb{Y}^\nu(T) \otimes  m)
+\displaystyle\int^T_s\pig[l_{x}(\mathbb{Y}^\nu(\tau),u^\nu(\tau))+D_{X}F(\mathbb{Y}^\nu(\tau) \otimes  m)\pig]d\tau
-\int^T_s\displaystyle\sum_{j=1}^{n}\mathbbm{r}_j^{\nu}(\tau)dw_{j}(\tau)\)},\h{-10pt}
\label{backward, X^nu}
\end{empheq}
\begin{flalign}
\text{subject to}&&
l_{v}(\mathbb{Y}^\nu(s),u^\nu(s))+\mathbb{Z}^\nu(s)=0.&&
\label{1st order condition, X^nu}
\end{flalign}

\noindent {\bf Part 1. Differentiability of $V(X \otimes m,t)$ in $X$:}\\
According to Remark \ref{rem4-1}, we can use a common space of controls $L_{\mathcal{W}_{tX^{1}X^{2}}}^{2}(t,T;\mathcal{H}_{m})$ for two problems (\ref{forward, X^nu})-(\ref{1st order condition, X^nu}) with $\nu=1,2$. Consider the dynamics (\ref{forward FBSDE})-(\ref{backward FBSDE}) using the control $u^{2}(s)$ with the initial condition $X^{1}$, the corresponding trajectory is $\mathbb{Y}^2(s)+X^{1}-X^{2}$, and then by definition

\begin{equation}
\begin{aligned}
&\h{-10pt}V(X^{1} \otimes  m,t)-V(X^{2} \otimes  m,t)\\
\leq&\:\int_{t}^{T}\mbox{\fontsize{10}{10}\selectfont\(
\bigg\{\mathbb{E}\left[\displaystyle\int_{\mathbb{R}^{n}}
l\big(\mathbb{Y}^2(s)+X^{1}-X^{2},u^{2}(s)\big)
-l\big(\mathbb{Y}^2(s),u^{2}(s)\big) dm(x)\right]
+F((\mathbb{Y}^2(s)+X^{1}-X^{2}) \otimes  m)-F(\mathbb{Y}^2(s) \otimes  m))\bigg\}ds\)}\\
&+\mathbb{E}\left[\int_{\mathbb{R}^{n}}h(\mathbb{Y}^2(T)+X^{1}-X^{2})-h(\mathbb{Y}^2(T))dm(x)\right]
+F_{T}((\mathbb{Y}^2(T)+X^{1}-X^{2}) \otimes  m)-F(\mathbb{Y}^2(T) \otimes  m)\\
=&\:\int_{0}^{1}\int_{t}^{T}
\Big\langle
l_{x}(\mathbb{Y}^2(s)+\theta(X^{1}-X^{2}),u^{2}(s))+D_{X}F((\mathbb{Y}^2(s)+\theta(X^{1}-X^{2})) \otimes  m),X^{1}-X^{2}\Big\rangle_{\mathcal{H}_m} dsd\theta\\
&+\int_{0}^{1}\Big\langle
h_{x}(\mathbb{Y}^2(T)+\theta(X^{1}-X^{2}))+D_{X}F_{T}((\mathbb{Y}^2(T)+\theta(X^{1}-X^{2})) \otimes  m),X^{1}-X^{2}
\Big\rangle_{\mathcal{H}_m}d\theta.
\label{eq:ApB14}
\end{aligned}
\end{equation}
From Assumptions \textbf{A(ii)}'s (\ref{assumption, bdd of lxx, lvx, lvv}), \textbf{A(iii)}'s (\ref{assumption, bdd of h, hx, hxx}) and \textbf{B(ii)}'s (\ref{assumption, bdd of D^2F, D^2F_T}), we further obtain
\begin{equation}
\begin{aligned}
&V(X^{1} \otimes  m,t)-V(X^{2} \otimes  m,t)\\
&\leq\left\langle\int_{t}^{T}l_{x}(\mathbb{Y}^2(s),u^{2}(s))+D_{X}F(\mathbb{Y}^2(s) \otimes  m)ds+h_{x}(\mathbb{Y}^2(T))+D_{X}F_T(\mathbb{Y}^2(T) \otimes  m),X^{1}-X^{2}\right\rangle_{\mathcal{H}_m}\\
&\h{10pt}+ \pig[c_{h}+c_{T}+(c_{l}+c)T\pig]\|X^{1}-X^{2}\|^{2}_{\mathcal{H}_m}.
\end{aligned}
\label{4289}
\end{equation}
Substituting (\ref{backward, X^nu}) into (\ref{4289}), we see that
\begin{equation}
V(X^{1} \otimes  m,t)-V(X^{2} \otimes  m,t)
\leq\pig\langle \mathbb{Z}^2(t),X^{1}-X^{2}\pigr\rangle_{\mathcal{H}_m}
+ \pig[c_{h}+c_{T}+(c_{l}+c)T\pig]\|X^{1}-X^{2}\|^{2}_{\mathcal{H}_m}.
\label{eq:ApB15}
\end{equation}
By interchanging the role of $X^{1}$ and $X^{2}$, we obtain the reverse inequality

\begin{equation}
 \mbox{\fontsize{9.7}{10}\selectfont\(
V(X^{1} \otimes  m,t)-V(X^{2} \otimes  m,t)\geq\pig\langle \mathbb{Z}^2(t),X^{1}-X^{2}\pigr\rangle_{\mathcal{H}_m}
- \pig[c_{h}+c_{T}+(c_{l}+c)T\pig]\pig\|X^{1}-X^{2}\pigr\|^{2}_{\mathcal{H}_m}
+\pig\langle
\mathbb{Z}^1(t)-\mathbb{Z}^2(t),X^{1}-X^{2}
\pigr\rangle_{\mathcal{H}_m}
\)}.
\label{eq:ApB16}
\end{equation}
Next, using the first order conditions (\ref{1st order condition, X^nu}) and the fact that $\mathbb{Y}^1(s)-\mathbb{Y}^2(s)$ is of finite variation, we have
\begin{equation}
\begin{aligned}
&\h{-15pt}\dfrac{d}{ds}\pig\langle\mathbb{Z}^1(s)-\mathbb{Z}^2(s),\mathbb{Y}^1(s)-\mathbb{Y}^2(s)\pigr\rangle_{\mathcal{H}_m}\\
=\:&-\pig\langle l_{v}(\mathbb{Y}^1(s),u^{1}(s))-l_{v}(\mathbb{Y}^2(s),u^{2}(s)),u^{1}(s)-u^{2}(s)\pigr\rangle_{\mathcal{H}_m}\\
&-\pig\langle l_{x}(\mathbb{Y}^1(s),u^{1}(s))-l_{x}(\mathbb{Y}^2(s),u^{2}(s))+D_{X}F(\mathbb{Y}^1(s) \otimes  m)-D_{X}F(\mathbb{Y}^2(s) \otimes  m),
\mathbb{Y}^1(s)-\mathbb{Y}^2(s)\pigr\rangle_{\mathcal{H}_m}.
\end{aligned}
\end{equation}
Integrating from $t$ to $T$, we obtain 
\begin{equation}
\begin{aligned}
&\h{-10pt}\pig\langle
\mathbb{Z}^1(t)-\mathbb{Z}^2(t),X^{1}-X^{2}
\pigr\rangle_{\mathcal{H}_m}\\
=\:&\pig\langle h_{x}(\mathbb{Y}^1(T))-h_{x}(\mathbb{Y}^2(T))+D_{X}F_{T}(\mathbb{Y}^1(T) \otimes  m)-D_{X}F_T(\mathbb{Y}^2(T) \otimes  m),\mathbb{Y}^1(T)-\mathbb{Y}^2(T)\pigr\rangle_{\mathcal{H}_m}\\
&+\int_{t}^{T}\pig\langle
l_{v}(\mathbb{Y}^1(s),u^{1}(s))-l_{v}(\mathbb{Y}^2(s),u^{2}(s)),u^{1}(s)-u^{2}(s)\pigr\rangle_{\mathcal{H}_m}ds\\
&+\int_{t}^{T}\pig\langle
l_{x}(\mathbb{Y}^1(s),u^{1}(s))-l_{x}(\mathbb{Y}^2(s),u^{2}(s))+D_{X}F(\mathbb{Y}^1(s) \otimes  m)-D_{X}F(\mathbb{Y}^2(s) \otimes  m),\mathbb{Y}^1(s)-\mathbb{Y}^2(s)\pigr\rangle_{\mathcal{H}_m}ds\\
\geq&\:
 -\mbox{\fontsize{9.5}{10}\selectfont\((c'_{h}+c'_{T})\pig\|\mathbb{Y}^1(T)-\mathbb{Y}^2(T)\pigr\|^{2}_{\mathcal{H}_m}
 +\left(\lambda-2c_l\kappa_{14}  \right)\displaystyle\int_{t}^{T}\pig\|u^{1}(s)-u^{2}(s)\pigr\|^{2}_{\mathcal{H}_m}ds
-\left(c_l'+\dfrac{c_{l}}{2\kappa_{14}}+c'\right)\int_{t}^{T}\pig\|\mathbb{Y}^1(s)-\mathbb{Y}^2(s)\pigr\|^{2}_{\mathcal{H}_m}ds\)},
\end{aligned}
\label{eq:ApB17}
\end{equation}
for some positive constant $\kappa_{14}$. From the fact that $\mathbb{Y}^1(s)-\mathbb{Y}^2(s)=X^{1}-X^{2}+\displaystyle\int_{t}^{s}u^{1}(\tau)-u^{2}(\tau)d\tau$, we get

\begin{equation}
\sup_{s\in [t,T]}
\pig\|\mathbb{Y}^1(s)-\mathbb{Y}^2(s)\pigr\|^{2}_{\mathcal{H}_m}
\leq(1+\kappa_{15})T\int_{t}^{T}\pig\|u^{1}(s)-u^{2}(s)\pigr\|^{2}_{\mathcal{H}_m}ds
+\left(1+\dfrac{1}{\kappa_{15}}\right)\pig\|X^{1}-X^{2}\pigr\|^{2}_{\mathcal{H}_m}\h{10pt} \text{and}
\label{est |YT1-YT2|}
\end{equation}
\begin{equation}
\int_{t}^{T}\pig\|\mathbb{Y}^1(s)-\mathbb{Y}^2(s)\pigr\|^{2}_{\mathcal{H}_m}ds
\leq(1+\kappa_{15})\dfrac{T^{2}}{2}\int_{t}^{T}\pig\|u^{1}(s)-u^{2}(s)\pigr\|^{2}_{\mathcal{H}_m}ds+T\left(1+\dfrac{1}{\kappa_{15}}\right)\pig\|X^{1}-X^{2}\pigr\|^{2}_{\mathcal{H}_m},
\label{est int|Y1-Y2|}
\end{equation}
for some $\kappa_{15} > 0$, therefore we obtain 
\begin{equation}
\begin{aligned}
&\h{-10pt}\pig\langle
\mathbb{Z}^1(t)-\mathbb{Z}^2(t),X^{1}-X^{2}
\pigr\rangle_{\mathcal{H}_m}\\
\geq&\:\left[\lambda-2c_l\kappa_{14}  
-\left(c_l'+\dfrac{c_{l}}{2\kappa_{14}}+c'\right)_+(1+\kappa_{15})\dfrac{T^{2}}{2}
-(c'_{h}+c'_{T})_+(1+\kappa_{15})T\right]
\int_{t}^{T}\pig\|u^{1}(s)-u^{2}(s)\pigr\|^{2}_{\mathcal{H}_m}ds\\
&-\left[\left(c_l'+\dfrac{c_{l}}{2\kappa_{14}}+c'\right)_+T\left(1+\dfrac{1}{\kappa_{15}}\right)
+(c'_{h}+c'_{T})_+\left(1+\dfrac{1}{\kappa_{15}}\right)\right]
\pig\|X^{1}-X^{2}\pigr\|^{2}_{\mathcal{H}_m}.
\label{eq:ApB18}
\end{aligned}
\end{equation}
Employing the assumption (\ref{uniqueness condition}) and choosing suitable $\kappa_{14}$ and $\kappa_{15}$ so that the coefficient of the first term on the right hand side of (\ref{eq:ApB18}) is positive. Combining the inequalities in (\ref{eq:ApB16}) and (\ref{eq:ApB15}), together with (\ref{eq:ApB18}), it
follows that 
\begin{equation}
\Big|V(X^{1} \otimes  m,t)-V(X^{2} \otimes  m,t)
-\pig\langle\mathbb{Z}^2(t),X^{1}-X^{2}\pigr\rangle_{\mathcal{H}_m}
\Big|\leq A_4\pig\|X^{1}-X^{2}\pigr\|^{2}_{\mathcal{H}_m},
\label{eq:ApB19}
\end{equation}
where $A_4$ is positive and clearly independent of $X^1$ and $X^2$. This proves immediately the result (\ref{D_X V = Z}). 

\noindent {\bf Part 2. Lipschitz continuity of $D_XV(X \otimes m,t)$ in $X$:}\\
Note from (\ref{eq:ApB18}) that we have also proven
the following estimate 

\begin{equation}
\int_{t}^{T}\big\|u^{1}(s)-u^{2}(s)\big\|^{2}_{\mathcal{H}_m}ds
\leq A_5\pig\langle \mathbb{Z}^1(t)-\mathbb{Z}^2(t),X^{1}-X^{2}
\pigr\rangle_{\mathcal{H}_m}
+A_6\big\|X^{1}-X^{2}\big\|^{2}_{\mathcal{H}_m},
\label{eq:ApB200}
\end{equation}
for constants $A_5$ and $A_6$ depending only on $\lambda$, $c_l$, $c'$, $c_h'$, $c_T'$, $c_l'$ and $T$. Next, Young's inequality also gives
\begin{equation}
\begin{aligned}
&\h{-10pt}\pig\langle\mathbb{Z}^1(t)-\mathbb{Z}^2(t),X^{1}-X^{2}\pigr\rangle_{\mathcal{H}_m}\\
=\:&\left\langle\int_{t}^{T}l_{x}(\mathbb{Y}^1(s),u^{1}(s))-l_{x}(\mathbb{Y}^2(s),u^{2}(s))+D_{X}F(\mathbb{Y}^1(s) \otimes  m)-D_{X}F(\mathbb{Y}^2(s) \otimes  m)ds,X^1-X^2 \right\rangle_{\mathcal{H}_m}\\
&+\pig\langle h_{x}(\mathbb{Y}^1(T))-h_{x}(\mathbb{Y}^2(T))+D_{X}F_{T}(\mathbb{Y}^1(T) \otimes  m)-D_{X}F(\mathbb{Y}^2(T) \otimes  m),X^{1}-X^{2}\pigr\rangle_{\mathcal{H}_m}\\
\leq\:& \kappa_{16}c_l^2\int_{t}^{T}\pig\|u^1(s)-u^2(s)\pigr\|^2_{\mathcal{H}_m}ds
+\dfrac{1}{4}\left(\dfrac{1}{\kappa_{16}}
+\dfrac{1}{\kappa_{17}}
+\dfrac{1}{\kappa_{18}}
+\dfrac{1}{\kappa_{19}}
+\dfrac{1}{\kappa_{20}}
\right)\pig\|X^{1}-X^{2}\pigr\|^2_{\mathcal{H}_m}\\
&+(\kappa_{17} c^2_{l}+\kappa_{18} c^2)
\int_{t}^{T}\pig\|\mathbb{Y}^1(s)-\mathbb{Y}^2(s)\pigr\|^2_{\mathcal{H}_m}ds
+ (\kappa_{19}c_{h}^2+\kappa_{20}c_{T}^2)\pig\|\mathbb{Y}^1(T)-\mathbb{Y}^2(T)\pigr\|^2_{\mathcal{H}_m}.
\end{aligned}
\label{eq:ApB201}
\end{equation}
Substituting the inequalities in (\ref{est |YT1-YT2|}), (\ref{est int|Y1-Y2|}), (\ref{eq:ApB201}) into (\ref{eq:ApB200}), we see 
\begin{equation*}
\begin{aligned}
&\h{-15pt}\mbox{\fontsize{10.5}{10}\selectfont\(
\left( 1-\kappa_{16} c_l^2 A_5 -  (\kappa_{17} c^2_{l}+\kappa_{18} c^2)A_5(1+\kappa'_{15})\dfrac{T^2}{2} 
-(\kappa_{19}c_{h}^2+\kappa_{20}c_{T}^2)A_5(1+\kappa'_{15})T
\right)\displaystyle\int_{t}^{T}\big\|u^{1}(s)-u^{2}(s)\big\|^{2}_{\mathcal{H}_m}ds\)}\\
&\h{10pt}\leq 
\mbox{\fontsize{10.5}{10}\selectfont\(\left[A_6+\dfrac{A_5}{4}\displaystyle\sum^{20}_{i=16}\dfrac{1}{\kappa_{i}}
+A_5(\kappa_{17} c^2_{l}+\kappa_{18} c^2)\left(1+\dfrac{1}{\kappa_{15}'}\right)T
+A_5(\kappa_{19}c_{h}^2+\kappa_{20}c_{T}^2)
\left(1+\dfrac{1}{\kappa_{15}'}\right)
\right]\big\|X^{1}-X^{2}\big\|^{2}_{\mathcal{H}_m}.\)}
\end{aligned}
\end{equation*}
Choosing $\kappa_{i}$'s small enough, we have
\begin{equation}
\int_{t}^{T}\pig\|u^{1}(s)-u^{2}(s)\pigr\|^{2}_{\mathcal{H}_m}ds\leq A_7 \pig\|X^{1}-X^{2}\pigr\|^{2}_{\mathcal{H}_m},
\label{est int |u1-u2|}
\end{equation}
for some $A_7>0$ depending only on $\lambda$, $c$, $c_l$, $c_h$, $c_T$, $c'$, $c_h'$, $c_T'$, $c_l'$ and $T$.
Plugging back (\ref{est int |u1-u2|}) into (\ref{est |YT1-YT2|}) and (\ref{est int|Y1-Y2|}), it implies 
\begin{equation}
\sup_{s\in [t,T]}
\pig\|\mathbb{Y}^1(s)-\mathbb{Y}^2(s)\pigr\|^{2}_{\mathcal{H}_m}
\leq A_9\pig\|X^{1}-X^{2}\pigr\|^{2}_{\mathcal{H}_m}
\h{5pt} \text{and} \h{5pt}
\int^T_t\pig\|\mathbb{Y}^1(s)-\mathbb{Y}^2(s)\pigr\|^{2}_{\mathcal{H}_m}ds
\leq A_{10}\pig\|X^{1}-X^{2}\pigr\|^{2}_{\mathcal{H}_m}.
\label{sup |Y^1-Y^2|, int |Y^1-Y^2|}
\end{equation}
Hence, by using (\ref{D_X V = Z}), we have 
$$ \pigl\|D_{X}V(X^{1}  \otimes  m,t)-D_{X}V(X^{2}  \otimes  m,t)\pigr\|_{\mathcal{H}_m}
= \big\|\mathbb{Z}^1(t)-\mathbb{Z}^2(t)\big\|_{\mathcal{H}_m},$$
together with (\ref{bdd Z}), (\ref{sup |Y^1-Y^2|, int |Y^1-Y^2|}), the backward equations of (\ref{backward, X^nu}), Assumption \textbf{A(ii)}'s (\ref{assumption, bdd of lxx, lvx, lvv}),
\textbf{A(iii)}'s (\ref{assumption, bdd of h, hx, hxx}),
\textbf{B(ii)}'s (\ref{assumption, bdd of D^2F, D^2F_T}), the result follows.
\hfill $\blacksquare$ 

\subsubsection{Proof of Proposition \ref{prop holder of V D_X V in time}}\label{app, prop holder of V DX V in time}

We begin with (\ref{eq:4-26}). For the ease of notation, we take $t_1=t+\epsilon$ and $t_2=t$ with $t \in [0,T)$ and $\epsilon \in [0,T-t]$. From the optimality principle and (\ref{eq:3-16}), we have

\begin{equation}
\begin{aligned}
V(X \otimes  m,t)-V(X \otimes  m,t+\epsilon)
=\:&\int_{t}^{t+\epsilon}
\mathbb{E}\left[\int_{\mathbb{R}^{n}}l(\mathbb{Y}_{tX}(s),u_{tX}(s))dm(x)\right]
+F(\mathbb{Y}_{tX}(s) \otimes  m)\;ds\\
&+V(\mathbb{Y}_{tX}(t+\epsilon) \otimes  m,t+\epsilon)-V(X \otimes  m,t+\epsilon).
\end{aligned}
\label{eq:ApB22}
\end{equation}
The first term on the right hand side of (\ref{eq:ApB22}) can be estimated by using Assumption \textbf{A(ii)}'s (\ref{assumption, bdd of l, lx, lv}), (\ref{assumption, bdd of F F_T}) and (\ref{bdd Y, Z, u, r}) to obtain 
\begin{equation}
\begin{aligned}
\int_{t}^{t+\epsilon}
\left|\mathbb{E}\left[\int_{\mathbb{R}^{n}}l(\mathbb{Y}_{tX}(s),u_{tX}(s)dm(x)\right]\right|
+\pig| F(\mathbb{Y}_{tX}(s) \otimes  m) \pig|\;ds
&\leq \epsilon (c_l+c) \sup_{s\in[t,T]}\pig(1+ \big\|\mathbb{Y}_{tX}(s)\big\|_{\mathcal{H}_m}^2 
+\big\|u_{tX}(s)\big\|_{\mathcal{H}_m}^2\pig)\\
&\leq \epsilon  A_{11}  \pig(1+ \big\|X\big\|_{\mathcal{H}_m}^2 \pig),
\end{aligned}
\label{1st term in V(t)-V(t+e)}
\end{equation}
for some $A_{11} > 0$ depending on $n$, $\lambda$, $\eta$, $c$, $c_l$, $c_h$, $c_T$, $c'$, $c_h'$, $c_T'$, $c_l'$ and $T$. The Lipschitz continuity of $D_X V$ in (\ref{lip cts of D_XV in X}) and the relation in (\ref{d_theta F = D_X F}) imply that

\begin{equation}
\begin{aligned}
&\pig|
V(\mathbb{Y}_{tX}(t+\epsilon) \otimes  m,t+\epsilon)-V(X \otimes  m,t+\epsilon)-\pig\langle D_{X}V(X \otimes  m,t+\epsilon),\mathbb{Y}_{tX}(t+\epsilon)-X\pigr\rangle_{\mathcal{H}_m}
\pig|\\
&= \left| \int^1_0\pigl\langle D_{X}V\big((X+\theta (\mathbb{Y}_{tX}(t+\epsilon)-X))  \otimes  m,t+\epsilon\big)
-D_{X}V(X \otimes  m,t+\epsilon)
,\mathbb{Y}_{tX}(t+\epsilon)-X\pigr\rangle_{\mathcal{H}_m} d \theta\right|\\
&\leq C_6\pig\|\mathbb{Y}_{tX}(t+\epsilon)-X\pigr\|^{2}_{\mathcal{H}_m}.
\end{aligned}
\label{eq:ApB23}
\end{equation}
Since $D_{X}V(X \otimes  m,t+\epsilon)$ is $\sigma(X)$-measurable and $X \in L^2_{\mathcal{W}_t^{\indep} }(\mathcal{H}_m)$, we can remove the stochastic integral by taking the expectation so as to obtain 
$$
\pig\langle D_{X}V(X \otimes  m,t+\epsilon),\mathbb{Y}_{tX}(t+\epsilon)-X\pigr\rangle_{\mathcal{H}_m}
=\left\langle D_{X}V(X \otimes  m,t+\epsilon),\int_{t}^{t+\epsilon}u_{tX}(s)ds\right\rangle_{\mathcal{H}_m}.
$$
Due to (\ref{D_X V = Z}) and (\ref{bdd Y, Z, u, r}), we further have
\begin{equation}
\begin{aligned}
\Big|\pig\langle D_{X}V(X \otimes  m,t+\epsilon),\mathbb{Y}_{tX}(t+\epsilon)-X\pigr\rangle_{\mathcal{H}_m}\Big|
\leq \epsilon A_{12}\pig(1+ \big\|X\big\|_{\mathcal{H}_m}^2 \pig).
\end{aligned}
\label{<D_X V, Y-X>}
\end{equation}
Substituting (\ref{1st term in V(t)-V(t+e)}), (\ref{eq:ApB23}), (\ref{<D_X V, Y-X>}) into (\ref{eq:ApB22}), we obtain 
$$
\pig| V(X \otimes  m,t)-V(X \otimes  m,t+\epsilon)\pig|
\leq \epsilon  A_{11}  \pig(1+ \big\|X\big\|_{\mathcal{H}_m}^2 \pig)
+C_6\pig\|\mathbb{Y}_{tX}(t+\epsilon)-X\pigr\|^{2}_{\mathcal{H}_m}
+ \epsilon A_{12} \pig(1+ \big\|X\big\|_{\mathcal{H}_m}^2 \pig),
$$
which leads to the result (\ref{eq:4-26}) by using (\ref{est. |Ys-X|^2}) to deal with the middle term on the right hand side.

We now turn to the proof of (\ref{cts of D_X V in t}) by applying (\ref{D_X V(s)=Z(s)}), which is equivalent to establish the inequality 
$\pig\|\mathbb{Z}_{t+\epsilon,X}(t+\epsilon)-\mathbb{Z}_{tX}(t)\pigr\|_{\mathcal{H}_m}
\leq C_8\pig(\epsilon^{\frac{1}{2}}+\epsilon\|X\|_{\mathcal{H}_m}\pig)$. From the BSDE in (\ref{backward FBSDE}), we can write 
\begin{align*}
 \mbox{\fontsize{9.9}{10}\selectfont\(
 \mathbb{Z}_{t+\epsilon,X}(t+\epsilon)-\mathbb{Z}_{tX}(t)
=\mathbb{Z}_{t+\epsilon,X}(t+\epsilon)
-\mathbb{E}\pig(\mathbb{Z}_{tX}(t+\epsilon) \pig| X\pig) 
-\displaystyle\int_{t}^{t+\epsilon}
\mathbb{E}\left[l_{x}\pig(\mathbb{Y}_{tX}(\tau),u_{tX}(\tau)\pig)
+D_{X}F\pig(\mathbb{Y}_{tX}(\tau)  \otimes  m\pig) \middle| X\right]d\tau.\)}
\end{align*}
Applying (\ref{bdd Y, Z, u, r}), Assumption \textbf{A(i)}'s (\ref{assumption, bdd of l, lx, lv}) and \textbf{B(i)}'s (\ref{assumption, bdd of DF, DF_T}), it yields

\[
\pig\|\mathbb{Z}_{t+\epsilon,X}(t+\epsilon)-\mathbb{Z}_{tX}(t)\pigr\|_{\mathcal{H}_m}
\leq \pig\|\mathbb{Z}_{t+\epsilon,X}(t+\epsilon)-\mathbb{Z}_{tX}(t+\epsilon)\pigr\|_{\mathcal{H}_m}
+A_{13} \epsilon\pig(1+\|X\|_{\mathcal{H}_m}\pig).
\]
Note that $\mathbb{Z}_{tX}(t+\epsilon)=\mathbb{Z}_{t+\epsilon,\mathbb{Y}_{tX}(t+\epsilon)}(t+\epsilon)$, using Proposition \ref{prop4-2}, we further rewrite
\[
\pig\|\mathbb{Z}_{t+\epsilon,X}(t+\epsilon)-\mathbb{Z}_{tX}(t)\pigr\|_{\mathcal{H}_m}
\leq C_{6}\pig\|\mathbb{Y}_{tX}(t+\epsilon)-X\pigr\|_{\mathcal{H}_m}
+A_{13}\epsilon\pig(1+\|X\|_{\mathcal{H}_m}\pig)
\]
 and thus the result (\ref{cts of D_X V in t}) is obtained immediately due to (\ref{est. |Ys-X|^2}). \hfill $\blacksquare$

\subsubsection{Proof of Proposition \ref{prop cts of DXX V}}\label{app, prop cts of DXX V}
\mycomment{
\noindent {\bf Part 1. Convergence:}\\
Recall that the quadruple $\pig( \mathbb{Y}_{tX\Psi}^{\epsilon}(s),
\mathbb{Z}_{tX\Psi}^{\epsilon}(s),
u_{tX\Psi}^{\epsilon}(s),
\mathbbm{r}_{tX\Psi,j}^{\epsilon}(s)\pig)$
solve
\mycomment{
\begin{empheq}[left=\h{-15pt}\empheqbiglbrace]{align}
\mathbb{Y}_{tX\Psi}^{\epsilon}(s)
=\:&\mathcal{X}+\int_{t}^{s}u_{tX\Psi}^{\epsilon}(\tau)d\tau;
\label{finite diff. forward}\\
\mathbb{Z}_{tX\Psi}^{\epsilon}(s)
=\:&\mbox{\fontsize{10}{10}\selectfont\(
\displaystyle\int_{0}^{1}h_{xx}\pig(\mathbb{Y}_{tX}(T)+\theta\epsilon \mathbb{Y}_{tX\Psi}^{\epsilon}(T)\pig)\mathbb{Y}_{tX\Psi}^{\epsilon}(T) +D_{X}^{2}F_{T}\pig((\mathbb{Y}_{tX}(T)+\theta\epsilon \mathbb{Y}_{tX\Psi}^{\epsilon}(T)) \otimes  m\pig)(\mathbb{Y}_{tX\Psi}^{\epsilon}(T))d\theta\)}\nonumber\\
&+\int^T_s\int_{0}^{1}l_{xx}\pig(\mathbb{Y}_{tX}(\tau)+\theta\epsilon \mathbb{Y}_{tX\Psi}^{\epsilon}(\tau),u_{tX}(\tau) \pig)\mathbb{Y}_{tX\Psi}^{\epsilon}(\tau)\nonumber\\
&\pushright{+l_{xv}\pig(\mathbb{Y}_{t,X+\epsilon\mathcal{X}}(\tau),u_{tX}(\tau)+\theta\epsilon u_{tX\Psi}^{\epsilon}(\tau)\pig)u_{tX\Psi}^{\epsilon}(\tau)d\theta d\tau}\h{-1pt}\nonumber\\
&+\int^T_s\int_{0}^{1}D_{X}^{2}F\pig((\mathbb{Y}_{tX}(\tau)+\theta\epsilon \mathbb{Y}_{tX\Psi}^{\epsilon}(\tau)) \otimes  m\pig)(\mathbb{Y}_{tX\Psi}^{\epsilon}(\tau))d\theta-\sum_{j=1}^n r_{tX\mathcal{X},j}^{\epsilon}(\tau)dw_{j}(\tau)d\tau;
\label{finite diff. backward}
\end{empheq}}
subject to
\begin{equation}
\int_{0}^{1}l_{vx}\pig(\mathbb{Y}_{tX}(s)+\theta\epsilon \mathbb{Y}_{tX\Psi}^{\epsilon}(s),u_{tX}(s)\pig)\mathbb{Y}_{tX\Psi}^{\epsilon}(s)+l_{vv}\pig(\mathbb{Y}_{t,X+\epsilon\mathcal{X}}(s) ,u_{tX}(s)+\theta\epsilon u_{tX\Psi}^{\epsilon}(s)\pig)
u_{tX\Psi}^{\epsilon}(s)d\theta+\mathbb{Z}_{tX\Psi}^{\epsilon}(s)=0.
\label{1st order finite diff.}
\end{equation}

We then consider the subsequences of $\mathbb{Y}_{tX\Psi}^{\epsilon}(s)$, $\mathbb{Z}_{tX\Psi}^{\epsilon}(s)$, $u_{tX\Psi}^{\epsilon}(s)$, 
$\mathbbm{r}_{tX\mathcal{X},j}^{\epsilon}(s)$
which converge weakly to $\mathcal{U}_{tX\mathcal{X}}(s),\mathcal{Y}_{tX\mathcal{X}}(s),\mathcal{Z}_{tX\mathcal{X}}(s),\mathcal{R}_{tX\mathcal{X},j}(s)$, in $L_{\mathcal{W}_{tX\mathcal{X}}}^{2}(t,T;\mathcal{H}_{m})$
It follows immediately that 

\begin{equation}
\mathcal{Y}_{tX\mathcal{X}}(s)=\mathcal{X}+\int_{t}^{s}\mathcal{U}_{tX\mathcal{X}}(\tau)d\tau\label{eq:8-71}
\end{equation}
 Define 

\begin{align*}
L_{tX\mathcal{X}}^{\epsilon}(s)=\int_{0}^{1}l_{vx}\pig(\mathbb{Y}_{tX}(s)+\theta\epsilon \Delta^\epsilon_\Psi \mathbb{Y}_{tX} (s),
u_{tX}(s)&+\theta\epsilon \Delta^\epsilon_\Psi u_{tX}(s)\pig)\Delta^\epsilon_\Psi \mathbb{Y}_{tX} (s)\\
&+l_{vv}\pig(\mathbb{Y}_{tX}(s)+\theta\epsilon \Delta^\epsilon_\Psi \mathbb{Y}_{tX} (s),u_{tX}(s)+\theta\epsilon \Delta^\epsilon_\Psi u_{tX}(s)\pig)\Delta^\epsilon_\Psi u_{tX}(s)d\theta
\end{align*}
We want to show that as $\epsilon \to 0$,

\begin{equation}
L_{tX\mathcal{X}}^{\epsilon}(s)\longrightarrow L_{tX\mathcal{X}}(s)=l_{vx}\pig(\mathbb{Y}_{tX}(s),u_{tX}(s)\pig)\mathcal{Y}_{tX\mathcal{X}}(s)+l_{vv}\pig(\mathbb{Y}_{tX}(s),u_{tX}(s)\pig)\mathcal{U}_{tX\mathcal{X}}(s)\label{eq:8-72}
\end{equation}
in $L_{\mathcal{W}_{tX\mathcal{X}}}^{2}(t,T;\mathcal{H}_{m})$ weakly.
Let us take an arbitrary $\Gamma_{tX\mathcal{X}}(s)$ in $L_{\mathcal{W}_{tX\mathcal{X}}}^{2}(t,T;\mathcal{H}_{m})$ and consider
\begin{align*}
\mbox{\fontsize{10.2}{10}\selectfont\(\displaystyle\int_{t}^{T}\pig\langle L_{tX\mathcal{X}}^{\epsilon}(s),\Gamma_{tX\mathcal{X}}(s)\pigr\rangle_{\mathcal{H}_m} ds\)}
=\:&\mbox{\fontsize{10.2}{10}\selectfont\(
\displaystyle\int_{t}^{T}\left\langle
\Delta^\epsilon_\Psi \mathbb{Y}_{tX} (s),
\displaystyle\int_{0}^{1}l_{xv}\pig(\mathbb{Y}_{tX}(s)+\theta\epsilon \Delta^\epsilon_\Psi \mathbb{Y}_{tX} (s),u_{tX}(s)+\theta\epsilon \Delta^\epsilon_\Psi u_{tX}(s)\pig)\Gamma_{tX\mathcal{X}}(s)d\theta\right\rangle_{\mathcal{H}_m} ds\)}\\
&+\mbox{\fontsize{10}{10}\selectfont\(\displaystyle\int_{t}^{T}\left\langle \Delta^\epsilon_\Psi u_{tX}(s),
\int_{0}^{1}l_{vv}\pig(\mathbb{Y}_{tX}(s)+\theta\epsilon \Delta^\epsilon_\Psi \mathbb{Y}_{tX} (s),u_{tX}(s)+\theta\epsilon \Delta^\epsilon_\Psi u_{tX}(s)\pig)\Gamma_{tX\mathcal{X}}(s)d\theta\right\rangle_{\mathcal{H}_m}ds\)}
\end{align*}
Note that $\mathbb{Y}_{tX}(s)+\theta\epsilon \mathbb{Y}_{tX\Psi}^{\epsilon}(s)\longrightarrow \mathbb{Y}_{tX}(s)$ and $u_{tX}(s)+\theta\epsilon u_{tX\Psi}^{\epsilon}(s)\longrightarrow u_{tX}(s)$
in $L_{\mathcal{W}_{tX\mathcal{X}}}^{2}(t,T;\mathcal{H}_{m})$ as $\epsilon \to 0$. Moreover, second-order derivatives of $l$ are continuous and bounded by Assumptions (\ref{assumption, bdd of l, lx, lv}),
(\ref{assumption, cts of lxx, lvx, lvv, hxx}), it follows by dominated Lebesgue convergence theorem that 

\[
\mbox{\fontsize{10.5}{10}\selectfont\(
\displaystyle\int_{0}^{1}l_{xv}\pig(\mathbb{Y}_{tX}(s)+\theta\epsilon \Delta^\epsilon_\Psi \mathbb{Y}_{tX} (s),u_{tX}(s)+\theta\epsilon \Delta^\epsilon_\Psi u_{tX}(s)\pig)\Gamma_{tX\mathcal{X}}(s)d\theta
\longrightarrow l_{xv}\pig(\mathbb{Y}_{tX}(s),u_{tX}(s)\pig)\Gamma_{tX\mathcal{X}}(s)\)}
\h{5pt} \text{in}\;L_{\mathcal{W}_{tX\mathcal{X}}}^{2}(t,T;\mathcal{H}_{m})
\]
and similar argument for $\displaystyle\int_{0}^{1}l_{vv}\pig(\mathbb{Y}_{tX}(s)+\theta\epsilon \Delta^\epsilon_\Psi \mathbb{Y}_{tX} (s),u_{tX}(s)+\theta\epsilon \Delta^\epsilon_\Psi u_{tX}(s)\pig)\Gamma_{tX\mathcal{X}}(s)d\theta$. Therefore 

\begin{align*}
\int_{t}^{T}\pig\langle L_{tX\mathcal{X}}^{\epsilon}(s),\Gamma_{tX\mathcal{X}}(s)
\pigr\rangle_{\mathcal{H}_m} ds
\longrightarrow\:&
\int_{t}^{T}\pig\langle
\mathcal{Y}_{tX\mathcal{X}}(s),l_{xv}\pig(\mathbb{Y}_{tX}(s),u_{tX}(s)\pig)\Gamma_{tX\mathcal{X}}(s)\pigr\rangle_{\mathcal{H}_m}  ds\\
&+\int_{t}^{T}\pig\langle\mathcal{U}_{tX\mathcal{X}}(s),l_{vv}\pig(\mathbb{Y}_{tX}(s),u_{tX}(s)\pig)\Gamma_{tX\mathcal{X}}(s)
\pigr\rangle_{\mathcal{H}_m} ds.
\end{align*}
in $L_{\mathcal{W}_{tX\mathcal{X}}}^{2}(t,T;\mathcal{H}_{m})$ weakly. Transposing the matrices $l_{xv}$ and $l_{vv}$, we obtain 

\[
\int_{t}^{T}\pig\langle
L_{tX\mathcal{X}}^{\epsilon}(s),\Gamma_{tX\mathcal{X}}(s)\pigr\rangle_{\mathcal{H}_m} ds
\longrightarrow \int_{t}^{T}\pig\langle L_{tX\mathcal{X}}(s),\Gamma_{tX\mathcal{X}}(s)\pigr\rangle_{\mathcal{H}_m} ds
\]
in $L_{\mathcal{W}_{tX\mathcal{X}}}^{2}(t,T;\mathcal{H}_{m})$ weakly, and since $\Gamma_{tX\mathcal{X}}(s)$ is arbitrary, the result (\ref{eq:8-72})
follows. So (\ref{1st order finite diff.}) yields 

\begin{equation}
l_{vx}\pig(\mathbb{Y}_{tX}(s),u_{tX}(s)\pig)\mathcal{Y}_{tX\mathcal{X}}(s)+l_{vv}\pig(\mathbb{Y}_{tX}(s),u_{tX}(s)\pig)\mathcal{U}_{tX\mathcal{X}}(s)+\mathcal{Z}_{tX\mathcal{X}}(s)=0\label{eq:8-73}
\end{equation}
A similar reasoning implies 

\begin{equation}
\begin{aligned}
&\mbox{\fontsize{10}{10}\selectfont\(\displaystyle\int_{0}^{1}l_{xx}\pig(\mathbb{Y}_{tX}(s)+\theta\epsilon \Delta^\epsilon_\Psi \mathbb{Y}_{tX} (s),u_{tX}(s)+\theta\epsilon \Delta^\epsilon_\Psi u_{tX}(s)\pig)
\Delta^\epsilon_\Psi \mathbb{Y}_{tX} (s)+l_{xv}\pig(\mathbb{Y}_{tX}(s)+\theta\epsilon \Delta^\epsilon_\Psi \mathbb{Y}_{tX} (s),u_{tX}(s)+\theta\epsilon \Delta^\epsilon_\Psi u_{tX}(s)\pig)\Delta^\epsilon_\Psi u_{tX}(s)d\theta\)}\\
&\longrightarrow 
l_{xx}\pig(\mathbb{Y}_{tX}(s),u_{tX}(s)\pig)\mathcal{Y}_{tX\mathcal{X}}(s)+l_{vx}\pig(\mathbb{Y}_{tX}(s),u_{tX}(s)\pig)\mathcal{U}_{tX\mathcal{X}}(s)
\end{aligned}
\label{eq:8-74}
\end{equation}
in $L_{\mathcal{W}_{tX\mathcal{X}}}^{2}(t,T;\mathcal{H}_{m})$ weakly. Now, thanks to Assumption (\ref{assumption, cts of D^2F}), we can also state that 

\begin{equation}
\int_{0}^{1}D_{X}^{2}F\pig((\mathbb{Y}_{tX}(s)+\theta\epsilon \Delta^\epsilon_\Psi \mathbb{Y}_{tX} (s)) \otimes  m\pig)(\Delta^\epsilon_\Psi \mathbb{Y}_{tX} (s))d\theta
\longrightarrow D_{X}^{2}F\pig(\mathbb{Y}_{tX}(s) \otimes  m\pig)(\mathcal{Y}_{tX\mathcal{X}}(s))
\label{eq:8-75}
\end{equation}
in $L_{\mathcal{W}_{tX\mathcal{X}}}^{2}(t,T;\mathcal{H}_{m})$ weakly. Also 

\begin{equation}
\begin{aligned}
&\int_{0}^{1}h_{xx}\pig(\mathbb{Y}_{tX}(T)+\theta\epsilon \Delta^\epsilon_\Psi \mathbb{Y}_{tX} (T)\pig)\Delta^\epsilon_\Psi \mathbb{Y}_{tX} (T)d\theta+\int_{0}^{1}D_{X}^{2}F_{T}((\mathbb{Y}_{tX}(T)+\theta\epsilon \Delta^\epsilon_\Psi \mathbb{Y}_{tX} (T)) \otimes  m)(\Delta^\epsilon_\Psi \mathbb{Y}_{tX} (T))d\theta\\
&\longrightarrow
h_{xx}\pig(\mathbb{Y}_{tX}(T)\pig)\mathcal{Y}_{tX\mathcal{X}}(T)
+D_{X}^{2}F_{T}\pig(\mathbb{Y}_{tX}(T) \otimes  m\pig)(\mathcal{Y}_{tX\mathcal{X}}(T))
\end{aligned}
\label{eq:8-76}
\end{equation}
weakly in the subspace of \ensuremath{\mathcal{H}_{m}} consisting of variables \ensuremath{\mathcal{W}_{tX\mathcal{X}}^{T}}-measurable. From relation (\ref{finite diff. backward}), by conditioning and letting $\epsilon\rightarrow0$
we obtain 

\begin{equation}
\begin{aligned}
\mathcal{Z}_{tX\mathcal{X}}(s)
=\mathbb{E}\bigg[\:&\int_{s}^{T}l_{xx}(\mathbb{Y}_{tX}(\tau),u_{tX}(\tau))\mathcal{Y}_{tX\mathcal{X}}(\tau)+l_{vx}(\mathbb{Y}_{tX}(\tau),u_{tX}(\tau))\mathcal{U}_{tX\mathcal{X}}(\tau)
+D_{X}^{2}F(\mathbb{Y}_{tX}(\tau) \otimes  m)(\mathcal{Y}_{tX\mathcal{X}}(\tau))d\tau\\
&+h_{xx}(\mathbb{Y}_{tX}(T))\mathcal{Y}_{tX\mathcal{X}}(T)+D_{X}^{2}F_{T}(\mathbb{Y}_{tX}(T) \otimes  m)(\mathcal{Y}_{tX\mathcal{X}}(T))\bigg|\ensuremath{\mathcal{W}_{tX\mathcal{X}}^{s}}\bigg].
\end{aligned}
\label{eq:8-77}
\end{equation}
This implies that $\mathcal{Z}_{tX\mathcal{X}}(s)$ is an It\^o process satisfying (\ref{backward eq Z_tX psi, no expand in u}). To check that $\mathcal{R}_{tX\mathcal{X},j}(s)$ is the weak limit
of $\mathbbm{r}_{tX\mathcal{X},j}^{\epsilon}(s)$ in $L_{\mathcal{W}_{tX\mathcal{X}}}^{2}(t,T;\mathcal{H}_{m}),$ we consider an It\^o process in $L_{\mathcal{W}_{tX\mathcal{X}}}^{2}(t,T;\mathcal{H}_{m})$
of the form 

\begin{equation}
d\Gamma_{tX\mathcal{X}}(s)=a_{tX\mathcal{X}}(s)ds+\sum_{j=1}^n b_{tX\mathcal{X},j}(s)dw_{j}(s)
\h{1pt},\h{10pt}
\Gamma_{tX\mathcal{X}}(0)=0
\label{eq:8-79}
\end{equation}
Computing $\dfrac{d}{ds}\pig\langle \Delta^\epsilon_\Psi \mathbb{Z}_{tX} (s),\Gamma_{tX\mathcal{X}}(s)
\pigr\rangle_{\mathcal{H}_m}$,
$\dfrac{d}{ds}\pig\langle
\mathbb{Z}_{tX\mathcal{X}}(s),\Gamma_{tX\mathcal{X}}(s)\pigr\rangle_{\mathcal{H}_m}$,
and integrating, we check easily (we omit details) that 

\[
\sum_{j=1}^{n}\int_{t}^{T}\pig\langle \mathbbm{r}_{tX\mathcal{X},j}^{\epsilon}(s),b_{tX\mathcal{X},j}(s)\pigr\rangle_{\mathcal{H}_m}ds
\longrightarrow\sum_{j=1}^{n}\int_{t}^{T}
\pig\langle\mathcal{R}_{tX\mathcal{X},j}(s),b_{tX\mathcal{X},j}(s)
\pigr\rangle_{\mathcal{H}_m}ds.
\]

Since $b_{tX\mathcal{X},j}(s)$ are arbitrary in $L_{\mathcal{W}_{tX\mathcal{X}}}^{2}(t,T;\mathcal{H}_{m})$, the weak convergence of $\mathbbm{r}_{tX\mathcal{X},j}^{\epsilon}(s)$ to $\mathcal{R}_{tX\mathcal{X},j}(s)$ in $L_{\mathcal{W}_{tX\mathcal{X}}}^{2}(t,T;\mathcal{H}_{m})$ is obtained. 

\noindent {\bf Part 2. Existence of second-order G\^ateaux derivative:}\\
Since we have also 

\begin{align*}
\pig\langle\mathcal{X},\Delta^\epsilon_\Psi \mathbb{Z}_{tX} (t)
\pigr\rangle_{\mathcal{H}_m}
=\Bigg\langle \mathcal{X},
\int_{0}^{1}\:& l_{xx}\pig(\mathbb{Y}_{tX}(s)+\theta\epsilon \Delta^\epsilon_\Psi \mathbb{Y}_{tX} (s),u_{tX}(s)+\theta\epsilon \Delta^\epsilon_\Psi u_{tX}(s)\pig)\Delta^\epsilon_\Psi \mathbb{Y}_{tX} (s)\\
&+l_{vx}\pig(\mathbb{Y}_{tX}(s)+\theta\epsilon \Delta^\epsilon_\Psi \mathbb{Y}_{tX} (s),u_{tX}(s)+\theta\epsilon \Delta^\epsilon_\Psi u_{tX}(s)\pig)\Delta^\epsilon_\Psi u_{tX}(s)\\
&+D_{X}^{2}F\pig((\mathbb{Y}_{tX}(s)+\theta\epsilon \Delta^\epsilon_\Psi \mathbb{Y}_{tX} (s)) \otimes  m\pig)(\Delta^\epsilon_\Psi \mathbb{Y}_{tX} (s))d\theta\:\Bigg\rangle_{\mathcal{H}_m}.?
\end{align*}

Due to the continuities of $l_{xx}$, $l_{vx}$ and $D^2_X F$ in assumptions (\ref{assumption, cts of lxx, lvx, lvv, hxx}), (\ref{assumption, cts of D^2F}), we can state that 

\begin{equation}
\pig\langle \mathcal{X},\Delta^\epsilon_\Psi \mathbb{Z}_{tX} (t)\pigr\rangle_{\mathcal{H}_m}
\longrightarrow\pig\langle\mathcal{X},
D\mathbb{Z}_{tX}^\Psi(t)\pigr\rangle_{\mathcal{H}_m}.(?missed terminal in the above)
\label{eq:8-80}
\end{equation}}

The proof of this proposition is a modification of the proof for Theorem 5.2 in \cite{BGY20} and Lemma \ref{lem, Existence of J flow}. By using (\ref{DX^2 V = DZ}), we need to prove that as $k \to \infty$,
\begin{equation}
D\mathbb{Z}_{t_k X_{k}}^{\Psi_{k}}(t_{k})
\longrightarrow
D\mathbb{Z}_{tX}^{\Psi}(t)
\h{10pt}
\text{ in $\mathcal{H}_m$.}
\label{eq:8-86}
\end{equation}
Fixing $s \in (t,T]$, we can assume that $t_{k}<s$ for large enough $k$. From (\ref{bdd Y, Z, u, r, epsilon}) and Lemma \ref{lem, Existence of Frechet derivatives}, the linearity of the system (\ref{J flow of FBSDE}) deduces that
\begin{equation}
\pig\|D\mathbb{Z}_{t_k X_{k}}^{\Psi_{k}}(t_{k})
-D\mathbb{Z}_{t_k X_{k}}^{\Psi}(t_{k})
\pigr\|_{\mathcal{H}_m}
\leq C'_4\|\Psi_k -\Psi\|_{\mathcal{H}_m},
\label{bdd of DU DY DZ, approx of J flow}
\end{equation}
where $C'_4$ is independent of $k$. Note that $D\mathbb{Z}_{t_k X_{k}}^{\Psi_{k}}(t_{k})
-D\mathbb{Z}_{tX}^{\Psi}(t)
=\pig[D\mathbb{Z}_{t_k X_{k}}^{\Psi_{k}}(t_{k})
-D\mathbb{Z}_{t_kX_k}^{\Psi}(t_k)\pig]
+\pig[D\mathbb{Z}_{t_k X_{k}}^{\Psi}(t_{k})
-D\mathbb{Z}_{tX}^{\Psi}(t)\pig]
$, in order to prove the claim (\ref{eq:8-86}), it is sufficient to prove that the second term converges to $0$, that is, in the $\mathcal{H}_m$-sense,
\begin{equation}
D\mathbb{Z}_{t_k X_{k}}^{\Psi}(t_{k})
\longrightarrow
D\mathbb{Z}_{tX}^{\Psi}(t)
\h{10pt}
\text{ in $\mathcal{H}_m$.}
\label{eq:8-88}
\end{equation}
From (\ref{bdd Y, Z, u, r, epsilon}), (\ref{est. |Ys-X|^2}), and the flow property (\ref{flow property}), we obtain that, for $s \in [t_k,T]$,
\begin{align}
\big\|\mathbb{Y}_{t_{k}X_{k}}(s)-\mathbb{Y}_{tX}(s)\big\|_{\mathcal{H}_m}
&\leq\big\|\mathbb{Y}_{t_{k}X_{k}}(s)
-\mathbb{Y}_{t_{k}X}(s)\big\|_{\mathcal{H}_m}
+\big\|\mathbb{Y}_{t_{k}X}(s)
-\mathbb{Y}_{t_k\mathbb{Y}_{tX}(t_k)}(s)\big\|_{\mathcal{H}_m}\nonumber\\
&\leq A_{14}\pig[\big\|X_{k}-X\big\|_{\mathcal{H}_m} + (t_k-t)^{1/2} \pig]
\label{eq:8-110}
\end{align}
and similar estimates hold for $\big\|\mathbb{Z}_{t_{k}X_{k}}(s)-\mathbb{Z}_{tX}(s)\big\|_{\mathcal{H}_m}$ and  $\big\|u_{t_{k}X_{k}}(s)-u_{tX}(s)\big\|_{\mathcal{H}_m}$ by arguing in the same way as in (\ref{eq:8-110}), with $A_{14}$ independent of $k$. In order to facilitate the comparison against $X_k$'s and $\Psi_k$'s, we introduce the $\sigma$-algebras, for $s\in(t,T]$,
$$\mathcal{\widetilde{W}}_{tX\Psi}^{s}
:=\bigvee_{\{j\in \mathbb{N}|t_{j}<s\}}
\left(\mathcal{W}_{t_{j}X_{j}\Psi_j}^{s}\bigvee\mathcal{W}_{t}^{t_{j}}\right).$$ Note that $X$ is $\mathcal{W}_{tX\Psi}^{s}$-measurable, so $X$ is also $\widetilde{\mathcal{W}}_{tX\Psi}^{s}$-measurable as $\mathcal{W}_{tX\Psi}^{s}\subset\mathcal{\widetilde{W}}_{tX\Psi}^{s}$. 
\mycomment{
To extend the definition of the processes $
D \mathbb{Y}_{t_kX_{k}}^\Psi(\tau)$, $D \mathbb{Z}_{t_kX_{k}}^\Psi(\tau)$, 
$D u_{t_kX_{k}}^\Psi(\tau)$, for $\tau\in[t,t_{k})$, we define 

\begin{equation}
D \mathbb{Y}_{t_kX_{k}}^\Psi(\tau)=\Psi,\h{10pt} 
D u_{t_kX_{k}}^\Psi(\tau)=0,\h{10pt} 
D \mathbb{Z}_{t_kX_{k}}^\Psi(\tau)
=\mathbb{E}\pig[D \mathbb{Z}_{t_kX_{k}}^\Psi(t_{k})\pig|\widetilde{\mathcal{W}}_{tX\Psi}^{\tau}\pig].
\label{eq:8-15}
\end{equation}
These extensions are compatible with the original definitions in (\ref{J flow of FBSDE}) since, for $\tau \in [t,t_k)$, 
$$ D \mathbb{Y}_{t_kX_{k}}^\Psi(\tau) =\Psi = \Psi + \int_t^\tau D u_{t_kX_{k}}^\Psi(r)dr 
\h{10pt} \h{20pt}\text{and}$$
{\small
\begin{align*}
&D \mathbb{Z}_{t_kX_{k}}^\Psi(\tau)\\
&=\mathbb{E}\pig[D \mathbb{Z}_{t_kX_{k}}^\Psi(t_{k})\pig|\widetilde{\mathcal{W}}_{tX\Psi}^{\tau}\pig]\\
&=\mathbb{E}\left[D \mathbb{Z}_{t_kX_{k}}^\Psi(t_{k})
-\sum^n_{j=1}\int^{t_k}_\tau D \mathbbm{r}_{t_kX_k,j}^\Psi(r)dw_j(r)
\Big|\widetilde{\mathcal{W}}_{tX\Psi}^{\tau}\right]\\
&=h_{xx}(\mathbb{Y}_{t_kX_k}(T))D \mathbb{Y}_{t_kX_k}^\Psi (T)
+D_{X}^2F_{T}(\mathbb{Y}_{t_kX_k}(T)  \otimes  m)(D \mathbb{Y}_{t_kX_k}^\Psi (T))\\
&\h{10pt}+\displaystyle\int^T_s\bigg\{l_{xx}(\mathbb{Y}_{t_kX_k}(\tau),u_{t_kX_k}(\tau))D \mathbb{Y}_{t_kX_k}^\Psi (\tau)
+l_{xv}(\mathbb{Y}_{t_kX_k}(\tau),u_{t_kX_k}(\tau))
D u_{t_kX_k}^\Psi (\tau)
+D^2_{X}F(\mathbb{Y}_{t_kX_k}(\tau)  \otimes  m)\pig(
D\mathbb{Y}_{t_kX_k}^\Psi (\tau)\pig)\bigg\}d\tau \h{-10pt}\\
&\h{10pt}-\int^T_s\sum_{j=1}^{n}D\mathbbm{r}^\Psi_{t_kX_k,j}(\tau)dw_{j}(\tau)
\end{align*}}}
For simplicity, we denote the respective processes $
D \mathbb{Y}_{t_kX_{k}}^\Psi(s)$, $D \mathbb{Z}_{t_kX_{k}}^\Psi(s)$, 
$D u_{t_kX_{k}}^\Psi(s)$, $D \mathbbm{r}_{t_kX_{k}}^\Psi(s)$ by $D\mathbb{Y}_{k}(s)$, $D\mathbb{Z}_{k}(s)$, $Du_{k}(s)$, $D\mathbbm{r}_{k}(s)$. Fix a $\tau^* \in (t,T]$, no matter how close to $t$, we can find $N^* \in \mathbb{N}$ large enough such that $t_k<\tau^*$ for any $k>N^*$. The processes $D\mathbb{Y}_{k}(s)$, $Du_{k}(s)$ and
$D\mathbb{Z}_{k}(s)$ remain bounded in $L_{\widetilde{\mathcal{W}}_{tX\Psi}}^{\infty}(t_k,T;\mathcal{H}_{m})$, and $D\mathbbm{r}_{k,j}(s)$ remains bounded in $L_{\widetilde{\mathcal{W}}_{tX\Psi}}^{2}(t_k,T;\mathcal{H}_{m})$, due to (\ref{bdd Y, Z, u, r, epsilon}) (the bounds hold no matter what the filtration is chosen). 

\noindent {\bf Step 1. Weak Convergence:}\\
To prove the convergence in (\ref{eq:8-88}), we suppose by contradiction that there is a subsequence of $k$, namely, $k_l$, such that $D\mathbb{Z}_{k_l}$ diverges from $D\mathbb{Z}_{tX}^\Psi$ in $\mathcal{H}_m$ as $l \to \infty$; if there were a subsequence of the subsequence $k_l$, namely, $k'_j$, such that $D\mathbb{Z}_{k'_j}$ converges to $D\mathbb{Z}_{tX}$ in $\mathcal{H}_m$ as $j\to \infty$, then the assumption is violated. Therefore, without loss of generality, in light of Banach-Alaoglu theorem, we can pick subsequences of the processes, without relabeling, such that they converge weakly to $\mathscr{D}\mathbb{Y}_{\infty}(s)$, $\mathscr{D}u_{\infty}(s)$,
$\mathscr{D}\mathbb{Z}_{\infty}(s)$,  $\mathscr{D}\mathbbm{r}_{\infty,j}(s)$ in $L_{\widetilde{\mathcal{W}}_{tX\Psi}}^{2}(\tau^*,T;\mathcal{H}_{m})$, respectively, and then it is sufficient to prove the strong convergence in $\mathcal{H}_m$ and show that  $\mathscr{D}\mathbb{Y}_{\infty}(s) = D\mathbb{Y}^\Psi_{tX}(s)$, $\mathscr{D}\mathbb{Z}_{\infty}(s) = D\mathbb{Z}^\Psi_{tX}(s)$,
$\mathscr{D}u_{\infty}(s) = Du^\Psi_{tX}(s)$, $\mathscr{D}\mathbbm{r}_{\infty,j}(s) = D\mathbbm{r}^\Psi_{tX,j}(s)$.

Then, using the continuity of $u(Y,Z)$ and the strong convergences of $\mathbb{Y}_{t_k X_k}$, $\mathbb{Z}_{t_k X_k}$ in (\ref{eq:8-110}), we have
$$
\mathscr{D}\mathbb{Y}_{\infty}(s)\:=\Psi+\int_{t}^{s}
\diff_y  u(\mathbb{Y}_{tX}(\tau),\mathbb{Z}_{tX}(\tau)) \mathscr{D}\mathbb{Y}_{\infty}(\tau) 
+\diff_z  u(\mathbb{Y}_{tX}(\tau),\mathbb{Z}_{tX}(\tau)) \mathscr{D}\mathbb{Z}_{\infty}(\tau) 
d\tau;$$
indeed, for instance,
\begin{align*}
&\h{-10pt}\int^s_{\tau^*}
\diff_y  u(\mathbb{Y}_{t_kX_k}(\tau),\mathbb{Z}_{t_kX_k}(\tau)) D\mathbb{Y}_{k}(\tau)d\tau\\
=\:&\int^s_{\tau^*} 
\Big[\diff_y  u(\mathbb{Y}_{t_kX_k}(\tau),\mathbb{Z}_{t_kX_k}(\tau))
-\diff_y  u(\mathbb{Y}_{tX}(\tau),\mathbb{Z}_{tX}(\tau))\Big]D\mathbb{Y}_{k}(\tau)
+\diff_y  u(\mathbb{Y}_{tX}(\tau),\mathbb{Z}_{tX}(\tau)) D\mathbb{Y}_{k}(\tau)d\tau,
\end{align*}
where the first term strongly converges to zero in $L_{\widetilde{\mathcal{W}}_{tX\Psi}}^{2}(\tau^*,T;\mathcal{H}_{m})$ and the second term  converges weakly to $\int^s_{\tau^*} \diff_y  u(\mathbb{Y}_{tX}(\tau),\mathbb{Z}_{tX}(\tau))\mathscr{D}\mathbb{Y}_{\infty}(\tau)d\tau$. 
Next, we aim to show that
\begin{equation}
\h{-10pt}\begin{aligned}
\mathscr{D}\mathbb{Z}_{\infty}&(s)\\ =\mathbb{E}\Bigg[&h_{xx}(\mathbb{Y}_{tX}(T))\mathscr{D}\mathbb{Y}_{\infty} (T)
+D_{X}^2F_{T}(\mathbb{Y}_{tX}(T)  \otimes  m)(\mathscr{D}\mathbb{Y}_{\infty} (T))\\
&\h{-1pt}+\displaystyle\int^T_s 
l_{xv}(\mathbb{Y}_{tX}(\tau),u_{tX}(\tau))
\mathscr{D}u_{\infty} (\tau)
+\pig[l_{xx}(\mathbb{Y}_{tX}(\tau),u_{tX}(\tau))
+D^2_{X}F(\mathbb{Y}_{tX}(\tau)  \otimes  m)\pig]
\pig(
\mathscr{D}\mathbb{Y}_{\infty}  (\tau)\pig) d\tau 
\Bigg|\widetilde{\mathcal{W}}_{tX\Psi}^{s}\Bigg].\h{-30pt}
\end{aligned}
\label{DZ_infty}
\end{equation}
\mycomment{
{\small
\begin{equation}
\h{-10pt}\begin{aligned}
\mathscr{D}\mathbb{Z}_{\infty}(s)=\:& h_{xx}(\mathbb{Y}_{tX}(T))\mathscr{D}\mathbb{Y}_{\infty} (T)
+D_{X}^2F_{T}(\mathbb{Y}_{tX}(T)  \otimes  m)(\mathscr{D}\mathbb{Y}_{\infty} (T))\\
&\h{-7pt}+\displaystyle\int^T_s
l_{xv}(\mathbb{Y}_{tX}(\tau),u_{tX}(\tau))
\mathscr{D}u_{\infty} (\tau)
+\pig[l_{xx}(\mathbb{Y}_{tX}(\tau),u_{tX}(\tau))
+D^2_{X}F(\mathbb{Y}_{tX}(\tau)  \otimes  m)\pig]\pig(
\mathscr{D}\mathbb{Y}_{\infty}  (\tau)\pig) d\tau \\
&\h{-7pt}-\sum^n_{j=1}\int^T_s \mathscr{D}\mathbbm{r}_{\infty,j}(\tau)dw_j(\tau).
\end{aligned}
\label{DZ_infty}
\end{equation}}}
Using the tower property, we obtain, for any bounded random vector $\varphi$ adapted to $\widetilde{\mathcal{W}}_{tX\Psi}^{s}$, 
{\footnotesize
\begin{align}
&\h{-10pt}\mathbb{E}\Big[ D\mathbb{Z}_{k}(s) \varphi\Big]\nonumber\\
=\:&\mathbb{E}\Bigg\{
\mathbb{E}\Bigg[\Big[h_{xx}(\mathbb{Y}_{t_kX_k}(T))
+D_{X}^2F_{T}(\mathbb{Y}_{t_kX_k}(T)  \otimes  m)\Big](D\mathbb{Y}_{k} (T))
\Bigg|\widetilde{\mathcal{W}}_{tX\Psi}^{s}\Bigg]\varphi\Bigg\}\nonumber\\
&+\mathbb{E}\Bigg\{\mathbb{E}\bigg[\int_s^T
l_{xv}(\mathbb{Y}_{t_{k}X_{k}}(\tau),u_{t_{k}X_{k}}(\tau))
Du_{k} (\tau)
+\Big[l_{xx}(\mathbb{Y}_{t_{k}X_{k}}(\tau),u_{t_{k}X_{k}}(\tau))
+D_{X}^{2}F(\mathbb{Y}_{t_{k}X_{k}}(\tau) \otimes  m)\Big](D\mathbb{Y}_k(\tau)) d\tau\Bigg|\widetilde{\mathcal{W}}_{tX\Psi}^{s}\bigg]\varphi\Bigg\}\nonumber\\
=\:&
\mathbb{E}\Bigg\{\Big[h_{xx}(\mathbb{Y}_{t_kX_k}(T))
+D_{X}^2F_{T}(\mathbb{Y}_{t_kX_k}(T)  \otimes  m)\Big](D\mathbb{Y}_{k} (T))\varphi\Bigg\}\nonumber\\
&+\mathbb{E} \bigg\{\int_s^T
l_{xv}(\mathbb{Y}_{t_{k}X_{k}}(\tau),u_{t_{k}X_{k}}(\tau))
Du_{k} (\tau)\varphi
+\Big[l_{xx}(\mathbb{Y}_{t_{k}X_{k}}(\tau),u_{t_{k}X_{k}}(\tau))
+D_{X}^{2}F(\mathbb{Y}_{t_{k}X_{k}}(\tau) \otimes  m)\Big](D\mathbb{Y}_k(\tau))\varphi d\tau \bigg\}\nonumber\\
=\:&
\mathbb{E}\Bigg\{\Big[h_{xx}(\mathbb{Y}_{t_kX_k}(T))
+D_{X}^2F_{T}(\mathbb{Y}_{t_kX_k}(T)  \otimes  m)
-h_{xx}(\mathbb{Y}_{tX}(T))
-D_{X}^2F_{T}(\mathbb{Y}_{tX}(T)  \otimes  m)\Big](D\mathbb{Y}_{k} (T))\varphi\Bigg\}\label{E(DZk)phi, line1}\\
&+\mathbb{E}\Bigg\{\Big[h_{xx}(\mathbb{Y}_{tX}(T))
+D_{X}^2F_{T}(\mathbb{Y}_{tX}(T)  \otimes  m)\Big](D\mathbb{Y}_{k} (T))\varphi\Bigg\}\label{E(DZk)phi, line2}\\
&+\mathbb{E} \left\{\int_s^T\Big[
l_{xv}(\mathbb{Y}_{t_{k}X_{k}}(\tau),u_{t_{k}X_{k}}(\tau))
-l_{xv}(\mathbb{Y}_{tX}(\tau),u_{tX}(\tau))
\Big]Du_{k} (\tau)\varphi d\tau \right\}
+\mathbb{E} \left\{\int_s^T\Big[
l_{xv}(\mathbb{Y}_{tX}(\tau),u_{tX}(\tau))
\Big]Du_{k} (\tau)\varphi d\tau \right\}\label{E(DZk)phi, line3}\\
&+\mathbb{E} \bigg\{\h{-3pt}\int_s^T\h{-5pt}
\Big\{\pig[l_{xx}(\mathbb{Y}_{t_{k}X_{k}}(\tau),u_{t_{k}X_{k}}(\tau))
-l_{xx}(\mathbb{Y}_{tX}(\tau),u_{tX}(\tau))\pig]
+\pig[D_{X}^{2}F(\mathbb{Y}_{t_{k}X_{k}}(\tau) \otimes  m)
-D_{X}^{2}F(\mathbb{Y}_{tX}(\tau) \otimes  m)\pig]\Big\}(D\mathbb{Y}_k(\tau))\varphi d\tau \bigg\}\label{E(DZk)phi, line4}\\
&+\mathbb{E} \bigg\{\int_s^T
\Big[l_{xx}(\mathbb{Y}_{tX}(\tau),u_{tX}(\tau))
+D_{X}^{2}F(\mathbb{Y}_{tX}(\tau) \otimes  m)\Big](D\mathbb{Y}_k(\tau))\varphi d\tau \bigg\},
\label{E(DZk)phi, line5}
\end{align}}
\h{-3.6pt}see Remark 3.7 and the proof of Proposition 5.2 in \cite{BGY20} for the change of conditioning $\sigma$-algebra in the above. The $\mathcal{H}_m$-boundedness of $D\mathbb{Y}_k$ in (\ref{bdd Y, Z, u, r, epsilon}), the continuities of the second-order derivatives of $l$, $h$, $F$ and $F_T$ in Assumptions \textbf{A(iv)}'s (\ref{assumption, cts of lxx, lvx, lvv, hxx}), \textbf{B(iii)}'s (\ref{assumption, cts of D^2F}) and \textbf{B(iv)}'s (\ref{assumption, cts of D^2F_T}), the strong convergences of $\mathbb{Y}_{t_{k}X_{k}}$ and $\mathbb{Z}_{t_{k}X_{k}}$ in (\ref{eq:8-110}) deduce the respective convergences to zero in, (1): line (\ref{E(DZk)phi, line1}), (2): first term in (\ref{E(DZk)phi, line3}), (3): line (\ref{E(DZk)phi, line4}), see more details in Step 1 of the proof of Lemma \ref{lem, Existence of J flow}. The term in (\ref{E(DZk)phi, line2}), second term in (\ref{E(DZk)phi, line3}), and (\ref{E(DZk)phi, line5}) weakly converge as linear functionals; therefore $D \mathbb{Z}_k(s)$ weakly converges to the right hand of (\ref{DZ_infty}) in $L_{\widetilde{\mathcal{W}}_{tX\Psi}}^{2}(\tau^*,T;\mathcal{H}_{m})$.
\mycomment{
Next, we aim to show that
{\small
\begin{equation}
\h{-10pt}\begin{aligned}
\mathscr{D}\mathbb{Z}_{\infty}(s)= \mathbb{E}&\Bigg[h_{xx}(\mathbb{Y}_{tX}(T))\mathscr{D}\mathbb{Y}_{\infty} (T)
+D_{X}^2F_{T}(\mathbb{Y}_{tX}(T)  \otimes  m)(\mathscr{D}\mathbb{Y}_{\infty} (T))\\
&\h{-7pt}+\displaystyle\int^T_sl_{xx}(\mathbb{Y}_{tX}(\tau),u_{tX}(\tau))\mathscr{D}\mathbb{Y}_{\infty} (\tau)
+l_{xv}(\mathbb{Y}_{tX}(\tau),u_{tX}(\tau))
\mathscr{D}u_{\infty} (\tau)
+D^2_{X}F(\mathbb{Y}_{tX}(\tau)  \otimes  m)\pig(
\mathscr{D}\mathbb{Y}_{\infty}  (\tau)\pig) d\tau 
\Bigg|\widetilde{\mathcal{W}}_{tX\Psi}^{s}\Bigg].\h{-30pt}
\end{aligned}
\label{DZ_infty}
\end{equation}}
Using the continuities of the second-order derivatives of $l$, $h$, $F$ and $F_T$ in Assumptions \textbf{A(iv)}'s (\ref{assumption, cts of lxx, lvx, lvv, hxx}), Assumption \textbf{B(iii)}'s (\ref{assumption, cts of D^2F}) and Assumption \textbf{B(iv)}'s (\ref{assumption, cts of D^2F_T}), the strong convergences of $\mathbb{Y}_{t_{k}X_{k}}$ and $\mathbb{Z}_{t_{k}X_{k}}$ in (\ref{eq:8-110}), we obtain, for any bounded random vector $\varphi$ adapted to $\widetilde{\mathcal{W}}_{tX\Psi}^{s}$,
{\small
\begin{align*}
&\mathbb{E}\Bigg\{\mathbb{E}\bigg[\int_s^T
l_{xx}(\mathbb{Y}_{t_{k}X_{k}}(\tau),u_{t_{k}X_{k}}(\tau))
D\mathbb{Y}_k(\tau)
+l_{xv}(\mathbb{Y}_{t_{k}X_{k}}(\tau),u_{t_{k}X_{k}}(\tau))
Du_{k} (\tau)+D_{X}^{2}F(\mathbb{Y}_{t_{k}X_{k}}(\tau) \otimes  m)(D\mathbb{Y}_k(\tau)) d\tau\Bigg|\widetilde{\mathcal{W}}_{tX\Psi}^{s}\bigg]\varphi\Bigg\}\\
&\longrightarrow 
\mathbb{E}\Bigg\{\mathbb{E}\bigg[\int_s^T l_{xx}(\mathbb{Y}_{tX}(\tau),u_{tX}(\tau))
\mathscr{D}\mathbb{Y}_{\infty} (\tau)
+l_{xv}(\mathbb{Y}_{tX}(\tau),u_{tX}(\tau))
\mathscr{D}u_{\infty} (\tau)
+D^2_{X}F(\mathbb{Y}_{tX}(\tau)  \otimes  m)\pig(
\mathscr{D}\mathbb{Y}_{\infty}  (\tau)\pig)d\tau\Bigg|\widetilde{\mathcal{W}}_{tX\Psi}^{s}\bigg]\varphi\Bigg\}.
\end{align*}}
 For instance, consider the integral
\begin{align*}
&\mathbb{E}\left\{\mathbb{E}\Bigg[
\int_s^T l_{xx}(\mathbb{Y}_{t_{k}X_{k}}(\tau),u_{t_{k}X_{k}}(\tau))
D\mathbb{Y}_k(\tau)d\tau
\Bigg|\widetilde{\mathcal{W}}_{tX\Psi}^{s}\Bigg]\varphi\right\}\\
&=\mathbb{E}\left\{\mathbb{E}\Bigg[\int_s^T\Big[l_{xx}(\mathbb{Y}_{t_{k}X_{k}}(\tau),u_{t_{k}X_{k}}(\tau))
-l_{xx}(\mathbb{Y}_{tX}(\tau),u_{tX}(\tau))
\Big]
D\mathbb{Y}_k(\tau)
+l_{xx}(\mathbb{Y}_{tX}(\tau),u_{tX}(\tau))
D\mathbb{Y}_k(\tau)d\tau\Bigg|\widetilde{\mathcal{W}}_{tX\Psi}^{s}\Bigg]\varphi\right\}\\
\end{align*}
where the first term strongly converges to zero in $\mathcal{H}_m$ due to the $\mathcal{H}_m$-boundedness of $D\mathbb{Y}_k$ in (\ref{bdd Y, Z, u, r, epsilon}), the continuity of $l_{xx}$ in Assumptions \textbf{A(iv)}'s (\ref{assumption, cts of lxx, lvx, lvv, hxx}), and the strong convergences of $\mathbb{Y}_{t_{k}X_{k}}$ and $\mathbb{Z}_{t_{k}X_{k}}$ in (\ref{eq:8-110}); and the second term  converges weakly to $\int_s^Tl_{xx}(\mathbb{Y}_{tX}(\tau),u_{tX}(\tau))
\mathscr{D}\mathbb{Y}_\infty(\tau) d\tau$ as a linear operator. Likewise, we also have the weak convergence of
{\footnotesize
$$
h_{xx}(\mathbb{Y}_{t_kX_k}(T))D\mathbb{Y}_{k} (T)
+D_{X}^2F_{T}(\mathbb{Y}_{t_kX_k}(T)  \otimes  m)(D\mathbb{Y}_{k} (T))
\longrightarrow
h_{xx}(\mathbb{Y}_{tX}(T))D\mathbb{Y}_{\infty} (T)
+D_{X}^2F_{T}(\mathbb{Y}_{tX}(T)  \otimes  m)(D\mathbb{Y}_{\infty} (T)), \h{5pt} \text{as $k \to \infty$,}
$$}
\h{-4pt}in $L^2_{\widetilde{\mathcal{W}}_{tX\Psi}}$. }
Hence, (\ref{DZ_infty}) follows. Necessarily, due to the uniqueness of the linear FBSDE system (\ref{J flow of FBSDE}), we see that $\mathscr{D}\mathbb{Y}_{\infty}(s)=D\mathbb{Y}_{tX}^\Psi(s)$, $\mathscr{D}\mathbb{Z}_{\infty}(s)=D\mathbb{Z}_{tX}^\Psi(s)$ and $\mathscr{D}u_\infty = Du^\Psi_{tX}(s)$.
We next show the strong convergences.

\noindent {\bf Step 2. Estimate of $\pig\langle D\mathbb{Z}_k(t_k),\Psi \pigr\rangle_{\mathcal{H}_m} - \pig\langle D\mathbb{Z}^\Psi_{tX}(t),\Psi \pigr\rangle_{\mathcal{H}_m}$:}\\
For any bounded test function $\Psi \in L^2_{\mathcal{W}_t}(\mathcal{H}_m)$, we see that
\begin{align}
\lim_{k \to \infty}\pig\langle D\mathbb{Z}_k(t_k),\Psi \pigr\rangle_{\mathcal{H}_m}
\h{-5pt}=\lim_{s \to t^+}\lim_{k \to \infty}\pig\langle D\mathbb{Z}_k(t_k)-D\mathbb{Z}_k(s),\Psi \pigr\rangle_{\mathcal{H}_m} 
\h{-5pt}+\lim_{s \to t^+}\lim_{k \to \infty}\pig\langle D\mathbb{Z}_k(s),\Psi \pigr\rangle_{\mathcal{H}_m}
\label{<DZk(tk), Psi> to <DZ(t), Psi>+Jk}
\end{align}
Define
\begin{align*}
\mathscr{J}_k :&= \pig\langle D\mathbb{Z}_k(t_k)-D\mathbb{Z}_k(s),\Psi \pigr\rangle_{\mathcal{H}_m} \\
&=\h{-3pt}\displaystyle\int^{s}_{t_k}\h{-3pt}
\bigg\langle l_{xv}(\mathbb{Y}_{t_{k}X_{k}}(\tau),u_{t_{k}X_{k}}(\tau))
Du_{k} (\tau)
+\Big[l_{xx}(\mathbb{Y}_{t_{k}X_{k}}(\tau),u_{t_{k}X_{k}}(\tau))
+D_{X}^{2}F(\mathbb{Y}_{t_{k}X_{k}}(\tau) \otimes  m)\Big](D\mathbb{Y}_k(\tau)),
\Psi
\bigg\rangle_{\mathcal{H}_m}\h{-5pt} d\tau.
\end{align*}
We estimate $\mathscr{J}_k$ by using Assumptions \textbf{A(ii)}'s (\ref{assumption, bdd of lxx, lvx, lvv}) and \textbf{B(ii)}'s (\ref{assumption, bdd of D^2F, D^2F_T}), 
\begin{align*}
\big|\mathscr{J}_k\big| \leq\:&\int^{s}_{t_k}\Big( c_l\big\|
Du_{k} (\tau)\bigr\|_{\mathcal{H}_m} 
+(c_l+c)\big\|D\mathbb{Y}_k(\tau) \big\|_{\mathcal{H}_m}\Big)\|\Psi\|_{\mathcal{H}_m} d\tau.
\end{align*}
Due to the $\mathcal{H}_m$-boundedness of $D\mathbb{Y}_k(\tau)$, $Du_{k} (\tau)$ and $\Psi$, we obtain that $\mathscr{J}_k  \longrightarrow 0$ as $t_k \to t^+$ and then $s \to t^+$. The first term on the right hand side of (\ref{<DZk(tk), Psi> to <DZ(t), Psi>+Jk}) is zero. By It\^o's lemma, \eqref{bdd Y, Z, u, r, epsilon}, Assumptions \textbf{A(ii)}'s (\ref{assumption, bdd of lxx, lvx, lvv}) and \textbf{B(ii)}'s (\ref{assumption, bdd of D^2F, D^2F_T}), the equation in (\ref{J flow of FBSDE}) shows that for $s_1>s_2$,
\begin{align*}
&\pig\|D \mathbb{Z}_k (s_1) - D \mathbb{Z}_k (s_2)\pigr\|_{\mathcal{H}_m}^2\\
&\leq 2\int^{s_1}_{s_2}
\Big\langle 
l_{vx}\pig(\mathbb{Y}_{t_k X_k}(\tau),u_{t_k X_k}(\tau)\pig)
D \mathbb{Y}_{k} (\tau)
+l_{xv}\pig(\mathbb{Y}_{t_k X_k}(\tau),u_{t_k X_k}(\tau)\pig)
D u_{k}(\tau),D \mathbb{Z}_k (s_1) - D \mathbb{Z}_k (\tau)\Big\rangle_{\mathcal{H}_m} d\tau\\
&\h{10pt}+2\int^{s_1}_{s_2}\h{-3pt}
\bigg\langle 
D_{X}^{2}F(\mathbb{Y}_{t_{k}X_{k}}(\tau) \otimes  m)(D\mathbb{Y}_k(\tau)),
D \mathbb{Z}_k (s_1) - D \mathbb{Z}_k (\tau)\bigg\rangle_{\mathcal{H}_m} \h{-6pt} d\tau\\
&\leq 2(s_1-s_2)C_4'(4c_l^2+c^2)\|\Psi\|_{\mathcal{H}_m}^2 
+2\int^{s_1}_{s_2}\h{-3pt}
\pig\|
D \mathbb{Z}_k (s_1) - D \mathbb{Z}_k (\tau)\pigr\|_{\mathcal{H}_m}^2 d\tau.
\end{align*}
By Gr\"onwall's inequality, we have
$$\pig\|D \mathbb{Z}_k (s_1) - D \mathbb{Z}_k (s_2)\pigr\|_{\mathcal{H}_m}^2
\leq 2(s_1-s_2)C_4'(4c_l^2+c^2)\|\Psi\|_{\mathcal{H}_m}^2
\Big(1+2(s_1-s_2)e^{2(s_1-s_2)}\Big).$$
The process $D\mathbb{Z}_k(s)$ is equi-continuous for $s\in [t_k,T]$ (independent of $k$) by studying its equation in (\ref{J flow of FBSDE}), together with the uniform boundedness of $D\mathbb{Z}_k(s)$ in (\ref{bdd Y, Z, u, r, epsilon}), Arzel\`a-Ascoli theorem implies that $\displaystyle\lim_{k \to \infty}\pig\langle D\mathbb{Z}_k(s),\Psi \pigr\rangle_{\mathcal{H}_m}
=\pig\langle D\mathbb{Z}^\Psi_{tX}(s),\Psi \pigr\rangle_{\mathcal{H}_m}$ for each $s \in [\tau^*,T]$. Also, the weak limit $D\mathbb{Z}^\Psi_{tX}(s)$ is also continuous for $s\in [t,T]$ by studying its equation in (\ref{DZ_infty}), therefore it yields $\displaystyle\lim_{s \to t^+}\lim_{k \to \infty}\pig\langle D\mathbb{Z}_k(s),\Psi \pigr\rangle_{\mathcal{H}_m} = \pig\langle D\mathbb{Z}^\Psi_{tX}(t),\Psi \pigr\rangle_{\mathcal{H}_m}$. Hence, it yields from (\ref{<DZk(tk), Psi> to <DZ(t), Psi>+Jk}) that 
\begin{equation}
\lim_{k \to \infty}\pig\langle D\mathbb{Z}_k(t_k),\Psi \pigr\rangle_{\mathcal{H}_m} =\pig\langle D \mathbb{Z}^\Psi_{tX}(t),\Psi \pigr\rangle_{\mathcal{H}_m}.
\label{<DZk(tk), Psi> to <DZ(t), Psi>}
\end{equation}
On the other hand, by noting that $D\mathbb{Y}_k(s)$ is a finite variation process, we apply traditional It\^o lemma to $\pig\langle D\mathbb{Y}_k(s), D\mathbb{Z}_k(s) \pigr\rangle_{\mathcal{H}_m}$ to give the equality
\begin{equation}
\scalemath{0.95}{
\h{-15pt}\begin{aligned}
\pig\langle D&\mathbb{Z}_k(t_k),\Psi  \pigr\rangle_{\mathcal{H}_m}\\
=\:&
\Big\langle
h_{xx}(\mathbb{Y}_{t_kX_k}(T))D\mathbb{Y}_{k} (T)
+D_{X}^2F_{T}(\mathbb{Y}_{t_kX_k}(T)  \otimes  m)(D\mathbb{Y}_{k} (T))
,D\mathbb{Y}_k(T)
\Big\rangle_{\mathcal{H}_m}\\
&+\int_{t_k}^{T}
\Big\langle 
l_{vx}\pig(\mathbb{Y}_{t_k X_k}(\tau),u_{t_k X_k}(\tau)\pig)
D \mathbb{Y}_{k} (\tau)
+l_{vv}\pig(\mathbb{Y}_{t_k X_k}(\tau),u_{t_k X_k}(\tau)\pig)
D u_{k}(\tau),D u_k(\tau)\Big\rangle_{\mathcal{H}_m} d\tau\\
&+\int_{t_k}^{T}\h{-3pt}
\bigg\langle l_{xv}(\mathbb{Y}_{t_{k}X_{k}}(\tau),u_{t_{k}X_{k}}(\tau))
Du_{k} (\tau)
+\Big[l_{xx}(\mathbb{Y}_{t_{k}X_{k}}(\tau),u_{t_{k}X_{k}}(\tau))
+D_{X}^{2}F(\mathbb{Y}_{t_{k}X_{k}}(\tau) \otimes  m)\Big](D\mathbb{Y}_k(\tau)),
D\mathbb{Y}_k(\tau)
\bigg\rangle_{\mathcal{H}_m} \h{-6pt} d\tau.\h{-60pt}
\end{aligned}}
\label{<DZ_k(t_k),Psi>}
\end{equation}
By defining the operators $Q_{1k}:=h_{xx}(\mathbb{Y}_{t_kX_k}(T))
+D_{X}^2F_{T}(\mathbb{Y}_{t_kX_k}(T)  \otimes  m)$, $
Q_{2k}:=2l_{xv}(\mathbb{Y}_{t_{k}X_{k}}(\tau),u_{t_{k}X_{k}}(\tau))$,
$Q_{3k}:=l_{vv}(\mathbb{Y}_{t_{k}X_{k}}(\tau),u_{t_{k}X_{k}}(\tau))$, $Q_{4k}:=l_{xx}(\mathbb{Y}_{t_{k}X_{k}}(\tau),u_{t_{k}X_{k}}(\tau))
+D_{X}^{2}F(\mathbb{Y}_{t_{k}X_{k}}(\tau) \otimes  m)$, we can also write 
\begin{equation}
\begin{aligned}
\pig\langle\mathscr{D}\mathbb{Z}_\infty(t), \Psi \pigr\rangle_{\mathcal{H}_m}
=\:&
\pig\langle D\mathbb{Z}^\Psi_{tX}(t), \Psi \pigr\rangle_{\mathcal{H}_m}\\
=\:&\Big\langle
Q_{1\infty}(D\mathbb{Y}^\Psi_{tX} (T)),
D\mathbb{Y}^\Psi_{tX}(T)
\Big\rangle_{\mathcal{H}_m} 
+\int_{t}^{T}
\Big\langle Q_{2\infty}
(Du^\Psi_{tX} (\tau)),
D\mathbb{Y}^\Psi_{tX} (\tau)
\Big\rangle_{\mathcal{H}_m} d\tau
\\
&+\int_{t}^{T}
\Big\langle
Q_{3\infty}(Du^\Psi_{tX} (\tau)),
Du^\Psi_{tX}(\tau)
\Big\rangle_{\mathcal{H}_m} d \tau 
+\int_{t}^{T}\Big\langle Q_{4\infty}(D\mathbb{Y}^\Psi_{tX}(\tau)),
D\mathbb{Y}^\Psi_{tX} (\tau)
\Big\rangle_{\mathcal{H}_m} d\tau,\h{-40pt}
\end{aligned}
\label{<DZ_infty(t),Psi>}
\end{equation}
where $Q_{i\infty}$ represents the strong limit of $Q_{ik}$ in $\mathcal{H}_m$, for example, $Q_{1\infty} = h_{xx}(\mathbb{Y}_{tX}(T))
+D_{X}^2F_{T}(\mathbb{Y}_{tX}(T)  \otimes  m)$, where $\mathbb{Y}_{tX}$ is the strong limit of $\mathbb{Y}_{t_kX_k}$.

Referring to (\ref{<DZ_k(t_k),Psi>}) and (\ref{<DZ_infty(t),Psi>}), we consider $\pig\langle D\mathbb{Z}_k(t_k),\Psi \pigr\rangle_{\mathcal{H}_m} - \pig\langle D\mathbb{Z}^\Psi_{tX}(t),\Psi \pigr\rangle_{\mathcal{H}_m}$ by first checking the term

\begin{align}
&\h{-10pt}\Big\langle
Q_{1k}(D\mathbb{Y}_{k} (T)),
D\mathbb{Y}_k(T)
 \Big\rangle_{\mathcal{H}_m} 
 -\Big\langle
Q_{1\infty}(D\mathbb{Y}^\Psi_{tX} (T)),
D\mathbb{Y}^\Psi_{tX}(T)
 \Big\rangle_{\mathcal{H}_m} \nonumber\\
=\:&\Big\langle
Q_{1k}\pig(D\mathbb{Y}_{k} (T)-D\mathbb{Y}^\Psi_{tX} (T)\pig)
 ,D\mathbb{Y}_{k} (T)-D\mathbb{Y}^\Psi_{tX} (T)
 \Big\rangle_{\mathcal{H}_m} 
+\Big\langle
Q_{1k}(D\mathbb{Y}_{k} (T))
-Q_{1\infty}(D\mathbb{Y}^\Psi_{tX} (T)),D\mathbb{Y}^\Psi_{tX}(T)
 \Big\rangle_{\mathcal{H}_m} \nonumber\\
&
+\Big\langle
Q_{1k}(D\mathbb{Y}^\Psi_{tX} (T)),D\mathbb{Y}_{k}(T)
-D\mathbb{Y}^\Psi_{tX}(T)
\Big\rangle_{\mathcal{H}_m}.
\label{Q1k_Q1infty}
\end{align}
Recalling the $\mathcal{H}_m$-boundedness of $D\mathbb{Y}_k$ in (\ref{bdd Y, Z, u, r, epsilon}), the continuities of $h_{xx}$ and $D^2_X F_T$ in Assumptions \textbf{A(iv)}'s (\ref{assumption, cts of lxx, lvx, lvv, hxx}) and \textbf{B(iv)}'s (\ref{assumption, cts of D^2F_T}), together with the strong convergence of $\mathbb{Y}_{t_{k}X_{k}}$ in (\ref{eq:8-110}), it guarantees that \\$Q_{1k}(D\mathbb{Y}^\Psi_{tX} (T))$ converges strongly to $Q_{1\infty}(D\mathbb{Y}^\Psi_{tX} (T))$, and $Q_{1k}(D\mathbb{Y}_{k} (T))$ converges weakly to $Q_{1\infty}(D\mathbb{Y}^\Psi_{tX} (T))$ as a linear functional, see more details in Step 2A in the proof of Lemma \ref{lem, Existence of J flow}. Using these two convergences and (\ref{Q1k_Q1infty}), it yields that
$$
\scalemath{0.94}{
\Big\langle
Q_{1k}(D\mathbb{Y}_{k} (T)),
D\mathbb{Y}_k(T)
 \Big\rangle_{\mathcal{H}_m} 
 -\Big\langle
Q_{1\infty}(D\mathbb{Y}^\Psi_{tX} (T)),
D\mathbb{Y}^\Psi_{tX}(T)
 \Big\rangle_{\mathcal{H}_m} 
-\Big\langle
Q_{1k}\pig(D\mathbb{Y}_{k} (T)-D\mathbb{Y}^\Psi_{tX} (T)\pig)
 ,D\mathbb{Y}_{k} (T)-D\mathbb{Y}^\Psi_{tX} (T)
 \Big\rangle_{\mathcal{H}_m} }
$$
tends to $0$ as $k \to \infty$. By estimating the remaining terms in $ \pig\langle D\mathbb{Z}_k(t_k),\Psi \pigr\rangle_{\mathcal{H}_m} - \pig\langle D\mathbb{Z}^\Psi_{tX}(t),\Psi \pigr\rangle_{\mathcal{H}_m}$ involving $Q_{2k}$, $Q_{3k}$ and $Q_{4k}$ similarly, we can deduce
$$ \pig\langle D\mathbb{Z}_k(t_k),\Psi \pigr\rangle_{\mathcal{H}_m} - \pig\langle D\mathbb{Z}^\Psi_{tX}(t),\Psi \pigr\rangle_{\mathcal{H}_m}
-\mathscr{J}'_k
-\mathscr{J}''_k\longrightarrow 0, $$
where
\begin{align*}
\mathscr{J}'_{k}:=&\Big\langle \pig[h_{xx}(\mathbb{Y}_{t_{k}X_{k}}(T))
+D_{X}^{2}F_{T}(\mathbb{Y}_{t_{k}X_{k}}(T) \otimes  m)\pig](D\mathbb{Y}_k(T)-D\mathbb{Y}^\Psi_{tX}(T)),D\mathbb{Y}_k(T)-D\mathbb{Y}^\Psi_{tX}(T)\Big\rangle_{\mathcal{H}_m}\\
&+\int_{t_{k}}^{T}\Big\langle
l_{xx}(\mathbb{Y}_{t_{k}X_{k}}(s),u_{t_{k}X_{k}}(s))
(D\mathbb{Y}_{k}(s)-D\mathbb{Y}^\Psi_{tX}(s)),D\mathbb{Y}_{k}(s)-D\mathbb{Y}^\Psi_{tX}(s)\Big\rangle_{\mathcal{H}_m}\\
&\h{20pt}+2\Big\langle l_{xv}(\mathbb{Y}_{t_{k}X_{k}}(s),u_{t_{k}X_{k}}(s))(Du_k(s)-Du^\Psi_{tX}(s)),D\mathbb{Y}_k(s)-D\mathbb{Y}^\Psi_{tX}(s)\Big\rangle_{\mathcal{H}_m}\\
&\h{20pt}+\Big\langle l_{vv}\pig(\mathbb{Y}_{t_kX_k}(s),u_{t_kX_k}(s)\pig)(Du_k(s)-Du^\Psi_{tX}(s)),Du_k(s)-Du^\Psi_{tX}(s)\Big\rangle_{\mathcal{H}_m}\\
&\h{20pt}+\Big\langle D_{X}^{2}F(\mathbb{Y}_{t_{k}X_{k}}(s) \otimes  m)(D\mathbb{Y}_k(s)-D\mathbb{Y}^\Psi_{tX}(s)),D\mathbb{Y}_k(s)-D\mathbb{Y}^\Psi_{tX}(s)\Big\rangle_{\mathcal{H}_m}ds
\end{align*}
and
\begin{align*}
\mathscr{J}''_{k}:=&
\int^{t_k}_t
\Big\langle 
l_{vx}\pig(\mathbb{Y}_{tX}(s),u_{tX}(s)\pig)
D\mathbb{Y}^\Psi_{tX}(s) 
+l_{vv}\pig(\mathbb{Y}_{tX}(s),u_{tX}(s)\pig)
Du^\Psi_{tX}(s),D u^\Psi_{tX}(s)\Big\rangle_{\mathcal{H}_m} \\
&+
\bigg\langle l_{xv}(\mathbb{Y}_{tX}(s),u_{tX}(s))
D u^\Psi_{tX}(s)
+\Big[l_{xx}(\mathbb{Y}_{tX}(s),u_{tX}(s))
+D_{X}^{2}F(\mathbb{Y}_{tX}(s) \otimes  m)\Big](D\mathbb{Y}^\Psi_{tX}(s)),
D\mathbb{Y}^\Psi_{tX}(s)
\bigg\rangle_{\mathcal{H}_m} \h{-6pt} ds.
\end{align*}
Assumptions \textbf{A(ii)}'s (\ref{assumption, bdd of lxx, lvx, lvv}), \textbf{B(ii)}'s (\ref{assumption, bdd of D^2F, D^2F_T}) and the bounds in (\ref{bdd Y, Z, u, r, epsilon}) yield
\begin{align}
|\mathscr{J}''_{k}|\leq\:&
\int^{t_k}_t
\pig\|l_{vx}\pig(\mathbb{Y}_{tX}(s),u_{tX}(s)\pig)
D\mathbb{Y}^\Psi_{tX}(s) 
+l_{vv}\pig(\mathbb{Y}_{tX}(s),u_{tX}(s)\pig)
Du^\Psi_{tX}(s)\pigr\|_{\mathcal{H}_m}
\pig\|D u^\Psi_{tX}(s)\pigr\|_{\mathcal{H}_m}\nonumber \\
&+
\Big\|l_{xv}(\mathbb{Y}_{tX}(s),u_{tX}(s))
D u^\Psi_{tX}(s)
+\Big[l_{xx}(\mathbb{Y}_{tX}(s),u_{tX}(s))
+D_{X}^{2}F(\mathbb{Y}_{tX}(s) \otimes  m)\Big](D\mathbb{Y}^\Psi_{tX}(s))\Bigr\|_{\mathcal{H}_m}\pig\|
D\mathbb{Y}^\Psi_{tX}(s)
\pigr\|_{\mathcal{H}_m} \h{-6pt} ds\nonumber \\
\leq\:&\int^{t_k}_t
2 c_l\pig\|
D\mathbb{Y}^\Psi_{tX}(s) \pigr\|_{\mathcal{H}_m}
\pig\|D u^\Psi_{tX}(s)\pigr\|_{\mathcal{H}_m} 
+c_l
\pig\|D u^\Psi_{tX}(s)\pigr\|_{\mathcal{H}_m}^2
+(c_l+c)\pig\|
D\mathbb{Y}^\Psi_{tX}(s)
\pigr\|_{\mathcal{H}_m}^2 \h{-6pt} ds\nonumber \\
&\h{-10pt}\longrightarrow 0, \h{20pt} \text{ as $k \to \infty$.}\label{J'' to 0}
\end{align}
Together with the fact that $\pig\langle D\mathbb{Z}_k(t_k),\Psi \pigr\rangle_{\mathcal{H}_m} - \pig\langle D\mathbb{Z}^\Psi_{tX}(t),\Psi \pigr\rangle_{\mathcal{H}_m} \longrightarrow 0$ in (\ref{<DZk(tk), Psi> to <DZ(t), Psi>}) and $\mathscr{J}''_{k}\longrightarrow 0$ in (\ref{J'' to 0}), we have $\mathscr{J}'_k \longrightarrow 0$ as $k \to \infty$. Assumptions \textbf{A(v)}'s (\ref{assumption, convexity of l}), \textbf{A(vi)}'s (\ref{assumption, convexity of h}) and \textbf{B(v)(b)}'s (\ref{assumption, convexity of D^2F, D^2F_T}) imply
\begin{equation}
\begin{aligned}
\mathscr{J}'_{k}\geq\int_{t_{k}}^{T}
\lambda\big\|Du_k(s)-Du^\Psi_{tX}(s)\big\|_{\mathcal{H}_m}^{2}
-(c'_{l}+c') \big\|D\mathbb{Y}_k(s)-D\mathbb{Y}^\Psi_{tX}(s)\big\|_{\mathcal{H}_m}^{2}ds
-(c'_{h}+c'_{T})\big\|D\mathbb{Y}_k(T)-D\mathbb{Y}^\Psi_{tX}(T)\big\|_{\mathcal{H}_m}^{2}.
\end{aligned}
\label{est Jk >}
\end{equation}

\noindent {\bf Step 3. Strong Convergence:}\\
The equations of $D\mathbb{Y}_k$ and $D\mathbb{Y}^\Psi_{tX}$ give
\begin{align*}
D\mathbb{Y}_k(s) - D\mathbb{Y}^\Psi_{tX}(s)
=\int^s_{t_k} Du_k(\tau) - Du^\Psi_{tX}(\tau) d\tau
-\int^{t_k}_tDu^\Psi_{tX}(\tau) d\tau,
\end{align*}
hence, it yields
\begin{equation}
\scalemath{0.94}{
\begin{aligned}
\big\|D\mathbb{Y}_k(T)-D\mathbb{Y}^\Psi_{tX}(T)\big\|_{\mathcal{H}_m}^{2}
&\leq2T\int_{t_{k}}^{T}
\big\|Du_{k}(s)-Du^\Psi_{tX}(s)\big\|^{2}_{\mathcal{H}_m}ds
+2(t_{k}-t)^{2}\sup_{s}\big\|Du^\Psi_{tX}(s)\big\|^{2}_{\mathcal{H}_m} \h{5pt} \text{ and}\\
\int_{t_{k}}^{T}\big\|D\mathbb{Y}_k(s)-D\mathbb{Y}^\Psi_{tX}(s)\big\|_{\mathcal{H}_m}^{2}ds
&\leq T^{2}\int_{t_{k}}^{T}\big\|Du_{k}(s)-Du^\Psi_{tX}(s)\big\|^{2}_{\mathcal{H}_m}ds
+2T(t_{k}-t)^{2}
\sup_{s}\big\|Du^\Psi_{tX}(s)\big\|^{2}_{\mathcal{H}_m}.
\end{aligned}}
\label{est of yk}
\end{equation}
By putting (\ref{est of yk}) into (\ref{est Jk >}), the assumption (\ref{uniqueness condition}), the fact that $t_{k}\downarrow t$ and
$\mathscr{J}'_{k}\rightarrow0$ give the convergence $\int_{t_k}^{T}\pig\|Du_k (s)-Du^\Psi_{tX}(s)\pigr\|^{2}_{\mathcal{H}_m}ds\longrightarrow 0$. Therefore, the second inequality in (\ref{est of yk}) shows the strong convergence of $D\mathbb{Y}_k(s)$ to $D\mathbb{Y}^\Psi_{tX}(s)$ in $L_{\widetilde{\mathcal{W}}_{tX\Psi}}^{2}(\tau^*,T;\mathcal{H}_{m})$. Hence, it is then easy to conclude that
\begin{equation}
\h{-5pt}\begin{aligned}
D\mathbb{Z}_{k}&(s)\\ =\mathbb{E}\Bigg[&h_{xx}(\mathbb{Y}_{t_kX_k}(T))D\mathbb{Y}_{k} (T)
+D_{X}^2F_{T}(\mathbb{Y}_{t_kX_k}(T)  \otimes  m)(D\mathbb{Y}_{k} (T))\\
&\h{-7pt}+\displaystyle\int^T_s 
l_{xv}(\mathbb{Y}_{t_kX_k}(\tau),u_{t_kX_k}(\tau))
Du_{k} (\tau)
+\pig[l_{xx}(\mathbb{Y}_{t_kX_k}(\tau),u_{t_kX_k}(\tau))
+D^2_{X}F(\mathbb{Y}_{t_kX_k}(\tau)  \otimes  m)\pig]
\pig(
D\mathbb{Y}_{k}  (\tau)\pig) d\tau 
\Bigg|\widetilde{\mathcal{W}}_{tX\Psi}^{s}\Bigg]\\
&\h{-25pt}\longrightarrow D\mathbb{Z}^\Psi_{tX}(s)
\end{aligned}
\end{equation}
strongly in $L_{\widetilde{\mathcal{W}}_{tX\Psi}}^{2}(\tau^*,T;\mathcal{H}_{m})$, by using the continuities of the second-order derivatives of $l$, $h$, $F$ and $F_T$ in Assumptions \textbf{A(iv)}'s (\ref{assumption, cts of lxx, lvx, lvv, hxx}), \textbf{B(iii)}'s (\ref{assumption, cts of D^2F}) and \textbf{B(iv)}'s (\ref{assumption, cts of D^2F_T}), the strong convergences of $\mathbb{Y}_{t_{k}X_{k}}$ and $\mathbb{Z}_{t_{k}X_{k}}$ in (\ref{eq:8-110}), and the strong convergences of $D\mathbb{Y}_k$ and $Du_k$. Moreover, It\^o lemma yields
\begin{equation}
\begin{aligned}
&\h{-15pt}\big\| D\mathbb{Z}_k(s) - D\mathbb{Z}^\Psi_{tX}(s) \big\|_{\mathcal{H}_m}^2
+\int^T_s\big\| D\mathbbm{r}_{k,j}(\tau) - D\mathbbm{r}_{tX,j}^\Psi(\tau) \big\|_{\mathcal{H}_m}^2 d \tau\\
=\:&\Big\| \pig[h_{xx}(\mathbb{Y}_{t_kX_k}(T))
+D_{X}^2F_{T}(\mathbb{Y}_{t_kX_k}(T)  \otimes  m)\pig](D\mathbb{Y}_{k} (T)) \\
&\h{150pt}- \pig[h_{xx}(\mathbb{Y}_{tX}(T))D\mathbb{Y}_{k} (T)
+D_{X}^2F_{T}(\mathbb{Y}_{tX}(T)  \otimes  m)\pig](D\mathbb{Y}_{tX}^\Psi (T)) \big\|_{\mathcal{H}_m}^2\\
&+\int^T_s\pig\langle
l_{xv}(\mathbb{Y}_{t_kX_k}(\tau),u_{t_kX_k}(\tau))Du_k(\tau)
-l_{xv}(\mathbb{Y}_{tX}(\tau),u_{tX}(\tau))
Du^\Psi_{tX}(\tau),D\mathbb{Z}_k(\tau) - D\mathbb{Z}^\Psi_{tX}(\tau)
\pigr\rangle_{\mathcal{H}_m}d\tau \\
&+\int^T_s\pig\langle
\pig[l_{xx}(\mathbb{Y}_{t_kX_k}(\tau),u_{t_kX_k}(\tau))+D^2_{X}F(\mathbb{Y}_{t_kX_k}(\tau)  \otimes  m)\pig]
\pig(D\mathbb{Y}_{k}  (\tau)\pig)
,D\mathbb{Z}_k(\tau) - D\mathbb{Z}^\Psi_{tX}(\tau)
\pigr\rangle_{\mathcal{H}_m}d\tau\\
&-\int^T_s\pig\langle
\pig[l_{xx}(\mathbb{Y}_{tX}(\tau),u_{tX}(\tau))+D^2_{X}F(\mathbb{Y}_{tX}(\tau)  \otimes  m)\pig]
\pig(D\mathbb{Y}_{tX}^\Psi(\tau)\pig)
,D\mathbb{Z}_k(\tau) - D\mathbb{Z}^\Psi_{tX}(\tau)
\pigr\rangle_{\mathcal{H}_m}d\tau\\
\leq\:&2\Big\|\pig[ h_{xx}(\mathbb{Y}_{t_kX_k}(T))
+D_{X}^2F_{T}(\mathbb{Y}_{t_kX_k}(T)  \otimes  m) - h_{xx}(\mathbb{Y}_{tX}(T))
-D_{X}^2F_{T}(\mathbb{Y}_{tX}(T)  \otimes  m)\pig](D\mathbb{Y}_{tX}^\Psi (T)) \Big\|_{\mathcal{H}_m}^2\\
&+2\Big\| \pig[h_{xx}(\mathbb{Y}_{tX}(T))D\mathbb{Y}_{k} (T)
+D_{X}^2F_{T}(\mathbb{Y}_{tX}(T)  \otimes  m)\pig](D\mathbb{Y}_{k}(T)-D\mathbb{Y}_{tX}^\Psi (T)) \Big\|_{\mathcal{H}_m}^2\\
&+c_l\int^T_s\Big(\pig\|Du_k(\tau)\pigr\|_{\mathcal{H}_m}
+\pig\|Du^\Psi_{tX}(\tau)\pigr\|_{\mathcal{H}_m}\Big)
\pig\|D\mathbb{Z}_k(\tau) - D\mathbb{Z}^\Psi_{tX}(\tau)
\pigr\|_{\mathcal{H}_m}d\tau \\
&+(c_l+c)\int^T_s\Big(\pig\|
D\mathbb{Y}_{k}  (\tau)\pigr\|_{\mathcal{H}_m}
+\pig\|D\mathbb{Y}_{tX}^\Psi(\tau)\pigr\|_{\mathcal{H}_m}
\Big)\pig\|D\mathbb{Z}_k(\tau) - D\mathbb{Z}^\Psi_{tX}(\tau)
\pigr\|_{\mathcal{H}_m}d\tau.
\end{aligned}
\label{|DZ_k-DZ|}
\end{equation}
Since $D\mathbb{Y}_k (\cdot)$, $D\mathbb{Z}_k (\cdot)$, $Du_k (\cdot)$ strongly converge to $D\mathbb{Y}_{tX}^\Psi (\cdot)$, $D\mathbb{Z}_{tX}^\Psi (\cdot)$, $Du_{tX}^\Psi (\cdot)$ in $L_{\widetilde{\mathcal{W}}_{tX\Psi}}^{2}(\tau^*,T;\mathcal{H}_{m})$ respectively, then for almost every $s \in [\tau^*,T]$, $\big\|D\mathbb{Y}_k(s)-D\mathbb{Y}^\Psi_{tX}(s)\big\|_{\mathcal{H}_m}$, $\big\|Du_k(s)-Du^\Psi_{tX}(s)\big\|_{\mathcal{H}_m}$ and $\big\|D\mathbb{Z}_k(s)-D\mathbb{Z}^\Psi_{tX}(s)\big\|_{\mathcal{H}_m}$ converge to zero, along some subsequence due to an usual application of Borel-Cantelli lemma. Due to the $\mathcal{H}_m$-boundedness of $D\mathbb{Y}_k$, $Du_k$ in (\ref{bdd Y, Z, u, r, epsilon}), the continuities of the second-order derivatives of $h$, $F_T$ in Assumptions \textbf{A(iv)}'s (\ref{assumption, cts of lxx, lvx, lvv, hxx}) and \textbf{B(iv)}'s (\ref{assumption, cts of D^2F_T}), the strong convergences of $\mathbb{Y}_{t_{k}X_{k}}$ and $\mathbb{Z}_{t_{k}X_{k}}$ in (\ref{eq:8-110}) and the convergence in (\ref{est of yk}), an application of  dominated convergence theorem to (\ref{|DZ_k-DZ|}) implies that $\big\|D\mathbb{Z}_k(s)-D\mathbb{Z}^\Psi_{tX}(s)\big\|_{\mathcal{H}_m}$ converges to zero for every $s \in [\tau^*,T]$. Hence, for $s\in [\tau^*,T]$, \begin{align*}
&\lim_{k \to \infty}\pig\|
D\mathbb{Z}_k(t_k)-D\mathbb{Z}^\Psi_{tX}(t)\pigr\|_{\mathcal{H}_m}\\
&\leq\lim_{k \to \infty}\pig\|
D\mathbb{Z}_k(t_k)-D\mathbb{Z}_k(s)\pigr\|_{\mathcal{H}_m}
+\lim_{k \to \infty}\pig\|
D\mathbb{Z}_k(s)-D\mathbb{Z}^\Psi_{tX}(s)\pigr\|_{\mathcal{H}_m}
+\pig\|
D\mathbb{Z}^\Psi_{tX}(s)-D\mathbb{Z}^\Psi_{tX}(t)\pigr\|_{\mathcal{H}_m}\\ 
&=\lim_{k \to \infty}\pig\|
D\mathbb{Z}_k(t_k)-D\mathbb{Z}_k(s)\pigr\|_{\mathcal{H}_m}
+\pig\|D\mathbb{Z}^\Psi_{tX}(s)
-D\mathbb{Z}^\Psi_{tX}(t)\pigr\|_{\mathcal{H}_m}.
\end{align*}
By taking $s\to t^+$, we have
\begin{align}
\lim_{k \to \infty}\pig\|
D\mathbb{Z}_k(t_k)-D\mathbb{Z}^\Psi_{tX}(t)\pigr\|_{\mathcal{H}_m}
\leq\lim_{s \to t^+}\lim_{k \to \infty}\pig\|
D\mathbb{Z}_k(t_k)-D\mathbb{Z}_k(s)\pigr\|_{\mathcal{H}_m}.
\label{ineq. of |DZ_k(t_k)-DZ(t)|}
\end{align}
Using Jensen's and Cauchy-Schwarz inequalities, we consider the term
\begin{align}
&\h{-10pt}\pig\|
D\mathbb{Z}_k(t_k)-D\mathbb{Z}_k(s)\pigr\|_{\mathcal{H}_m}^2\nonumber\\
=\:& 
\int_{\mathbb{R}^n}\mathbb{E}\Bigg\{\bigg|\mathbb{E}\bigg[\displaystyle\int^s_{t_k}
l_{xv}(\mathbb{Y}_{t_kX_k}(\tau),u_{t_kX_k}(\tau))
Du_{k} (\tau)\nonumber\\
&\h{120pt}+\pig[l_{xx}(\mathbb{Y}_{t_kX_k}(\tau),u_{t_kX_k}(\tau))+D^2_{X}F(\mathbb{Y}_{t_kX_k}(\tau)  \otimes  m)\pig]
\pig(D\mathbb{Y}_{k}  (\tau)\pig) d\tau
\bigg|\widetilde{\mathcal{W}}_{tX\Psi}^{s}\bigg]\bigg|^2\Bigg\} dm(x)\nonumber\\
\leq\:& 
(s-t_k)\int_{\mathbb{R}^n}\mathbb{E}\Bigg\{\mathbb{E}\bigg[\displaystyle\int^s_{t_k}\Big|
l_{xv}(\mathbb{Y}_{t_kX_k}(\tau),u_{t_kX_k}(\tau))
Du_{k} (\tau)\nonumber\\
&\h{120pt}+\pig[l_{xx}(\mathbb{Y}_{t_kX_k}(\tau),u_{t_kX_k}(\tau))+D^2_{X}F(\mathbb{Y}_{t_kX_k}(\tau)  \otimes  m)\pig]
\pig(D\mathbb{Y}_{k}  (\tau)\pig)\Big|^2 d\tau
\bigg|\widetilde{\mathcal{W}}_{tX\Psi}^{s}\bigg]\Bigg\} dm(x)\nonumber\\
\leq\:& 
(s-t_k)\int_{\mathbb{R}^n}
\mathbb{E}\bigg[\displaystyle\int^s_{t_k}\Big|
l_{xv}(\mathbb{Y}_{t_kX_k}(\tau),u_{t_kX_k}(\tau))
Du_{k} (\tau)\nonumber\\
&\h{120pt}+\pig[l_{xx}(\mathbb{Y}_{t_kX_k}(\tau),u_{t_kX_k}(\tau))+D^2_{X}F(\mathbb{Y}_{t_kX_k}(\tau)  \otimes  m)\pig]
\pig(D\mathbb{Y}_{k}  (\tau)\pig)\Big|^2 d\tau
\bigg] dm(x)\nonumber\\
\leq\:& 
3(s-t_k)\displaystyle\int^s_{t_k}
c^2_l\big\|
Du_{k} (\tau)\big\|_{\mathcal{H}_m}^2 +(c_l^2+c^2)
\big\|D\mathbb{Y}_{k}  (\tau)\big\|_{\mathcal{H}_m}^2 d\tau\nonumber\\
\leq\:& 
3(s-t_k)\displaystyle\int^s_{t_k}
c^2_l\sup_k\big\|
Du_{k} (\tau)\big\|_{\mathcal{H}_m}^2 +(c_l^2+c^2)\sup_k
\big\|D\mathbb{Y}_{k}  (\tau)\big\|_{\mathcal{H}_m}^2 d\tau.
\label{DZ_k(t_k)-DZ_k(s) to 0}
\end{align}
Therefore,
\begin{align*}
\sup_k\pig\|
D\mathbb{Z}_k(t_k)-D\mathbb{Z}_k(s)\pigr\|_{\mathcal{H}_m}^2
\leq
3(s-t_k)\displaystyle\int^s_{t_k}
c^2_l\sup_k\big\|
Du_{k} (\tau)\big\|_{\mathcal{H}_m}^2 +(c_l^2+c^2)\sup_k
\big\|D\mathbb{Y}_{k}  (\tau)\big\|_{\mathcal{H}_m}^2 d\tau.
\end{align*}
Since $\big\|Du_{k}(\tau)\big\|_{\mathcal{H}_m}$ and $\big\|D\mathbb{Y}_{k}  (\tau)\big\|_{\mathcal{H}_m}$ are uniformly bounded, then putting (\ref{DZ_k(t_k)-DZ_k(s) to 0}) into (\ref{ineq. of |DZ_k(t_k)-DZ(t)|}) yields the convergence in  (\ref{eq:8-86}).

\mycomment{
Let us set 
$\mathcal{Y}_k^*(s)
=D\mathbb{Y}^{\Psi_k}_{t_kX_k}(s)
-D\mathbb{Y}^{\Psi_k}_{t_kX}(s)$, $\mathcal{Z}_k^*(s)
=D\mathbb{Z}^{\Psi_k}_{t_kX_k}(s)
-D\mathbb{Z}^{\Psi_k}_{t_kX}(s)$, $\mathcal{U}_k^*(s)
=Du^{\Psi_k}_{t_kX_k}(s)
-Du^{\Psi_k}_{t_kX}(s)$ and $\mathcal{R}_{k,j}^*(s)
=D\mathbbm{r}^{\Psi_k}_{t_kX_k,j}(s)
-D\mathbbm{r}^{\Psi_k}_{t_kX,j}(s)$.
As usual we will denote by $\widetilde{\mathcal{W}}_{tX\Psi}$ the filtration
of $\sigma$-algebras $\widetilde{\mathcal{W}}_{tX\Psi}^{s}$. The processes $\mathcal{Y}^*_{k}(s),\mathcal{Z}^*_{k}(s),\mathcal{U}_{k}^*(s),\mathcal{R}^*_{k,j}(s)$
belong to $L_{\widetilde{\mathcal{W}}_{tX\Psi}}^{2}(t_{k},T;\mathcal{H}_{m})$. It is convenient to extend the processes for $s \in [t,t_{k})$ as follows 

\begin{equation}
\mathcal{Y}^*_{k}(s)=0,
\mathcal{Z}^*_{k}(s)=\mathbb{E}[\widetilde{\mathcal{Z}}^{k}(t_{k})|\widetilde{\mathcal{W}}_{tX\Psi}^{s}],
\mathcal{U}^*_{k}(s)=0,
\mathcal{R}^*_{k,j}(s)=0
\label{eq:8-155}
\end{equation}
By the equations satisfied by $\mathcal{Y}^*_k(s)$ and $\mathcal{Z}^*_k(s)$, as well as the first order condition, we have

\begin{equation}
\begin{aligned}
0=\:&l_{vx}(\mathbb{Y}_{t_{k}X_{k}}(s),u_{t_{k}X_{k}}(s))\mathcal{Y}^*_{k}(s)+l_{vv}(\mathbb{Y}_{t_{k}X_{k}}(s),u_{t_{k}X_{k}}(s))\mathcal{U}^*_{k}(s)\\
&+\mathbb{E}\Bigg\{ \int_{s}^{T}l_{xx}(\mathbb{Y}_{t_{k}X_{k}}(\tau),u_{t_{k}X_{k}}(\tau))\mathcal{Y}^*_{k}(\tau)+l_{xv}(\mathbb{Y}_{t_{k}X_{k}}(\tau),u_{t_{k}X_{k}}(\tau))\mathcal{U}^*_{k}(\tau)+D_{X}^{2}F(\mathbb{Y}_{t_{k}X_{k}}(\tau) \otimes  m)(\mathcal{Y}^*_{k}(\tau))d\tau\\
&\h{30pt}+h_{xx}(\mathbb{Y}_{t_{k}X_{k}}(T))\mathcal{Y}^*_{k}(T)+D_{X}^{2}F(\mathbb{Y}_{t_{k}X_{k}}(T) \otimes  m)(\mathcal{Y}^*_{k}(T))\;\Bigg|\:\widetilde{\mathcal{W}}_{tX\Psi}^{s}\Bigg\} 
+\mathbb{E}\left\{\mathscr{J}^*_k(s)\middle|\:\widetilde{\mathcal{W}}_{tX\Psi}^{s}\right\},
\end{aligned}
\label{eq:8-156}
\end{equation}
where 
{\footnotesize
\begin{equation}
\begin{aligned}
\mathscr{J}^*_k&(s)\\
=\:&\Big[l_{vx}(\mathbb{Y}_{t_{k}X_{k}},u_{t_{k}X_{k}})(s)
-l_{vx}(\mathbb{Y}_{t_{k}X},u_{t_{k}X})(s)\Big]
D\mathbb{Y}^{\Psi_{k}}_{t_k X }(s)
+\Big[l_{vv}(\mathbb{Y}_{t_{k}X_{k}},u_{t_{k}X_{k}})(s)
-l_{vv}(\mathbb{Y}_{t_{k}X},u_{t_{k}X})(s)\Big]
Du^{\Psi_k}_{t_kX}(s)\\
&+\int_{s}^{T}\Big[l_{xx}(\mathbb{Y}_{t_{k}X_{k}},u_{t_{k}X_{k}})(\tau)
-l_{xx}(\mathbb{Y}_{t_{k}X},u_{t_{k}X})(\tau)\Big]
D\mathbb{Y}^{\Psi_k}_{t_k X_{k}}(\tau)
+\Big[l_{xv}(\mathbb{Y}_{t_{k}X_{k}}(\tau),u_{t_{k}X_{k}}(\tau))
-l_{xv}(\mathbb{Y}_{t_{k}X}(\tau),u_{t_{k}X}(\tau))\Big]Du^{\Psi_k}_{t_{k}X}(\tau)\\
&\h{255pt}+\Big[D_{X}^{2}F(\mathbb{Y}_{t_{k}X_{k}}(\tau) \otimes m)-D_{X}^{2}F(\mathbb{Y}_{t_{k}X}(\tau) \otimes  m)\Big]
(D\mathbb{Y}^{\Psi_k}_{t_{k}X}(\tau))
d\tau\\
&+\Big[h_{xx}(\mathbb{Y}_{t_{k}X_{k}}(T))
-h_{xx}(\mathbb{Y}_{t_{k}X}(T))\Big]
D\mathbb{Y}^{\Psi_k}_{t_{k}X}(T)
+\Big[D_{X}^{2}F(\mathbb{Y}_{t_{k}X_{k}}(T) \otimes  m)-D_{X}^{2}F(\mathbb{Y}_{t_{k}X}(T) \otimes  m)\Big](D\mathbb{Y}^{\Psi_k}_{t_{k}X}(T)).
\end{aligned}
\label{def J_k}
\end{equation}}
Since $\{\Psi_{k}\}_{k\in \mathbb{N}}$ is bounded in $\mathcal{H}_{m}$ and all quantities $D\mathbb{Y}^{\Psi_{k}}_{t_k X_k }(s),Du^{\Psi_{k}}_{t_k X_k }(s),D\mathbb{Y}^{\Psi_{k}}_{t_k X }(s),Du^{\Psi_{k}}_{t_k X }(s)$
are also bounded in $\mathcal{H}_{m}$ uniformly in $s$.

We aim to prove that
\begin{equation}
\mathbb{E}\left\{\mathscr{J}^*_k(s)\middle|\:\widetilde{\mathcal{W}}_{tX\Psi}^{s}\right\}
\longrightarrow0
\h{10pt}
\text{in $L^{1}(\Omega,\mathcal{A},P;L_{m}^{1}(\mathbb{R}^{n}))$, for any $s\in(t,T]$.}\label{eq:8-158}
\end{equation}
As 
\[
\mathbb{E}\left[
\int_{\mathbb{R}^n}
\left|\mathbb{E}\left\{\mathscr{J}^*_k(s)\middle|\:\widetilde{\mathcal{W}}_{tX\Psi}^{s}\right\}\right|dm(x)\right]
\leq \mathbb{E}\left[\int_{\mathbb{R}^n}\big|\mathscr{J}^*_k(s)\big|dm(x)\right],
\]
it is sufficient to show that $\mathbb{E}\left[\int_{\mathbb{R}^n}|\mathscr{J}^*_k(s)|dm(x)\right]\longrightarrow0$, for any $s \in (t,T]$. Consider, for instance, 

\[
\mathscr{J}^*_{1k}(s)
=\pig[l_{vx}(\mathbb{Y}_{t_{k}X_{k}}(s),u_{t_{k}X_{k}}(s))-l_{vx}(\mathbb{Y}_{t_{k}X}(s),u_{t_{k}X}(s))
\pig]D\mathbb{Y}_{t_k X}^{\Psi_{k}}(s),
\]
Cauchy-Schwarz inequality deduces

\begin{equation}
\mathbb{E}
\left[\int_{\mathbb{R}^n}
\pig|\mathscr{J}^*_{1k}(s)\pigr| dm(x)\right]
\leq
\pig\|l_{vx}(\mathbb{Y}_{t_{k}X_{k}}(s),u_{t_{k}X_{k}}(s))-l_{vx}(Y_{t_{k}X}(s),u_{t_{k}X}(s))\pigr\|_{\mathcal{H}_m}^2
\pig\|D\mathbb{Y}_{t_k X}^{\Psi_{k}}(s)\pigr\|_{\mathcal{H}_m}^2 \longrightarrow 0 \h{10pt} \text{as $k \to \infty$,}
\label{ineq E(L_k)<|l-l||DY| to 0}
\end{equation}
due to the continuity of $l_{vx}$, the strong convergence of $\mathbb{Y}_{t_{k}X_{k}}(s)$, $u_{t_{k}X_{k}}(s)$ and the $\mathcal{H}_m$-boundedness of $D\mathbb{Y}_{t_k X}^{\Psi_{k}}(s)$. Similarly, all terms in (\ref{def J_k}) tend to $0$ in $L^{1}(\Omega,\mathcal{A},P;L_{m}^{1}(\mathbb{R}^{n}))$. Since $\mathbb{E}\left[\int_{\mathbb{R}^n}|\mathscr{J}^*_k(s)|dm(x)\right]$ is also bounded for any $s \in (t_k,T]$ due to (\ref{ineq E(L_k)<|l-l||DY| to 0}), (\ref{def J_k}) and (\ref{bdd Y, Z, u, r, epsilon}), we can also
state, by dominated convergence theorem,

\begin{equation}
\int_{t}^{T}\mathbb{E}\left[\int_{\mathbb{R}^n}|\mathscr{J}^*_k(s)|\:dm(x)ds\right]
\longrightarrow 0.
\label{eq:8-159}
\end{equation}
Turning back to (\ref{eq:8-156}) and using the fact that $l_{vv}$
is an invertible matrix, we
obtain easily the inequality 
{\small
\begin{equation}
\begin{aligned}
\lambda \mathbb{E}\left[ \int_{\mathbb{R}^{n}}\big|\,\mathcal{U}^*_{k}(s)\big|dm(x) \right]
\leq\:& c_{l}\mathbb{E}\int_{\mathbb{R}^{n}}\big|\mathcal{Y}^*_{k}(s)\big|dm(x)
+(c_{l}+c)\int_{s}^{T}
\mathbb{E}\left[\int_{\mathbb{R}^{n}}\big|\mathcal{Y}^*_{k}(\tau)\big|dm(x)\right]d\tau\\
&+c_{l}\int_{s}^{T}\mathbb{E}\left[\int_{\mathbb{R}^{n}}\big|\,\mathcal{U}^*_{k}(\tau)\big|dm(x)\right]d\tau
+(c_{h}+c_{T})\mathbb{E}\left[
\int_{\mathbb{R}^{n}}\big|\mathcal{Y}^*_{k}(T)\big|dm(x)\right]
+\mathbb{E}\left[\int_{\mathbb{R}^{n}}\big|\mathscr{J}^*_k(s)\big|dm(x)\right]
\end{aligned}
\label{eq:8-160}
\end{equation}}
Using (\ref{eq:8-152}) and easy standard majorations we end up with 

\[
\left[\lambda-(2c_{l}+c_{h}+c_{T})T-(c_{l}+c)\dfrac{T^{2}}{2}\right]
\int_{t}^{T}(\mathbb{E}\int_{\mathbb{R}^{n}}\big|\,\mathcal{U}^*_{k}(s)\big|dm(x))\leq\int_{t}^{T}(\mathbb{E}\int_{\mathbb{R}^{n}}\big|\mathscr{J}^*_k(s)\big|dm(x))\rightarrow0
\]
From assumption (\ref{eq:5-2007}) and (\ref{eq:8-159}) we conclude
that $\int_{t}^{T}(\mathbb{E}\int_{\mathbb{R}^{n}}\big|\,\mathcal{U}^*_{k}(s)\big|dm(x))\rightarrow0.$
It is then easy to obtain $ \mathcal{Z}^*_k\longrightarrow0$
in $L^{1}(\Omega,\mathcal{A},P;L_{m}^{1}(\mathbb{R}^{n}))$ which concludes the
proof.} 

\hfill $\blacksquare$

\subsection{Proof of Statements in Section 5}

\subsubsection{Proof of Theorem \ref{prop bellman}}\label{app, prop bellman}
\noindent {\bf Part 1. Proof of (a):}\\
From the dynamical programming principle or optimality principle, we can write 
\[
0=\dfrac{1}{\epsilon}\int_{s}^{s+\epsilon}\int_{\mathbb{R}^{n}}
\mathbb{E}\pig[l(\mathbb{Y}_{tX}(\tau),u_{tX}(\tau))\pig]dm(x)
+F(\mathbb{Y}_{tX}(\tau) \otimes  m)d\tau+\dfrac{1}{\epsilon}\pig[V(\mathbb{Y}_{tX}(s+\epsilon) \otimes  m,s+\epsilon)-V(\mathbb{Y}_{tX}(s) \otimes  m,s)\pig].
\]
From the continuity of functions $s\longmapsto \mathbb{Y}_{tX}(s),u_{tX}(s)$, we have 

\begin{equation}
\mathbb{E}\left[
\int_{\mathbb{R}^{n}}l\pig(\mathbb{Y}_{tX}(s),u_{tX}(s)\pig)dm(x)\right]
+F(\mathbb{Y}_{tX}(s) \otimes  m)
+\dfrac{d}{ds}V(\mathbb{Y}_{tX}(s) \otimes  m,s)=0.
\label{eq:9-1}
\end{equation}
The assumptions in  (\ref{subspace of H_m indep. of W_t}) to (\ref{eq:3-112}) are fulfilled by the regularity properties of $V$ mentioned in (\ref{bdd V}), (\ref{lip cts of D_XV in X}), (\ref{cts of D_X V in t}), (\ref{eq:4-26}), (\ref{DX^2 V < psi}) and (\ref{ct of D^2_X V}) by Propositions \ref{prop bdd of V}-\ref{prop cts of DXX V}. Thus, it is legitimate to use Theorem \ref{ito thm} to replace the total derivative in time by the partial one, that is, for a.e. $s \in (t,T)$, 
\begin{align}
0=\:&\mathbb{E}\left[
\int_{\mathbb{R}^{n}}l\pig(\mathbb{Y}_{tX}(s),u_{tX}(s)\pig)dm(x)\right]
+F(\mathbb{Y}_{tX}(s) \otimes  m)+\dfrac{\partial}{\partial s}V(\mathbb{Y}_{tX}(s) \otimes  m,s)
+\pig\langle D_{X}V(\mathbb{Y}_{tX}(s) \otimes  m,s),u_{tX}(s)\pigr\rangle_{\mathcal{H}_m}\nonumber\\
&+\dfrac{1}{2}\sum_{j=1}^{n}
\Big\langle
D_{X}^{2}V(\mathbb{Y}_{tX}(s) \otimes  m,s)\pig(\eta^{j}\mathcal{N}^{j}_s\pig),\eta^{j}\mathcal{N}^{j}_s\Big\rangle_{\mathcal{H}_m}.\label{eq:9-3}
\end{align}
Recall from (\ref{D_X V(s)=Z(s)}) and the definition of Hamiltonian in (\ref{def. Hamiltonian}), the first and forth term in the first line of (\ref{eq:9-3}) are combined to give the Hamiltonian term, we thus obtain the equation in (\ref{eq. bellman eq.}).

\mycomment{
\begin{equation}
\begin{aligned}
\dfrac{\partial}{\partial s}V(\mathbb{Y}_{tX}(s) \otimes  m,s)
+\mathbb{E}\left[\int_{\mathbb{R}^{n}}H(\mathbb{Y}_{tX}(s),\mathbb{Z}_{tX}(s))dm(x)\right]
+F(\mathbb{Y}_{tX}(s) \otimes  m)&\\
+\dfrac{1}{2}\sum_{j=1}^{n}
\Big\langle
D_{X}^{2}V(\mathbb{Y}_{tX}(s) \otimes  m,s)\pig(\eta^{j}\mathcal{N}^{j}_s\pig),\eta^{j}\mathcal{N}^{j}_s\Big\rangle_{\mathcal{H}_m}&=0,
\end{aligned}
\label{eq:9-4}
\end{equation}
and 
\begin{equation}
V(\mathbb{Y}_{tX}(T) \otimes  m,T)=\mathbb{E}\left[\int_{\mathbb{R}^{n}}h(\mathbb{Y}_{tX}(T)dm(x)\right]
+F_{T}(\mathbb{Y}_{tX}(T) \otimes  m).
\label{eq:9-5}
\end{equation}}

\noindent {\bf Part 2. Proof of (b):}\\
Now if a functional $V^*(X \otimes  m,t)$ solves (\ref{eq. bellman eq.}) and satisfies the regularity properties in Propositions \ref{prop bdd of V}-\ref{prop cts of DXX V}, then $V^*(\mathbb{Y}_{tX}(s) \otimes  m,s)$ also satisfies  (\ref{eq:9-1}) by working backwards through the first use of mean-field It\^o Lemma in Theorem \ref{ito thm} and then the definition of Hamiltonian, together with equation (\ref{eq. bellman eq.}). By integrating (\ref{eq:9-1}) with respect to $s$ from $t$ to $T$, we obtain
\begin{equation}
V^*(X \otimes  m,t)=J_{tX}(u_{tX}).
\label{eq:9-6}
\end{equation}
Since the right hand side is the value function which is uniquely defined, the solution to (\ref{eq. bellman eq.}) is necessarily unique. Note that $X$ must be independent
of $\mathcal{W}_{t}$ by its definition; indeed, we can always construct the
Gaussian random variables $\mathcal{N}^j_t$ (for example, the construction in the proof of Theorem \ref{ito thm}; also see Remark 5.4 in \cite{BGY20}), so that this condition is satisfied. This concludes
the proof. \hfill $\blacksquare$

\addcontentsline{toc}{section}{References}
\bibliographystyle{abbrv} 
\bibliography{meanfieldtypecontrol}
\nocite{*} 

\end{document}